\documentclass[reqno]{amsart}
\usepackage[english]{babel}
\usepackage{fourier} 
\usepackage{graphicx}
\usepackage{amscd}
\usepackage{amsmath}
\usepackage{amsfonts}
\usepackage{amssymb}
\usepackage{mhequ}
\usepackage{bbm}
\usepackage{setspace}
\usepackage{enumerate}         
\usepackage{fixme}
\usepackage{color}
\usepackage{url}
\usepackage{amsthm}
\usepackage{microtype}

\usepackage{orcidlink}

\usepackage{amsaddr}

\usepackage{bm}
\usepackage{xy}
\usepackage{comment}
\usepackage{stmaryrd}
\usepackage{tikz-cd}
\usepackage{cprotect}
\usepackage{hyperref}

\usetikzlibrary{external}

\AtBeginEnvironment{tikzcd}{\tikzexternaldisable}
\AtEndEnvironment{tikzcd}{\tikzexternalenable}

\def\dash{\leavevmode\unskip\kern0.18em--\penalty\exhyphenpenalty\kern0.18em}
\def\slash{\leavevmode\unskip\kern0.15em/\penalty\exhyphenpenalty\kern0.15em}

\colorlet{darkblue}{blue!90!black}
\colorlet{darkred}{red!90!black}

\def\enlarge{edge[draw=none] (0,-0.1) edge[draw=none] (0,0.1)}

\theoremstyle{plain}
\newtheorem{theorem}{Theorem}[section]

\newtheorem{claim}[theorem]{Claim}

\newtheorem{corollary}[theorem]{Corollary}

\newtheorem{lemma}[theorem]{Lemma}

\newtheorem{prop}[theorem]{Proposition}

\newtheorem{definition}[theorem]{Definition}

\newtheorem{assumption}[theorem]{Assumption}
\theoremstyle{remark}
\newtheorem{remark}[theorem]{Remark}

\def\Vec{{\mathrm{Vec}}}
\def\VecB{{\mathrm{VecB}}}

\def\Iso{{\mathrm{Iso}}}
\def\symset{\mcb{s}}
\def\bbar#1{\bar{\bar #1}}

\numberwithin{equation}{section}

\newcommand{\rsto}{]\!\kern-1.8pt ]}
\newcommand{\lsto}{[\!\kern-1.7pt [}

\numberwithin{equation}{section}
\def\ob{\mathord{\mathrm{Ob}}}

\newcommand{\RR}{\mathbb{R}}

\newcommand{\PP}{\mathbb{P}}

\newcommand{\NN}{\mathbb{N}}

\newcommand{\cC}{\mathcal{C}}

\renewcommand{\Im}{\mathrm{Im}}

\newcommand{\id}{\operatorname{id}}
 
\newcommand{\diag}{\delta}

\let\eps\varepsilon

\newcommand{\pT}{\bar{T}}
\newcommand{\pTset}{{\mathtt{T}}}

\newcommand{\vertiii}[1]{{\left\vert\kern-0.25ex\left\vert\kern-0.25ex\left\vert #1 
		\right\vert\kern-0.25ex\right\vert\kern-0.25ex\right\vert}}

\DeclareMathOperator{\Tr}{Tr}
\DeclareMathOperator{\trr}{\mathfrak{tr}}

\DeclareMathOperator{\tr}{tr}
\DeclareMathOperator{\sign}{sign}

\newcommand{\supp}{\ensuremath{\mathsf{supp}}}

\newcommand{\fraks}{\mathfrak{s}}
\newcommand{\frakl}{\mathfrak{l}}
\newcommand{\frakt}{\mathfrak{t}}
\newcommand{\frakf}{\mathfrak{f}}
\newcommand{\frakk}{\mathfrak{k}}
\newcommand{\dep}{\operatorname{dep}}
\newcommand{\Con}{\operatorname{Con}}
\newcommand{\Per}{\operatorname{Per}}
\newcommand{\E}{\operatorname{E}}

\newcommand{\D}{\operatorname{\mathfrak{A}}}
\newcommand{\red}{\operatorname{red}}
\newcommand{\ind}{\operatorname{ind}}
\newcommand{\dif}{\operatorname{dif}}

\newcommand{\re}{\operatorname{red}_*}
\newcommand{\mft}{\mathfrak{t}}
\newcommand{\Poly}{\mathcal{P}}
\newcommand{\dVol}{d\operatorname{Vol}}
\newcommand{\Vol}{\operatorname{Vol}}

\def\proj{\mathbf{p}}
\def\PP{\mathbf{P}}
\def\reg{{\mathop{\mathrm{reg}}}}

\def\lll{{\mathop{\prec}}}
\def\rrr{{\mathop{\succ}}}

\def\loc{\mathrm{loc}}

\let\tto\rightsquigarrow

\newcommand{\mfL}{\mathfrak{L}}

\newcommand{\mfl}{\mathfrak{l}}

\def\hotimes{\mathbin{\hat\otimes}}

\def\Lab{\mathfrak{S}}

\def\Hom{\mathord{\mathrm{Hom}}}

\def\Iso{\mathord{\mathrm{Iso}}}
\def\TStruc{\mathord{\mathrm{TStruc}}}
\def\SSet{\mathord{\mathrm{SSet}}}

\def\Func{\mathbf{F}}

\def\symset{\mcb{s}}

\DeclareMathAlphabet{\mathbbm}{U}{bbm}{m}{n}

\DeclareFontFamily{U}{BOONDOX-calo}{\skewchar\font=45 }
\DeclareFontShape{U}{BOONDOX-calo}{m}{n}{
  <-> s*[1.05] BOONDOX-r-calo}{}
\DeclareFontShape{U}{BOONDOX-calo}{b}{n}{
  <-> s*[1.05] BOONDOX-b-calo}{}
\DeclareMathAlphabet{\mcb}{U}{BOONDOX-calo}{m}{n}
\SetMathAlphabet{\mcb}{bold}{U}{BOONDOX-calo}{b}{n}

\colorlet{symbols}{blue!90!black}

\colorlet{testcolor}{green!60!black}

\usetikzlibrary{calc}
\usetikzlibrary{shapes.misc}
\usetikzlibrary{shapes.symbols}
\usetikzlibrary{shapes.geometric}
\usetikzlibrary{snakes}
\usetikzlibrary{decorations}
\usetikzlibrary{decorations.markings}

\tikzset{
	root/.style={circle,fill=testcolor,inner sep=0pt, minimum size=2mm},
	broot/.style={circle,fill=gray,inner sep=0pt, minimum size=2mm},
	dot/.style={circle,fill=black,inner sep=0pt, minimum size=1mm},
		reddot/.style={circle,fill=red,inner sep=0pt, minimum size=1mm},
			bluedot/.style={circle,fill=blue,inner sep=0pt, minimum size=1mm},
	eps/.style={circle,fill=white,draw=symbols,inner sep=0pt,minimum size=0.8mm},
	int/.style={circle,fill=black,draw=black,inner sep=0pt,minimum size=0.7mm},
	var/.style={circle,fill=black!10,draw=black,inner sep=0pt, minimum size=2mm},
	dotred/.style={circle,fill=black!50,inner sep=0pt, minimum size=2mm},
	generic/.style={semithick,shorten >=1pt,shorten <=1pt},
	dist/.style={ultra thick,draw=testcolor,shorten >=1pt,shorten <=1pt},
	testfcn/.style={ultra thick,testcolor,shorten >=1pt,shorten <=1pt,<-},
	testfcnx/.style={ultra thick,testcolor,shorten >=1pt,shorten <=1pt,<-,
		postaction={decorate,decoration={markings,mark=at position 0.6 with {\drawx}}}},
	keps/.style={semithick,shorten >=1pt,shorten <=1pt,densely dashed,->},
	kprimex/.style={semithick,shorten >=1pt,shorten <=1pt,densely dashed,->,
		postaction={decorate,decoration={markings,mark=at position 0.4 with {\drawx}}}},
	kernel/.style={semithick,shorten >=1pt,shorten <=1pt,->},
	multx/.style={shorten >=1pt,shorten <=1pt,
		postaction={decorate,decoration={markings,mark=at position 0.5 with {\drawx}}}},
	kernelx/.style={semithick,shorten >=1pt,shorten <=1pt,->,
		postaction={decorate,decoration={markings,mark=at position 0.4 with {\drawx}}}},
	kernel1/.style={->,semithick,shorten >=1pt,shorten <=1pt,postaction={decorate,decoration={markings,mark=at position 0.45 with {\draw[-] (0,-0.1) -- (0,0.1);}}}},
	kernel2/.style={->,semithick,shorten >=1pt,shorten <=1pt,postaction={decorate,decoration={markings,mark=at position 0.45 with {\draw[-] (0.05,-0.1) -- (0.05,0.1);\draw[-] (-0.05,-0.1) -- (-0.05,0.1);}}}},
	kernelBig/.style={semithick,shorten >=1pt,shorten <=1pt,decorate, decoration={zigzag,amplitude=1.5pt,segment length = 3pt,pre length=2pt,post length=2pt}},
	rho/.style={dotted,semithick,shorten >=1pt,shorten <=1pt},
	renorm/.style={shape=circle,fill=white,inner sep=1pt},
	labl/.style={shape=rectangle,fill=white,inner sep=1pt},
	xi/.style={circle,fill=symbols!10,draw=symbols,inner sep=0pt,minimum size=1.2mm},
	xix/.style={crosscircle,fill=symbols!10,draw=symbols,inner sep=0pt,minimum size=1.2mm},
	xib/.style={circle,fill=symbols!10,draw=symbols,inner sep=0pt,minimum size=1.6mm},
	xibx/.style={crosscircle,fill=symbols!10,draw=symbols,inner sep=0pt,minimum size=1.6mm},
	not/.style={circle,fill=symbols,draw=symbols,inner sep=0pt,minimum size=0.5mm},
cumu2n/.style={inner sep=3pt},
cumu2/.style={draw=red!80,fill=red!40},
cumu2b/.style={draw=blue!80,fill=blue!40},
cumu2nv/.style={inner sep=3pt},
cumu2v/.style={draw=red!80,fill=white,very thick},
cumu3/.style={regular polygon, regular polygon sides=3,draw=red!80,rounded corners=3pt,fill=red!40,minimum size=5mm},
cumu4/.style={regular polygon, regular polygon sides=4,draw=red!80,rounded corners=3pt,fill=red!40,minimum size=7mm},
cumu5/.style={regular polygon, regular polygon sides=5,draw=red!80,rounded corners=3pt,fill=red!40,minimum size=7mm},
	>=stealth,
	not/.style={circle,fill=symbols,draw=symbols,inner sep=0pt,minimum size=0.5mm},
kernels2/.style={very thick,segment length=12pt},
	}
\makeatletter
\def\DeclareSymbol#1#2#3{%
	\expandafter\gdef\csname MH@symb@#1\endcsname{\tikzsetnextfilename{symbol#1}%
		\tikz[baseline=#2,scale=0.15,draw=symbols,line join=round]{#3}}%
	\expandafter\gdef\csname MH@symb@#1s\endcsname{\scalebox{0.75}{\tikzsetnextfilename{symbol#1}%
			\tikz[baseline=#2,scale=0.15,draw=symbols,line join=round]{#3}}}%
	\expandafter\gdef\csname MH@symb@#1ss\endcsname{\scalebox{0.65}{\tikzsetnextfilename{symbol#1}%
			\tikz[baseline=#2,scale=0.15,draw=symbols,line join=round]{#3}}}%
}
\def\<#1>{\ifthenelse{\boolean{mmode}}{\mathchoice{\csname MH@symb@#1\endcsname}{\csname MH@symb@#1\endcsname}{\csname MH@symb@#1s\endcsname}{\csname MH@symb@#1ss\endcsname}}{\csname MH@symb@#1\endcsname}}
\makeatother

\DeclareSymbol{X}{-2.4}{\node[eps] {};}
\DeclareSymbol{I}{0}{\draw[white] (-.5,0) -- (.5,0); \draw (0,0)  -- (0,1.5) {};}
\DeclareSymbol{1}{0}{\draw[white] (-.5,0) -- (.5,0); \draw (0,0)  -- (0,1.5) node[eps] {};}
\DeclareSymbol{2}{0}{\draw (-0.5,1.5) node[eps] {} -- (0,0) -- (0.5,1.5) node[eps] {};}
\DeclareSymbol{3}{0}{\draw (0,0) -- (0,1.5) node[eps] {}; \draw (-.7,1.3) node[eps] {} -- (0,0) -- (.7,1.3) node[eps] {};}
\DeclareSymbol{E3}{0}{\draw (0,0) -- (0,1.5) node[not] {}; \draw (-.7,1.3) node[not] {} -- (0,0) -- (.7,1.3) node[not] {}; \node[eps] (0,0) {};}

\DeclareSymbol{E5E4}{1}{
\draw (.5,2) node[not] {} -- (1.5,1.5) -- (2.5,2) node[not] {};
\draw (1.5,1.5) -- (1.5,2.7) node[not] {}; \draw (0.8,2.5) node[not] {} -- (1.5,1.5) -- (2.2,2.5) node[not] {};
\draw (0,0) -- (0,1.5) node[not] {}; \draw (-.7,1.3) node[not] {} -- (0,0) -- (.7,1.3) node[not] {};
\draw (-1.2,0.7) node[not] {} -- (0,0);
\draw (0,0) to[bend right=50] (1.5,1.5); 
\node[eps] at (0,0) {};\node[eps] at (1.5,1.5) {};
}
\DeclareSymbol{E4E4}{1}{
\draw (.5,2) node[not] {} -- (1.5,1.5) -- (2.5,2) node[not] {};
\draw (1,2.6) node[not] {} -- (1.5,1.5) -- (2,2.6) node[not] {};
\draw (0,0) -- (0,1.5) node[not] {}; \draw (-.7,1.3) node[not] {} -- (0,0) -- (.7,1.3) node[not] {};
\draw (-1.2,0.7) node[not] {} -- (0,0);
\draw (0,0) to[bend right=50] (1.5,1.5); 
\node[eps] at (0,0) {};\node[eps] at (1.5,1.5) {};
}
\DeclareSymbol{3E4}{1}{
\draw (1.5,1.5) -- (1.5,2.7) node[not] {}; \draw (0.8,2.5) node[not] {} -- (1.5,1.5) -- (2.2,2.5) node[not] {};
\draw (0,0) -- (0,1.5) node[not] {}; \draw (-.7,1.3) node[not] {} -- (0,0) -- (.7,1.3) node[not] {};
\draw (-1.2,0.7) node[not] {} -- (0,0);
\draw (0,0) to[bend right=50] (1.5,1.5); 
\node[eps] at (0,0) {};
}

\DeclareSymbol{3E3}{1}{
\draw (1.5,1.5) -- (1.5,2.7) node[not] {}; \draw (0.8,2.5) node[not] {} -- (1.5,1.5) -- (2.2,2.5) node[not] {};
\draw (0,0) -- (0,1.5) node[not] {}; \draw (-.7,1.3) node[not] {} -- (0,0) -- (.7,1.3) node[not] {};
\draw (0,0) to[bend right=50] (1.5,1.5); 
\node[eps] at (0,0) {};
}

\DeclareSymbol{E5}{0}{\draw (0,0) -- (0,1.5) node[not] {}; \draw (-.7,1.3) node[not] {} -- (0,0) -- (.7,1.3) node[not] {};
\draw (-1.2,0.7) node[not] {} -- (0,0) -- (1.2,0.7) node[not] {}; \node[eps] (0,0) {};}
\DeclareSymbol{E52}{-3}{\draw (0,0) -- (0,-1); \draw (1.5,-0.4) node[not] {} -- (0,-1) -- (-1.5,-0.4) node[not] {}; 
\draw (-1,0.5) node[not] {} -- (0,0) -- (1,0.5) node[not] {};
\draw (0,0) -- (0,1.2) node[not] {}; \draw (-.7,1) node[not] {} -- (0,0) -- (.7,1) node[not] {};\node[eps] (0,0) {};}
\DeclareSymbol{E32}{-3}{\draw (0,0) -- (0,-1); \draw (1.5,-0.4) node[not] {} -- (0,-1) -- (-1.5,-0.4) node[not] {}; 
\draw (0,0) -- (0,1.2) node[not] {}; \draw (-.7,1) node[not] {} -- (0,0) -- (.7,1) node[not] {};\node[eps] (0,0) {};}
\DeclareSymbol{E50}{-3}{\draw (0,0) -- (0,-1); 
\draw (-1,0.5) node[not] {} -- (0,0) -- (1,0.5) node[not] {};
\draw (0,0) -- (0,1.2) node[not] {}; \draw (-.7,1) node[not] {} -- (0,0) -- (.7,1) node[not] {};\node[eps] (0,0) {};}

\DeclareSymbol{3E2}{-3}{\draw (0,0) -- (0,-1); \draw (1,0) node[not] {} -- (0,-1) -- (-1,0) node[not] {}; \draw (0,0) -- (0,1.2) node[not] {}; \draw (-.7,1) node[not] {} -- (0,0) -- (.7,1) node[not] {};\node[eps] at (0,-1) {};}

\DeclareSymbol{31}{-2}{\draw (0,0) -- (0,-1) -- (1,0) node[eps] {}; \draw (0,0) -- (0,1.2) node[eps] {}; \draw (-.7,1) node[eps] {} -- (0,0) -- (.7,1) node[eps] {};}
\DeclareSymbol{30}{-2}{\draw (0,0) -- (0,-1); \draw (0,0) -- (0,1.2) node[eps] {}; \draw (-.7,1) node[eps] {} -- (0,0) -- (.7,1) node[eps] {};}
\DeclareSymbol{32}{-2}{\draw (0,0) -- (0,-1); \draw (1,0) node[eps] {} -- (0,-1) -- (-1,0) node[eps] {}; \draw (0,0) -- (0,1.2) node[eps] {}; \draw (-.7,1) node[eps] {} -- (0,0) -- (.7,1) node[eps] {};}
\DeclareSymbol{32b}{-4.3}{\draw (0,0.5) -- (0,-1); \draw (1,0) node[not] {} -- (0,-1) -- (-1,0) node[not] {};}
\DeclareSymbol{22}{-2}{\draw (0,0.3)  -- (0,-1); \draw (1,0) node[eps] {} -- (0,-1) -- (-1,0) node[eps] {};\draw (-.7,1) node[eps] {} -- (0,0.3) -- (.7,1) node[eps] {};}
\DeclareSymbol{20}{-2}{\draw (0,0)  -- (0,-1);\draw (-.7,1) node[not] {} -- (0,0) -- (.7,1) node[not] {};}
\DeclareSymbol{12}{-3}{\draw (0,0.3) -- (0,-1); \draw (1,0) node[not] {} -- (0,-1) -- (-1,0) node[not] {};\draw (-.7,1) node[not] {} -- (0,0.3);}
\DeclareSymbol{10}{-3}{\draw (0,0.3) -- (0,-1);\draw (-.7,1) node[not] {} -- (0,0.3);}
\DeclareSymbol{21}{-3}{\draw (0,0.3) -- (0,-1) -- (1,0) node[not] {};\draw (-.7,1) node[not] {} -- (0,0.3) -- (.7,1) node[not] {};}

\DeclareSymbol{1'}{0}{\draw[white] (-.5,0) -- (.5,0); \draw (0,0)  -- (0,1.5) node[not] {}; \draw ( -0.2,1)-- ( 0.2, 1);}

\DeclareSymbol{3'}{0}{\draw (0,0) -- (0,1.5) node[not] {};
\draw  (-0.25,1.1) -- ( 0.25, 1.1);
 \draw (-.7,1.3) node[eps] {} -- (0,0) -- (.7,1.3) node[eps] {};}

\DeclareSymbol{3''}{0}{\draw (0,0) -- (0,1.5) node[eps] {};
 \draw (-.7,1.3) node[not] {} -- (0,0) -- (.7,1.3) node[not] {};
 \draw  (-0.8,0.7) -- ( -0.25, 1);
  \draw  (0.8,0.7) -- ( 0.25, 1);
 }
 
 \DeclareSymbol{3'''}{0}{\draw (0,0) -- (0,1.5) node[not] {};
\draw  (-0.3,1.05) -- ( 0.3, 1.05);
 \draw (-.7,1.3) node[not] {} -- (0,0) -- (.7,1.3) node[not] {};
 \draw  (-0.8,0.7) -- ( -0.25, 1);
  \draw  (0.8,0.7) -- ( 0.25, 1);
 }
 
 \DeclareSymbol{3'2}{-2}{\draw (0,0)  -- (0,-1); \draw (1,0) node[eps] {} -- (0,-1) -- (-1,0) node[eps] {}; \draw (0,0) -- (0,1.4) node[not] {}; \draw (-.7,1) node[eps] {} -- (0,0) -- (.7,1) node[eps] {};
\draw  (-0.25,1) -- ( 0.25, 1); 
 }
 
 \DeclareSymbol{3''2}{-2}{\draw (0,0)  -- (0,-1); \draw (1,0) node[eps] {} -- (0,-1) -- (-1,0) node[eps] {}; \draw (0,0) -- (0,1.3) node[eps] {}; \draw (-.7,1) node[not] {} -- (0,0) -- (.7,1) node[not] {};
 \draw  (-0.8,0.5) -- ( -0.25, 0.8);
  \draw  (0.8,0.5) -- ( 0.25, 0.8); 
 }
 
  \DeclareSymbol{3'''2}{-2}{\draw (0,0)  -- (0,-1); \draw (1,0) node[eps] {} -- (0,-1) -- (-1,0) node[eps] {}; \draw (0,0) -- (0,1.4) node[not] {}; \draw (-.7,1) node[not] {} -- (0,0) -- (.7,1) node[not] {};
 \draw  (-0.8,0.5) -- ( -0.25, 0.8);
  \draw  (0.8,0.5) -- ( 0.25, 0.8); 
  \draw ( -0.25,1)-- ( 0.25, 1);
 }
 
  \DeclareSymbol{3'2'}{-2}{\draw (0,0)  -- (0,-1); \draw (1,0) node[eps] {} -- (0,-1) -- (-1,0) node[not] {}; \draw (0,0) -- (0,1.4) node[not] {}; \draw (-.7,1) node[eps] {} -- (0,0) -- (.7,1) node[eps] {};
\draw  (-0.25,1) -- ( 0.25, 1); 
\draw  (-0.8,-0.6) -- ( -0.4, -0.2);
 }
 
  \DeclareSymbol{3'2''}{-2}{\draw (0,0)  -- (0,-1); \draw (1,0) node[not] {} -- (0,-1) -- (-1,0) node[not] {}; \draw (0,0) -- (0,1.4) node[not] {}; \draw (-.7,1) node[eps] {} -- (0,0) -- (.7,1) node[eps] {};
\draw  (-0.25,1) -- ( 0.25, 1); 
\draw  (-0.8,-0.6) -- ( -0.4, -0.2);
\draw  (0.8,-0.6) -- ( 0.4, -0.2);
 }

   \DeclareSymbol{3''2'}{-2}{\draw (0,0)  -- (0,-1); \draw (1,0) node[eps] {} -- (0,-1) -- (-1,0) node[not] {}; \draw (0,0) -- (0,1.3) node[eps] {}; \draw (-.7,1) node[not] {} -- (0,0) -- (.7,1) node[not] {};
 \draw  (-0.8,0.5) -- ( -0.25, 0.8);
  \draw  (0.8,0.5) -- ( 0.25, 0.8); 
\draw  (-0.8,-0.6) -- ( -0.4, -0.2);
 }
  \DeclareSymbol{32'}{-2}{\draw (0,0) -- (0,-1); \draw (1,0) node[eps] {} -- (0,-1) -- (-1,0) node[not] {}; \draw (0,0) -- (0,1.3) node[eps] {}; \draw (-.7,1) node[eps] {} -- (0,0) -- (.7,1) node[eps] {};
\draw  (-0.8,-0.6) -- ( -0.4, -0.2);
 }
  \DeclareSymbol{3'0}{-2}{\draw (0,0)  -- (0,-1); 
   \draw (0,0) -- (0,1.4) node[not] {}; \draw (-.7,1) node[eps] {} -- (0,0) -- (.7,1) node[eps] {};
\draw  (-0.25,1) -- ( 0.25, 1); 
 }

\DeclareSymbol{1}{0}{\draw[white] (-.4,0) -- (.4,0); \draw (0,0)  -- (0,1.2) node[eps] {};}
\DeclareSymbol{2}{0}{\draw (-0.5,1.2) node[eps] {} -- (0,0) -- (0.5,1.2) node[eps] {};}
\DeclareSymbol{11}{0}{\draw (0,1.8) node[eps] {} -- (-0.7,0.9) -- (0,0) -- (0.7,1) node[eps] {};}
\DeclareSymbol{10}{0}{\draw (0,1.8) node[eps] {} -- (-0.8,0.9) -- (0,0);}
\DeclareSymbol{21}{0.7}{\draw (-1,1.8) node[eps] {} -- (-0.5,0.9); \draw (0,1.8) node[eps] {} -- (-0.5,0.9) -- (0,0) -- (0.5,0.9) node[eps] {};}
\DeclareSymbol{20}{0.7}{\draw (-1,1.8) node[eps] {} -- (-0.5,0.9); \draw (0,1.8) node[eps] {} -- (-0.5,0.9) -- (-0.5,0);}
\DeclareSymbol{210}{1.1}{\draw (-0.5,2.4) node[eps] {} -- (-1,1.6);\draw (-1.5,2.4) node[eps] {} -- (-0.5,0.8); \draw (0,1.6) node[eps] {} -- (-0.5,0.8) -- (-0.5,0);}
\DeclareSymbol{211}{1}{\draw  (-0.5,2.4) node[eps] {} -- (-1,1.6);\draw (-1.5,2.4) node[eps] {} -- (-0.5,0.8); \draw (0,1.6) node[eps] {} -- (-0.5,0.8) -- (0,0) -- (0.5,0.8) node[eps] {};}
\DeclareSymbol{22j}{0.7}{\draw (-1.5,1.8) node[eps] {} -- (-1,0.9) -- (0,0) -- (1,0.9) -- (1.5,1.8) node[eps] {};
\draw (-0.5,1.8) node[eps] {} -- (-1,0.9);\draw (0.5,1.8) node[eps] {} -- (1,0.9);}

\DeclareSymbol{2'}{0}{\draw (-0.65,1.3) node[not] {} -- (0,0) -- (0.6,1.3) node[eps] {};
\draw  (-0.8,0.7) -- ( -0.15, 1.1);} 

\DeclareSymbol{2''}{0}{\draw (-0.65,1.3) node[not] {} -- (0,0) -- (0.6,1.3) node[not] {};
\draw  (-0.8,0.7) -- ( -0.15, 1.1);
\draw  (0.8,0.7) -- ( 0.15, 1.1);
} 
 
\DeclareSymbol{2'1}{0.7}{\draw (-1,1.8) node[not] {} -- (-0.5,0.9); \draw (0,1.8) node[eps] {} -- (-0.5,0.9) -- (0,0) -- (0.5,0.9) node[eps] {};
\draw  (-1.1,1.2) -- ( -0.5, 1.6);} 

\DeclareSymbol{2''1}{0.7}{\draw (-1,1.8) node[not] {} -- (-0.5,0.9); \draw (0,1.8) node[not] {} -- (-0.5,0.9) -- (0,0) -- (0.5,0.9) node[eps] {};
\draw  (-1.1,1.2) -- ( -0.5, 1.6);
\draw  (.15,1.2) -- ( -0.45, 1.6);} 
 
\DeclareSymbol{2''1'}{0.7}{\draw (-1,1.8) node[not] {} -- (-0.5,0.9); \draw (0,1.8) node[not] {} -- (-0.5,0.9) -- (0,0) -- (0.55,1) node[not] {};
\draw  (-1.1,1.2) -- ( -0.5, 1.6);
\draw  (.15,1.2) -- ( -0.45, 1.6);
\draw  (0.8,0.4) -- ( 0, 0.8);
} 
\DeclareSymbol{2'1'}{0.7}{\draw (-1,1.8) node[not] {} -- (-0.5,0.9); \draw (0,1.8) node[eps] {} -- (-0.5,0.9) -- (0,0) -- (0.55,1) node[not] {};
\draw  (-1.1,1.2) -- ( -0.5, 1.6);
\draw  (0.8,0.4) -- ( 0, 0.8);} 

\DeclareSymbol{21'}{0.7}{\draw (-1,1.8) node[eps] {} -- (-0.5,0.9); \draw (0,1.8) node[eps] {} -- (-0.5,0.9) -- (0,0) -- (0.55,1) node[not] {};
\draw  (0.8,0.4) -- ( 0, 0.8);} 

\DeclareSymbol{2'0}{0.7}{\draw (-1,1.8) node[not] {} -- (-0.5,0.9); \draw (0,1.8) node[eps] {} -- (-0.5,0.9) -- (-0.5,0);
\draw  (-1.1,1.2) -- ( -0.5, 1.6);}

\DeclareSymbol{b2}{0}{
\draw[kernels2] (0,0) -- (-0.5,1.4);
\draw[kernels2] (0,0) -- (+0.5,1.4); 
 \draw  (-0.5,1.4) node[eps] {};
 \draw  (0.5,1.4) node[eps] {};
 }
 
\DeclareSymbol{bI}{0}{\draw[white] (-.5,0) -- (.5,0); \draw[kernels2] (0,0)  -- (0,1.5) {};} 

\DeclareSymbol{Xi2}{0}{\draw[white] (-.5,0) -- (.5,0); \draw (0,0)  -- (0,1.5) node[eps] {};
\draw (0,0) node[eps] {};}

 \DeclareSymbol{1'Xi}{0}{\draw[white] (-.5,0) -- (.5,0); \draw (0,0)  -- (0,1.5) node[not] {}; \draw ( -0.2,1)-- ( 0.2, 1);
 \draw (0,0) node[eps] {};
 }

\def\CCE{\mathbb{E}}
\def\CCV{\mathbb{V}}
\def\CCG{\mathbb{G}}
\def\CCT{\mathbb{T}}

\def\neigh{N}

\begin{document}
\title{Regularity Structures on Manifolds and Vector Bundles}

\author{Martin Hairer\orcidlink{0000-0002-2141-6561}}
\address{EPFL, Switzerland and Imperial College London, UK}
\email{martin.hairer@epfl.ch}

\author{Harprit Singh\orcidlink{0000-0002-9991-8393}}
\address{Imperial College London, UK}
\email{h.singh19@imperial.ac.uk}


\maketitle


\begin{abstract}
We develop a generalisation of the original theory of regularity structures, \cite{Hai14}, which is able to treat SPDEs on manifolds with values in vector bundles. Assume $M$ is a Riemannian manifold and $E\to M$ and $F^i\to M$ are vector bundles (with a metric and connection), this theory allows to solve subcritical equations of the form 
\begin{equation*}
\partial_t u + \mathcal{L}u = \sum_{i=0}^m G_i(u, \nabla u,\ldots, \nabla^n u)\xi_i\ ,
\end{equation*}
where $u$ is a (generalised) section of $E$, $\mathcal{L}$ is a uniformly elliptic operator on $E$ of order strictly greater than $n$, the $\xi_i$ are $F^i$-valued random distributions (e.g.\ $F^i$-valued white noises), and the $G_i:E\times TM^*\otimes E \times\ldots\times (TM^*)^{\otimes n} \otimes E \to L(F^i ,E)$ are local functions.

We apply our framework to three example equations which illustrate
that when $\mathcal{L}$ is a Laplacian it is possible in most cases to renormalise such equations by adding spatially homogeneous 
counterterms and we discuss in which cases more sophisticated renormalisation procedures (involving
the curvature of the underlying manifold) are required.
\end{abstract}
\thispagestyle{empty}

\setcounter{tocdepth}{1}

\tableofcontents

\section{Introduction}

The purpose of this work is to extend the theory of regularity structures as developed in \cite{Hai14} for solving singular partial differential equations on $\mathbb{R}^d$ or $\mathbb{T}^d$ to a theory able to treat equations on vector bundles over a Riemannian manifold. 
That is, assume $M$ is a Riemannian manifold and $E\to M$ and $F^i\to M$ are vector bundles (with a metric and connection), we develop a theory for solving subcritical equations of the form 
\begin{equation}\label{equation to solve}
\partial_t u + \mathcal{L}u = \sum_{i=0}^m G_i(u, \nabla u,\ldots, \nabla^n u)\xi_i\ ,
\end{equation}
where $u$ is a (generalized) section of the vector bundle $E \to M$ and $\mathcal{L}$ is a uniformly elliptic operator on $E$ of order grater than $n$, $\xi_i$ are $F^i$ valued random distributions (e.g. $F^i$-valued white noise) and $G_i:E\times TM^*\otimes E \times\ldots\times (TM^*)^{\otimes n} \otimes E \to L(F^i ,E)$ are local functions.

As in \cite{Hai14} and subsequent works \cite{BHZ19}, \cite{CH16}, \cite{BCCH20} the treatment of such equations can naturally be divided into the following steps:
\begin{enumerate}
\item Introducing regularity structures and models as general tool and developing an analytic framework allowing to reformulate and solve such equations as abstract equations ``driven'' by a model. 
\item Building a regularity structure for a specific equation (or system of equations) and a large enough renormalisation group, as well as identifying its action on the equation.
\item Identifying the correct renormalisation procedure and showing convergence of renormalised models.
\end{enumerate}
In this work we treat the first two steps in full generality and execute the last step for the following equations.
\begin{itemize}
\item Let $M$ be a compact Riemannian $2$-manifold, $f\in \cC^{\infty}(\RR)$ and $A(\cdot, \cdot)$ a smooth bilinear form on the tangent bundle $TM$. The g-PAM equation reads
\begin{equation}\tag{g-PAM}
\label{gpam_intro}
\partial_t u + \Delta u = A( \nabla u, \nabla u) +f(u)\xi, 
\end{equation}
where $\Delta$ is the Laplace--Beltrami operator\footnote{We use the sign convention more commonly used in geometry. Thus, Laplacians are positive operators.} and $\xi$ is spatial white noise constant in time. 

\item Let $M$ be a compact Riemannian 2 or 3-manifold and $E\to M$ a vector bundle with a metric $h$ and compatible connection $\nabla^{E}$. The natural geometric generalisation of the $\Phi^4$ equation reads
\begin{equ}\tag{$\Phi^4_3$}\label{phi4_intro}
\partial_t u + \Delta^E u =  -u\cdot |u|^2_h +\xi,
\end{equ}
where $\xi$ is an $E$-valued spacetime white noise, $|\cdot|_h$ the norm on $E$ induced by $h$ and $\Delta^E$ 
a generalised Laplacian acting on sections of $E$, an example being
the connection Laplacian induced by the connection $\nabla^E$ and the Levi-Civita connection, c.f.\ \eqref{eq:connection_laplacian}.

The invariant measure of this dynamic is of interest in constructive quantum field theory and can formally be written as
$$\exp\Big(-\int_M |\nabla^E\phi|^2_{g\times h} +V_{C,\lambda}(\phi)  \Big)d\mathfrak{L}(\phi)\ , $$ 
where
$|\cdot |_{g\times h}$ is the induced norm on $T^*M\otimes E$,
$V_{C,\lambda}(u)= C|u|_{h}^4- \lambda |u|_{h}^2 \ $ is a double well potential and $\mathfrak{L}(\phi)$ denotes the (non-existent) Lebesgue measure on sections of $E$. 

\item While it turns out that in many cases, such as the two above, the renormalisation procedure needed to solve singular SPDEs is insensitive to the curvature of the underlying manifold, this is not always the case. The equation 
\begin{equ}\tag{$\phi^3_4$}\label{eq:phi34_intro}
\partial_t u + \Delta u = u^2  +\xi
\end{equ}
on a compact Riemannian $4$ manifold is an example where an additional counter-term proportional to the scalar curvature of $M$ has to be introduced.
\end{itemize}

At this point we mention the work \cite{DDD19} which builds a minimal framework to solve Equation (\ref{gpam_intro}) in the special case $f=0$ and $g=\id$ within regularity structures. Of course flat compact manifolds, such as the torus or the Klein bottle could already be treated in \cite{Hai14}, since all flat compact manifolds are quotients of Euclidean space by a free crystallographic group. 
The theory of para-controlled distributions \cite{GIP15}, an alternative theory to deal with SPDEs, has been extended to the manifolds setting in \cite{BB16} to be able to deal with the analytic step of solving Equation (\ref{gpam_intro}), see also \cite{Mou22}. 
Furthermore, Equation~\ref{phi4_intro} was recently independently studied in \cite{BDFT23a}, \cite{BDFT23b} based on a transform introduced in \cite{JP23}. 

\subsection*{Outline of the first half} 
Section~\ref{section geometric setting} introduces the geometric setting for this work, in particular a notion of scaling $\fraks$ similar to \cite{Hai14}. 

In Section~\ref{basic facts on jets} we first recall known results about jets, the analogue of polynomials on manifolds. Then, the notion of an \textit{admissible realisation} introduced. 
Since in the setting here, i.e.\ in the presence of a connection and metric, there exists a canonical grading and norm on jet bundles $J^k$ which we use to construct a graded space $J$.  
The space $J$ will play the role of abstract polynomials and comes with a canonical inclusion 
$$J^k\hookrightarrow J \ $$
for each $k$.

Section~\ref{section Regularity structures and models} introduces the notion of a regularity structure and model. As in \cite{DDD19}, we only impose a condition informally expressible as
\begin{equ}\label{eq g}
\Pi_p \Gamma_{p,q}\sim \Pi_q
\end{equ}
instead of the rigid relation $\Pi_p \Gamma_{p,q}= \Pi_q$ in \cite{Hai14} and further introduce a bigrading $T= \bigoplus_{\alpha,\delta} T_{(\alpha,\delta)}$ where the parameter $\delta$ keeps track of a quantitative bound on the error implicit in \eqref{eq g}.  While this allows for more flexibility in the definition of jet models, it has the drawback that in order to formulate a general Schauder estimate in Section~\ref{Section Singular Kernels} it forces us to abandon the lower triangular structure of the maps $\Gamma$. 
To be able to treat equations with components taking values in different vector bundles such as $u$ and $\nabla u$ in Equation (\ref{gpam_intro}), we work with \textit{regularity structure ensembles}. In contrast to \cite{DDD19} our maps $\Gamma$ and $\Pi$ are defined globally. We introduce vector bundle valued jet regularity structures and models as examples.

Up to modifications of some proofs the ideas in Section~\ref{section modelled distributions and reconstruction}, where modelled distributions are introduced and a reconstruction theorem in our setting is provided, are the same as in \cite[Section 2]{Hai14}.

In Section~\ref{section Local operations} we show how to implement local operations on regularity structures. The notion of products is similar to the one in \cite{Hai14} with the caveat of having to use regularity structure ensembles in the bundle valued setting and the need to keep track of the bi-grading. Compositions with smooth local functions relies on interpreting a local function $G: E\to F$, where $\pi_E: E\to M$, $\pi_F:F\to M$ are vector bundles, as a section of the bundle $E\ltimes F:= \pi_E^* (F)$ and needs additional considerations not present in the flat setting. Lastly, we explain in this section, how to lift differential operators relying on the fact that every $k$th order differential operator $\partial: \mathcal{C}^\infty(E)\to \mathcal{C}^\infty(F)$ corresponds to a bundle morphism $T_\partial\in L(J^k(E), F)$ together with considerations from Section~\ref{basic facts on jets}.

In Section~\ref{Section Singular Kernels} we follow \cite[Section 5]{Hai14} to define integration against singular kernels between vector bundles. Due to (\ref{eq g})
which has as consequence that
\begin{equation}\label{locloc}
(I+J)\circ \Gamma\neq \Gamma\circ (I+J)
\end{equation}
(as opposed to the setting in \cite{Hai14}), we have to abandon the lower triangular structure of the maps $\Gamma$ and introduce a new operator to compensate \eqref{locloc}. While the presence of this operator technically complicates the analysis, it has a natural interpretation in terms of the ``elementary'' modelled distributions introduced in \cite{ST18}, see Remark~\ref{philosophical remark}.

Section~\ref{section singular modelled distributions} and Section~\ref{section solution to semilinear...} are except for technical changes close to \cite[Section 6, 7]{Hai14}. Since throughout the main part of the text we only work with the integral form of the equations of interest, Section~\ref{section:excoursion} contains a short discussion about the existence of fundamental solutions for differential operators of the form $\partial_t + \mathfrak{L}$.

This completes the analytic part of the theory. 

\begin{remark}\label{rem:flat space interpretation}
For readers mostly interested in the case of $M=\mathbb{R}^d$ or $\mathbb{T}^d$ the jet bundle can canonically be identified with polynomials, c.f.\ Sections~\ref{jets in local coordinates} and Remark~\ref{remark homogenious/trivial bundle}. From this point of view one can interpret many results as recovering known statements of \cite{Hai14} under weaker assumptions, see also Remark~\ref{rem:flat space interpretation2}.
\end{remark}

\subsection*{Outline of the second half}
In Section~\ref{section symmetric sets and Vector bundles} we recall the notion of symmetric sets developed in \cite{CCHS22}, and introduce vector bundle assignments, a direct modification of vector space assignments. We also recall the analogue of the functor $\Func_W$ therein.

In Section~\ref{section sets of trees for...} sets of (combinatorial decorated) trees which form the backbone of the constructions that follow are introduced. These, in contrast to previous works, are built from three sets of edge types $\mathcal{E}_+$, $\mathcal{E}_-$ and $\mathcal{E}_0$. The first encodes kernels, the second noises and the last serves as a place holder for jets. From these, similarly to \cite{BHZ19}, sets of trees $\mathfrak{T}, \mathfrak{T}_+$ and $\mathfrak{T}_-$ as well as coloured versions thereof are built from a rule.
The trees in $\mathfrak{T}_-$ are used later for negative renormalisation, contain no $\mathcal{E}_0$ type edges and are used to index spaces of differential operators $\mathfrak{Dif}$, while the trees in $\mathfrak{T}$ resp.\ $\mathfrak{T}_+$ are used in the construction of the regularity structure ensemble, resp.\ for positive renormalisation utilising on the machinery of symmetric sets.
The coloured versions of these sets of trees such as $\mathfrak{T}^{(-)}$ and $\mathfrak{T}_-^{(-)}$ are crucially used since both renormalisations are encoded by colouring operations instead of extractions.

In Section~\ref{section regularity structure, models,...} , after fixing a cutoff $\delta_0>0$, we construct vector bundles $\mathcal{T}$ and $\mathcal{T}_+$ from a subsets of $\mathfrak{T}$ and $\mathfrak{T}_+$. This is done by first turning trees into symmetric sets and then using the functor $\Func_W$,  informally this corresponds to ``attaching"  jets to edges of type  $\mathcal{E}_0$. Finally, $\mathcal{T}$, is equipped with a bi-grading and thus a regularity structure ensemble is constructed.
Then, we construct pre-models which turn elements of $\mathcal{T}$ into concrete distributions as well as the renormalisation group $\mathfrak{G}_-$. After investigating the action of $\mathfrak{G}_-$ on a pre-model $\pmb{\Pi}$, we show how to construct a model $(\Pi, \Gamma)$ from it. To do so, positive renormalisation is still described by ``admissible cuts", though in this setting, as a consequence of \eqref{eq g} one has to cut trees of negative homogeneity as well. 
Then, we study the action of $\mathfrak{G}_-$ on models. Since we do not restrict ourselves to constant coefficient differential operators, this action is not described by extraction operations but by colouring operations. 
For both renormalisations, due to the fact that all bundles are constructed from symmetric sets, significant care has to be taken to ensure that everything is well defined.
Lastly, explicit formulas for the renormalisation of canonical smooth models are obtained in this section.

Section~\ref{section application to spdes} starts with a subcritical system of SPDEs and constructs the required bundle assignments, edge types and rule. Then the machinery of the previous sections yields a regularity structure as well as a renormalisation group. First, for the canonical model for smooth noises we check that the corresponding (reconstructed) abstract solution agrees the classical solution to the system of equations. 
Then, in the central Proposition~\ref{prop:renormalised equation} of this section, we prove a formula for the renormalised equation, which is solved (in the classical sense) by the reconstructed abstract solution with respect to a renormalised model.

\begin{remark}\label{rem:flat space interpretation2}
In the context of Remark~\ref{rem:flat space interpretation} one can interpret the results of this second part of the article as giving an alternative non-algebraic point of view on the results in \cite{BHZ19}, \cite{BB21} and \cite{BCCH20}. In particular, our framework allows for a rather compact proof of 
 Proposition~\ref{prop:renormalised equation}, a slight generalisation of the main result of \cite{BCCH20} even in the case of the torus.
\end{remark}

\subsection*{Outline of concrete applications and stochastic estimates}
When working with non-translation invariant equations, one is faced with the problem that the strategy of regularity structures, when applied naively might yield an infinite-dimensional solution family. The renormalisation group constructed in Section~\ref{sec:renormalisation group} is indeed infinite dimensional. In Section~\ref{sec:canonical renormalisation}, we outline a general strategy to remedy this issue. 
%
In particular, we explain why for most singular SPDEs where the linear term is given by a generalised Laplacian, renormalisation constants are sufficient. More precisely, one expects the convergence results for such singular SPDEs on trivial vector bundles over the torus to translate directly to any (possibly non-trivial) vector bundle over a compact manifold.

In Section~\ref{sec:concrete applications} we turn to the treatment of specific examples of singular SPDEs, each of which illustrates certain aspects of the theory developed in this article.
\begin{itemize}
\item While \eqref{gpam_intro} is a scalar valued equation, the right-hand side involves the gradient of the solution, i.e.\ a section of $TM$ and the equation involves a rather general non-linearity. Since it requires only Wick renormalisation, it serves as an example illustrative of the general workings of the theory which does not many stochastic estimates.
In Theorem~\ref{thm:g-pam} we exhibit an (at most) two dimensional canonical solution family.
\item The $E$-valued~\ref{phi4_intro}-equation illustrates how the theory developed in this article
can be used to solve singular SPDEs taking values in general vector bundles. While the scalar valued analogue
on a compact manifold can essentially be treated as in \cite{Hai14},
when the vector bundle $E$ is non-trivial, a technical novelty arises when establishing the convergence of  
$$\left(\hat{\Pi}^\epsilon_q  \hat{\Gamma}^\epsilon_{q,p}- \hat{\Pi}^\epsilon_p\right) \left(\<3'2>_{_{(e_p)}}\right)\ , $$
see Lemma~\ref{lem:stochastic bounds 3'2}.
As in \cite{Hai14}, we exhibit an (at most) $2$-dimensional solution family in Theorem~\ref{thm:phi4_3}.
\item The last example, the~\ref{eq:phi34_intro} equation, is meant to illustrate an SPDE, where a geometric counterterm is needed to renormalise the model. Here we obtain a (at most) $5$-dimensional solution family in Theorem~\ref{thm:phi3_4}.
\end{itemize}

Typically, in the framework of this article, when showing that a family of models $Z^\epsilon=(\Pi^\epsilon, \Gamma^\epsilon)$ converges, one has to not only obtain appropriate stochastic estimates on terms of the form 
${\Pi}^\epsilon_p\tau (\phi^\lambda_p)$, but also additionally for 
$$\left(\Pi^\epsilon_q\Gamma_{q,p}^\epsilon \tau_p-\Pi^\epsilon_p\tau_p \right)(\phi^\lambda_q)\ .$$
While this can significantly increase the number of estimates needed, their general form is close to the Feynman diagrams usually encountered in regularity structures. Thus, for example for the equations treated here, the straightforward adaptations of the kernel estimates in \cite{Hai14} and \cite{HQ18} in Section~\ref{sec:Kernels estimates}, are mostly sufficient.
\subsection*{Acknowledgements}
HS gratefully acknowledges funding by a President’s PhD Scholarship from Imperial College London
and would like to thank S.\ Paycha for valuable discussions.
MH gratefully acknowledges funding from the Royal Society through 
a research professorship, grant RP\textbackslash R1\textbackslash 191065.

\section{Geometric Setting: Orthogonal Foliations of Manifolds}\label{section geometric setting}

In this section we motivate a natural geometric setup for incorporating a scaling $\mathfrak{s}$ in the sense of \cite{Hai14}.
Let $(M,g)$ be a (complete) $d$-dimensional Riemannian manifold and suppose one is given $n\leq d$ mutually orthogonal distributions\footnote{Here ``distribution'' is meant in the geometric sense, i.e.\ a smooth section of the Grassmann bundle over $M$.} $\mathcal{D}^1,\ldots, \mathcal{D}^n\subset TM$, such that $T_p M= \bigoplus_{i=1}^n \mathcal{D}^i_p$ together with an element $\mathfrak{s}\in \mathbb{N}^n$. Then one defines the scaled length of a tangent vector $$X_p= X_p^1+\ldots+X_p^n\in T_pM \ ,$$ where $X_p^i\in \mathcal{D}_p^i$ as
\begin{equation}\label{scaled norm}
|X_p|_\mathfrak{s}=\sum_{i=1}^n g_p(X_p^i,X_p^i)^{\frac{1}{2\mathfrak{s}_i}} \ .
\end{equation}
On geometric grounds it is natural to assume that these distributions are integrable and parallel. 
Under these assumptions the following theorem of de Rham, c.f.\  \cite[Ch. 4, Thm. 4.4]{BF06} holds.
\begin{theorem}\label{thm:cultural}
In addition to the above setup, assume that $M$ is simply connected. We denote by $\mathfrak{F}^1,\ldots,\mathfrak{F}^n$ the foliations of $M$ corresponding to $\mathfrak{D}^1,\ldots,\mathfrak{D}^n$ respectively. Then, for every $p\in M$, there exists a foliation preserving isometry from $(M,g)$ to the Riemannian product manifold
$(L^1, g^1)\times\ldots\times (L^n,g^n)$, mapping each leaf $\mathfrak{F}^i$ to $L^i$.
\end{theorem}
%
%
The above discussion motivates the following definition.
\begin{definition}\label{def scaling on manifold}
A scaling $\mathfrak{s}$ on a manifold $(M,g)$ is a decomposition $M=(M^1,g^1)\times\ldots\times (M^n,g^n)$ together with an element $\mathfrak{s}\in \mathbb{N}^n$. To a scaling $\mathfrak{s}$ on $(M,g)$ we associate the scaled ``norm'' on $TM$ defined in (\ref{scaled norm}), as well as the distance\footnote{Sometimes it will be more convenient to work with the (equivalent) scaled distance $d_\mathfrak{s}(p,q):=\sup_{i\in \{1,\ldots,n\}} d^i(p_i,q_i)^{1/\mathfrak{s}_i}$, which we shall do without elaborate explanation.} defined for $p,q\in M$ by
$$d_\mathfrak{s}(p,q)=\sum_{i=1}^n d^i(p_i,q_i)^{\frac{1}{\mathfrak{s}_i}}$$
where $d^i$ denotes the usual distance induced on $M^i$ by $g^i$ and where $p=(p_1,\ldots,p_n)$, $q=(q_1,\ldots,q_n) \in (M^1,g^1)\times\ldots\times (M^n,g^n)$ .
\end{definition}
%
%
The following simple lemma illustrates that the norm $|\cdot |_\mathfrak{s}$ and distance $d_\mathfrak{s}(\cdot, \ \cdot)$ fit together well.
\begin{lemma}
In the above setting for every point $p\in M$ and $X_p\in T_p M$ small enough the identity 
$$d_\mathfrak{s}(p,\exp_p{X_p}) = |X_p|_\mathfrak{s}\ ,$$
holds.
\end{lemma}
We call a chart $\phi : M\supset U \to V\subset \mathbb{R}^d$ an $\mathfrak{s}$-chart, if there exist charts 
$\phi^{i}: M^i\supset U^i \to V^i\subset \mathbb{R}^{d_i}$ for $i\in \{1,\ldots,n\}$ such that
$U=U^1\times\ldots\times U^n$ and $\phi=(\phi^1,\ldots,\phi^n)$. For a multi-index $k=(k^1,\ldots,k^n)$, where $k^i\in \mathbb{N}^{d_i}$, we write $$|k|_\mathfrak{s}=\sum_{i=1}^n \fraks_i\sum_{j=1}^{d_i} k^i_j \ .$$

In the sequel it will be useful to single out the vector fields $V$ on $M$, which take values in $TM^i\subset TM$. We call such a vector field an $\mathfrak{s}_i$-vector field. We say $W$ is a $\fraks$-vector field, if it is a $\mathfrak{s}_i$-vector field for some $i$ and declare the degree of a $\fraks$-vector field $V$ to be given by 
\begin{equation}\label{deg of vec}
\text{deg}_\fraks V =\begin{cases}  0 & \text{if } V=0 \ ,\\
\fraks_i & \text{if } V \text{ is an }\fraks_i\text{-vector field .} 
\end{cases}
\end{equation}
We extend this to tuples of $\fraks$-vector fields $ (V_1,\ldots,V_k)$ by 
$$ \text{deg}_\fraks (V_1,\ldots,V_k) = \sum_j \deg_\fraks(V_j) \ .$$

\begin{remark}
One could also work with real-valued scalings $\fraks\in \mathbb{R}_+^n$, but the restriction $\fraks\in \mathbb{N}^n$ simplifies notation later on. Furthermore, since in the theory of regularity structures one always has ``an epsilon of room'', this is not usually much of a restriction. 
\end{remark}

\begin{remark}
In the setting of Theorem~\ref{thm:cultural} the Frobenious theorem, c.f.\ \cite{Lun92}, characterises integrable distributions. 
The somewhat complementary setting where $\mathcal{D}^1$ is spanned by Hörmander vector fields, the simplest examples being Carnot groups,  was considered in \cite{MS23}.
\end{remark}

\subsection{Vector bundles and metric connections}\label{section bundles and metric}
Usually, when working with a vector bundle $\pi: E\to M$, we shall assume that it is equipped with a metric $\langle\cdot , \cdot\rangle_{E}$ and a connection $\nabla^E$ which is compatible with the metric, i.e.\ for any two sections $V,W$ of $E$ and $X_p\in T_p M$ one has
$$X_p(\langle V, W \rangle_E) =\langle \nabla_{X_p}^E V, W \rangle_E + \langle V, \nabla_{X_p}^E W \rangle_E \ , $$
where the left-hand side uses the usual identification of tangent vectors with derivations.
We shall write $(E,\nabla^E, \langle\cdot, \cdot\rangle_E)$ in this case. We shall denote by $\cC^\infty(E)$ the space of smooth sections of $E$ and by $\cC_c^\infty(E)$ those which are furthermore compactly supported. 

Let us recall some important constructions and examples of vector bundles. 
\subsubsection{Trivial bundle}\label{subsubsec:trivial}
Any real-valued function on $M$ can naturally be identified with a section of the trivial bundle $M\times \RR$. (Recall that there exist a canonical metric and connection on this bundle, the latter being independent of the Riemannian structure of $M$.)
\subsubsection{Tangent bundle}
We shall always equip the tangent bundle $TM$ with the Levi-Civita connection $\nabla^{TM}$ .
\subsubsection{Dual bundle}
Given a vector bundle $(E,\nabla^E, \langle\cdot, \cdot\rangle_E)$ as above, denote by $E^*$ the dual bundle and $(\cdot, \cdot): E^*|_p \otimes E|_p\to \RR$ the natural pairing above any point $p\in M$. We shall always equip $E^*$ with the canonically induced metric $\langle\cdot , \cdot\rangle_{E^*}$ and connection  $\nabla^{E^*}$, which are characterised by the following properties:
\begin{itemize}
\item If $\{e_i\}$ is an orthonormal basis of $E|_p$, then the unique basis $\{e_i^*\}$ of $E^*|_p$ satisfying $$(e_k^*,e_i)= \delta_{i,k} \ $$
 is an orthonormal basis of $E^*|_p$.
\item If $V, V^*$ are sections of $E$, resp. $E^*$ and $X_p\in T_p M$, then
$$\nabla_{X_p} (V^*,V) = (\nabla_{X_p}^{E^*} V^*, V) + (V^*,\nabla_{X_p}^{E}  V)\ .$$
\end{itemize}
One can observe that $\nabla^{E^*}$ is a metric connection with respect to $\langle\cdot , \cdot\rangle_{E^*}$.

\subsubsection{Product bundle}\label{section hat tensorproduct} 
Given two vector bundles $E\to M, \ F\to N$, let $E\hotimes F$ be the vector bundle over $M\times N$ with fibres $E\hotimes F|_{(p,q)}=  E|_p \otimes F|_q$ for $(p,q)\in M\times N$ and smooth structure characterised by the fact that if $V\in \cC^\infty(E),\ W\in \cC^\infty(F)$, then $(p,q)\mapsto V(p)\otimes W(q)$, for which we write $V(p)\hotimes W(q)$ from now, is an element of $\cC^\infty (E\hotimes F)$.

Furthermore $(E,\nabla^E, \langle\cdot, \cdot\rangle_E)$ and $(F,\nabla^F, \langle\cdot, \cdot\rangle_F)$ as above, then we shall always equip $E\hotimes F$ with the canonically induced metric $\langle\cdot , \cdot\rangle_{E\otimes F}$ and connection $\nabla^{E\otimes F}$, which are characterised by the following properties: 
\begin{itemize}
\item If $\{e_i\}$, $\{f_j\}$ is an orthonormal basis of $E|_p$ respectively $F|_q$, then $\{e_i\hotimes f_j\}_{i,j}$ is an orthonormal basis of $(E\otimes F)|_{(p,q)}$ .
\item If $V, W$ are local sections of $E$ resp. $F$, $X_p\in T_p M$ and $Y_q\in T_q N$, then $$\nabla_{(X_p,Y_q)}^{E\hotimes F} (V\hotimes W) = (\nabla_{X_p}^{E} V)\otimes W(q) +  V(p)\otimes \nabla_{Y_q}^{F} W \ ,$$
where we implicitly identified $T(M\times N)\simeq TM\hotimes TN$ under the canonical isomorphism.
\end{itemize}
One can observe that $\nabla^{E\hotimes F}$ is compatible with $\langle\cdot , \cdot\rangle_{E\hotimes F}$. 

\subsubsection{Tensor product bundle}
Given two vector bundles $(E,\nabla^E, \langle\cdot, \cdot\rangle_E)$ and $(F,\nabla^F, \langle\cdot, \cdot\rangle_F)$ with the same manifold $M$ as base space, let $E\otimes F:=\diag^*(E\hotimes F)$ where $\diag:M\to M\times M$ is the diagonal map and $\diag ^*$ the induced pullback of vector bundles. We equip $E\otimes F$ with the canonically induced metric $\langle\cdot , \cdot\rangle_{E\otimes F}$ and connection $\nabla^{E\otimes F}$, which are characterised by the following properties: 
\begin{itemize}
\item If $\{e_i\}$ and $\{f_j\}$ are orthonormal bases of $E|_p$ and $F|_p$, then $\{e_i\otimes f_j\}_{i,j}$ is an orthonormal basis of $(E\otimes F)|_p$.
\item If $V, W$ are local sections of $E$, resp. $F$ and $X_p\in T_p M$, then $$\nabla_{X_p}^{E\otimes F} (V\otimes W) = (\nabla_{X_p}^{E} V)\otimes W +  V\otimes \nabla_{X_p}^{F} W \ .$$
\end{itemize}
One can observe that $\nabla^{E\otimes F}$ is compatible with $\langle\cdot , \cdot\rangle_{E\otimes F}$. Furthermore they agree with the pulled back metric and connection under $\diag$ .

\subsubsection{Symmetric tensor product bundle}
Working with one vector bundle $(E,\nabla^E, \langle\cdot, \cdot\rangle_E)$, we write $E\otimes_s E:=\diag^*(E\hotimes_s E)$ for the subbundle of $E\otimes E$ consisting of symmetric tensors equipped with the metric and connection obtained by restriction of the ones on $E\otimes E$.

\subsubsection{Higher order derivatives}
The constructions above give a meaning to higher order derivatives of sections of $ E$ given inductively as $$(\nabla^E)^{n+1}:= \nabla^{T^*M^{\otimes n}\otimes E}\circ (\nabla^E)^n \ , $$
where $(\nabla^E)^1:=\nabla^E$ .

We shall usually simply write $\langle\cdot , \cdot \rangle$ and $\nabla$ omitting reference to the underlying vector bundle. 

\subsection{Bundle valued distributions}
Given a vector bundle $E\to M$, we write $\mathcal{D}(E)$ for the space of smooth compactly supported sections. We equip $$\mathcal{D}(E)= \bigcup_{K\subset M\ {\text{compact}}} \mathcal{C}^\infty_{c,K}(E)\ $$ with the inductive limit topology, where each $\mathcal{C}^\infty_{c,K}(E):=\{ \phi\in \mathcal{C}^\infty_{c}(E) \ : \supp (\phi)\subset K\}$ is equipped with the usual Fr\'{e}chet topology induced be the semi-norms $$\|\phi \|_k:= \sup\{ |\nabla^l \phi (x)|\ : \ x\in K, \ l\leq k\}\ .$$
The space of generalised sections (or $E$-valued distributions) $\mathcal{D}'(E)$ is defined as the dual space of $\mathcal{D}(E)$.
We shall usually identify $u\in \mathcal{C}^\infty (E)$ with an element of $\mathcal{D}'(E)$ via the canonical inclusion
$$\mathcal{C}^\infty (E)\hookrightarrow \mathcal{D}'(E),\qquad u\mapsto T_u \ , $$
where $T_u(\phi):= \int_M \langle u(p), \phi(p)\rangle_E  \, \dVol_p$ for $\phi\in \mathcal{D}(E)$ and 
$\dVol_p$ stands for integration in $p$ with respect to the usual Riemannian density. 

\section{Jet Bundles}\label{basic facts on jets}

In this section we collect the definitions and basic properties of the jet bundle of a vector bundle $\pi:E\to M$ where $M$ is equipped with a scaling $\fraks$ and at the same time we fix notation used in the sequel. For a beautiful exposition on jet bundles we refer to \cite{Pal65}.

\subsection{Definition and basic properties}\label{section def and basic prop}
Throughout this section we do not assume any structure other than $M$ being a smooth manifold and $\pi:E\to M$ being a smooth vector bundle.

For $k\in \mathbb{N},p\in M$ we define an equivalence relation on $\cC^\infty(M)$ by postulating that
$$f\sim_{k,p} g$$ if in some (and therefore any) $\fraks$-chart $(U,\phi)$ around $p=\phi (0) \in M$ and for all multiindices
$l$ with $|l|_{\fraks}\leq k$ one has $D^l(f\circ \phi)(0) =D^l(g\circ \phi)(0)$. We define  
$$J_p^k M := \cC^\infty(M)/\sim_{k,p}$$ and denote by $j^k_p f$ the equivalence class of $f\in \mathcal{C}^\infty(M)$.

For a vector bundle $E\to M$ we define the space 
$$J_p^k E:= \cC^\infty(E)/\sim_{E,k,p}\ ,$$
where two section $f,g\in \cC^\infty(E)$ are equivalent with respect to $\sim_{E,k,p}$ if in some (and therefore any) smooth local frame $\{e_i\}_i$ of $E$ around $p$ the component functions of $f=\sum_i f^i e_i$, $g=\sum_i g^i e_i$ satisfy $f^i\sim_{k,p} g^i$ for each $i$ .

We define $J^k E:= \bigcup_{p} J_p^k E$, which can be endowed with the structure of a vector bundle (c.f.\ \cite[Chapter 2.3]{Hir12} or \cite[Chapter 6]{Sau89}) and is called the $k$-jet bundle of $E$. To any smooth section $f\in \mathcal{C}^\infty (E)$ one can associate the following section of $J^k E$: 
$$j^k f :M \to J^k E,\qquad  p\mapsto  j_p^k f\ .$$
The map $j^k$ is called the prolongation map.

\begin{remark}
To make notation consistent one should write $J^k (M\times\RR)$, instead of  $J^k M$ but we shall usually opt for the latter for notational convenience.
\end{remark}

\begin{remark}
Classically jet bundles are defined in the absence of a non-trivial scaling. But since the scaling only causes minor modifications of the classical concepts, committing slight abuse of nomenclature, we use the classical terminology in this setting as well.
\end{remark}
\begin{remark}
The notion of jet bundles extends to fibre bundles, but the jet bundle of a fibre bundle is only an affine space in general, c.f.\ \cite{Sau89}.
\end{remark}
\begin{remark}
In the absence of a scaling $\mathfrak{s}$,  an equivalent condition 
for $f\sim_{k,p} g$ for two functions $f,g\in \cC^\infty(M)$ is to ask that for any smooth curve $\gamma \colon (-\epsilon, \epsilon) \to M$ 
such that $\gamma(0)=p$, one has $$\partial_s^n (f\circ \gamma)(s)|_{s=0}=\partial_s^n (g\circ \gamma)(s)|_{s=0}$$ for all $n\leq k\ .$
This is \textit{not} true in the presence of a non-trivial scaling $\mathfrak{s}$!
\end{remark}

\begin{remark}\label{alternative definition remark}
Let us mention an alternative characterisation of the space $J_p^k M$, resp.\ $J_p^k E$, for a proof in the absence of a scaling we refer to the discussion in \cite[Chapter IV.2]{Pal65}. 

For $p=(p_1,\ldots,p_n)\in M$ let $I^i_p:=\{f\in  \cC^\infty(M^i)\,:\, f(p)=0\}$ be the ideal of $\mathcal{C}^\infty(M^i)$ of functions vanishing at $p$.
Identify $I^i_p\subset \mathcal{C}^\infty (M)$ in the canonical way and define $\bar{I}_p^k$ as the smallest ideal containing $$\bigcup_{|l|_\fraks > k} (I_p^1)^{l^1} \cdots (I_p^n)^{l^n}\;.$$
Then one can define $J^k_p(M)= \mathcal{C}^\infty(M)/\bar{I}^k_p(M)$, which shows in particular, that $J^k(M)$ has a natural algebra structure.

Similarly for a vector bundle $E$, if one defines $Z^k_p= \bar{I}_p^k \cdot \cC^\infty(E)\subset \cC^\infty(E)$ , one can set $$J^k_p(E)= \mathcal{C}^\infty(E)/Z^k_p\;,$$ which implies that $J^k_p(E)$ is a $J^k_p(M)$-module. 
\end{remark}

\subsubsection{Pull-back of jets}\label{section pullback of jets}
Let $\phi:M\to N$ be a smooth map, the notion of pullback $$\phi^*: C^{\infty}(N)\to C^{\infty}(M)\ ,\qquad  g\mapsto g\circ \phi\ ,$$ induces a notion of pullback on jets. We define a map $\phi^*:J^k N \to J^k M$ determined for $p\in U$, $q=\phi(p)$ by 
\begin{equation}\label{eq pullback of scalar jets}
\phi^*|_q: J^k_q V \to J^k_p U \ , \qquad  j^k_q f\mapsto j^k_{p} (\phi^* f) =j^k_{p} (f\circ \phi) \ .
\end{equation}
Working on vector bundles $E\to M$, $F\to N$, let $\Phi: E\to F$ be a bundle morphism which restricts to isomorphisms on fibres.
Write $\Phi=(\phi, G)$ where $\phi:M\to N$ is a smooth map and $G(p):E|_p\to F|_{\phi(p)}$ is an invertible linear map for each $p\in M$,  such that one can write $\Phi(p,v)= (\phi(p), G(p)(v))$ for every $(p,v)\in E$. Such a bundle morphism induces a pullback $\Phi^*: \cC^\infty(F)\to \cC^\infty(E)$ by
\begin{equation}\label{eq pullback of general jets}
\big(\Phi^* g\big) (p)= G^{-1}(p) \big(g\circ \phi(p)\big)\ .
\end{equation}
One can thus define a map $\Phi^*:J^k F \to J^k E$ determined by the fact that for any $g\in\cC^\infty(F)$ and $p\in M$, $q=\phi(p)$
$$\Phi^*|_q: J^k_q F \to J^k_p E \ , \qquad  j^k_q g\mapsto j^k_{p} (\Phi^*g) \ .$$
One easily checks that the maps introduced in \eqref{eq pullback of scalar jets} and \eqref{eq pullback of general jets} are well defined.

\subsubsection{Differentiation of jets}
For $V$ a vector field in a neighbourhood of $p$ the map $$\nabla_V: J^k_p(E) \to J_p^{{\lfloor k-\max_{i}\fraks_i\rfloor}} (E),\quad 
j^k_p f \mapsto j^{{\lfloor k-\max_{i}\fraks_i\rfloor}}_p \nabla_Vf $$ is well defined. If $V$ is an $\fraks_i$-vector field, this can improved and the  map 
\begin{equation}\label{differentiation of jets}
\nabla_V: J^k_p(M) \to J_p^{{\lfloor k-\fraks_i\rfloor}}(M) ,\quad 
j^k_p f \mapsto j^{{\lfloor k-\fraks_i\rfloor}}_p \nabla_Vf \ 
\end{equation}
is well defined. This follows from a direct computation in charts, or alternatively from the fact that in the setting of Remark~\ref{alternative definition remark} if $r\in Z_p^k$, then $\nabla_{V}r \in Z_p^{\lfloor k-\fraks_i\rfloor}$.
Given a global $\fraks_i$-vector field $V$, this induces a smooth map $J^k E\to J^{\lfloor k-\fraks_i\rfloor} E$.

\subsubsection{Scalar valued jets in local coordinates}\label{jets in local coordinates}

Let $\mathbb{R}^d$ be equipped with a scaling $\mathfrak{s}$ in the sense of \cite[Equation (2.10)]{Hai14} and denote by $\Poly$ polynomials in indeterminates $X_1,\ldots,X_d$. If $M=U$ is an open subset of $\mathbb{R}^d$, we can and will identify $J^k M$ with $U\times \Poly^k$, where $\Poly^k\subset \Poly$ denotes the space of abstract polynomials of (scaled) degree up to $k$
and we identify for $f\in \mathcal{C}^\infty(M)$
$$j^k_p f= \sum_{|l|_\fraks\leq k} \frac{X^l}{l!} D^{l}f(p) \ .$$
Note that $\Poly^k$ comes with a natural algebra structure, where the product is given by the usual product on $\Poly$ composed with the projection onto $\Poly^k$ and, for any $p \in M$, this agrees with the algebra structure on $J^k_p$ previously mentioned. 
We now investigate the notion of pullback of jets under this identification. Let $U,V \subset \mathbb{R}^m$ and $\phi: U\to V$ be a smooth map. For $p\in U$, $q=\phi(p)$ one can express the map
 $$\phi^*|_q: J^k_q V \to J^k_p U \ , \qquad  j^k_q f\mapsto j^k_{p} (\phi^* f) =j^k_{p} (f\circ \phi) \ ,$$
under the above identification of jets with polynomials as follows: Denoting by $P^k_{q}\phi$ the polynomial of order $k$ associated to the $\mathbb{R}^d$-valued jet $j_q^k \phi$, the map $\phi^*$ is given by the algebra (bundle) morphism $\{q\}\times \Poly^k \to \{p\} \times \Poly^k$ determined by $X_i \mapsto \big( P_{q}\phi - p \big)_i$. Note that this map is not graded in general, but upper triangular with respect to the natural grading on polynomials.

\begin{remark}
The fact that the map $\phi^*|_q: J^k_q V \to J^k_p U$ does not preserve the grading on polynomials is consistent with the fact that in the setting of smooth manifolds, the jet bundle is only a filtered bundle as will be discussed in Section~\ref{section grading on jets}.
\end{remark}

If $(U,\phi)$ is a chart of $M$ at $p\in U \subset M$, then this induces a chart 
\begin{equation}\label{chart for jets}
\theta^k_{(U,\phi)}:J^k U \to \phi(U)\times \Poly^k
\end{equation}
determined by $\theta^k_{(U,\phi)}:  j^k_q f\mapsto j^k_{\phi(q)} (\phi^{-1})^* f$, where we can identify $ j^k_{\phi(q)} (\phi^{-1})^* f$ with a polynomial. The smooth structure of $JM$ agrees with the one given by such charts.

\subsubsection{Bundle valued jets in local coordinates}\label{bundle valued jets in local coordinates}

We extend the previous discussion to bundle valued jets. Let $M=U$ be an open subset of $\mathbb{R}^d$ as before, let $E=U\times\mathbb{R}^n$ be the trivial bundle over $U$, and denote by $\{e_i\}_{i=1,..,n}$ the canonical basis of $\mathbb{R}^n$. Then we can and will identify $J^k E$  with $U\times \Poly^k(\RR^n)$, where $\Poly^k(\RR^n)=\Poly^k \otimes \mathbb{R}^n$ denotes the space of abstract $\RR^n$-valued polynomials of (scaled) degree up to $k$. Note that $\Poly^k(\RR^n)$ is naturally a $\Poly^k$-module. We thus identify for $f=\sum_{i=1}^n f^i\otimes e_i$
$$j^k_p f= \sum_{|l|_\fraks\leq k} \sum_{i=1}^n \frac{D^{l}f^i(p)}{l!} X^l\otimes e_i \ .$$

Let $U,V \subset \mathbb{R}^d$ and $E=U\times \RR^n$, $F=V\times \mathbb{R}^n$ be trivial bundles. Let $\Phi: E\to F$ be a bundle morphism restricting to an isomorphism on fibres. Write $\Phi=(\phi, G)$ where $\phi:U\to V$ and $G:U\mapsto \text{GL}(\mathbb{R}^n)$ are smooth maps such that $\Phi(p,v)= (\phi(p), G(p)(v))$ for each $(p,v)\in E$. Then the pull-back map $\Phi^*: \cC^\infty(F)\to \cC^\infty(E)$ is given on $g= \sum_i g^i e_i \in\cC^\infty(F)$ by
$$\big(\Phi^* g\big) (p)= G^{-1}(p) \big(g\circ \phi(p)\big)= \sum_i (g^i\circ\phi)(p) \, G^{-1}(p)(e_i) \ .$$
To describe the map $$\Phi^*|_q: J^k_q F \to J^k_p E \ , \qquad  j^k_q g\mapsto j^k_{p} (\Phi^*g) \ $$ for $p\in U$, $q=\phi(p)$ explicitly, we use the algebra structure of $\Poly^k$ and that $\Poly^k(\RR^n)$ is a $\Poly^k$-module. Then, making the above identification of jets with polynomials and denoting for $\Phi= (\phi, G)$ by $\{g^{i,j}\}_{i,j}$ the components of $G^{-1}$, the map $\Phi^*|_q: J^k_q F\to J^k_p E$ is the unique linear map such that for every multi index $l$
\begin{equation}\label{equationn}
X^l  e_j\mapsto (P^k_{q}\phi - p)^l \cdot (P^k_p G^{-1})e_j =\sum_{i} (P_{q}\phi - p)^l \cdot (P^k_p g^{i,j}) e_i 
\end{equation}
where $P^k_{q}\phi$ is the polynomial of order $k$ associated to the $\mathbb{R}^d$ valued jet $j_q \phi$ and similarly for $P^k_p (G^{-1})$ and $P^k_p g^{i,j}$.

Now suppose $(\pi^{-1} U,\Phi)$ is a local trivialisation of $E$, then this induces a chart 
$\theta^k_{(\pi^{-1} U,\Phi)}: J^k(\pi^{-1} U) \to J^k(\phi(U)\times \mathbb{R}^n)=\phi(U)\times \Poly^k(\RR^n)$ determined by 
\begin{equation}\label{bundle jet chart}
\theta^k_{(\pi^{-1} U,\Phi)}:  j^k_q f \mapsto j^k_{\phi(q)} \big((\Phi^{-1})^* f\big) \ ,
\end{equation}
where we identify $j^k_{\phi(q)} \big((\Phi^{-1})^* f\big)$ with a polynomial. 

\begin{remark}
As in the scalar case one sees from (\ref{equationn}) the jet bundle is only a filtered bundle and the maps in \eqref{bundle jet chart} are charts for the jet bundle.
\end{remark}
\subsubsection{Differential operators}\label{section differential operators}
\begin{definition}\label{definition:differential_operator}
Given vector bundles $E, F$ over a manifold $(M,g)$ with a scaling $\fraks$, we call $\mathcal{A}: \mathcal{D}(E) \to \mathcal{D}(F)$ a (smooth) $k$th order differential operator if there exists $T_\mathcal{A} \in \mathcal{C}^\infty (L(J^kE , F))$ such that 
$ \mathcal{A}= T_\mathcal{A}\circ j^k\ .$
We denote by $\mathfrak{Dif}_k(E,F)$ the space of such differential operators. 

\end{definition}
Observe that if the scaling is trivial, this notion of $k$-th order differential operators agrees with the standard notion, c.f.\ \cite[Section~3]{Pal65}.
\subsection{Grading on jets}\label{section grading on jets}
In this section we observe that there is a canonical grading on $J^k E$, if the underlying manifold $M$ is equipped with a Riemannian metric and $E$ with a (metric) connection.
Note that for each $l<k$ the linear map $p_{k,l}: J^k E\to J^l E$, $j^k f \mapsto j^l f$ is well defined, which 
turns $(J^kE)_{k\in \mathbb{N}}$ into  a filtered bundle. 
One then has a \textit{highly non-canonical} isomorphism of vector bundles
$$J^k E \simeq \ker p_{k,k-1} \oplus \Im p_{k,k-1} \simeq \ker p_{k,k-1} \oplus J^{k-1} E\;,$$ 
so that one can write
$$J^k E \simeq \bigoplus_{l=0}^k \ker p_{l,l-1}\ ,$$
where we identify $J^{-1} E=M\times \{0\}$ and $ p_{0,-1}: J^{0} E\to J^{-1} E$ is the unique bundle morphism preserving the base space.
Thus, by choosing for each $k$ an isomorphism $ \ker p_{k,k-1} \oplus J^{k-1} E\to J^k E $, one constructs a grading on $J^k M$. It is natural to restrict our attention to maps which act on $\ker p_{k,k-1}$ as the identity, and thus it suffices to choose an injection 
$$i_{k-1,k}: J^{k-1} E\to J^k E \ ,$$
for which it is natural to further impose that $p_{k,k-1}\circ i_{k-1,k}= \id_{J^{k-1}(E)}$.

We first discuss the case when $E=M\times \RR$ is the trivial bundle and later turn to the more general case.
For $p\in M$ we define
\begin{equation}\label{def of i}
i_{k-1,k}: J^{k-1}_p M \to J^{k}_p M
\end{equation}
as follows. Let $\phi$ be an exponential $\fraks$-chart at $p$, i.e.\ an exponential chart at $p$ which is also an $\fraks$-chart, then for 
 $f\in \cC^\infty(M)$ set
$$i_{k-1,k}(j^{k-1}_p f) = j_p^k h \ ,$$
where $h$ is a function satisfying
\begin{enumerate}
\item $j^{k-1} h = j^{k-1} f$
\item For any multi-index $l$ with $|l|_{\mathfrak{s}}=k$, $D^l h (p)=0$. 
\end{enumerate}
\begin{lemma}
For each $k$ the map  $i_{k-1,k}$ is a well defined bundle morphism and furthermore it is a right inverse of the map $p_{k,k-1}$, i.e.\ $p_{k,k-1}\circ i_{k-1,k}= \id_{J^{k-1}M}$.
\end{lemma}
Before proceeding to the proof, let us point out that in the vector bundle case, these two conditions (respectively their naive adaptation) are in general not compatible as can be seen from the following example. Let $E=TM$, then in the same exponential chart $(\nabla_{\partial_1}\nabla_{\partial_2}-\nabla_{\partial_2}\nabla_{\partial_1} )f= R(\partial_1,\partial_2)f$ where $R$ denotes the Riemann curvature tensor.
\begin{proof}
To show that the map $i_{k-1,k}$ is well defined, it suffices to take $h$, in an exponential $\fraks$-chart, to be the polynomial of scaled degree $k$ satisfying these conditions. 
To see that it is a right inverse of $p_{k,k-1}$, note that 
$$p_{k,k-1}\circ i_{k-1,k} j_p^{k-1} f= p_{k,k-1} j^k_p h = j_p^{k-1} h = j_p^{k-1} f\ . $$
To observe that $i_{k-1, k}$ is indeed a bundle morphism, it remains to show that it is smooth. Fix some point $p\in M$, and let $U\subset M$ be a neighbourhood of $p$, such that there exists an orthonormal frame of $TU$ consisting of $\fraks$-vector fields. Fix such a frame $(V_1,\ldots,V_d)$ and denote by $r_{\tilde{p}}\in (0,\infty)$ the injectivity radius of the exponential map at ${\tilde{p}}\in M$. Denote by 
$$\phi_{\tilde{p}}: B_{r_{\tilde{p}}}(0)\to M\ ,\qquad (\lambda_1,\ldots,\lambda_d)\mapsto \exp_{\tilde{p}} (\lambda_1V_1({\tilde{p}}),\ldots,\lambda_d V_d({\tilde{p}}))$$ the induced exponential chart at ${\tilde{p}}\in U$. Write $\neigh=\{({\tilde{p}},q)\in M\times M\ | \  d({\tilde{p}},q)<r_{\tilde{p}}\}$. It is well known that $\neigh\ni ({\tilde{p}},q)\mapsto (\phi_{\tilde{p}})^{-1}(q)$ is smooth.

Let $O:=\{({\tilde{p}},q,v_q)\in U\times JU \ | \ d({\tilde{p}},q)< r_{\tilde{p}} \}$ and define $$\psi^k: U\times JU \supset O\to U\times \Poly^k, \qquad ({\tilde{p}},v_q)\mapsto \theta^k_{\tilde{p}}(v_q) \ ,$$
where $\theta^k_{\tilde{p}}=(\phi^{-1}_{\tilde{p}})^*$ is defined as in (\ref{chart for jets}). Observe that $\psi^k$ is smooth, which follows from the fact that the exponential chart map is smooth in both entries and the explicit transformation rule for jets in charts from Section~\ref{jets in local coordinates}. 
We define $$\theta^k: JU \to U\times \Poly^k, \qquad ({\tilde{p}},v_{\tilde{p}})\mapsto \psi^k({\tilde{p}},v_{\tilde{p}})\;.$$
Then $i_{k-1,k}$ can locally be factored as
$$J^{k-1}M\supset J^{k-1}U \overset{\theta^{k-1}}\longrightarrow U\times \Poly^{k-1}\hookrightarrow U\times \Poly^{k} \overset{(\theta^{k})^{-1}}\longrightarrow  J^{k}U\subset  J^{k}M \ ,$$
and since all maps involved are smooth, the claim follows.
\end{proof}

By defining for $l<k-1$ the maps $i_{l,k}:= i_{l, l+1}\circ \ldots\circ i_{k-1,k}$, we have fixed an injection
$$J^l M\supset \ker p_{{l},{(l-1)}}\to J^k M\ .$$
We set $(J^k M)_l:= i_{l,k}(\ker p_{{l},{(l-1)}})$ and obtain a (canonical) grading 
\begin{equation}
J^k M = \bigoplus_{l\leq k} (J^k M)_l \ .
\end{equation}
For $k\geq l$ we shall write $Q^k_l$ for the canonical projection 
$$J^k M\to (J^k M)_l\subset J^k M\ ,$$ 
which can explicitly be written as 
\begin{equation}
Q^k_l=  i_{l,k}\circ p_{k,l} -  i_{l-1,k}\circ p_{k,l-1} \ .
\end{equation}

Before we construct a (canonical) grading on the jet bundle of a genuine vector bundle $E$, we turn to a short discussion about higher order covariant derivatives.
\subsubsection{Higher order derivatives}\label{section higher order derivatives}
Denote by $T^*M$ the dual space of the tangent space of $M$. This space naturally decomposes as 
$$T^*M= T^*M^1\oplus \ldots\oplus T^*M^n \ .$$ For a multiindex $l=(l^1,\ldots,l^n)\in \mathbb{N}^n$ define
$$ |l|_\fraks:=\sum_{i=1}^n \fraks_i l^i \qquad \text{ and }\qquad  |l|:=\sum_{i=1}^n  l^i\ .$$
and define
$$(T^*M)^{\otimes_s l }:= \bigotimes_{i=1}^n (T^*M^i)^{\otimes_s l^i} \ ,$$ where
$\otimes_s$ denotes the symmetric tensor product. For $m\in \mathbb{N}$ set 
$$(T^*M)^{m}= \bigoplus_{|l|_\mathfrak{s} =m} (T^*M)^{\otimes_s l } \ .$$
Define for a multi-index  $l\in \mathbb{N}^n$
\begin{equation}\label{canonical differential op}
D^l: \mathcal{C}^\infty(M) \to \cC^\infty\big( (T^*M)^{\otimes_s l }\big) \ \text{ respectively }\  D^l: \mathcal{C}^\infty(E) \to \cC^\infty\big((T^*M)^{\otimes_s l }\otimes E \big)
\end{equation}
to be given as follows:
For $f\in \mathcal{C}^\infty(M) $ , respectively  $f\in \mathcal{C}^\infty(E)$ 
$$D^l f:= \nabla^{|l|} f\big|_{(TM)^{\otimes_s l }}\ ,$$
where we implicitly use the canonical identification $\big((TM)^{\otimes_s l }\big)^*= (T^*M)^{\otimes_s l }$ . 
We define the map
\begin{equation}\label{differential operator}
\bar D_{\fraks}^k: \mathcal{C}^\infty(E)\to\cC^\infty\big( \bigoplus_{m=0}^k (T^*M)^m\otimes E \big)\ ,\qquad f\mapsto \sum_{|l|_\mathfrak{s} \leq k} D^l f \ .
\end{equation}

\begin{remark}
Note that in the absence of a non-trivial scaling $\bar{D}^k f= \bigoplus_{m=0}^k \nabla^m f$.
\end{remark}

\subsubsection{Grading on bundle valued jets}

Once we have \eqref{canonical differential op}, we have a natural injection $i_{k-1,k}: J^{k-1} E\to J^k E$
by setting, for any section $f\in \cC^\infty(E)$, $$ i_{k-1,k}(j^{k-1}_p f)= j_p^k h \ ,$$
where $h$ is a section of $E$ satisfying
\begin{enumerate}
\item $j_p^{k-1} h = j_p^{k-1} f$
\item for any multi-index $l\in \mathbb{N}^n$ with $|l|_{\mathfrak{s}}=k$, $\big(D^l h\big) (p)=0$ in $(T^*M)_p^{\otimes_s l}\otimes E_p$ . 
\end{enumerate}
\begin{remark}
Note that in the case of the trivial bundle $E=M\times \RR$ this definition of $i_{k-1,k}$ agrees with the one in (\ref{def of i}) since $M\times \RR$ is a flat bundle and thus $D^l= \nabla^{|l|}$ in this case.
\end{remark}
\begin{lemma}\label{i well defined}
For each $k$ the map  $i_{k-1,k}$ is a well defined bundle morphism, and furthermore it is a right inverse of the map $p_{k,k-1}$, i.e.\ $p_{k,k-1}\circ i_{k-1,k}= \id_{J^{k-1}E}$.
\end{lemma}

\begin{proof}
The proof proceeds similarly to the scalar valued case, but with some modifications.
For $p\in M$, we again use $r_p$ to denote the injectivity radius of the exponential map. Given an orthonormal basis $\{e_{i,p}\}$ of $E_p$, we construct the following ``special'' orthonormal frame on $B_{r_p}(p)$ where for $q\in B_{r_p}(p)$
\begin{equation}\label{paralell}
e_{i,p}(q)\in E_q \text{ is obtained by by parallel translation of } e_{i,p}
\end{equation}
along the shortest geodesic connecting $p$ to $q$. We denote by $\Phi^k_p$ the trivialisation of $\pi^{-1}(B_{r_p}(p))\subset E$ induced by the exponential chart and the frame $e_{i,p}(\cdot)$.

To show that the map $i_{k-1,k}$ is well defined, it suffices to take $h$ in the trivialisation $\Phi^k_p$ to be the (unique) 
appropriate polynomial of scaled degree $k$ and to note that in this trivialisation 
$$j_p^{k-1} h = j_p^{k-1} f$$
is a condition on the part of the polynomial of scaled degree up to $k-1$,  while 
$$\big(D^l h\big) (p)=0  \ \text{for each } l\in \mathbb{N}^n \text{ such that}\ |l|_{\mathfrak{s}}=k$$ 
is a condition on the part of the polynomial of (scaled) homogeneity exactly $k$.

The map $i_{k-1,k}$ is a right inverse of $p_{k,k-1}$ since
$$p_{k,k-1}\circ i_{k-1,k} j_p^{k-1} f= p_{k,k-1} j^k_p h = j_p^{k-1} h = j_p^{k-1} f\ . $$

To observe that $i_{k-1, k}$ is indeed a bundle morphism, it remains to show that it is smooth. We proceed by a slightly more elaborate version of the construction for the $J^k M$ case.
Fix some point $p_0\in M$, and let $U\subset M$ be a neighbourhood of $p_0$, such that there exists an orthonormal frame of $TU$ consisting of $\fraks$ vector fields and an orthonormal frame of $\pi^{-1}(U)\subset E$. Fix such frames $\{V_1(p),\ldots,V_d(p)\}$ of $T_p M$ and $\{e_{p,1},\ldots,e_{p,m}\}$ of $E|_p$ and denote by $r_p\in (0,\infty)$ the injectivity radius of the exponential map at $p$. Denote by 
$$\phi_p: B_{r_p}(0)\to M\ ,\quad (\lambda_1,\ldots,\lambda_d)\mapsto \exp_p (\lambda_1V_1(p),\ldots,\lambda_d V_d(p)) $$
 the induced exponential chart at $p\in U$. Write $\neigh=\{(p,q)\in M\times M\ | \  d(p,q)<r_p\}$ and recall that $\neigh \ni (p,q)\mapsto (\phi_p)^{-1}(q)$ is smooth. 
 
Denote for $(p,q)\in \neigh$ by $e_{p,i}(q)\in E_q$ the vector obtained by parallel translating $e_{p,i}$ to $q$ as in (\ref{paralell}) and note that the maps $ e_i: \neigh\to E,\ (p,q)\mapsto e_{p,i}(q)$ are smooth by construction. Write $\Phi_p:  \pi^{-1}(B_{r_p}(p))\to B_{r_p}(0)\times \mathbb{R}^n$ for the induced trivialisation.

Let $O:=\{(p,q,v_q)\in U\times\pi^{-1}(U) \ | \ d(p,q)< r_p \}$ and define 
$$\Psi^k: U\times JU \supset O\to U\times \Poly^k(\mathbb{R}^n),\quad \ (p,v_q)\mapsto \theta^k_p(v_q) \ ,$$
where $\theta^k_p:=(\Psi_p^{-1})^*$ is defined as in (\ref{bundle jet chart}). Observe that the map $\Psi^k$ is smooth, which follows from the fact that the exponential chart map as well as the maps $e_i$ are smooth (in both entries) and the explicit transformation rule for jets  under change of bundle trivialisations (\ref{equationn}) . 
We define $$\theta^k: JU \to U\times \Poly^k, \quad (p,v_p)\mapsto \Psi^k(p,v_p) \ .$$
 
Then $i_{k-1,k}$ can locally be factored as
$$J^{k-1}E\supset J^{k-1}\pi^{-1}(U) \overset{\theta^{k-1}}\longrightarrow U\times \Poly^{k-1}(\mathbb{R}^n)\hookrightarrow U\times \Poly^{k}(\mathbb{R}^n) \overset{(\theta^{k})^{-1}}\longrightarrow  J^{k}\pi^{-1}(U)\subset  J^{k}E \ $$
and since all involved maps are smooth, the claim follows.
\end{proof}

Exactly as in the scalar valued case, we set $i_{l,k}:= i_{l, l+1}\circ \ldots\circ i_{k-1,k}$ and thus single out a ``canonical'' injection
$$J^l E\supset \ker p_{{l},{(l-1)}}\to J^k E\ .$$
Set $(J^k E)_l:= i_{l,k}(\ker p_{{l},{(l-1)}})$ to obtain a grading 
\begin{equation}\label{grading of J^k}
J^k E = \bigoplus_{l\leq k} (J^k E)_l \ .
\end{equation}
For $k\geq l$ we shall write $Q^k_l$ for the associated canonical projection 
$$J^k E\to (J^k M)_l\subset J^k E\ ,$$ 
which can explicitly be written as 
\begin{equation}\label{explicit formula for jet projection}
Q^k_l=  i_{l,k}\circ p_{k,l} -  i_{l-1,k}\circ p_{k,l-1} \ .
\end{equation}

We define a map $${\pmb{D}}_\fraks^k: JE^k\to \bigoplus_{m=0}^k (T^*M)^m\otimes E$$ as the unique bundle morphism such that the following diagram commutes
\[\begin{tikzcd}
C^{\infty}(E) \arrow{r}{j^k_{\cdot}} \arrow[swap]{dr}{\bar{D}_\fraks^k}  & J^k E \arrow{d}{{\pmb{D}_\fraks^k}}\\
&\bigoplus_{m=0}^k (T^*M)^m\otimes E \ ,
\end{tikzcd}  
\] 
where $\bar{D}_\fraks^k$ was introduced in (\ref{differential operator}).
We recover the following (in the non-scaled setting classical) result, c.f.\ the Corollary after Theorem 7 in \cite[Chapter IV. 9]{Pal65}. The result also substantiates the claim that the grading constructed above is ``canonical''.
\begin{prop}\label{prop relation to cotangent}
The map $\pmb{D}^k$ is well defined and an isomorphism of graded vector bundles.
\end{prop}

\begin{proof}
We first construct and show uniqueness of the map $\pmb{D}^k_p$ which makes the following diagram commute.
\[\begin{tikzcd}
C^{\infty}(E) \arrow{r}{j_p^k} \arrow[swap]{dr}{\bar{D}^k|_p}  & J_p^k E \arrow{d}{\pmb{D}_p^k}\\
& \bigoplus_{m=0}^k (T_p^*M)^k\otimes E|_p \ .
\end{tikzcd}  
\] 
Indeed, existence, uniqueness and the fact that $\pmb{D}^k_p:   \bigoplus_{m=0}^k J_p^k E\to  (T_p^*M)^k\otimes E|_p$ is an isomorphism follow from the fact that $\ker j_p^k= \ker \bar{D}^k|_p$ as can, for example, be seen by going to an exponential chart at $p$. 
It follows directly from the construction of the grading that $\pmb{D}^k_p$ is graded.

To see that the thus defined map $\pmb{D}^k$ with the property 
$\pmb{D}^k|_p=\pmb{D}^k_p$ is indeed a bundle morphism, it suffices to check smoothness. This can for example be seen similarly as in the smoothness part of the proof of Lemma~\ref{i well defined} by using a map $\theta$ defined in a neighbourhood of $p$. Alternatively it follows directly from \cite[Chapter~IV.3, Theorem~3]{Pal65}.
\end{proof}

We fix the scalar product on $J^kE$ making $\pmb{D}^k$ an isomorphism of inner product bundles.
\begin{remark}\label{norm remark}
Let $(U,\phi)$ be an exponential $\fraks$-chart at $p$ and denote by $\partial_i$ the coordinate vector fields. Then one sees from the above construction, that 
\begin{align*}
|j^k_p f|_n:= |Q^k_n j^k_p f|&\lesssim \sup_{\deg(\partial_{i_1},\ldots,\partial_{i_m})=n} |(\nabla_{\partial_{i_1}}\circ\ldots\circ \nabla_{\partial_{i_m}} f)(p)|\\
&=\sup_{\deg(\partial_{i_1},\ldots,\partial_{i_m})=n} |(\nabla^m_{(\partial_{i_1},\ldots,\partial_{i_m})} f)(p)|  \ , 
\end{align*}
for every $f\in \mathcal{C}^\infty (E)$ .
\end{remark}

The following is a direct consequence of the preceding discussion.
\begin{corollary}\label{lemma gradedness exponential chart}
For a point $p$ denote by $\phi_p$ $\fraks$-normal coordinates in a neighbourhood $U$ of $p$. Then, the map $\theta_p:=(\phi_p^{-1})^*|_{J^k_p M}$ is graded with respect to the grading introduced in (\ref{grading of J^k}) .
Similarly, in the vector bundle valued case, denote by $\Phi_p$ a trivialisation of $\pi^{-1}(U)\subset E$ constructed as in (\ref{paralell}) from an orthonormal frame $\{e_{i,p}\}_i$ of $E|_p$ and an underlying $\fraks$-exponential chart. Then the map 
$\theta_p:=(\Phi^{-1}_p)^*|_{J^k_p M}$ is graded.
\end{corollary}

\begin{remark}
Note that the search for a canonical map $i_{k-1,k}$ in this section is equivalent to the search for a canonical splitting of the \textit{jet bundle exact sequence} 
 \[ \begin{tikzcd}
0  \arrow{r} & (T^*M)^k\otimes E \arrow{r}  & J^k(E) \arrow{r}{p_{k,k-1}} & J^{k-1}(E)   \arrow{r} & 0 \ ,
\end{tikzcd}
\] 
c.f.\ \cite[Chapter~IV, Thm.~1]{Pal65}.
A differential operator $\mathcal{A}\in \mathfrak{Dif}_k(E,T^*M)^k\otimes E)$ such that $T_\mathcal{A}$ splits the jet bundle exact sequence is called a \textit{$k$th total differential}.\footnote{It turns out that in the setting of a trivial scaling, $k$th total differentials can be characterised as exactly those differential operators $\mathcal{A}$ whose symbol, c.f.\ Section~\ref{section:excoursion}, is given by $\sigma_k(\mathcal{A})(v_p)(e_p)= v_p^{\otimes k}\otimes e_p$, see \cite[Chapter~IV.9, Thm.~5]{Pal65}. } The operator
$$\mathcal{A}=\big(\sum_{l\in \mathbb{N}^n : |l|_\fraks=k } D^l \big)\, : \mathcal{C}^\infty(E)\to \mathcal{C}^\infty((T^*M)_p^k\otimes E )\ ,$$ 
where $D^l$ was defined in (\ref{canonical differential op}), has this property. 
\end{remark}

\subsection{Infinite jet bundle}\label{section infinite jet bundle}
At this point it is natural to introduce the infinite jet bundle $J^\infty E$ as the inverse limit of the inverse system of vector bundles $(J^k E, p_{k,l})_{l,k\in \mathbb{N}: \ l\leq k}$. 
It is straightforward to see that $(J^k E, p_{k,l})_{l,k\in \mathbb{N}: \ l\leq k}$ is a perfect inverse system and that there is a canonical isomorphism $J_p^\infty E\simeq \mathcal{C}^\infty(E)/\sim_{\infty,p}$
where the equivalence relation $\sim_{\infty,p}$ is defined analogously to the relation $\sim_{k,p}$ introduced in Section~\ref{basic facts on jets}. This in particular implies that for each $k\in \mathbb{N}$ there is a unique projection 
$$p_{\infty,k}: J^\infty E\to J^kE$$
such that the family of maps $\{p_{\infty, k}\}_{k\in \mathbb{N}}$ satisfies the identity $p_{\infty,l}=p_{\infty,k}\circ  p_{k,l}$ for all $l\leq k$. The analogous statement also holds for the maps $i_{k,\infty}$.

In particular, it follows that the maps $\{Q^k_l\}_{k,l\in \mathbb{N}}$ introduced in \eqref{explicit formula for jet projection} extend to a map $Q_l:J^\infty E\to J^\infty E$ given by 
\begin{equation}\label{def of proj}
Q_l=i_{l,\infty}\circ p_{\infty,l} -  i_{l-1,\infty}\circ p_{\infty,l-1} \ .
\end{equation}

We define the spaces $(JE)_l= Q_l J^\infty E$, which come with a natural scalar scalar product induced from the one on $J^kE$ .
We define the graded space
$$JE:= \bigoplus_{m=0}^\infty (JE)_m \ ,$$
which in view of Proposition~\ref{prop relation to cotangent} is canonically isomorphic to $\bigoplus_{m=0}^\infty (T^*M)^m\otimes E$.
Note the strict inclusion
$JE\subsetneq J^\infty E$.

\begin{remark}
The reason for introducing the space $JE\subset J^\infty E$ is that working with $J^\infty E$ would introduce non-trivial convergence considerations in subsequent sections, while this can be circumvented on $JE$. In principle it would have been possible to work with $J^k M$ for some big enough $k\in \mathbb{N}$ throughout this article, but our decision to work with $JE$ seems more natural from the point of view of \cite{Hai14}, as $J_pE$ is the canonical analogue of the space of polynomials of arbitrary but finite degree at $p$, while $J_p^\infty$ allows for infinite power series.
\end{remark}

\begin{remark}
From now on we shall use the notation
$$j_p f:= j^\infty_p f \ ,$$
since this will make many expressions less cluttered. Note that $j_p f\notin J_p M$ in general, but this will cause no confusion, since the space $J^\infty M$ will play no role from Section~\ref{section Regularity structures and models} onwards.
\end{remark}
\begin{remark}\label{rem isometry}
One can see that $JE$ and $JM\otimes E$ are canonically isometrically isomorphic, since it follows from Proposition~\ref{prop relation to cotangent} that for each $k$ the vector bundles $J^k E$ and $J^kM \otimes E$ are isometrically isomorphic. 
\end{remark}

\subsection{Jets and analytic bounds}
We make some remarks on the growth of a function $f$ depending on its jet $j^k_p f$.
\begin{lemma}\label{pointwise bound}
Suppose $W=(W_1,\ldots,W_n)$ is a tuple of $\fraks$-vector fields defined on $B_1(p)$. Then there exists a constant (depending on $W$), such that the following bound holds uniformly over $f, g\in \mathcal{C}^\infty(M)$ satisfying $j^k_p f= j^k_p g$
\begin{align*}
 |(\nabla_{W_1}\cdots \nabla_{W_n} f)(q)&-(\nabla_{W_1}\cdots \nabla_{W_n} g)(q)|\\
 & \lesssim_{W}  d_\fraks(p,q)^{\lfloor k+1-\deg_\fraks(W)\rfloor\vee 0}\ \big\| |j_\cdot^{\lfloor k+\max_i \fraks_i\rfloor\vee \deg(W)} (f-g)|\big\|_{L^\infty(B_1)} ,
\end{align*}
where $\deg_\fraks(W)$ was defined in (\ref{deg of vec}) . 
\end{lemma}
\begin{proof}

We show the claim for $|j_p^k f|=|j_p^k f|=1$, the general case follows from linearity. We begin with the case $n=0$:
Let $h= f-g$, since $j^k_p h=0$ the discussion in this Section~\ref{jets in local coordinates} implies that in any $\fraks$-coordinate chart $(U,\phi)$ around $p$ the Taylor polynomial $j^k_{\phi(p)} (h\circ \phi^{-1})$ of order $k$ vanishes. Thus it follows from \cite[Prop. A1]{Hai14}, that one has a bound of the form $|h\circ \phi^{-1}(x)|\lesssim |x-\phi(p)|_\fraks^{k+1}$, which in turn implies the bound.
In the bundle valued case the same argument applies replacing the $\fraks$-chart $(U,\phi)$ by a trivialisation.

If $n>0$, note that $j^{\lfloor k-\text{deg}_\fraks(W)\rfloor}_p(\nabla_{W_1}\cdots \nabla_{W_n} h)=0$ by (\ref{differentiation of jets}) and thus the claim follows as above.
\end{proof}

\subsection{Admissible realisations}
In this section we define a notion of ``admissible realisation'' of $J E$ for a smooth vector bundle $E$ which will lay the groundwork for defining the analogue of the polynomial regularity structure. In short, an admissible realisation of $J E$ associates to each element $v\in J E$ a concrete function. We shall first define admissible realisations on $J^kE$ and obtain admissible realisations of $JE$ by a (soft) limiting argument. The only place where more than the smooth structure on $E$ is needed is when formulating quantitative analytic bounds on admissible realisations (Lemma~\ref{easy crucial bound}, Lemma~\ref{bound on admissable realisations} and Lemma~\ref{comparing admissible realisations}).

\begin{definition}\label{def admissible realisation}
We call a map $R^k: J^k E \to \mathcal{C}^\infty(E)$ an admissible realisation of $J^k E$ if the following properties are satisfied:
\begin{itemize}
\item For each $p\in M$, the map $R^k|_{J^k_pE}$ is a right inverse of $j_p^k$, i.e.\  $j_p^k \circ R^k|_{J^k_pE}= \id|_{J^k_p E}$.
\item The map $R^k$ is linear on each fiber of $J^k E$.
\item The map $J^k E \times M \to E, ((p, v_p),q) \mapsto R^k( (p, v_p))(q)$ is smooth.
\end{itemize}
We shall often just write $R^k(v_p)$ instead of $R^k(p,v_p)$ or $R^k_{E|_p}(v_p)$.
\end{definition}

Let $F\mapsto N$ be a vector bundle and $\Phi:E\to F$ bundle isomorphism. Then the notion of pullback of sections $\Phi^*: \mathcal{C}^\infty (F)\to \mathcal{C}^\infty (E)$ induces the notion of pullback $\Phi^*: JF\to JE$ discussed in Section~\ref{section pullback of jets}. This in turn induces a 
notion of pullback of admissible realisations.
\begin{definition}\label{Pushforward of admissible realisations}
 Let $\Phi:E\to F$ be a bundle isomorphism  and
$R^k$ be an admissible realisation of $JF$, then $\Phi^* R^k$ is defined as the unique map $JE\to \mathcal{C}^\infty(E)$ such that the following diagram commutes
\[ \begin{tikzcd}
J^kF  \arrow{r}{\Phi^*} \arrow[swap]{d}{R^k} &J^kE \arrow{d}{\Phi^* R^k} \\%
\mathcal{C}^\infty (F) \arrow{r}{\Phi^*}& \mathcal{C}^\infty (E)
\end{tikzcd}
\]
\end{definition}
Note that this diagram defines the map $\Phi^* E^k$ uniquely and that the smoothness and linearity conditions in Definition~\ref{def admissible realisation} are trivially satisfied. To conclude that $\Phi^* R^k$ is an admissible realisation, it suffices to observe that for $p\in M$, $v_p\in J_pE$ one has $j_p^k (\Phi^* R^k(v_p))=v_p$ . Indeed, writing $\bar{p}= \phi(p)$, one has
$$j_p^k ((\Phi^* R^k)(v_p)) = j_p^k (\Phi^* (R^k ((\Phi^{-1})^* v_p)))= \Phi^*j_{\bar{p}}^k R^k((\Phi^{-1})^* v_p)=\Phi^* (\Phi^{-1})^* v_p=v_p\ .$$

\begin{remark}
One can see that the notion of pullback of admissible realisations also makes sense if $\Phi=(\phi, G)$ is a bundle morphism restricting to isomorphisms on fibres and such that $\phi$ is an immersion.
\end{remark}

\begin{prop}\label{existence of admissible realisations}
For any vector bundle $E\to M$ and any $k\in \mathbb{N}$ there exists an admissible realisation of $J^k E$. 
\end{prop}
\begin{proof}
Let us fix some function $\psi: \mathcal{C}^\infty_c (B_1(0))$, such that $\psi|_{B_{1/2}(0)} = 1$. For $r>0$ we denote by $\psi_r$ the map given by $\psi_r(x)= \psi(\frac{x}{r})$ and extended to be vanish outside of $B_r(0)$. 

First assume that $V \subset W\subset \mathbb{R}^d$ are such that $\text{dist}(V,W^c)>r$ for some $r>0$. 
Then we define the map,
$$ R_{V,W,r}: V\times \Poly^k(\mathbb{R}^m)\to \mathcal{C}^\infty(W)\otimes\mathbb{R}^n,\quad (p, v_p)\mapsto \psi_r (\cdot - p)(\Pi_p v_p)(\cdot)\ ,$$
where $\Pi_p$ is the map that maps an abstract polynomial in $\Poly^k(\mathbb{R}^n)$ to the corresponding concrete $\mathbb{R}^n$ valued polynomial function centred at $p$.
Note that 
\begin{enumerate}
\item $j_p^k R_{V,W,r}(v_p) = v_p$ for all $p\in V$,
\item\label{1} the map $R_{V,W,r}$ is $\mathcal{C}^\infty(V)$-linear,
\item\label{2} the map $J^k(V\times \mathbb{R}^m)\times W \to \mathbb{R}^m,\ ((p,v_p), q)\mapsto R_{V,W,r}(v_p)(q)$ is smooth.
\end{enumerate}

Now we construct an admissible realisation on an arbitrary vector bundle $E\to M$. Let $\{(U_i,\phi_i)\}_{i\in I}$ be a locally finite atlas of $M$ such that there exist $V_i\subset \phi (U_i)$,  satisfying $\text{dist}(V_i,\phi(U_i)^c):=r_i>0$ and $M=\bigcup_i \phi_i^{-1}(V_i)$ and such that there exists a partition of unity $\mathbb{1}_i: M\to [0,1]$ subordinate to the cover $\{V_i\}$ as well as trivialisations $\Phi_i= (\phi_i, G_i)$ of $\pi^{-1}(U_i)$. Then we define $R^k: JE \to C_c^\infty(E)$ by the following formula:
\begin{equation}\label{formula for admissible realisation}
R^k(p,v_p)(q):= \sum_i  \mathbb{1}_i(p)\, \Phi_i^*\big(R_{V_i,\phi_i(U_i),r_i}(\theta^k_i v_p )\big)(q) \,
\end{equation}
where $\theta^k_i:= (\Phi_i^{-1})^* : J^k(\pi^{-1}(V))\to V\times \Poly^k(\mathbb{R}^m)$  .

Next we check that this is indeed an admissible realisation. The second and third condition in Definition~\ref{def admissible realisation} follow directly from (\ref{1}) and (\ref{2}) above.
For the first condition note that one has 
$$j^k_p\big(\Phi_i^*R_{V_i,\phi_i(U_i),r_i}((\Phi_i^{-1})^* v_p )\big) = v_p ,$$
by definition. 
Thus it suffices to show that, if $f , f_1,\ldots,f_n$ are smooth
sections of a vector bundle defined in a neighbourhood of $p\in M$ satisfying $j^k_p f = j^k_p f_1 =\ldots= j^k_p f_n$ and $\eta_1,\ldots,\eta_n$ are smooth functions satisfying $\sum_{i=1}^n \eta=1$ in a neighbourhood of 
$p$, then $$j_p^k \big(\sum_{i=1}^n \eta_i f_i\big) = j^k_p f\ .$$ But since this holds on the trivial bundle $\mathbb{R}^d\times \mathbb{R}^n$ it follows in the general case.
\end{proof}
The next lemma contains bounds which are used for the construction of the analogue of the polynomial model. Assume that $M$ is Riemannian and $(E, \nabla, \langle \cdot , \cdot\rangle)$ is as in Section~\ref{section bundles and metric} and recall for $k>l$ the map $Q^k_l:  J^k E\to (J^k E)_l$ from (\ref{def of proj}).
\begin{lemma}\label{easy crucial bound}
Let $R^k$ be an admissible realisation of $J^k E$. For every compact set $K\subset M$ the bound
$$|Q^k_m j_q^k R(v_p)|_m\lesssim_K d_\fraks(p,q)^{(n-m)\vee 0}|v_p|_n$$
holds uniformly over $p,q\in K$,  $m,n\leq k$ and $v_p\in (J^k_pE)_n$.
\end{lemma}
\begin{proof}
Let $(U,\phi)$ be a chart around $p$ and $\Phi$ a compatible trivialisation of $\pi^{-1}(U)$. Note that it suffices to show the claim for $q\in U$, by continuity and compactness if follows for $q$ further away from $p$. Since $(\Phi^{-1})^*R(v_p) $ is a smooth function on $\phi(U)$ vanishing of order $n$ at $\phi(p)$, it follows that the coefficients in front of monomials of homogeneity $m$ of the abstract polynomial $j_{\phi(q)}^k (\Phi^{-1})^*R(v_p)$ vanish of order $n-m$. Since the map $\Phi^*$ is upper triangular, the same is true for $ j^k_q R ( v_p)$.
\end{proof}

\begin{definition}
We call a family of admissible realisations $\{R^k\}_{k\in\mathbb{N}}$ for ${J^k E}$ consistent, if for each $k\geq l$ and each map $p_{k,l}:J^k E \to J^lE$ introduced at the beginning of Section~\ref{section grading on jets} one has the identity
$$R^l= R^k \circ p_{k,l} \ .$$
We call a map $R: JE\to \mathcal{C}^\infty(E)$ an admissible realisation of $JE$ if there exists a family of admissible realisations  $\{R^k\}_{k\in\mathbb{N}}$ of $J^k E$ such that for each $k\geq l$ one has the identity 
\begin{equation}\label{admissible re}
R|_{(JE)_l} = R^k\circ p_{\infty,k}|_{(JE)_l}\ .
\end{equation}
\end{definition}
\begin{remark}
Note that there is an obvious one to one correspondence between consistent families of  admissible realisations of $\{J^kE\}_k$ and admissible realisations of $JE$. In particular every consistent family of  admissible realisations gives rise to an admissible realisation of $JE$.
\end{remark}

\begin{prop}
For any vector bundle $E\to M$ there exists a family of admissible realisations $\{R^k\}_k$ which is consistent and therefore there exists an admissible realisation of $JE$. 
\end{prop}
\begin{proof}
In view of the previous remark it suffices to show that there exists a consistent family of admissible realisations.
Note that the admissible realisations constructed for each $k\in \mathbb{N}$ in the proof of Proposition~\ref{existence of admissible realisations} would form a consistent admissible realisation, if we could replace each $\theta_i^k$ by $\theta_i^\infty$ in equation (\ref{formula for admissible realisation}). 
To this end, define for $V,W$ as in the proof of Proposition~\ref{existence of admissible realisations} and any sequence $\{r_n\}_{n\in \mathbb{N}}$ such that $d(V,W^c)>r_n$ for all $n$ the following partial map

$$R_{V,W,\{r^n\}_{n\in \mathbb{N}}}: J^\infty(V\times \mathbb{R}^m) \nrightarrow \mathcal{C}^\infty(W)\otimes \mathbb{R}^m$$ 
which on $J (V\times \mathbb{R}^m)_n\subset V\times P(\mathbb{R}^m)$ agrees with $R_{V,W,r_n}$.

Now the point is, that depending on the chart maps and trivialisation $(U_i,\Phi_i)$ one can choose sequences $\{r_{n,i}\}_{n\in \mathbb{N}}$ such that no convergence issues arise when setting $k=\infty$ and replacing $R_{V_i,\phi_i(U_i),r_i}$ by $R_{V_i,\phi_i(U_i),\{r_{n,i}\}_{n\in \mathbb{N}}}$ in  (\ref{formula for admissible realisation}).
\end{proof}

Note that the bounds obtained for admissible realisations of $J^k E$ in Lemma~\ref{easy crucial bound} translate directly into bounds for admissible realisations $R$ of $JE$.
\begin{lemma}\label{bound on admissable realisations}
Let $R$ be an admissible realisation of $JE$.
For every $k\in \mathbb{N}$ and $K\subset M$ compact, there exists a constant $C_{k,K}$ such that for all $n,m\leq k$ the bound
\begin{equation}\label{norm on admissible }
|j_q R(v_p)|_m \lesssim C_{k,K} d_\fraks (p,q)^{(n-m)\vee 0} |v_p|
\end{equation}
holds uniformly over $p,q\in K$ and
$v_p\in (T_p E)_n$ .
\end{lemma}
We point out the following direct consequence of Lemma~\ref{pointwise bound} .
\begin{lemma}\label{comparing admissible realisations}
Let $R, \tilde R$ be two admissible realisations of $J E$, then for every compact set $K$ and $n\in \mathbb{N}$
$$|R(\tau_p)(q)-\tilde R(\tau_p) (q)|\lesssim_{R, \tilde R, K} |\tau | d_{\fraks}(p,q)^{k+n}$$
uniformly over $p,q \in K$, $\tau_p\in (J_p E)_{<k}$ .
\end{lemma}

\subsubsection{Further conditions on admissible realisations}\label{remark algebraic structure}
We point out that the notion of admissible realisation is rather general and depending on the context one might want to impose further conditions such as the following.
\begin{enumerate}
\item\label{remark algebraic structure JM} Recall, that $JM$ is has a natural algebra-bundle structure. Thus one might want impose that $R$ is an algebra-bundle morphism. Indeed, the existence of such admissible realisation is clear: since given any admissible realisation $\tilde{R}$ of $JM$ one can simply define $R$ fibre wise as the unique algebra morphism satisfying $R(\mathbf{1}_p)=1$ where $\mathbf{1}_p:=j_p 1 \in (JM_p)_0$ and $R|_{(JM)_1}= \tilde{R}_{(JM)_1}$.
\item\label{remark algebraic structure JE} Recall that $JE$ is naturally a $JM$ module, i.e there is a fibre wise multiplication. Suppose one is given an admissible realisation $R$ of $JM$ which is a bundle algebra morphism. Then, one could restrict attention to admissible realisations $\bar{R}$ of $JE$ which are compatible with this multiplicative structure, i.e.\ such that the following diagram commutes
\[ \begin{tikzcd}
JM\otimes JE  \arrow{r}{m} \arrow[swap]{d}{R\otimes \bar{R}} & J E \arrow{d}{\bar{R}} \\%
\mathcal{C}^\infty (M)\otimes \mathcal{C}^\infty(E) \arrow{r}{m} & \mathcal{C}^\infty (E)\ .
\end{tikzcd}
\]

\item\label{tensor products on jets} More generally, given two vector bundles $E$ and $F$, recall that in general there is a canonical product 
$$m: JE\otimes JF\to J(E\otimes F) \ ,$$
specified by the fact that for $e\in \mathcal{C}^\infty(E)$, $f\in \mathcal{C}^\infty(F)$ one has $$Q_k j(e\otimes g) =\sum_{k=m+n}  m(Q_{n} je, Q_m jf)\ .$$ Similarly to above one might restrict attention to admissible realisations $R, R'$ and $R''$ such that the following diagram commutes
\[ \begin{tikzcd}
JE\otimes JF  \arrow{r}{m} \arrow[swap]{d}{R\otimes {R'}} & J (E\otimes F) \arrow{d}{{R''}} \\%
\mathcal{C}^\infty (E)\otimes \mathcal{C}^\infty(F) \arrow{r}{} & \mathcal{C}^\infty (E\otimes F)\ ,
\end{tikzcd}
\]
where the lower arrow denotes the pointwise tensor product.
\item\label{compatibility with trace} 
Observe that for any vector bundles $E, F$, the canonical trace map $\tr: E\otimes E^* \otimes F\to F$ extends to a trace map 
$\trr  : J(E\otimes E^* \otimes F)\to JF$ on jets, where the latter is characterised by the fact that for $g\in \mathcal{C}^\infty(E\otimes E^* \otimes F)$ one has $ \trr j g = j \tr g $ .
\end{enumerate}
For a fixed manifold $M$, we give a universal construction of an admissible realisation for any vector bundle (with a connection) such that the above properties are always satisfied.
First, for a vector bundle $E\to M$, recall the isomorphism $JE\simeq JM\otimes E$.
\begin{enumerate}
\item Fix a multiplicative admissible realisation $\tilde{R}$ as in Point~\ref{remark algebraic structure JM} of $JM$ with the property that $\supp(v_p)\subset B_{r_p}$ for all $v_p$ in $J_pM$ where $r_p$ denotes the exponential radius at $p$. 
\item For $e_p\in E|_p$ denote by $e_p(\cdot)$ the vector field obtained by parallel transport as in (\ref{paralell}). 
\item Define $R: JE\simeq  JM\otimes E \to \mathcal{C}^\infty (E)$ by
$$JM\times E\to \mathcal{C}^\infty (E), \qquad  (v_p, e_p)\mapsto R(v_p)(\cdot)e_p(\cdot)\ .$$
\end{enumerate}
It is straightforward to check that the above properties are satisfied.


\section{Regularity Structures and Models on Vector Bundles}\label{section Regularity structures and models}
As already seen in the introduction, when solving equations such as \eqref{equation to solve}, we need to work with several vector bundles simultaneously. In order to do so we first introduce some convenient notion.
\subsection{Maps between families of spaces}\label{sec:maps_between_famlillies_of_spaces}
For a family of vector spaces or vector bundles $V=\{V_i\}_{i\in I}$ we shall write $v\in V$ to mean $v\in \bigcup_i V_i$. 
Given a second family $W= \{W_j\}_{j\in J}$ and a map $\sigma: I\to J$ we define 
\begin{equ}[e:defLVW]
L_\sigma (V,W):=  \prod_{i \in I} L(V_i,W_{\sigma(i)}) = \prod_{i \in I}  W_{\sigma(i)}\otimes V_i^*, \quad \hat{L}_\sigma (V,W):= \prod_{i \in I}  W_{\sigma(i)}\hotimes V_i^*\ .
\end{equ}
Given $A\in L_\sigma(V,W)$ and $v \in V$, we write $Av$ to denote $A_{i}v$, where $i\in I$ is the index such 
that $v\in V_i$. (In particular $v \in V_i$ implies that $Av \in W_{\sigma(i)}$.)
For a third family $U=\{U_k\}_{k\in K}$ and $\sigma': K\to I$ and $B\in L_{\sigma'}(U,V)$ the composition 
$AB\in L_{\sigma\circ \sigma'}(V,W)$ is defined by component-wise composition, i.e.\ $(AB)_{k}= A_{\sigma'(k)} \circ B_{k}$.

Furthermore, if the index sets agree, i.e.\ $I=J$ and $\sigma=\id$, we simply write $L(V,W)$ and $\hat{L} (V,W)$ instead of $L_{\id}(V,W)$ and $\hat{L}_{\id} (V,W)$.
We use a similar notation for spaces of sections, such as $\mathcal{D}(V)= \{ \mathcal{D} (V_i) \}_{i\in I}$ and $\mathcal{D}'(V)= \{ \mathcal{D}' (V_i) \}_{i\in I}$ for vector bundle valued test functions and distributions.

\begin{remark}\label{remark useful notation}
We shall sometimes also use the notation $L(V,W)$ where $V$ is a family of vector bundles, while $W$ is only a family of vector spaces. In this case we implicitly identify $W_j$ with the trivial bundle over one point and mean $L(V,W):= \hat{L}(V,W)$.
\end{remark} 

\begin{remark}
In this article, when working with the notions introduced above, we shall often work with finite families of finite dimensional vector bundles. 
In this case one has the canonical isomorphisms 
$$L_\sigma (V,W) \simeq \bigoplus_{i \in I}  W_{\sigma(i)}\otimes V_i^*, \quad \hat{L}_\sigma (V,W)\simeq \bigoplus_{i \in I}  W_{\sigma(i)}\hotimes V_i^*\ . $$
\end{remark}

%

\subsection{Regularity structure ensemble}
We henceforth assume that the manifold $M$ is equipped with a Riemannian structure and any vector bundle $E\to M$ is equipped with a metric and compatible connection as in Section~\ref{section bundles and metric}. 

The appearance of several vector bundles in equations motivates the definition of a regularity structure ensemble.
\begin{definition}\label{definition regularity structure}
Let $M$ be a smooth manifold, a regularity structure ensemble consists of an indexing set $\mathfrak{L}$ and a triple $\mathcal{T}=(\{{E}^\mathfrak{t}\}_{\mathfrak{t}\in \mathfrak{L}},\{{T}^\mathfrak{t}\}_{\mathfrak{t}\in \mathfrak{L}}, L)$ consisting of the following elements:
\begin{enumerate}
\item For each $\frakt\in \mathfrak{L}$, ${E}^\mathfrak{t}$ is a finite dimensional vector bundle over $M$ equipped with a metric and a compatible connection.
\item\label{condition on T} For each $\frakt\in \mathfrak{L}$, $T^{\frakt}= \bigoplus_{(\alpha,\delta)\in A\times \triangle} T^{\frakt}_{\alpha,\delta}$ is a finite dimensional vector bundle over $M$ bi-graded by $A\subset \RR$ and $\triangle\subset \RR\cup \{+\infty\}$ which are bounded from below, discrete sets.
\item\label{condition on T_0} The bundle $T^{\frakt}_{0,:}:= \bigoplus_{\delta\in \triangle} T^{\frakt}_{0,\delta}$ comes with a canonical isomorphism $T^{\frakt}_0\simeq E^{\frakt}$.
\item For every $\alpha\in A$ the bundle $T^\frakt_{\alpha,\infty}$ is trivial.
\item\label{condition on L} $L$ is a sub-bundle of $\hat{L}(T,T)$ such that for every $p,q,r\in M$ composition maps $L_{p,q}\times L_{q, r}$ to $L_{p,r}$.
\end{enumerate}
\end{definition}

\begin{remark}
Let us elaborate on the definition of a regularity structure ensemble:
\begin{itemize}
\item The spaces $T^\frakt$ and the $A$-grading play the same role as in \cite{Hai14}. A model, c.f.\ Definition~\ref{def model}, shall turn elements of $T^\frakt$ into $E^\frakt$-valued distributions. The second grading $\triangle$ plays the role of ``precision'', c.f.\ again Definition~\ref{def model} and the definition of a model in \cite{DDD19}. In contrast to \cite{DDD19} we allow different values of $\delta$ for different elements of $T^\frakt$, this is natural when introducing products on regularity structures, c.f.\ Section~\ref{section Local operations}, and when renormalising models, c.f.\ Section~\ref{section construction of models}.
\item The isomorphism Point~(\ref{condition on T_0}) is crucial when extending non-linear functions on the vector bundle $E^\frakt$ to functions on $T^\frakt$, c.f.\ Section~\ref{section nonlinearities}. (Using the vocabulary introduced below, $T^\frakt_0$ will always be part of a ``jet sector''.)
\item The reason that the action of $L$ on $T^\frakt$ is in general not lower triangular in $A$ stems from the fact that when $\delta<+\infty$ an additional ``higher order'' term appears when lifting singular kernels, c.f.\ Section~\ref{Section Singular Kernels} and in particular Remark~\ref{philosophical remark}.
\end{itemize}
\end{remark}

\begin{remark}
In order to lighten notation, whenever we write $\tau_p\in T^\frakt_{\alpha,\delta}$ the subscript $p$ signifies that $\tau_p$ belongs to the fibre of $T^\frakt_{\alpha,\delta}$ above $p\in M$.
\end{remark}

\begin{remark}
We set $T^{\frakt}_{\alpha,:}:=  \bigoplus_{\delta \in \triangle} T_{\alpha,\delta}^\frakt$ and $T_{:,\delta}^\frakt=\bigoplus_{\alpha\in A} T^\frakt_{\alpha,\delta}$. 
For $\gamma\in \mathbb{R}$ we introduce 
$T^{\frakt,\gamma}:=\bigoplus_{\alpha<\gamma< \delta} T^\frakt_{\alpha,\delta}\ .$
Lastly, for $\alpha\in A$ we define the canonical projection maps $Q_{\alpha}: T^\frakt\to T_{\alpha, :}^\frakt$.

\end{remark}

Analogously to \cite[Def 2.5]{Hai14}, we can define what a sector of a regularity structure ensemble is.
\begin{definition}
Given a regularity structure ensemble $\mathcal{T}=(\{E^{\mathfrak{t}}\}_{t\in \mathfrak{L}},\{T^{\mathfrak{t}}\}_{t\in \mathfrak{L}}, L)$, a sector ensemble of regularity $\underline{\alpha}\in \mathbb{R}^{\mathfrak{L}}$ and precision $\delta\in \mathbb{R}\cup \{+\infty\}$ is a family $\{V^\frakt\}_{\frakt \in \mathfrak{L}}$ of bi-graded subbundles $V^\frakt=\bigoplus_{\alpha'\in A, \delta' \in \triangle} V^\frakt_{\alpha', \delta'}\subset T^\frakt$ with $V^\frakt_{\alpha', \delta'}\subset T^\frakt_{\alpha', \delta'}$ satisfying the following properties for each $\frakt \in \mathfrak{L}$
\begin{itemize}
\item $V^{\frakt}_{\alpha', \delta'} = M\times \{0\}\subset T^{\frakt}_{\alpha', \delta'}$ whenever $\alpha'<\underline{\alpha}_\mathfrak{t}$ or $\delta'<\delta$.
\item For all $p,q\in M$ and all $\Gamma\in L_{p,q}$ one has $\Gamma_{p,q}V^{\frakt}_q\subset V^{\frakt}_p$ .
\item For every $\alpha'\in A$, $\delta'\in \triangle$,
there exists a subbundle $\bar{V}^{\frakt}_{\alpha', \delta'}\subset T^{\frakt}_{\alpha', \delta'}$, such that
$$T^{\frakt}_{\alpha', \delta'} = V^{\frakt}_{\alpha', \delta'}\oplus  \bar{V}^{\frakt}_{\alpha', \delta'}\ .$$
\end{itemize}
As usual in the theory of regularity structures, we call the bundles $V^\frakt$ sectors of 
regularity $\underline{\alpha}_\frakt$. If $\underline{\alpha}=0$, then the sector (or 
sector ensemble) is called function-like.
\end{definition}
\begin{definition}
We say that a regularity structure ensemble 
$\mathcal{T}=(\{E^{\mathfrak{t}}\}_{t\in \mathfrak{L}},\{T^{\mathfrak{t}}\}_{t\in \mathfrak{L}}, L)$
 is contained in a regularity structure ensemble
 $\hat{\mathcal{T}}=(\{E^{\mathfrak{t}}\}_{t\in \hat{\mathfrak{L}}},\{\hat{T}^{\mathfrak{t}}\}_{t\in \hat{\mathfrak{L}}},\hat L)$
  if
\begin{itemize}
\item one has $\mathfrak{L}\subset \hat{\mathfrak{L}}$ ,
\item for each $\frakt\in \mathfrak{L}$ there is a graded injection $\iota: T^{\frakt} \to \hat T^{\frakt}$ ,
\item the spaces $\iota T^\frakt\subset \hat{T}^\frakt$ are sectors and $\iota^{-1} \circ \hat L \circ \iota = L$.
\end{itemize}
\end{definition}

Before we turn to the definition of a model, we introduce the following useful space of test functions. Let $V\to M$ be a vector bundle (equipped with the usual structure), for $r\in \mathbb{N}, \lambda\in (0,1), p\in M$ let 
\begin{equation}\label{good textfunctions}
\mathcal{B}_p^{r,\lambda}(V) := \Big \{\phi \in \cC_c^\infty\big(V \big)\  \Big| \ \supp( \phi) \subset B_\lambda(p) \ , \ 
 \| |j^r_\cdot \phi|_k \|_{L^\infty}<\frac{1}{\lambda^{|\fraks|+k}} \text{ for all } k\leq r\Big\}\ .
\end{equation}
We shall often write $\mathcal{B}_p^{r,\lambda}$ or $\mathcal{B}_p^{\lambda}$ instead of $\mathcal{B}_p^{r,\lambda}(V)$ if $r$ resp. $V$ is fixed. 
Furthermore, instead of just writing $\phi\in\mathcal{B}_p^{r,\lambda}$, we shall sometimes use the notation $\phi_p^\lambda$ to \textit{emphasise} $\lambda\in (0,1]$ and $p\in M$.
Lastly, we use the usual notational convention for families of vector bundles.
\begin{remark}\label{decomposing test functions}
Instead of requiring bounds on the jet of a section, one could require bounds on the (covariant) derivatives in (\ref{good textfunctions}). This is equivalent, as can be seen from the discussion around Proposition~\ref{prop relation to cotangent}.
\end{remark}

\begin{remark}
We point out the following useful property of the spaces $\mathcal{B}_p^{r,\lambda}$. For every $c\in (0,1)$ there exists $N\in \mathbb{N}$, such that the following property holds: For each $p$ there exist $N$ points $p_1,\ldots,p_N$, such that 
each $\phi^\lambda_p \in \mathcal{B}_p^{r,\lambda}$ can be written as
$$\phi^\lambda_p= \sum_{i=1}^N \phi^{c\lambda}_{p_i}\ ,$$
where $\phi^{c\lambda}_{p_i}\in \mathcal{B}_{p_i}^{r,c\lambda}$. The points $p_i\in M$ can furthermore be chosen to satisfy $d_\fraks (p, p_i)\leq \lambda $.
\end{remark}
\begin{remark}\label{translated test function remark}
Fix some constant $D>1$ and $K\subset M$ compact, observe that there exists a constant $C=C(D,K)$ such that uniformly over $\lambda <\frac{1}{D}$, $p\in K$ and $d(p,\bar{p})<D$ one has $$\phi^\lambda_{\bar{p}} \in C\mathcal{B}_p^{r,D\lambda} \ .$$
\end{remark}

\begin{remark}\label{rem frame testfunctions}
Let $\{e_i\}_i$ be a frame of $\pi^{-1}(U)\subset V$. Then there exist $C, C'$ depending on this frame, such that 
$$\phi^\lambda_p= \sum \phi^{\lambda,i}_p e_i  \in \mathcal{B}_p^{r,\lambda}(V)\quad \Rightarrow\quad  \phi^{\lambda,i}_p\in C\mathcal{B}_{p}^{r,\lambda}(M) \text{ for all } i$$
and 
$$ \phi^{\lambda,i}_p\in \mathcal{B}_{p}^{r,\lambda}(M)\text{ for all } i\quad \Rightarrow\quad \phi^\lambda_p= \sum \phi^{\lambda,i}_p e_i  \in C'\mathcal{B}_p^{r,\lambda}(E) \ . $$
\end{remark}

\begin{definition}\label{def model}
A model for a regularity structure ensemble $\mathcal{T}$ consists of a pair $Z = (\Pi, \Gamma)$, where
\begin{itemize}
\item $\Pi$ is an element of $L(T, \mathcal{D}'(E)\big)$ (in the notation of Remark~\ref{remark useful notation}),
\item $\Gamma$ is a continuous section of $L$.
\end{itemize}
Furthermore, a model is required to satisfy the following analytic conditions for every $\frakt \in \mathfrak{L}$, compact subset $K\in M$ and for all $p,q\in K$, $\alpha \in A$, $\delta \in \triangle$ and $\tau_p\in (T_p)^\frakt_{\alpha,\delta}$:
\begin{enumerate}
\item $|\Pi_p \tau_p (\phi^\lambda_p) | \lesssim_{K,\alpha} \lambda^\alpha |\tau_p|_\alpha$,
\item $|\Gamma_{p,q} \tau_q |_{\beta} \lesssim_{K,\alpha} d_\fraks(p,q)^{(\alpha-\beta)\vee 0}|\tau_q|_\alpha$ for all $\beta\in A$ and $p,q\in K$,
\item $|(\Pi_p \Gamma_{p,q} \tau_q - \Pi_q \tau_q) (\phi_p^\lambda) | \lesssim_{K,\alpha} \lambda^\delta |\tau_q|_\alpha$ for $p,q\in K$,
\end{enumerate}
uniformly\footnote{Here it is understood that $\lambda^{+\infty}=0$ for all $\lambda\in (0,1 )$.} over $\lambda\in (0,1 )$ and $\phi_p^\lambda \in \mathcal{B}_p^{r,\lambda}(E^\frakt)$ for fixed $r>|{\min A}|$. 
We denote by $\mathcal{M}_{\mathcal{T}}$ the space of all models for $\mathcal{T}$. 
\end{definition}
From now on, whenever we are working with a regularity structure or regularity structure ensemble we implicitly fix some integer $r>|{\min A}|$.
\begin{remark}
In practice models usually satisfy for $\tau_p \in T^\frakt_0$
$$\Pi_p \tau_p\in \mathcal{C}^\infty (E^\frakt), \qquad \Pi_p \tau_p (p)=\tau_p\ , $$ where the second equality makes use of the canonical isomorphism $ T^\frakt_0\simeq E^\frakt$ in Definition~\ref{definition regularity structure}.  
\end{remark}
\begin{remark}\label{remark homogenious/trivial bundle}
Given a regularity structure $\mathcal{T}=(E, T, L)$ such that the bundles 
$T_{\alpha, \delta}$ are all trivial, one can reformulate the notions of a regularity structure and model in a way closer to the definition in \cite{Hai14}.
\begin{enumerate}
\item For each $\alpha\in A$, $\delta\in \triangle$, let $N_{\alpha,\delta}\in \mathbb{N}$ and $\{\tilde{\tau}^i_{\alpha,\delta}\}_{i\in\{1,\ldots,N_{_{\alpha,\delta}}\} }$ be a global frame of $T_{\alpha, \delta}$.
Define the vector spaces $\tilde{T}_{\alpha,\delta}=\text{span}\{  \tilde{\tau}^i_{\alpha,\delta} \ : i=1,\ldots, N_{\alpha,\delta}\ \}$ and 
$$\tilde{T}= \bigoplus_{\alpha,\delta}\tilde{T}_{\alpha,\delta}\ .$$
\item Then, let $\tilde{L}$ consist of all linear maps $\tilde{\Gamma}: \tilde{T}\to \tilde{T}$ such that there exist $p,q\in M$ and $\Gamma_{p,q}\in L$ such that for every $\tilde{\tau}\in \tilde{T}$ one has
\begin{equation}\label{eq:local rhs}
\big(\tilde{\Gamma} \tilde{\tau}\big)(p)= \Gamma_{p,q} (\tilde{\tau}(q)) \ . 
\end{equation}
\item Given a model $Z=(\Pi, \Gamma)$ for $\mathcal{T}$ this induces the maps
\begin{enumerate}
\item $\tilde{\Pi}: M\to L(\tilde{T}, \mathcal{D}'(E))$ given by 
$$
\tilde{\Pi}_p \tilde{\tau}: = \Pi_p \left(\tilde{\tau}(p)\right)\ ,
$$
\item $\tilde{\Gamma}: M\times M\to \tilde{L}$ is the unique map such that each $p,q\in M$ and $\tilde{\tau}\in \tilde{T}$ the term $\left(\tilde{\Gamma}_{p,q}\tilde{\tau}\right)(p)$ agrees with the right-hand side of \eqref{eq:local rhs}.
\end{enumerate}
\end{enumerate}
The bundles $T^\frakt$ appearing when treating singular SPDEs will in general only be trivial if the jet bundles $JE^\frakt$ are trivial.

Whenever $M$ is a compact manifold, it is possible to artificially put oneself in the setting of Remark~\ref{remark homogenious/trivial bundle} by enlarging $T_{\alpha,\delta}$ since every vector bundle over a compact manifold can be seen as a summand \dash with respect to the direct sum of vector bundles\dash  of a trivial vector bundle.
\end{remark}
Next we introduce semi-norms on models.
\begin{definition}\label{def:norm on models}
Given a regularity structure ensemble  $\mathcal{T}$, we introduce the following ``seminorms'' 
on models $Z = (\Pi, \Gamma)$:
\begin{itemize}
\item For $c>0$ and a compact set $K\subset M$, let 
$$\| \Pi\|_{c,K}:= \sup_{\lambda<c}\sup_{\frakt\in \mathfrak{L}} \sup_{p\in K} \sup_{\alpha\in A}\sup_{\tau_p\in T^\frakt_{\alpha,:}}   \sup_{\phi\in \mathcal{B}^{r,\lambda}_p} \frac{|\Pi_p \tau_p (\phi) |}{\lambda^\alpha |\tau_p|} \ .$$
We also write $\| \Pi\|_{ K}: = \| \Pi\|_{ 1,K}$.
\item Define $$\|\Gamma\|_{K}:=\sup_{\frakt\in \mathfrak{L}}\sup_{p,q\in K: d_\fraks(p,q)<1} \sup_{\alpha,\beta\in A}\sup_{\tau_q\in T^\frakt_{\alpha,:}} \frac{|\Gamma_{p,q}\tau_q|_\beta}{d_\fraks(p,q)^{(\alpha-\beta)\vee 0}|\tau_q|} \ . $$
\item Set $$\llbracket Z \rrbracket_{ c,K}:= \sup_{\lambda<c}\sup_{\frakt\in \mathfrak{L}} \sup_{p,q\in K} \sup_{\alpha\in A} \sup_{\delta\in\triangle}\sup_{\tau_q\in T^\frakt_{\alpha,\delta}}  \sup_{\phi\in \mathcal{B}^{r,\lambda}_p} \frac{|(\Pi_p \Gamma_{p,q} \tau_q - \Pi_q \tau_q) (\phi) |}{\lambda^\delta |\tau_q|} $$
and, as before, $\llbracket Z  \rrbracket_{ K}:=\llbracket Z  \rrbracket_{ 1,K}$.
\end{itemize}
We  set $\|Z\|_{ c,K}:=\| \Pi\|_{ c,K}+\|\Gamma\|_{ K}+\llbracket Z \rrbracket_{ c,K}$.
%
%
%
%
Similarly, we define for a second model $\tilde{Z}=(\tilde{\Pi},\tilde{\Gamma})$ 
$$\llbracket Z ;\tilde{Z} \rrbracket_{ c,K}:= \sup_{\lambda<c}\sup_{p,q\in K} \sup_{\alpha\in A}\sup_{\delta \in \triangle}\sup_{\tau_q\in T_{\alpha,\delta}}  \sup_{\phi\in \mathcal{B}^{r,\lambda}_p} \frac{|(\Pi_p \Gamma_{p,q} \tau_q - \Pi_q \tau_q- \tilde{\Pi}_p \tilde{\Gamma}_{p,q} \tau_q + \tilde{\Pi}_q \tau_q) (\phi) |}{\lambda^\delta |\tau_q|} $$
and set $\|Z;\tilde{Z}\|_{ c,K}=\| \Pi-\tilde{\Pi}\|_{ c,K}+\|\Gamma-\tilde{\Gamma}\|_{ K}+\llbracket Z ;\tilde{Z} \rrbracket_{ c,K}$.
\end{definition}

The following technical lemma will be useful in the sequel.
\begin{lemma}\label{technical lemma}
Let $\tilde Z= (\tilde{\Pi}, \tilde{\Gamma})$ be such that for each compact set $K\subset M$, there exists a constant $c_K$, such that properties (1), (2), and (3) of Definition~\ref{def model} are satisfied for $\lambda\leq c_K$, or equivalently, for each $K$ compact there exists $c_K>0$, such that $\|\tilde Z\|_{c_K,K}< \infty$. Then $\tilde Z$ is a model. 
Furthermore for every $C>0$ the bounds 
$$\|\tilde Z\|_{ 1,K}\lesssim_{C,c_{\bar{K}}} \|\tilde Z\|_{ c_{\bar{K}},\bar K}\;,\qquad
\|Z; \tilde Z\|_{ 1,K}\lesssim_{C,c_{\bar{K}}}  \|Z;\tilde Z\|_{ c_{\bar{K}},\bar K}\;, 
$$
hold uniformly over models $Z$, $\tilde Z$ satisfying $\|Z\|_{c_{\bar{K}},\bar K} \vee \|\tilde Z\|_{c_{\bar{K}},\bar K} < C$, where $\bar{K}$ is the 2-fattening of $K$.
\end{lemma}
\begin{proof}
We check the required bounds separately:
\begin{enumerate}
\item Note that the bound on $\Gamma_{p,q}\tau$ holds by assumption.
\item Next we check the bound on $\Pi$. We use Remark~\ref{decomposing test functions} with $c=c_{\bar{K}}$ and decompose $\phi^\lambda_p \in \mathcal{B}^{r,\lambda}_{p}$ as $\phi^\lambda_p=\sum_{i=1}^N \phi^{c_{\bar{K}}\lambda}_{p_i}$. Thus we can write for $\tau_p \in T^\frakt_{\alpha,\delta}$, $p,q\in K$, $\lambda\in (0,1)$ 
\begin{align*}
|\Pi_p \tau_p (\phi_p^\lambda)| &\leq \sum_{i=1}^N |\Pi_p \tau_p (\phi^{c_{\bar{K}}\lambda}_{p_i})|\\
& \leq \sum_{i=1}^N |(\Pi_p \tau_p - \Pi_i \Gamma_{p_i, p}\tau_p)  (\phi^{c_{\bar{K}}\lambda}_{p_i}) |+| \Pi_{p_i} \Gamma_{p_i, p}\tau_p (\phi^{c_{\bar{K}}\lambda} _{p_i}) | \\
&\leq \llbracket Z \rrbracket_{ c_{\bar{K}},\bar K} |\tau_p| N \cdot (c_k\lambda)^\delta + \|\Pi\|_{ c_{\bar{K}},\bar K} \|\Gamma\|_{ \bar K} |\tau_p|\sum_i \sum_{\beta\in A} d_\fraks (p_i,p)^{(\alpha-\beta)\vee 0}(c_{\bar{K}}\lambda)^\beta\\
&\lesssim_{c_{\bar{K}}} (\llbracket Z \rrbracket_{ c_{\bar{K}},\bar K} + \|\Pi\|_{ c_{\bar{K}},\bar K} \|\Gamma\|_{ \bar K})|\tau_p|\lambda^{\alpha}
\end{align*}
where we used the fact that $d_\fraks (p_i,p) \leq \lambda$. Observe that the obtained bound is uniform over $p\in K$, $\alpha\in A $, $\tau_p\in T^\frakt_{\alpha,\delta}$ and $\lambda\in (0,1)$, $\phi_p^\lambda\in \mathcal{B}^{r,\lambda}_p$ .
Thus $$\|\Pi\|_{K}\lesssim_{c_{\bar{K}}} \llbracket Z \rrbracket_{ c_{\bar{K}},\bar K} + \|\Pi\|_{ c_{\bar{K}},\bar K} \|\Gamma\|_{ \bar K} \ .$$
Similarly one finds for a second model $\tilde Z=(\tilde{\Pi},\tilde{\Gamma})$
$$\|\Pi-\tilde{\Pi}\|_{K}\lesssim\llbracket Z ;\tilde{Z} \rrbracket_{ c_{\bar{K}},\bar K} + \|\Pi-\tilde{\Pi}\|_{ c_{\bar{K}},\bar K} \|\Gamma\|_{ \bar K}+\|\tilde{\Pi}\|_{ c_{\bar{K}},\bar K} \|\Gamma-\tilde{\Gamma}\|_{ \bar K} \ .$$
\item We proceed similarly to obtain the desired bound on $\Pi_p\tau_p-\Pi_q\Gamma_{q,p}\tau_p$. We treat only the case $\lambda\geq c_{\bar{K}}$, since for $\lambda< c_{\bar{K}}$ the desired bound holds by assumption. W.l.o.g assume $\tau_p\in T_{\alpha,\delta}$ for $\delta<+\infty$, then
\begin{align*}
(\Pi_q \tau_q -\Pi_p\Gamma_{p,q}\tau_q)(\phi_p^\lambda)& = \sum_{i=1}^N  (\Pi_q \tau_q -\Pi_p\Gamma_{p,q}\tau_q)(\phi_{p_i}^{c_{\bar{K}}\lambda})\\
&= \sum_{i=1}^N (\Pi_q \tau_q - \Pi_{p_i}\Gamma_{p_i, q}\tau_q)(\phi_{p_i}^{c_{\bar{K}}\lambda})\\
&+ \sum_{i=1}^N (\Pi_{p_i}\Gamma_{p_i, p}\Gamma_{p,q}\tau_q -\Pi_p\Gamma_{p,q}\tau_q )(\phi_{p_i}^{c_{\bar{K}}\lambda})\\
&+\sum_{i=1}^N (\Pi_{p_i}\Gamma_{p_i, q} \tau_q- \Pi_{p_i}\Gamma_{p_i, p}\Gamma_{p,q}\tau_q)(\phi_{p_i}^{c_{\bar{K}}\lambda})
\end{align*}
\begin{enumerate}
\item The summands of the first sum can bounded as follows:
\begin{align*}
|(\Pi_q \tau_q - \Pi_{p_i}\Gamma_{p_i, q}\tau_q)(\phi_{p_i}^{c_{\bar{K}}\lambda})|&\leq \llbracket Z \rrbracket_{ c_{\bar{K}},\bar K} |\tau| (c_{\bar{K}}\lambda)^ \delta
\end{align*}
\item For the summands in the second sum one has
\begin{align*}
|(\Pi_{p_i}\Gamma_{p_i, p}\Gamma_{p,q}\tau_q -\Pi_p\Gamma_{p,q}\tau_q )(\phi_{p_i}^{c_{\bar{K}}\lambda})|
&\leq \sum_{\beta\in A} 
|(\Pi_{p_i}\Gamma_{p_i, p} Q_{\beta} \Gamma_{p,q}\tau_q -\Pi_p Q_{\beta} \Gamma_{p,q}\tau_q )(\phi_{p_i}^{c_{\bar{K}}\lambda})|\\
&\leq \llbracket Z  \rrbracket_{ c_{\bar{K}},\bar K} \sum_\beta |Q_{\beta} \Gamma_{p,q}\tau_q| (c_{\bar{K}}\lambda)^\delta\\
&\lesssim_K \llbracket Z  \rrbracket_{ c_{\bar{K}},\bar K} \|\Gamma\|_{  K} |\tau_q| (c_{\bar{K}}\lambda)^ \delta
\end{align*}
\item For the summands in the last sum we use the following very crude bounds:
\begin{align*}
|\Pi_{p_i}\Gamma_{p_i, q} \tau_q (\phi_{p_i}^{c_{\bar{K}}\lambda})| &\leq \|\Pi\|_{ 1,K}  \sum_{\beta}|Q_\beta \Gamma_{p_i, q}\tau_q| (c_{\bar{K}}\lambda )^\beta\\
 &\lesssim_K  \|\Pi\|_{ c_{\bar{K}},\bar K} \|\Gamma\|_{  K}|\tau_q| \lambda^{\min A}\\
 &\lesssim_K \|\Pi\|_{ c_{\bar{K}},\bar K} \|\Gamma\|_{  K}|\tau_q| c_{\bar{K}}^{\min A -\delta} \lambda^\delta\\
 &\lesssim_{K,{c_{\bar{K}}}}\|\Pi\|_{ c_{\bar{K}},\bar K} \|\Gamma\|_{  K}|\tau_q|\lambda^\delta
\end{align*}
where in the last line we used that $\lambda^{\min A}\leq c_{\bar{K}}^{\min A -\delta} \lambda^\delta$ for $\lambda\geq c_{\bar{K}}$.
Very similarly one obtains 
\begin{align*}
|\Pi_{p_i}\Gamma_{p_i, p}\Gamma_{p,q}\tau_q (\phi_{p_i}^{c_{\bar{K}}\lambda}) |&\lesssim_{K,{c_{\bar{K}}}} \|\Pi\|_{ c_{\bar{K}},\bar K} \|\Gamma\|^2_{  \bar K}|\tau_q|\lambda^\delta
\end{align*}
for $\lambda\geq c_{\bar{K}}$.
\end{enumerate}
Putting this together,
$$|(\Pi_q \tau_q -\Pi_p\Gamma_{p,q}\tau_q)(\phi_p^\lambda)|\lesssim_{ c_{\bar{K}},\bar K} (1+\|\Gamma\|_{  \bar K}) (\llbracket Z  \rrbracket_{ c_{\bar{K}},\bar K}  +\|\Pi\|_{ c_{\bar{K}},\bar K} \|\Gamma\|_{  \bar K} )    |\tau_q| \lambda^ \delta $$
holds uniformly over $p,q\in K$, $\tau_q\in (T_q)_{\alpha,\delta}^\frakt,$ and $\lambda\in (0,1)$, $\phi_p^\lambda\in \mathcal{B}^{r,\lambda}_p$ .

For the last claim, if $\tilde Z=(\tilde{\Pi},\tilde{\Gamma})$ is a second model, one similarly finds 
\begin{align*}
\llbracket 
Z;\tilde{Z}  \rrbracket_{{ c_{\bar{K}},\bar K}} \lesssim & \|\Gamma-\tilde{\Gamma}\|_{  \bar K} \left(\llbracket Z\rrbracket_{ c_{\bar{K}},\bar K}  +\|\Pi\|_{ c_{\bar{K}},\bar K} \|\Gamma\|_{  \bar K} \right)  \\
&+(1+\|\tilde\Gamma\|_{  \bar K}) \left(\llbracket Z ;\tilde{Z} \rrbracket_{ c_{\bar{K}},\bar K}  +\|\Pi-\tilde{\Pi}\|_{ c_{\bar{K}},\bar K} \|\Gamma\|_{  \bar K} +\|\tilde \Pi\|_{ c_{\bar{K}},\bar K} \|\Gamma-\tilde{\Gamma}\|_{  \bar K} \right)  \ .
\end{align*}
\end{enumerate}
\end{proof}

Next, we make the following useful observation, which holds in the flat setting as well, though is less important there. 
\begin{lemma}\label{dislocated testfunction lemma}
Fix a regularity structure and a model. For $D>0$, denote by $K^D$, the $D$-enlargement of any set $K$. Then, for every compact set $K\subset M$ uniformly in
$p\in K$ and $\tilde{p}\in K$ satisfying $d(p,\tilde{p})<D\lambda$ as well as $\tau_p\in T_{\alpha,:}$ and $\phi^\lambda_{\tilde{p}}\in B^{r,\lambda}_{\tilde{p}}$
$$|\Pi_p \tau_p (\phi^\lambda_{\tilde{p}})| \lesssim_{D}\big(\|\Pi\|_{K^D}\|\Gamma\|_{K^D}+\llbracket Z  \rrbracket_{K^D } \big) |\tau_p| \lambda^{\alpha}$$
for $\lambda \in (0,1)$.
For any second model $\tilde{Z}=(\tilde{\Pi},\tilde{\Gamma})$ the analogous bound 
$$|\big(\Pi_p-\tilde{\Pi}_p\big) \tau_p (\phi^\lambda_{\tilde{p}})| \lesssim_{D}\big(\|\Pi-\tilde\Pi\|_{K_{D+1}}\|\tilde\Gamma\|_{K_{D+1}}+\|\Pi\|_{K^D}\|\Gamma-\tilde{\Gamma}\|_{K^D }+\llbracket Z  ;\tilde Z  \rrbracket_{K^D } \big) |\tau_p| \lambda^{\alpha} \ ,$$
holds.
\end{lemma}
\begin{proof}
First assume $\lambda< \frac{1}{D}$, then
\begin{align*}
\Pi_p\tau_p(\phi^\lambda_{\tilde{p}})&=(\Pi_p\tau_p-\Pi_{\tilde{p}}\Gamma_{\tilde{p},p}\tau_p)(\phi^\lambda_{\tilde{p}}) + (\Pi_{\tilde{p}}\Gamma_{{\tilde{p}},p}\tau_p)(\phi_{\tilde{p}}^\lambda)\\
&=(\Pi_p\tau_p-\Pi_{\tilde{p}}\Gamma_{\tilde{p},p}\tau_p)(\phi^\lambda_{\tilde{p}}) + \sum_\beta(\Pi_{\tilde{p}}Q_\beta \Gamma_{{\tilde{p}},p}\tau_p)(\phi_{\tilde{p}}^\lambda)\\
&\lesssim \llbracket Z  \rrbracket_{K^D }|\tau_p|\lambda^ \delta + \|\Pi\|_{K^D }\sum_{\beta }|Q_\beta \Gamma_{{\tilde{p}},p}\tau_p|\lambda^{\beta}\\
&\lesssim \llbracket Z  \rrbracket_{K^D }|\tau_p|\lambda^ \delta + \|\Pi\|_{K^D }\|\Gamma\|_{K^D }|\tau_p|\sum_{\beta } (D\lambda)^{(\alpha-\beta)\vee 0}\lambda^{\beta} \ .
\end{align*}
A similar calculation works the second part of the claim for $\lambda< \frac{1}{D}$.
The general case follows along the lines of the proof of Lemma~\ref{technical lemma}.
\end{proof}

\subsubsection{Jet regularity structures}
We define jet regularity structures in the case $|\mathfrak{L}|=1$ and use the obvious adaptation for cases $|\mathfrak{L}|>1$.
\begin{definition}\label{definition polynomial reg}
For a vector bundle $E\to M$, an $E$-valued jet regularity structure $\bar{\mathcal{T}}:=(E,\bar T, L)$ of precision $\delta_0$ consists of a bi-graded finite dimensional subbundle \linebreak
$\bar{T}= \bigoplus_{\alpha\in A, \delta \in \triangle} \bar{T}_{\alpha,\delta} \subset JE$ where 
\begin{itemize}
\item $A\subset \mathbb{N}$ and for every $\alpha\in A$ one has $\bar{T}_{\alpha,:}=\bigoplus_{\delta \in \triangle} \bar{T}_{\alpha,\delta}=(JE)_\alpha$, where the grading on $JE$ was introduced in (\ref{grading of J^k}).
\item The isomorphism $\bar{T}_{0}\simeq E$ is the canonical isomorphism $\bar{T}_{0}=(JE)_0 \simeq E$.
\item $L$ consists of all maps satisfying the properties in Definition~\ref{definition regularity structure} Point \eqref{condition on L}.
\end{itemize}
An $E$-valued jet regularity structure is said to be of \textit{precision} $\delta_0$ if it satisfies 
$A= \mathbb{N}\cap [0, \delta_0)$ and $\triangle\subset [ \delta_0, + \infty ]$. It is called the \textit{standard} $E$-valued jet regularity structure if $\triangle=\{\delta_0\} .$
\end{definition}
Now we define the analogue of the polynomial model.
\begin{definition}\label{definition polynomial model}
Let $\bar{\mathcal{T}}:=(E,\bar T, L)$ be an $E$-valued jet regularity structure. An admissible realisation $R$ of $JE$ is said to induce a model for $\bar{\mathcal{T}}$ if the maps $Z=(\Pi, \Gamma)$ defined as
\begin{itemize}
\item  $\Pi_p v_p = R (v_p)$ for $v_p\in \bar{T}_{\alpha,\delta}$,
\item  $\Gamma_{q,p}v_p= {Q}_{<\delta} j_q R (v_p)$ for $ v_p \in \bar{T}_{\alpha,\delta}$,
where ${Q}_{<\delta}: JE\to (JE)_{<\delta}$ denotes the canonical projection defined in (\ref{def of proj}) ,
\end{itemize}
satisfy the properties of a model. Then $Z$ is called a jet model for $\bar{\mathcal{T}}$ .
\end{definition} 

\begin{prop}\label{proposition on polynomial model}
For any $\delta_0>0$, every admissible realisation $R$ of $JE$ induces a model for the standard jet regularity structure of precision $\delta_0$. 
\end{prop}

\begin{proof}
Set $k= \max\{n\in \mathbb{N}\ : \ n<\delta_0\}$ and define $(\Pi, \Gamma)$ as in Definition~\ref{definition polynomial model}.
Recall the canonical identification $(JE)_0\simeq E$ induced by $Q_0 j_pf= f(p)$ for all $f\in \mathcal{C}^\infty (E)$ and 
that $\Pi_p v_p (p)=0$ for $v_p\in (T_p)_{>0}$ since for all $\tau_p\in \bar T_p$
 $$j_p \Pi_p \tau_p = j_p R(\tau_p) = \tau_p \ .$$  Thus 
\begin{equation}\label{good for bounds bellow}
|\Pi_p\tau_p (q)|= |R (Q_{<k} j_q \Pi_p\tau_p)(q)| = |R (Q_{<k} j_q R \tau_p)(q)| =|Q_0 j_q R \tau_p | \ 
\end{equation}
and
\begin{equation}\label{good for bound below 2}
|\Gamma_{p,q} \tau_q |_m = |j_p R\tau_q |_m
\end{equation}

We check that the properties of a model are satisfied.
\begin{itemize}
\item Note that (\ref{good for bounds bellow}) combined with Lemma~\ref{bound on admissable realisations} implies that for $v_p\in (T_p)_n $
$$ |\Pi_p\tau_p (q)|\lesssim d_\fraks (p,q)^n \ .$$
\item The same Lemma~\ref{bound on admissable realisations}  combined with (\ref{good for bound below 2})  implies for $v_p\in (T_p)_n $
\begin{equation}
|\Gamma_{p,q} \tau_q |_m = |Q_{\leq k}j_p R\tau_q |_m \lesssim d_\fraks (p,q)^{(n-m)\vee 0} \ .
\end{equation}
\item To check the last condition, denote by $\bar K:= \{p\in M \ : \ d(p,K)\leq 1\}$ the $1$-fattening of $K$. Then, since
$$\Pi_p \Gamma_{p,q} \tau_q - \Pi_q \tau_q =  R Q_{\leq k} j_{p} R\tau_q - R \tau_q $$
it follows by taking  $f= R\tau_q$ and $g= R (Q_{\leq k} j_{p} f)$ in Lemma~\ref{pointwise bound} that 
$$|(\Pi_p \Gamma_{p,q} \tau_q - \Pi_q \tau_q) (r) | \lesssim_{R,\delta} d_\fraks (p,r)^{\delta_0} |\tau_q|_\alpha \ .$$
\end{itemize}
\end{proof}

\begin{remark}
When $E=\mathbb{R}^d \times \RR^n$ is the trivial bundle over $M=\RR^d$ one can identify jets with polynomials as in Section~\ref{bundle valued jets in local coordinates}. Then the polynomial model in the sense of \cite{Hai14} induces a model for the any $E$ valued-jet regularity structure with $\triangle=\{+\infty\}$ (up to an appropriate identification of the different fibres of $\bar{T}$, c.f.\ Remark~\ref{remark homogenious/trivial bundle}, and a truncation to make the space finite dimensional).
\end{remark}

\begin{remark}\label{rem:nonstandard_polynomial}
Note that the definition of ``non-standard'' jet regularity structures allows some flexibility on the analytic conditions one imposes on the corresponding model. These conditions are particularly interesting in the case that $\infty\in \triangle$ and $\bar{T}_\infty\neq \{0\}$. 
\begin{itemize}
\item For example, when $E=M \times \RR$ is the trivial bundle one can always set $\bar{T}_{0,\infty}= (JM)_0$ and obtain a model for it by working with admissible realisations $R$ mapping $j_p 1\mapsto 1$ for every $p\in M$.
\item More generally, if there exists a global parallel section $e$ of $E$, it it can be useful to restrict attention to admissible realisations satisfying 
$Ej_p e = e$. An important example is the metric on some bundle $F$, which is a parallel section of $E=F^*\oplus F^*$.
\end{itemize}
Given an arbitrary jet regularity structure, in general it is a non-trivial geometric question whether there exists an admissible realisation inducing a model for it.
\end{remark}
\begin{remark}
Note that in this section there was no mention of the algebra structure of jet bundles, c.f.\ Remark~\ref{alternative definition remark}, which will be used again in Section~\ref{section Local operations}. There it will be crucial that flexible notion of admissible realisations allows one to impose further conditions with respect to this structure, c.f.\ Section~\ref{remark algebraic structure}. 
\end{remark}

\subsection{Operations on models}
\begin{definition}\label{definition push forward}
Let $E\to M$, $F\to N$ be vector bundles and $\Phi=(\phi, G): E\to F$ a bundle isomorphism. There is a natural notion of push-forward of a regularity structure  $\mathcal{T}_E=(E,T,L)$ on $M$ to a regularity structure $\Phi_*\mathcal{T}_E = (F,\phi_*T,\phi_*L)$ on $N$, where $\phi_*T,\phi_*L$ are the pullback-vector bundles under $\phi^{-1}$, resp. $\phi^{-1}\times \phi^{-1}$ and the isomorphism $(\phi_*T)_0\simeq F$ is given by composing the isomorphism $T_0\simeq E$ with $\Phi$.

The corresponding concept of push-forward of models 
$$\mathcal{M}_{\mathcal{T}_E}\to\mathcal{M}_{\Phi_*\mathcal{T}_M}, \qquad Z \mapsto \Phi_*Z \ ,$$
c.f.\ \cite[Lemma 19]{DDD19} for the scalar case, is given as follows: Let $Z=(\Pi,\Gamma)$ be a model for $\mathcal{T}_E$, then for $\bar{p}, \bar{q}\in N$
\begin{equs}
(\Phi_*\Gamma)_{\bar{p},\bar{q}}&:=\Gamma_{\phi^{-1}(\bar{p}),\phi^{-1}(\bar{p})}\\
(\Phi_*\Pi)_{\bar{q}}&:= \Phi_* \circ \Pi_{\phi^{-1}\bar{p}} \ ,
\end{equs}
where the last $\Phi_*$ denotes the push-forward of distributions $\Phi_*: \mathcal{D}'(E)\to \mathcal{D}'(F)$.
\end{definition}
\begin{remark}
In the special case of the jet regularity structure, this simply corresponds to the pushforward of admissible realisations, see Remark~\ref{Pushforward of admissible realisations}.
\end{remark}

%
\begin{remark}\label{pushing testfunctions}
Note that in the setting of Definition~\ref{definition push forward}, for every compact set $K\subset N$, there exist $C=C(K,\Phi)>0$, $c=c(K,\Phi)\in (0,\infty)$ such that uniformly over $\bar p \in K$ and $\lambda\in (0,1)$
 $$\mathcal{C}^\infty_c (F)\supset\Phi^*( \mathcal{B}^\lambda_{\bar{p}}(F) )\subset C\mathcal{B}^{c\lambda}_{p}(E) \subset \mathcal{C}^\infty_c (E)\ ,$$
where we write $p=\phi^{-1}(\bar{p})$.
\end{remark}

Since our definitions differ slightly from the ones in \cite{DDD19}, we state and prove the following lemma, which is close to \cite[Lemma 19]{DDD19}. 
\begin{lemma}\label{lemma push forward}
The pushforward map of Definition~\ref{definition push forward} is continuous. For every compact set $K\subset N$ there exists a constant $c_K$, such that
\begin{align}
&\|\Phi_* \Pi \|_{K,c_K}\lesssim_{\phi} \| \Pi \|_{\phi^{-1}( K)}\label{first}\\
&\llbracket \Phi_* Z  \rrbracket _{K} \lesssim_{\phi} \llbracket Z  \rrbracket_{\phi^{-1}( K)}\label{second}\\
&\|\Phi_*\Gamma \|_{c_K,K} \lesssim_{\phi} \|\Gamma \|_{\phi^{-1}( K)} \label{third}
\end{align}
and the analogous bound holds for the difference of two models.
Therefore, by Lemma~\ref{technical lemma}, the bound 
$$\|\Phi_* Z \|_{K}\lesssim_{C,K} \|Z \|_{\phi^{-1}(\bar K)}$$
holds uniformly over models $Z$ satisfying $\|Z \|_{\phi^{-1}(\bar K)}<C$, as well as the similar bound for two models $Z, \bar{Z}$
$$\|\Phi_* Z;\Phi_* \bar{Z} \|_{K, \delta}\lesssim_{C,K} \|Z;\bar{Z} \|_{\phi^{-1}(\bar K)} \ .$$
\end{lemma}
\begin{proof}
Recall that by Remark~\ref{pushing testfunctions} there exist $C=C(K,\phi)>0$, $c=c(K,\phi)\in (0,\infty)$ such that uniformly over $\bar p \in K$ and $\lambda\in (0,1)$
 $$\mathcal{C}^\infty_c (F)\supset\Phi^*( \mathcal{B}^\lambda_{\bar{p}}(F) )\subset C\mathcal{B}^{c\lambda}_{p}(E) \subset \mathcal{C}^\infty_c (E)\ ,$$
where we write $p=\phi^{-1}(\bar{p})$.
For $\tau_{\bar{p}}\in (\phi_*T)_{\bar p}$, denote by $\tau_p$ the corresponding element of $T_p$. We have
$$(\Phi_* \Pi)_{\bar p} \tau_{\bar p} \rho^\lambda_{\bar p} =\Pi_p \tau_p ( \Phi^*\rho_{\bar p}^\lambda) \ .$$
Similarly writing $q=\phi^{-1}(\bar{q})$, we have 
$$\big((\Phi_* \Pi)_{\bar q} (\phi_*\Gamma_{\bar q, \bar p})\tau_{\bar{q}} - (\Phi_* \Pi)_{\bar p} \tau_{\bar{q}} \big)\rho^\lambda_{\bar q}=
(\Pi_{ q} \Gamma_{q, p}\tau_{q} -  \Pi_{ p} \tau_{q}) (\Phi^*\rho_{\bar q}^\lambda) \ .$$
Thus, if we set $c_K:= \frac{1}{c}\wedge 1$ the bounds (\ref{first}) and (\ref{second}) follow.  The bound (\ref{third}) follows directly from the Lipschitz continuity of $\phi^{-1}$.
\end{proof}
\begin{remark}\label{remark localisation}
In the sequel we will want to pushforward a regularity structure and model, using chart maps\slash trivialisations. To do so, we need to give a meaning to a model $Z$ on $M$ restricted to an open set $V$. Let $U\subset M$ be such that $V\Subset U\subset M$ and $\mathbb{1}_V\in C_c^\infty (U)$ such that $ \mathbb{1}_V|_V=1$. 
Define $\mathcal{T}_M|_U:= (A, T|_U, L|_{U\times U})$. Then, we define
$$Z|_{V}:=(\tilde{\Pi},\tilde{\Gamma}) \ ,$$
where 
$$\tilde{\Pi}_p\tau (\cdot):= \mathbb{1}_V(\cdot)\Pi_p \tau (\cdot)$$ and
$$\tilde \Gamma_{p,q}=  \Gamma_{p,q} \ .$$
We say that $\bar Z$ is a restriction of $Z$ to $V$ subordinate to $U$. Note that this definition \textit{does} depend on the choice of $\mathbb{1}_V\in C_c^\infty (U)$.  It is clear how to pushforward such restricted models.
In particular it will be useful to pushforward restricted models using bundle trivialisations.

It follows from the definitions that restriction is a continuous operation, i.e.\ one has for each compact set $K\subset U$
$$\|\tilde Z\|_{K}\lesssim \|Z\|_{K} \ ,$$
where the implicit constant depends on the function $\mathbb{1}_V$ . 
\end{remark}
\begin{remark}
Given a regularity structure on $M$, assume $U\subset M$ is open. We call a model on $Z=(\Pi, \Gamma)$ on $U$ compactly supported, if for all $p\in U$ and $\tau_p\in T_p$, $\Pi_p\tau_p$ is compactly supported in $U$. 
If $U\subset \RR^d$ is convex, then there exists  extension of the regularity structure and model $Z$ to a model on all of $\RR^d$. The first claim follows from the fact that any bundle over a convex base space is trivial and a corresponding model is easily constructed as well.
\end{remark}

\section{Modelled Distributions and Reconstruction}\label{section modelled distributions and reconstruction}
Now we define modelled distributions analogously to \cite{Hai14}.
\begin{definition}
Given a regularity structure ensemble $\mathcal{T}=(\{E^{\mathfrak{t}}\}_{t\in \mathfrak{L}},\{T^{\mathfrak{t}}\}_{t\in \mathfrak{L}}, L)$ and a model $Z = (\Pi, \Gamma)$, we define for $\frakt\in \mathfrak{L}$, $\gamma\in \mathbb{R}$ the space $\mathcal{D}_Z^\gamma(T^\frakt)$ to consist of all continuous sections $f$ of $T^{\frakt,\gamma}= \bigoplus_{\alpha<\gamma, \delta>\gamma } T^\frakt_{\alpha,\delta}$ satisfying for every compact set $K\subset M$
\begin{equation}\label{norm on modelled distr.}
\interleave f \interleave_{\gamma ; K}:=\| f\|_{\gamma ; K} +\sup_{\substack{p,q\in K,\\ d_\fraks(p,q)<1}} \sup_{\beta <\gamma} \frac{|f(q)-Q_{<\gamma}\Gamma_{q,p} f (p)|_{\beta}}{d_\fraks(p,q)^{\gamma-\beta}} <+\infty\ ,
\end{equation}
where
$$\| f\|_{\gamma ; K} := \sup_{p\in K} \sup_{\alpha<\gamma} |f(p)|_\alpha \ .$$
For a second model $\tilde Z = (\tilde \Pi, \tilde{\Gamma})$ and $\tilde{f}\in \mathcal{D}_{\tilde{Z}}^\gamma(T^\frakt)$, we define
$$ \interleave f ; \tilde{f} \interleave_{\gamma ; K} := \| f-\bar{f}\|_{\gamma ; K}+\sup_{\substack{p,q\in K,\\ d_\fraks(p,q)<1}} \sup_{\beta <\gamma} \frac{|f(q)-Q_{<\gamma}\Gamma_{q,p} f(p)-(\tilde f (q)-Q_{<\gamma}\tilde{\Gamma}_{q,p}\tilde{f}(p))|_{\beta}}{d_\fraks(p,q)^{\gamma-\beta}} \ .$$
\end{definition}
%
%
\begin{remark}\label{pushforward of modelled distributions}
Note that in the setting of Definition~\ref{definition push forward}, one has the obvious notion of the push-forward of a modelled distributions 
$$\mathcal{D}_Z^\gamma(T^\frakt)\to \mathcal{D}_{(\Phi_*Z)}^\gamma(\Phi_* T^\frakt),\qquad f(\cdot) \mapsto \Phi_*f(\cdot)\ . $$
It is clear that this map is continuous.
\end{remark}

\begin{theorem}\label{reconstruction theorem}
Given a regularity structure ensemble $\mathcal{T}=(\{E^{\mathfrak{t}}\}_{t\in \mathfrak{L}},\{T^{\mathfrak{t}}\}_{t\in \mathfrak{L}}, L)$ and a model $Z = (\Pi, \Gamma)$, there exists for every $\gamma\in (0, +\infty)$ a unique operator
$\mathcal{R}_Z\in L( \mathcal{D}_Z^\gamma, \mathcal{D}'(E))$, called reconstruction operator associated to $Z$, such that
\begin{equation}\label{reconstruction cond}
\sup_{p\in K} \sup_{\lambda<1} \sup_{\phi_p^\lambda\in \mathcal{B}^{r,\lambda}_p} \frac{|\langle \mathcal{R}_Z f-\Pi_p f(p),\phi_p^\lambda
\rangle|}{\lambda^\gamma}  \lesssim_{c,K} \interleave f\interleave_{\gamma ; K} \|Z\|_{\bar K}
\end{equation}
uniformly over all $\frakt\in \mathfrak{L}$, $f\in \mathcal{D}^\gamma_Z(T^\frakt)$ and $\phi^\lambda_p \in \mathcal{B}_p^{r,\lambda}(E^\frakt)$ as well as all models $Z=(\Pi,\Gamma)$ satisfying $\|Z\|_{\bar{K}}\leq c$. 
Furthermore, the map $Z\mapsto \mathcal{R}_Z$ is continuous and if we denote by $\bar{Z}=(\bar{\Pi}, \bar{\Gamma})$ a second model for $\mathcal{T}$, $\mathcal{R}_{\bar{Z}}$ the associated reconstruction operator, and $\bar{\mathcal{D}}^{{\gamma}}_{\tilde{Z}}$ the corresponding space of modelled distributions, then the inequality
\begin{align}
\sup_{p\in K} \sup_{\lambda<1} \sup_{\eta\in \mathcal{B}_p^{r,\lambda}} &\frac{|\langle \mathcal{R}_Z f-\Pi_p \bar{f}(p)-\mathcal{R}_{\bar{Z}}\bar{f}-\bar{\Pi}_p f(p),\eta^\lambda\rangle|}{\lambda^\gamma}  \nonumber\\
\lesssim_{c,K} & \interleave f,\bar{f} \interleave_{\gamma ; K}  \|Z\|_{\bar{K}}
+\interleave \bar{f}\interleave_{\gamma ; K}  \|Z;\bar{Z}\|_{\bar{K}}
\end{align}
holds uniformly over models $Z,\bar{Z}\in\mathcal{M}_\mathcal{T}$, satisfying $\|\tilde Z\|_{\bar{K}},\|\tilde Z\|_{\bar{K}} \leq c$ and ${f}\in \mathcal{D}^{{\gamma}}_Z,\  \bar{f}\in \bar{\mathcal{D}}^{\gamma}_{\bar{Z}}$.
\end{theorem}

\begin{proof}
We divide the proof into several steps.
First, we argue that it suffices to show the bound for $\lambda$ small enough, instead of $\lambda\leq 1$, i.e.\ the following claim.
\begin{claim}\label{claimmm}
Assume that there exists an operator $\mathcal{R}_Z: \mathcal{D}_Z^\gamma\to \mathcal{D}'(E)$, such that for every $K$ there exists $c_K$ for which
\begin{equation}\label{easier to prove reconstruction cond}
\sup_{p\in K} \sup_{\lambda<c_K} \sup_{\eta^{\lambda}_{p}\in \mathcal{B}^{r,\lambda}_p} \frac{|\langle \mathcal{R}_Z f-\Pi_p f(p),\eta_p^\lambda
\rangle|}{\lambda^\gamma}  \lesssim_{c,K} \interleave f\interleave_{\gamma ; \bar K} \|Z\|_{ \bar K}
\end{equation}
Then $\mathcal{R}_Z f$ satisfies (\ref{reconstruction cond}).
\end{claim}

\begin{proof}[Proof of Claim~\ref{claimmm}]
We argue similarly to Lemma~\ref{technical lemma} and use Remark~\ref{decomposing test functions} to write 
\begin{align*}
\langle\mathcal{R}_Z f-\Pi_p f(p),\eta_p^\lambda \rangle &= \sum_{i=1}^N \Big(\mathcal{R}_Z f-\Pi_p f(p)\Big) (\eta_{p_i}^{c_{\bar K}\lambda})\\
&= \sum_i \big(\mathcal{R}_Z f-\Pi_{p_i} f(p_i)\big) (\eta_{p_i}^{c_{\bar K}\lambda})\\
&+\sum_i \big(\Pi_{p_i} f(p_i) - \Pi_{p_i}\Gamma_{p_i,p} f(p)\big)(\eta_{p_i}^{c_{\bar K}\lambda})\\
&+\sum_i \big(\Pi_{p_i}\Gamma_{p_i,p} f(p)-\Pi_p f(p)\big) (\eta_{p_i}^{c_{\bar K}\lambda})
\end{align*}
\begin{enumerate}
\item For summands of the first type, we can use the assumption to obtain
$$\Big|\big(\mathcal{R}_Z f-\Pi_{p_i} f(p_i)\big) (\eta_{p_i}^{c_{\bar K}\lambda})\Big|\lesssim_{c,K}\lambda^\gamma \interleave f \interleave_{\gamma ; \bar K}\|Z\|_{\bar K} \ .$$
\item For the terms in the second sum one has 
$$\Big|\big(\Pi_{p_i} f(p_i) - \Pi_{p_i}\Gamma_{p_i,p} f(p)\big)(\eta_{p_i}^{c_{\bar K}\lambda})\Big|\leq \lambda^\gamma \|\Pi\|_{ \bar{K}} \interleave f\interleave_{\gamma ; \bar K} \ .$$
\item The summands in the last line can be bounded by 
$$\Big | \big(\Pi_{p_i}\Gamma_{p_i,p} -\Pi_p \big)f(p) (\eta_{p_i}^{c_{\bar K}\lambda}) \Big |\leq \lambda^{\gamma} \llbracket Z  \rrbracket_{\bar K} \interleave f\interleave_{\gamma ; \bar K} .$$
\end{enumerate}
\end{proof}
\item The rest of the proof essentially adapts from \cite{DDD19}. The argument goes as follows.
\begin{enumerate}
\item\label{flat} First recall the theorem on the trivial bundle $\mathbb{R}^d\times \mathbb{R}^m$, c.f.\ \cite{Hai14}, \cite{OW19}, \cite{MW20}, \cite{ST18} or \cite{FH20}, all proofs adapt easily to our slightly more general notion of models.

Next prove reconstruction for models compactly supported on the domain $U\subset M$ of a trivialisation $\Phi: \pi^{-1}(U)\to \mathbb{R}^d\times \RR^m$. This can be done by push-forward:
Assume $Z=(\Pi,\Gamma)$ is such a model, and denote by $\bar Z= (\bar \Pi, \bar \Gamma ):= \Phi_*Z$ the model pushed forward to $\mathbb{R}^d$. Let $f\in \mathcal{D}^\gamma_{Z}$ and denote by $\bar f \in \mathcal{D}^\gamma_{\bar Z} $ the pushed forward modelled distribution. Set 
$$\mathcal{R}_Z f := (\Phi^{-1})_* \mathcal{R}_{\bar Z} \bar f \ .$$
We check that this satisfies (\ref{easier to prove reconstruction cond}). Let $x= \phi (p)$ , then
$$\big(\mathcal{R}_Z-\Pi_p f(p)\big)(\phi_p^\lambda)= \big(\mathcal{R}_{\bar Z} \bar f - \bar{\Pi}_x \bar f(x) \big) ((\Phi^{-1})^*\phi^\lambda_p) \ .$$
Thus the claim follows by combining the previous step (\ref{flat}) with Remark~\ref{pushing testfunctions} and the bounds on the pushforward of models and modelled distributions in Lemma~\ref{lemma push forward} and Remark~\ref{pushforward of modelled distributions}. 
\item Fix a locally finite partition $\mathbb{1}_i$ of unity subordinate to trivialisations $\Phi_i$. For a general model $Z=(\Pi,\Gamma)$, define the compactly supported models
$Z_i:=(\mathbb{1}_i(\cdot)\Pi, \Gamma )$ 
and set
$$\mathcal{R}_Z:=\sum_i \mathcal{R}_{Z_i} \ .$$
Observe, that $\mathcal{D}^\gamma_Z\subset \mathcal{D}^\gamma_{Z_i} $. 
Then the estimate (\ref{easier to prove reconstruction cond}) follows from the identity 
$$\mathcal{R}_Zf -\Pi_pf(p)= \sum_i \big(\mathcal{R}_{Z_i}f - \mathbb{1}_i(\cdot) \Pi_pf(p)\big) $$
combined with the previous step.
\end{enumerate}
\end{proof}
%
By a straightforward adaptation of the proof in \cite{Hai14} and by going to $\fraks$-charts on $M$ one finds the following.
\begin{prop}[Hai14, Prop. 7.2]\label{local reconstruction theorem}
Under the assumptions of Theorem~\ref{reconstruction theorem} one has the bound
$$ |(\mathcal{R}f -\Pi_p f(p) )(\phi^\lambda_p)|\lesssim \lambda^\gamma \sup_{p,q\in B_{\lambda}(p)} \sup_{\alpha<\gamma} \frac{|f(p)-\Gamma_{p,q}f(q)|_\alpha}{d_\fraks(p,q)^{\gamma-\alpha}} \ ,$$
where the implicit constant is of order $1+\|Z\|_{B_p(2\lambda)}$ . Furthermore, the analogue bound holds for the difference of two models.
\end{prop}

\begin{corollary}\label{prop:reconstruction for continous models}
Suppose, in the setting of Theorem~\ref{reconstruction theorem}, the model $Z=(\Pi,\Gamma)$ has the property that that there exists an $\epsilon>0$ such that for every $\tau_p\in T^{\frakt}$
$\Pi_p \tau_p \in \mathcal{C}^\epsilon (E^\frakt)$, then for every $f\in \mathcal{D}^\gamma_Z$ one has the identity
$$\mathcal{R}_Z f (p)= \left(\Pi_p f(p)\right) (p)\ .$$
\end{corollary}

\subsection{More on jet models} 
The following observation is trivial for polynomials in the flat setting, but for admissible realisations on manifolds worth pointing out. 
\begin{lemma}\label{relating}
Let $R$ be an admissible realisation of $JE$. Then, for every compact set $K$ and $k\in \mathbb{N}$, there exist constants $c,C>0$, such that uniformly over $p\in K$, $\tau_p\in (JE)_{<k}$ and $\lambda\in (0,1)$
$$c \sum_{l=1}^k \lambda^l |Q_l \tau_p | \leq \sup_{\phi^\lambda_p \in \mathcal{B}^\lambda_p } |\langle R\tau_p, \phi^\lambda_p\rangle | \leq C\sum_{l=1}^k \lambda^l |Q_l \tau_p |  $$
\end{lemma}

\begin{proof}
First we proof the lemma in the scalar valued case.
The upper bound follows directly from the fact that the jet model is a model. The lower bound can be obtained as follows:
\begin{enumerate}
\item It follows from Remark~\ref{norm remark}, that if $(\phi^\lambda_p)_{\lambda\in (0,1), p\in M}$ is a family of Dirac sequences, such that for some $c>0$,  $c\phi_p^\lambda\in \mathcal{B}^{r+k,\lambda}_p$ and we denote by $\partial_i$ the coordinate vector fields of a $\fraks$-exponential chart at $p$, then 
\begin{equation}\label{convergence}
|Q_n \tau_p|\leq \max_{|l|_\fraks=n} \lim_{\lambda\to 0} |\langle R\tau_p, (\partial^l)^* \phi^\lambda_p\rangle |\ .
\end{equation}
Fix $(\phi^\lambda_p)_{\lambda\in \mathbb{N}, p\in M}$ such that for $\lambda\to 0$, 
$\phi^\lambda_p\to \delta_p$ converges uniformly in $p$ over compact sets (in some distribution space), then the same also holds for $(\partial^l)^* \phi^\lambda_p\to (\partial^l)^*\delta_p$ . Therefore the limit in (\ref{convergence}) can  be taken to be uniform over $p$ in compacts. 
\item Thus by the previous step, for every $K$ compact, there exists $\lambda_K>0$, such that for all $\lambda \leq \lambda_{K,k}$, $p\in K$ 
\begin{align*}
\sum_{n\leq k} \lambda^n  |Q_n \tau_p| &
< 2\sum_{n\leq k} \lambda^n \max_{|l|_\fraks=n}  |\langle R\tau_p, \partial^l \phi^\lambda_p\rangle |\\
&\lesssim_k \sup_{n}\big( \lambda^n \max_{|l|_\fraks=n} |\langle R\tau_p, \partial^l \phi^\lambda_p\rangle |\big)\\
&=\max_{\psi^\lambda\in \tilde{\mathcal{B}}^\lambda_{p}} |\langle R\tau_p, \psi^\lambda_p\rangle | \ ,
\end{align*}
where we set $\tilde{\mathcal{B}}^\lambda_{p}= \big\{ \lambda^{|l|_\fraks} \partial^l \phi^\lambda_p \ | \ |l_\fraks|\leq k\big\}\subset\mathcal{ B}^{r,\lambda}_p.$
To see the first inequality one first observe it for the case $\tau_p\in T_n$.
\item The claim for the case $\lambda=1$ is straight forward, thus we conclude the general case by interpolation.
\end{enumerate}
This completes the proof in the scalar valued case. The bundle valued case follows by combining the lemma in this case with Remark~\ref{rem isometry} and Remark~\ref{rem frame testfunctions}.
\end{proof}
The following is a direct application of Lemma~\ref{relating} .
\begin{lemma}\label{bounding abstract pol}
Let $\delta\in (0,+\infty)$ and $\bar{\mathcal{T}}:=(E, \bar T, L)$ be a jet regularity structure such that $\triangle \subset (\delta,+\infty]$ and $Z=(\Pi,\Gamma)$ a jet model. Let $\gamma\notin \NN$ s.t. $\gamma<\delta$ and suppose $f$ is a continuous section of $\bar T_{<\gamma,;}$ .
\begin{enumerate}
\item If $f$ satisfies for every $K$ compact a bound of the form
$$|\big(\Pi_p f(p) - \Pi_p \Gamma_{p, q} f(q) \big)(\phi^\lambda_p) |\lesssim_K \lambda^{\gamma}$$ uniformly over $p,q\in K$ and $\lambda\in [2^{-1}d_\fraks(p,q), d_\fraks(p,q)]$, then the following bound holds
$$|f(p)-\Gamma_{p, q} f(q)|_m\lesssim_K d_\fraks(p,q)^{\gamma-m} \ $$
for $m<\gamma$
and thus $f\in \mathcal{D}^\gamma_Z$ .
\item If $f \in \mathcal{D}^\gamma_Z(\bar{T})$ satisfies uniformly over $K$ compact and $p\in K$ a bound of the form
$$|\Pi_p f(p) (\phi_p^\lambda)|\lesssim \lambda^\gamma$$ 
then $f=0$.
\end{enumerate}

\end{lemma}

\subsection{H\"older functions}
A function $f: \RR^d\mapsto \mathbb{R}$ is said to belong to the H\"older space\footnote{We point out that for integer $\alpha\in \mathbb{N}$ this space is different from the usual notion of $C^\alpha$ spaces.} $\mathcal{C}^\alpha_{\mathfrak{s}}(\RR^d)$ for some $\alpha>0$, if for every compact set $K\subset \RR$ the following semi-norm is finite
$$\| f\|_{{c}^\alpha_K} := \sup_{x\in K} \inf_{P\in \mathcal{P}^{\lfloor \alpha \rfloor} } \sup_{y\in K : |x-y|_\fraks \leq 1} \frac{|f(y)-\Pi_x P(y)|}{|x-y|_\fraks^{\alpha}} \ ,$$ 
where $\mathcal{P}^{\lfloor \alpha \rfloor}$ is as in Section~\ref{jets in local coordinates} and $\Pi_x: \mathcal{P}^{\lfloor \alpha \rfloor}\to \mathcal{C}^\infty(\RR)$ is the linear map specified by $\Pi_x X^k(y) =  (y-x)^k$.  It is a classical fact that for $f\in \mathcal{C}^\alpha_{\mathfrak{s}}$, there exists for each $x\in \RR$ a unique $P_x(f)\in \mathcal{P}^{\lfloor \alpha \rfloor}$ such that
$$\sup_{y : |x-y|_\fraks \leq 1} \frac{|f(y)-\Pi_xP_x(f) (y)|}{|x-y|_\fraks^{\alpha}} <+\infty \ .$$ 
Fix a norm $| \cdot |$ on the finite dimensional vector space $\mathcal{P}^{\lfloor \alpha \rfloor}$ and define 
$$\| f\|_{\mathcal{C}^\alpha(K)}:= \sup_{x\in K} | P_x(f) |  + \| f\|_{{c}^\alpha_K} \ .$$
(Note that $\| f\|_{\mathcal{C}^\alpha_K}$ is a norm on functions in ${\mathcal{C}_\mathfrak{s}^\alpha}$ which are supported in $K$. )
We extend the definition of $\mathcal{C}^\alpha_{\mathfrak{s}}$ to manifolds, resp.\ vector bundles in the usual manner.
\begin{definition}\label{definition Holder section}
Let $E\to M$ be an $n$-dimensional vector bundle. Fix a locally finite partition of unity $\{\mathbb{1}_i\}_{i\in \NN}$ of $M$ subordinate to an atlas of $\fraks$-charts $\{(U_i,\phi_i)\}_{i\in \mathbb{N}}$ and let $\{\Phi_i\}_{i\in \mathbb{N}}$ be subordinate trivialisations, i.e.\ $\Phi_i= (G_i,\phi_i)$. 
We say that $f\in \mathcal{C}^\alpha (E)$, if for any compact set $K \subset M$
$$\|f\|_{\mathcal{C}^\alpha_{\mathfrak{s}}(K)}:=\sum_{j=1}^n \sum_{i\in \mathbb{N}:  \supp \mathbb{1}_i\cap K \neq \emptyset} \left\|\big((\Phi_i^{-1})^* f_i\big)_j\right\|_{\mathcal{C}^\alpha_{\mathfrak{s}}(\phi^{-1}(K\cap \supp \mathbb{1}_i) )} <\infty \ , $$
where $f_i:= \mathbb{1}_i \cdot f$.
\end{definition}
\begin{remark}
It is straightforward to see that the semi-norms obtained by using a different partition of unity and local trivialisations are equivalent.
\end{remark}

\begin{lemma}\label{lemma equivalent characterisations of holder}
Fix $\alpha \in \mathbb{R}_+\setminus \mathbb{N}$. For a continuous section $f\in \mathcal{C}(E)$ the following properties are equivalent:
\begin{enumerate}
\item\label{111} $f\in \mathcal{C}^\alpha_{\mathfrak{s}}(E)$
\item\label{222} There exists a unique continuous section $\{F_p\}_{p\in M}$ of $(JE)_{<\alpha}$, such that for any admissible realisation $R$ 
\begin{equation}\label{equation jet characterisation}
\sup_{p\in K} \inf_{F_p\in (J_p E)_{<\alpha} } \sup_{q\in M : d_\fraks(p,q) \leq 1} \frac{|f(q)-RF_p(q)|}{d_\fraks(p,q)^{\alpha}}<\infty 
\end{equation}
for every compact set $K\subset M$.
\item\label{444} There exists a unique section $F$ of $JE_{<\alpha}$ such that for any jet regularity structure $\bar{\mathcal{T}}=(E,\bar{T},L)$ such that $\triangle\subset(\alpha, +\infty]$ and $JE_{<\alpha}\subset \bar{T}$ one has $F\in \mathcal{D}_Z^\alpha(\bar{T})$ for any jet model $Z$ and furthermore 
$$\mathcal{R}_ZF=f\ .$$
\end{enumerate}
\end{lemma}

\begin{proof}
One easily sees that  (\ref{111}) and (\ref{222}) are equivalent.
To show (\ref{222}) implies (\ref{444}) we check that $F$ as defined in (\ref{222}) is a modelled distribution and $\mathcal{R}F=f$ .
Let $p,q\in M$ and $\lambda\in [2^{-1}d_\fraks(p,q), 2 d_\fraks(p,q)]$, since
\begin{align*}
\Pi_p(F(p)- \Gamma_{p,q}F(q))(\phi^\lambda_p)=& (\Pi_pF(p)-f)(\phi^\lambda_p)+  (f-\Pi_q F(q))((\phi^\lambda_p))\\
&+
(\Pi_q F(q)- \Gamma_{p,q}F(q))(\phi^\lambda_p)\ ,
\end{align*}
one obtains
$$|\Pi_p(F(p)- \Gamma_{p,q}F(q))(\phi^\lambda_p)|\lesssim \lambda^{\alpha}$$
by using (\ref{222}) for the first term, (\ref{222}) combined with Remark~\ref{translated test function remark} for the second term and the bound in the definition of a model for the third term. 
Thus the claim follows from Lemma~\ref{bounding abstract pol}. 

Clearly (\ref{444}) implies (\ref{222}) by the reconstruction theorem, Theorem~\ref{reconstruction theorem}.
\end{proof}

\begin{remark}\label{Lifting smooth functions}
Let $E\to M$ be a vector bundle over a compact manifold $M$ and $\alpha\notin \NN$ as well as $\bar{\mathcal{T}}=(E,\bar{T},L)$ be a jet regularity structure such that $\triangle\subset(\alpha, +\infty]$ and $JE_{<\alpha}\subset \bar{T}$. It follows from the previous proposition that the reconstruction operator $\mathcal{R}_Z: \mathcal{D}^\alpha \to \mathcal{C}^\alpha$ is a bijective bounded map. 
Therefore by the open mapping theorem of functional analysis it has a bounded inverse. It follows from the proof of Lemma~\ref{lemma equivalent characterisations of holder} that this inverse is explicitly given by the map 
$$\mathcal{C}^\alpha\to \mathcal{D}^\alpha, \quad f\mapsto (i_{[\alpha],\infty}\circ j^{[\alpha]}) f \ .$$
\end{remark}

\section{Local Operations on Modelled Distributions}\label{section Local operations}
In this section we introduce linear and non-linear operations on modelled distributions.

\subsection{Products and traces}
\begin{definition}\label{definition product}
Given a regularity structure ensemble $\mathcal{T}=(\{E^\mathfrak{t}\}_{\mathfrak{t}\in \mathfrak{L}},\{T^\mathfrak{t}\}_{\mathfrak{t}\in \mathfrak{L}}, L )$, $\mathfrak{t}_3\in \mathfrak{L}$ and two sectors $V^{\mathfrak{t}_1}\subset T^{{\mathfrak{t}_1}}, {V}^{\mathfrak{t}_2}\subset T^{\mathfrak{t}_2}$ of regularity $\underline{\alpha}_{1}$ resp. $\underline{\alpha}_{2}$, 
a bilinear bundle morphism ${\star: V^{\mathfrak{t}_1}\times V^{\mathfrak{t}_2}\to T^{\mathfrak{t}_3}}$ is called an abstract  (tensor) product if the following properties are satisfied.
\begin{enumerate}
\item\label{asdfasdf} $E^{\frakt_1} \otimes E^{\frakt_2}=E^{\frakt_3}$ .
\item For any $\tau_p^1 \in V^{\mathfrak{t}_1}_{\alpha}$, $\tau_p^2\in V^{\mathfrak{t}_2}_{\alpha'}$ ones has $\tau_p^1\star  \tau_p^2\in T^{\mathfrak{t}_3}_{\alpha+\alpha'}$ .
\end{enumerate}
The product $\star$ is called $\gamma$-precise, if furthermore the sectors $V^{\mathfrak{t}_1}_{\alpha}$, $V^{\mathfrak{t}_2}_{\alpha'}$ are $\gamma$-precise and for every $\tau_p^1 \in V^{\mathfrak{t}_1}_{\alpha}$, $\tau_p^2\in V^{\mathfrak{t}_2}_{\alpha'}$ such that $\alpha+\alpha'<\gamma$ one has
\begin{enumerate}
\item $\tau_p^1\star  \tau_p^2\in T^{\mathfrak{t}_3}_{\alpha+\alpha', >\gamma}\ $,
\item\label{asdff} 
for all $q\in M$ and $\Gamma_{q,p}\in L_{q,p}$ $$Q_{<\gamma,:}(\Gamma_{q,p}\tau_p^1\star \Gamma_{q,p}\tau_p^2)= Q_{<\gamma,:}\Gamma_{q,p}(\tau_p^1\star \tau_p^2)\ .$$
\end{enumerate}
Natural extensions of this definition are given as follows.
\begin{itemize}
\item Products between more than two sectors $V^{\mathfrak{t}_1}\times\ldots\times V^{\mathfrak{t}_n}$ are defined in the analogous manner.
\item When $E^\frakt_1=E^\frakt_2$ above, one defines symmetric tensor products by replacing the tensor product in \eqref{asdfasdf} by a symmetric tensor product and further enforcing the product $\star$ be commutative.
\end{itemize}
We shall often use the notation $\tau_p^1\star_\gamma  \tau_p^2:=Q_{<\gamma} (\tau_p^1\star  \tau_p^2)$.
\end{definition}

\begin{remark}
This definition is analogous to the definition of a product in \cite{Hai14} except for the following differences. On the one hand we need to keep track of the vector bundles, which is straightforward and is done using sector ensembles, secondly, we need to introduce the notion of the product being $\gamma$-precise.\footnote{In \cite{Hai14} the notion of being $\gamma$-regular is introduced. We prefer to to use different nomenclature, since this notion though similar, is distinct from being $\gamma$-precise.}
At first glance it would seem natural to impose that $\star$ maps $V^{\mathfrak{t}_1}_{\alpha,\delta}\times V^{\mathfrak{t}_2}_{\alpha',\delta'}$ into $T^{\mathfrak{t}_3}_{\alpha+\alpha', \bar{\delta}}$, where ${\bar{\delta}=\big(\delta+(\underline{\alpha}_{2}\wedge 0)\big)\wedge\big( \delta'+(\underline{\alpha}_{1}\wedge 0)\big)}$. But the additional freedom gained by not enforcing this choice is crucial when discussing aspects of renormalisation, c.f.\ Section~\ref{section regularity structure, models,...}.
\end{remark}
\begin{remark}
Given a regularity structure ensemble $\mathcal{T}=(\{E^\mathfrak{t}\}_{\mathfrak{t}\in \mathfrak{L}},\{T^\mathfrak{t}\}_{\mathfrak{t}\in \mathfrak{L}}, L )$ and two sectors $V^{\mathfrak{t}_1}\subset T^{{\mathfrak{t}_1}}, {V}^{\mathfrak{t}_2}\subset T^{{\mathfrak{t}_2}}$ of regularity $\underline{\alpha}_1$ resp.\ $\underline{\alpha}_2$, one can always define a tensor product between $V^{\mathfrak{t}_1}$ and ${V}^{\mathfrak{t}_2}$ as follows.
\begin{itemize}
\item Extend the set $\mathfrak{L}$ by a new element $\mathfrak{t}_3$, i.e.\ let $\bar{\mathfrak{L}}= \mathfrak{L}\cup \{\mathfrak{t}_3\}$,
\item set $E^{\frakt_3}:=E^{\frakt_1} \otimes E^{\frakt_2}$,
\item set for example 
\begin{equation}\label{eq:example_prod_construction}
T^{\frakt_3}_{\alpha,\delta}:=\bigoplus_{\substack{\alpha'+\alpha''=\alpha,\\ \big(\delta'+(\underline{\alpha}_{2}\wedge 0)\big)\wedge\big( \delta''+(\underline{\alpha}_{1}\wedge 0)\big)=\delta }} V^{\frakt_1}_{\alpha',\delta'} \otimes V^{\frakt_2}_{\alpha'',\delta''}\ ,
\end{equation}
\item extend $ L_{p,q}$ to act on $T^{\frakt_3}$ multiplicatively,
\item set $\bar{\mathcal{T}}:=(\{E^\mathfrak{t}\}_{\mathfrak{t}\in \bar{\mathfrak{L}}},\{T^\mathfrak{t}\}_{\mathfrak{t}\in \bar{\mathfrak{L}}}, L )$.
\end{itemize}

Observe that in general there is no canonical extension of a non smooth model $Z$ on $\mathcal{T}$ to a model on $\bar{\mathcal{T}}$. An important exception occurs when $V^{\mft_1}$ or $V^{\mft_2}$ is a jet sector and $Z$ is a model that acts on that sector as a jet model.
\end{remark}

\begin{remark}
Note that while the construction in the previous remark is rather canonical, in certain situations there is more structure present which makes other constructions more natural. A very important example is the case of jet regularity structures.
Assume in the setting of the previous remark that
$V^{\mathfrak{t}_1}\subset J{E^{\mathfrak{t}_1}}$ and ${V}^{\mathfrak{t}_2}\subset J{E^{\mathfrak{t}_2}}$, then it is natural define the product to be the restriction of the canonical product 
$$ m: J{E^{\mathfrak{t}_1}}\otimes  J{E^{\mathfrak{t}_2}} \to  J ( {E^{\mathfrak{t}_1}}\otimes{E^{\mathfrak{t}_2}})$$
introduced in Section~\ref{remark algebraic structure}.
In order to satisfy Property \eqref{asdff} of being a $\gamma$-regular product for some $\gamma>0$, needs to restrict the space $L$ accordingly.
This in turn makes the space of jet models smaller, but by Point~\ref{tensor products on jets} in Section~\ref{remark algebraic structure} it is non-empty.
We shall call the product $m$ on jet sectors the \textit{canonical product on jets} and write $\star$ for it. 
\end{remark}

The following can be proved exactly as \cite[Theorem 4.7 and Prop. 4.10]{Hai14}.
\begin{lemma}\label{lemma product}
Let $V^{\mathfrak{t}_1}\subset T^{\mathfrak{t}_1}, {V}^{\frakt_2}\subset T^{\frakt_2}$ be two sectors of regularity $\alpha_1$ resp.\ $\alpha_2$ and let $\star: V^{\mathfrak{t}_1}\times  V^{\mathfrak{t}_2}\to T^{\mathfrak{t}_3}$ be an abstract tensor product. Let $Z$ be a model and $\gamma_1,\gamma_2>0$. Assuming the product $\star$ is $\gamma$-precise for $\gamma= (\gamma_1+\alpha_2) \wedge (\gamma_2+\alpha_1)$, the map
$$
\star_\gamma:\ \mathcal{D}^{\gamma_1}(V^{\frakt_1})\times \mathcal{D}^{\gamma_2}(V^{\frakt_2}) \to \mathcal{D}^{(\gamma_2+\alpha_1) \wedge (\gamma_1+\alpha_2) }(T^{\frakt_3}), \qquad
(f_1,f_2)\mapsto f_1\star_\gamma f_2:= Q_{\gamma} f_1\star f_2
$$
%
is well defined, continuous and the bound 
$$\interleave f_1\star_\gamma f_2  \interleave_{\gamma ; K}\lesssim \interleave f_1\interleave_{\gamma_1 ; K}\interleave f_2 \interleave_{\gamma_2 ; K}(1+ \|\Gamma\|_{ K})$$
holds.
If furthermore $\tilde Z = (\tilde \Pi, \tilde{\Gamma})$ is a second model and $\tilde{f}_1\in \mathcal{D}_{\tilde{Z}}^{\gamma_1}(V^{\frakt_1})$, $\tilde{f}_2\in \mathcal{D}_{\tilde{Z}}^{\gamma_2}(V^{\frakt_2})$ one has the bound 
$$\interleave f_1\star_\gamma f_2 ; \tilde{f}_1\star_\gamma \tilde{f}_2 \interleave_{\gamma ; K}\lesssim_C \interleave f_1 ; \tilde{f}_1 \interleave_{\gamma_1 ; K}+\interleave f_2 ;  \tilde{f}_2 \interleave_{\gamma_2 ; K}+ \|\Gamma-\tilde{\Gamma}\|_{ K}$$
uniformly over  $f_i, \tilde{f}_i$ satisfying $\interleave f_i \interleave_{\gamma_i ; K}, \interleave  \tilde{f}_i \interleave_{\gamma_i ; K}<C$ and models satisfying 
$\|\Gamma\|_{ K}, \|\tilde{\Gamma}\|_{K}<C$.
\end{lemma}

\subsubsection{Trace operator}
Let $E, F$ be vector bundles, recall the trace operator $\tr: E^*\otimes E\otimes F\to F$  acting on elementary tensors as $e_p^*\otimes e_p \otimes f_p\mapsto (e_p^*,e_p) f$. Denote by $I$ the $E\otimes E^*$ tensor field, such that $(I_p,e_p^*\otimes e_p)= (e_p^*, e_p)$ and observe that the dual map of $\tr$ is given by $$F\to E^*\otimes E\otimes F \quad f\mapsto I\otimes f\ . $$
Thus there is a natural extension of the trace operator 
$$\tr:\mathcal{D}'(E^*\otimes E\otimes F)\to \mathcal{D}'(F) $$
characterised by the fact that for each $\xi\in \mathcal{D}'(E^*\otimes E\otimes F)$ and $\phi \in \mathcal{D}(F) $ one has
$$\langle \tr \xi, \phi \rangle =  \langle \xi, I \otimes \phi \rangle \ .$$

\begin{definition}\label{def abstract trace}
Let $E^{\frakt_1}= E^*\otimes E\otimes F$, $E^{\frakt_2}= F$. An abstract trace operator is a bi-graded bundle morphism $\mathfrak{\trr}: T^{\mathfrak{t}_1}\mapsto T^{\mathfrak{t}_2}$ satisfying 
\begin{itemize}
\item for any $\tau_p\in T^{\frakt_1}_{\alpha,\delta}$ one has $\trr \tau_p \in T^{\frakt_2}_{\alpha,\delta}$
\item writing $\iota_1: T^{\frakt_1}_0\to E^{\frakt_1}$, $\iota_2: T^{\frakt_2}_0\to E^{\frakt_2}$ for the canonical isomorphisms one has 
$$\iota_2\circ \trr = \tr \circ \iota_1  ,$$
\item for each $\Gamma_{p,q}\in L_{p,q}$ one has $\Gamma_{p,q} \trr =\trr\Gamma_{p,q} $
\end{itemize}
A model realizes an abstract trace operator if $$\Pi_p(\trr \tau_p) = \tr (\Pi_p\tau_p) \ .$$
\end{definition}
\begin{remark}\label{remark trace}
It is obvious from the definition that for $f\in \mathcal{D}^\gamma (V^{\mathfrak{t}_1})$ one has $\trr f\in \mathcal{D}^\gamma(T^{\mathfrak{t}_2})$ and that
$$\tr \mathcal{R}f= \mathcal{R}\trr f\ .$$
\end{remark}
\begin{remark}
In the rest of this section we shall denote by $\tr$ trace operators on vector bundles contracting different indices and it will be clear from the context which indices are contracted. In this case the corresponding abstract trace $\trr$ is defined analogously.
When working with SPDEs, the vector bundles and indexing sets will be treated more systematically but in order to not distract from the main analytic content of this section this is postponed to Section~\ref{section sets of trees for...}.
\end{remark}

\begin{remark}
In the setting where $E^\frakt= E^*\otimes E\otimes F$ but the regularity structure ensemble does not carry an abstract trace, there is an obvious extension of the regularity structure ensemble which carries a trace operator as in Definition~\ref{def abstract trace}. Furthermore, in this case models have a canonical extension.
In the case of jet regularity structures ensembles, there is a canonical trace described in Point~\ref{compatibility with trace} of Section~\ref{remark algebraic structure}. (In which case one needs to restrict the choice of admissible realisation accordingly.)
\end{remark}

\subsubsection{Linear functionals}\label{section linear}
Let $\mathcal{T}=(\{E^\mathfrak{t}\}_{\mathfrak{t}\in \mathfrak{L}},\{T^\mathfrak{t}\}_{\mathfrak{t}\in \mathfrak{L}}, L )$ be a regularity structure ensemble and $Z$ a model for it. For $A: E^{\frakt_1} \to F$, a vector bundle morphism and $V^{\mathfrak{t}_1}$ be a sector of regularity $\alpha$, we aim to define for $f\in \mathcal{D}^\gamma (V^{\mathfrak{t}_1})$ a modelled distribution $A(f) \in \mathcal{D}^\gamma\ . $
First identify $A$ with a section of $E^{\frakt_2}:=(E^{\frakt_1})^*\otimes F$ and assume the regularity structure ensemble contains $(JE^{\frakt_2})_{<\gamma+|\alpha|}$ on which $Z$ acts as a jet model. Set $\bar{A}= Q_{< \gamma+|\alpha|} j A\in \mathcal{D}^{\gamma+|\alpha|} (E^{\frakt_2})$.
Assuming the regularity structure ensemble carries a $\gamma$-regular tensor product in the sense of Definition~\ref{definition product} on $J E^{\frakt_2}\times V^{\mathfrak{t}_1}$ and an abstract trace extending the canonical trace
$$ \tr:  E^{\frakt_2}\otimes E^{\frakt_1} \to F \ ,$$ we define
$$A(f):=\trr\big( \bar{A} \star_\gamma f\big) \ .$$
\begin{lemma}\label{lem lin map}
In the above setting, $A(f) \in \mathcal{D}^\gamma $ and the following continuity bound holds
$$\interleave A(f) \interleave_{\gamma ; K}\lesssim_A \interleave f \interleave_{\gamma_1 ; K}(1+ \|\Gamma\|_{ K})$$
as well as the analogue to the second bound of Lemma~\ref{lemma product}.
If one further assumes that the model $Z$ realises $\tr$ for $\trr$ and that for any $\tau^2_p\in (JE^{\frakt_2})_{<\gamma+|\alpha|}$ and $\tau^1_p\in V^{\frakt_1}$ one has 
$\Pi_p(\tau_p^2\star\tau^1_p)= \Pi_p\tau_p^2 \otimes\Pi_p\tau^1_p,$
then 
$$\mathcal{R}_Z A(f)= A(\mathcal{R}_Zf) \ .$$
\end{lemma}
\begin{proof}
The first part of the claim follows directly from Lemma~\ref{lemma equivalent characterisations of holder} , Lemma~\ref{lemma product} and Remark~\ref{remark trace}. The second part follows easily using the uniqueness of the reconstruction operator.
\end{proof}

\begin{remark}
Note that starting with $f\in \mathcal{D}^\gamma (V^\mathfrak{t})$ one can always extend the regularity structure and model canonically to carry all the structure needed to define $A(f)\in \mathcal{D}^\gamma$.
\end{remark}

\begin{remark}\label{rem:multilinear map} Note that a multilinear map $B:E^{\times n}\to F$ can be seen as a linear map $B: E^{\otimes n}\to F$ and thus as a section $B\in \mathcal{C}^\infty((E^*)^{\otimes n}\otimes F)$. Therefore, assuming the appropriate structure on the regularity structure ensemble, one can define for $f \in \mathcal{D}^\gamma (V) $
\begin{equation}\label{multilinear}
B(f):=\trr Q_{<(\gamma+(n-1)\alpha )\wedge \gamma} \big(\bar{ B} \star f^{\star n} \big) \ ,
\end{equation}
where  $\bar{B}=Q_{<\gamma + n| \alpha |} j B$. One finds similarly as above that if $V$ has regularity $\alpha$, then map 
$$B: \mathcal{D}(V^\mathfrak{t})\to \mathcal{D}^{(\gamma+(n-1)\alpha )\wedge \gamma}\ , \quad f \mapsto B(f)(\cdot):=Q_{<(\gamma+(n-1)\alpha )\wedge \gamma} \trr \big( (j_\cdot B) \star f^{\star n} (\cdot)\big)  $$
is well defined and continuous. Though this time, under the assumption for the second part of Lemma~\ref{lem lin map} one can only conclude the identity 
$$ \mathcal{R} B(f)= B (\mathcal R (Q_{<\gamma}f^{\star n})) \ .$$
\end{remark}

\subsubsection{Differential operators}
In this section we lift $k$-th order differential operators\footnote{See Section~\ref{section differential operators} for the definition of a $k$-th order differential operator in the setting of a non-trivial scaling $\mathfrak{s}$.} to the regularity structure. 
\begin{definition}\label{lift of differential}
Given a sector $V^{\mathfrak{t}_1}\subset T^{\frakt_1}$, a bundle morphism $\D: V^{{\frakt_1}}\to T^{{\mathfrak{t}_2}}$ is called an abstract $k$-th order differential operator if
\begin{itemize}
\item $\D\tau_p \in T^{{\frakt_2}}_{\alpha-k,\delta-k}$ for every $\tau_p \in V^{\mathfrak{t}_1}_{\alpha,\delta}$,
\item $\D \Gamma_{p,q}=\Gamma_{p,q}\D$ for each $\Gamma_{p,q}\in L_{p,q}\ . $ 
\end{itemize}
A model $Z$ is said to realise a $k$-th order differential operator $\mathcal{A}: \mathcal{C}^\infty(E^{\frakt_1})\to \mathcal{C}^\infty (E^{\frakt_2})$ for $\D$, if 
\begin{itemize}
\item $\D$ is an abstract $k$-th order differential operator,
\item $\Pi_p \D \tau_p= \partial \Pi_p \tau_p$ .
\end{itemize}
\end{definition}
The proof of the following Lemma adapts ad verbatim from the flat setting, c.f.\ \cite[Prop.~5.28]{Hai14}
\begin{lemma}
Let $V^{\mathfrak{t}_1}\subset T^{\frakt_1}$ be a sector and $\D: V^{{\frakt_1}}\to T^{{\mathfrak{t}_2}}$ an abstract $k$-th order differential operator. 
If $f\in \mathcal{D}^\gamma (V^{\mathfrak{t}_1})$, then $\D f\in\mathcal{D}^{\gamma-k} (V^{\mathfrak{t}_2})$ and if furthermore $\gamma>k$ and a model $Z$ realises $\mathcal{A} : \mathcal{C}^\infty(E^{\frakt_1})\to \mathcal{C}^\infty (E^{\frakt_2})$ for $\D$, then 
$$\mathcal{R}_Z\D f= \partial \mathcal{R}_Z f\ .$$
\end{lemma}

While this definition allows to lift any differential operator, the next remark describes a way to lift all $k$-th order operators ``at once''. 
\begin{remark}
Suppose the operators $\{D^l\}_{|l|_\fraks\leq k}$ from Section~\ref{section higher order derivatives} have been lifted to abstract differential operators $\{\mathcal{D}^l \}_{|l|_\fraks\leq k}$ on the regularity structure ensemble and for every $l\in \mathbb{N}^d$ such that $ |l|_\fraks\leq k$ the model realises $D^l$ for $\mathcal{D}^l$. Recall, that any $k$-th order differential operator on $E$, by definition factors through the jet bundle $J^k(E)\simeq \bigoplus_{m=0}^k (T^*M)^m\otimes E $, i.e for each $k$-th order differential operator $\mathcal{A}: \mathcal{C}^\infty(E)\mapsto \mathcal{C}^\infty(F)$ there exists a unique bundle morphism $T_\mathcal{A}\in L\big(\bigoplus_{m=0}^k (T^*M)^m\otimes E, F\big)$ such that $\mathcal{A}= T_\mathcal{A}\circ (\oplus_{|l|_\fraks\leq k} {D}^l)$ up to canonical isomorphism.
Thus, it is natural to set
$$\D f(p)=T_\mathcal{A} (\oplus_{|l|_\fraks\leq k} \mathcal{D}^l f)(p), $$
where the linear functional $T_\mathcal{A}$ is lifted as in Section~\ref{section linear}.
\end{remark}

\subsection{General nonlinearities}\label{section nonlinearities}
Next we shall define the composition of a modelled distribution with a smooth nonlinearity $G: E\to F$ with the property that $G(e_p)\in F|_p$ for each $e_p\in E|_p$. Let $\pi_E: E\to M $ be the projection map and let $E\ltimes F:= \pi_E^*F$ be the pull-back bundle of $F$ through $\pi_E$. The starting point for our construction will be the observation that $G$ as above can naturally be seen as a section of $E\ltimes F$ and an investigation of the space $J^\infty(E\ltimes F)$.
\subsubsection{Properties of $J^\infty (E\ltimes F)$}
First we introduce the following useful operator.
\begin{definition}\label{def alternative jet}
For a vector bundle $\tilde{F}\to M$ define the operator $$j^E_\cdot: \mathcal{C}^\infty(E\ltimes \tilde F)\to \mathcal{C}^\infty (E\ltimes J^\infty\tilde F) $$ as follows. Recall the construction associating to each $e_p\in E|_p$ a vector field $e_p(\cdot)$ by parallel translation in (\ref{paralell}). For $H\in \mathcal{C}^\infty(E\ltimes \tilde F)$ let 
$j_\cdot^E(H)(\cdot)\in \mathcal{C}^\infty (E\ltimes J^\infty_p\tilde F) $ be given by $j_p^E(H)(e_p)= j_p H( e_p(\cdot))$. This is well defined since $H( e_p(\cdot))\in \mathcal{C}^\infty(\tilde{F})$.
\end{definition}
Let $P(E):= \bigoplus_n (E^*)^{\otimes_s n}$, then one has the following lemma.
\begin{lemma}\label{lemma: polynomial special case for nonlinearity}
Denote by $i_F: \cC^{\infty}(P(E)\otimes F)\to \cC^{\infty}(E\ltimes F)$ and $i_{JF}: \cC^{\infty}(P(E)\otimes J^\infty F)\to C^{\infty}(E\ltimes J^\infty F)$ the canonical injections. Then the following diagram commutes
\[ \begin{tikzcd}
\mathcal{C}^\infty(P(E)\otimes F)  \arrow{r}{i_F} \arrow[swap]{d}{j_\cdot} &C^{\infty}(E\ltimes F)\arrow[swap]{d}{j_\cdot^E}  \\%
\mathcal{C}^\infty (J^\infty(P(E)\otimes F))\simeq \mathcal{C}^\infty (P(E)\otimes J^\infty F) \arrow{r}{i_{JF}}& \mathcal{C}^\infty (E \ltimes (J^\infty F)) \ .
\end{tikzcd}
\]
\end{lemma}
\begin{proof}
We show that $ j^E_p (i_F(A))= i_{JF} (j_p A)  $
for every for $A\in C^{\infty}(P(E)\otimes F)$.
Indeed for each $p$ we can locally write $$A(q) = \sum_I \alpha_I(q) (e^*_{1,p}(q))^{\otimes i_1}\otimes\ldots\otimes (e^*_{n,p}(q))^{\otimes i_n}\ ,$$
where $I$ runs over a finite subset of $\mathbb{N}^{\dim E}$, $\{e_{i,p}\}_{i=1}^{\dim E}$ is an orthonormal basis of $E_p$, $e_{i,p}(\cdot)$ are the corresponding local sections as in Definition~\ref{def alternative jet} and $e^*_{i,p}(\cdot)$ are the dual sections.
Thus 
$$ A(e_p(q))=\sum_{I } \alpha_I(q) \langle e^*_{1,p}(q), e_p(q)\rangle^{i_1}\cdots\langle e^*_{n,p}(q), e_p(q)\rangle^{i_n} \ ,$$
which implies that
$$j^E_p (i_F(A))(e_p)= \sum_{I }  \langle e^*_{1,p}(p), e_p(p)\rangle^{i_1}\cdots\langle e^*_{n,p}(p), e_p(p)\rangle^{i_n} j_p\alpha_I, $$
since each $\langle e^*_{j,p}(q), e_p(q)\rangle$ is constant.
On the other hand we have 
$$j_p A\simeq \sum_I  (e^*_{1,p}(q))^{\otimes i_1}\otimes\ldots\otimes (e^*_{n,p}(q))^{\otimes i_n}\otimes j_p\alpha_I$$
under the canonical isomorphism $ J^\infty(P(E)\otimes F)\simeq P(E)\otimes J^\infty F$ which implies the claim.
\end{proof}

\begin{remark}\label{translation replaced by admissible realisation}
Observe that for any $H\in \mathcal{C}^\infty(E\ltimes \tilde F)$ 
$$j^E_p(H)(e_p)=j_p H( Re_p (\cdot))\ ,$$ where $R$ is an admissible realisation of $JE$ and $e_p\in E|_p$ is interpreted as an element of $(J_p E)_0 \subset J_p E$ using the canonical isomorphism.
\end{remark}

\begin{remark}\label{remark interchanging derivatives}
Clearly, for any $H\in \mathcal{C}^\infty(E\ltimes \tilde F)$ and $g\in \mathcal{C}^\infty(E)$
$$D^n j_p H( g(\cdot))= j_p D^n H( g(\cdot)) \ $$
up to canonical isomorphism.
\end{remark}

The next Lemma provides the intuition on which the rest of this section is based. 
\begin{lemma}\label{lemma:canonical isom}
There is a canonical isomorphism 
$J (E\ltimes F)\simeq E \ltimes \big( P (E) \otimes JF \big)$ .
\end{lemma}
\begin{proof}
Not that the connection on $E$ identifies a canonical horizontal direction complementing the vertical subspace of the tangent bundle of $E$ with the property that 
$T_{\text{hori}}E \simeq E\ltimes TM$ and $T_{\text{vert}}E \simeq E\ltimes E$.
Therefore we can decompose the cotangent space as $T^* E= T_{\text{hori}}^*E\oplus T_{\text{vert}}^*E$, note the canonical isomorphisms
$$J(E\ltimes F)\simeq \big(\bigoplus_k (T^*E)^{\otimes_s k}\big)\otimes (E\ltimes F) \simeq (E\ltimes F)\otimes \bigoplus_k (T_{\text{hori}}^*E\oplus T_{\text{vert}}^*E)^{\otimes_s k}\ ,$$
where the fist isomorphism was given in Section~\ref{section infinite jet bundle}.
Since $$(T_{\text{hori}}^*E\oplus T_{\text{vert}}^*E)^{\otimes_s k}\simeq \bigoplus_{m+n=k} (T_{\text{hori}}^*E)^{ \otimes_s n}\otimes (T_{\text{vert}}^*E)^{\otimes_s m}\ ,$$ 
and since
$T_{\text{hori}}^*E \simeq E\ltimes T^*M$ and $T_{\text{vert}}^*E \simeq E\ltimes E^*$
we obtain 
\begin{align*}
J(E\ltimes F) & \simeq  E\ltimes \big(\bigoplus_{m,n} (T^*M)^{ \otimes_s n}\otimes F\otimes  (E^*)^{\otimes_s m}\big)\\
&\simeq (E\ltimes JF)\otimes \big( E\ltimes\big(\bigoplus_{n} (E^*)^{\otimes_s n} \big)  \big) \\
& \simeq\big( E\ltimes\big( \bigoplus_{n} (E^*)^{\otimes_s n}\big) \big)\otimes (E\ltimes JF) \\
& \simeq E \ltimes \big(\bigoplus_{n} (E^*)^{\otimes_s n}\otimes JF  \big) \\
& = E \ltimes \big( P(E)\otimes JF \big) \ .
\end{align*}
\end{proof}

\begin{remark}
In view of of the isomorphism $J(E\ltimes F)\simeq E\ltimes (JF\otimes P(E))$, one sees that the following diagram commutes for any $\gamma>0$
\[\begin{tikzcd}
C^{\infty}(E\ltimes F) \arrow{r}{Q_{<\gamma}j_{\cdot}} \arrow[swap]{dr}{Q_{<\gamma}\sum_n\frac{1}{n} D^n j^E_\cdot}  &  \mathcal{C}^\infty (J(E\ltimes F))\arrow{d}{\sim} \\
& \cC^{\infty} (E \ltimes (JF\otimes P(E))) \ .
\end{tikzcd}  
\] 
We define the grading on $J(E\ltimes F)$ as the grading that turns the isomorphism $\sim$ into a graded isomorphism.
\end{remark}

\subsubsection{Lifting general nonlinearites}\label{section:lifting general nonlinearities}
We have seen in Lemma~\ref{lemma:canonical isom} that for every $p\in M$ there is a canonical isomorphism 
\begin{equation}\label{isomorphism for composeition}
J_{p,e_p}(E\ltimes F)\simeq J_p(F)\otimes P(E)_p \ .
\end{equation}
We shall decompose a modelled distribution $f$ with values in $\bar V$, a function-like sector, as
$f(p)=\bar{f}_p+ \tilde{f}_p$ where $\bar{f}= Q_0 f\in \bar{V}_0\simeq E$ and $\tilde{f}= f-\bar{f}\in \bar{V}_{>0}$ and formally write
\begin{equation}\label{wished for}
G(f)(p)= \big(j_{p, \bar{f}_p}G\big) (\tilde{f}_p)\ .
\end{equation}
We define the right-hand side of (\ref{wished for}) in two steps, for now ignoring truncations and the appropriate assumptions on the sector ensemble. The rigorous definition under the appropriate assumptions is given below.
\begin{enumerate}
\item First observe that since $E|_p$ is a vector space for each $p\in M$ and $e_p\in E_p$,
for any section $G\in \mathcal{C}^\infty(E\ltimes F)$ there is a canonical definition of $D^n G(p, e_p)\in (E|_p^*)^{\otimes_s n}\otimes F|_p$ as the total derivative in the ``vertical direction''. 
\item We identify $j_{p,e_p} G= \sum_n \frac{1}{n!} j^E_p (D^n G)(e_p)$, where each $j^E_\cdot (D^n G)$ is interpreted as a section of $E \ltimes \big((E^*)^{\otimes_s n}\otimes JF \big)$. Thus for $f= \bar{f}+\tilde f$, we formally set
\begin{equation}\label{composition}
G(f)(p)= \big(j_{p, \bar{f}_p}G\big) (\tilde{f}_p):= \sum_n \frac{1}{n!}\trr\big(j_p^E(D^nG)(\bar{f}_p ) \star \tilde{f}_p^{\star n})\big) \ ,
\end{equation}
with the convention that $\tilde{f}_p^{\star 0}=1\in \RR$. 
\end{enumerate}

We make the following assumption in order to make \eqref{composition} rigorous for modelled distributions in $\mathcal{D}^\gamma(\bar{V})$ for some $\gamma>0$.
\begin{assumption}\label{Assumption Nonlinearites}
Let $\mathcal{T}=(\{E^\mathfrak{l}\}_{\mathfrak{l}\in \mathfrak{L}}, \{T^\mathfrak{l}\}_{\mathfrak{l}\in \mathfrak{L}},L)$ be a regularity structure ensemble and $\{V^\mathfrak{t}\}_{\mathfrak{t}\in{\mathfrak{L}}}$ a sector ensemble such that the following properties are satisfied.
\begin{enumerate}
\item Let $\bar{V}=V^{\bar{\frakt}}\subset T^{\bar{\frakt}}$ be a function-like sector and let $\zeta= \min \big\{\alpha \in A \cap (0,\infty) \ : \ \bar{V}_{\alpha} \neq \{ 0\} \big\}$. 
Set $\bar{V}^{0}= M\times \mathbb{R}$ to be the trivial sector and $\bar{V}^1= \bar{V}$. Assume that for every $n\in\mathbb{N}$ such that $n-1\leq \frac{\gamma}{\zeta}$
\begin{enumerate}
\item $V^n$ is a sector containing $J(E^{\otimes n})_{<\gamma}$
\item for every $k,l\in \mathbb{N}$ such that $k+l<n$ there is a $\gamma$-precise abstract tensor product 
$$\star: \bar{V}^k \times \bar{V}^l\to \bar{V}^n$$ 
which agrees with the canonical product on jets.
\end{enumerate}

\item For every $k\in \mathbb{N}$ such that $k-1\leq \frac{\gamma}{\zeta}$ the sector ensemble contains sectors $W^k$ and $U^k$ with the property that 
\begin{enumerate}
\item $\big(J((E^*)^{\otimes_s k}\otimes F)\big)_{<\gamma}\simeq \big((E^*)^{\otimes_s k}\otimes JF)\big)_{<\gamma} \subset W^k$ and there is a $\gamma$-precise abstract tensor product 
$$\star: W^k \times \bar{V}^k\to U^k$$
which again agrees with the canonical product on jets.
\item For each $k$ there is an abstract trace $\trr: U^k \to W^0$ which agrees with the canonical trace on jets.
\end{enumerate}
\item the products $\star$ are associative in the obvious sense and for $n\in \mathbb{N}$ and $\tau\in \bar{V}$ we write 
$$\tau^{\star n} :=\tau \star \tau^{\star n-1}\quad \text{and} \quad \tau^{\star_\gamma n} :=\tau \star_\gamma \tau^{\star_\gamma n-1}$$
whenever the expressions make sense.
\end{enumerate}
\end{assumption}
%

\begin{theorem}\label{theorem composition}
Let $\mathcal{T}$ be a regularity structure ensemble, $\{V^{\mathfrak{l}}\}_{\frakl \in \bar{\mathfrak{L}}}$ as in Assumption~\ref{Assumption Nonlinearites}, $Z=(\Pi,\Gamma)$ a model for $\mathcal{T}$ and
 $G\in \mathcal{C}^\infty(E\ltimes F)$. For $f\in \mathcal{D}^\gamma(\bar{V})$ define\footnote{Observe that this sum is finite due to the projection in the definition of $\star_\gamma$ . More precisely only terms such that $n\leq [\gamma/\zeta]$ contribute.}
$$G_\gamma(f)(p):= \sum_n \frac{1}{n!}\trr\big(j_p^E(D^nG)(\bar f_p ) \star_\gamma \tilde{f}_p^{\star_\gamma n}\big)\in \mathcal{D}^\gamma(U^0) \ ,$$ 
the map
$f\mapsto G_\gamma(f)$ is locally Lipschitz continuous in the sense that for every $C>0$,
\begin{equation}\label{nonlinear difference bound1}
\interleave G_\gamma(f)-G_\gamma(g)\interleave_{\gamma ; K}\lesssim \interleave f-g\interleave_{\gamma ; K}
\end{equation}
uniformly over $f,g$ satisfying $\interleave f\interleave_{\gamma ; K} , \interleave g\interleave_{\gamma ; K}\leq C$.
\end{theorem}

\begin{remark}\label{rem:upgrade_trick}
In the setting of Theorem~\ref{theorem composition} it can be shown along the lines of \cite[Proposition 3.11]{HP15} that if $\tilde Z = (\tilde \Pi, \tilde{\Gamma})$ is a second model that agrees with $Z$ on jets and the manifold $M$ is compact, then
\begin{equation}\label{nonlinear difference bound2}
\interleave G_\gamma(f);G_\gamma(\tilde{g}) \interleave_{\gamma ; K} \lesssim \interleave f;\tilde{g}\interleave_{\gamma ; K}+ \|Z;\tilde{Z}\| \ ,
\end{equation}
uniformly over $f,\tilde{g}$ satisfying $\interleave f\interleave_{\gamma ; K}, \interleave \tilde{g} \interleave_{\gamma ; K} \leq C$
and models satisfying $\|Z\|_K$, $\|\tilde{Z}\|_K\leq C$. 
Indeed, let us sketch the only difference in the set-up when applying the same trick as in the proof of \cite[Proposition 3.11]{HP15}, which is the fact that we identify jets, instead of only constants and that we work with regularity structure ensembles.
Thus, we construct a new regularity structure ensemble $\hat{\mathcal{T}}=(\{E^\mathfrak{l}\}_{\mathfrak{l}\in \mathfrak{L}}, \{\hat{T}^\mathfrak{l}\}_{\mathfrak{l}\in \mathfrak{L}},\hat{L})$ where
$$\hat{T}^\mathfrak{l}= \hat{T}^\mathfrak{l} \oplus \hat{T}^\mathfrak{l}/{\sim}$$
where $(\tau_1, \tau_2)\sim (\bar{\tau}_1, \bar{\tau}_2)$ if $\tau_1=\bar{\tau}_1 + \sigma$ and $ \tau_2=\bar{\tau}_2 -\sigma $ for some jet $\sigma\in JE^\mathfrak{l}$.
and $\hat{L}\subset L\oplus L$ consists of all elements $(\Gamma_1, \Gamma_2)\in L\oplus L$ such that  $\Gamma_1|_{JE^\mathfrak{t}}= \Gamma_2|_{JE^\mathfrak{t}}$.
From here onwards the proof of \cite[Proposition 3.11]{HP15} adapts mutatis mutandis.
\end{remark}

\begin{remark}
Note that the non-linearity \eqref{equation to solve} fits into the setting of this section by taking $E$ to be a direct sum of vector bundles.
\end{remark}
\begin{remark}
Note that in the case $G: \mathbb{R}\to \mathbb{R}$, the definition above agrees up to canonical isomorphisms with the one in \cite{Hai14}. Furthermore, in view or Lemma~\ref{lemma: polynomial special case for nonlinearity}, in the special case of polynomial nonlinearites, this construction is consistent with the one in Remark~\ref{rem:multilinear map}.
\end{remark}
\begin{proof}
The theorem is easily checked for $\gamma\leq \zeta$, thus we assume $\gamma>\zeta$. To check that $G_\gamma (f)\in \mathcal{D}^\gamma$,
we focus on bounding $|G_\gamma (f)(q)-\Gamma_{q,p} G_\gamma (f)|_\kappa$ as the bound on $|G_\gamma (f)(q)|_\kappa$ follows straightforwardly.
Denote by $R$ an admissible realisation of $JE$, using Remark~\ref{translation replaced by admissible realisation} we find that
$$G_\gamma(f)(p)= \sum_{n=0}^{[\gamma/\zeta]} \trr \frac{1}{n!} j_p D^n G( R\bar f_p(\cdot)) \star_\gamma \tilde{f}_p^{\star_\gamma n}$$
Note that for $\kappa<\gamma$
$$Q_\kappa\left(\Gamma_{q,p}G_\gamma(f)(p)- \sum_{n=0}^{[\gamma/\zeta]}\trr \frac{1}{n!} \Gamma_{q,p} j_p D^n G( R\bar{f}_p (\cdot)) \star_\gamma \Gamma_{q,p}\tilde{f}_p^{\star_\gamma n}\right)=0$$ and that
\begin{align*}
&\left| \sum_{n=0}^{[\gamma/\zeta]} \trr \frac{1}{n!} j_q D^n G( R\bar{f}_p(\cdot)) \star_\gamma \Gamma_{q,p}\tilde{f}_p^{\star_\gamma n}   - \sum_{n=0}^{[\gamma/\zeta]} \trr \frac{1}{n!} \Gamma_{q,p} j_p D^n G( R\bar{f}_p (\cdot)) \star_\gamma \Gamma_{q,p}\tilde{f}_p^{\star_\gamma n} \right|_\kappa \\
&\lesssim \sum_{n=0}^{[\gamma/\zeta]}  \sum_{\alpha+\beta= \kappa} \left| Q_{<\gamma-\beta} \big(j_q D^n G( R\bar{f}_p(\cdot))  - \Gamma_{q,p} j_p D^n G( R\bar{f}_p (\cdot)\big) \right|_\alpha |\Gamma_{q,p}\tilde{f}_p^{\star_\gamma n}  |_\beta \\
&\leq \sum_{n=0}^{[\gamma/\zeta]}  \sum_{\alpha+\beta= \kappa} d_\fraks (p,q)^{\gamma-\beta-\alpha} |\Gamma_{q,p}\tilde{f}_p^{\star_\gamma n}  |_\beta \\
&\lesssim d_\fraks (p,q)^{\gamma-\kappa} \ ,
\end{align*}
where we used Lemma~\ref{lemma equivalent characterisations of holder} together with the fact that $D^n G( R\bar{f}_p(\cdot)) $ is a smooth function in the third inequality. Thus,
it suffices to estimate
$$G(f)(q)-\sum_{n=0}^{[\gamma/\zeta]} \trr \frac{1}{n!} j_q D^n G( R\bar{f}_p(\cdot)) \star_\gamma \Gamma_{q,p}\tilde{f}_p^{\star_\gamma n} \ .$$
By Remark~\ref{remark interchanging derivatives}
$$j_q D^n G( R\bar{f}_q(\cdot))= D^n j_q G( R\bar{f}_q(\cdot))\ ,$$
therefore 
\begin{align*}
 &\Big|G(f)(q) -\sum_{n=0}^{[\gamma/\zeta]}\trr \frac{1}{n!} j_q D^n G( R\bar{f}_p(\cdot)) \star_\gamma \Gamma_{q,p}\tilde{f}_p^{\star_\gamma n}\Big|_\kappa\\
= & \Big|\sum_{n=0}^{[\gamma/\zeta]} \trr \frac{1}{n!} j_q D^n G( R\bar{f}_q(\cdot)) \star_\gamma \tilde{f}_q^{\star_\gamma n } 
-\sum_{n=0}^{[\gamma/\zeta]} \trr \frac{1}{n!} j_q D^n G( R\bar{f}_p(\cdot)) \star_\gamma \Gamma_{q,p}\tilde{f}_p^{\star_\gamma n}\Big|_\kappa \\
\leq & \Big|\sum_{n=0}^{[\gamma/\zeta]} \trr \frac{1}{n!} j_q D^n G( R\bar{f}_q(\cdot)) \star \tilde{f}_q^{\star_\gamma n}\\
&\qquad -\sum_{n=0}^{[\gamma/\zeta]}\sum_{k=0}^{[\gamma/\zeta]-n} \trr \frac{1}{n!} \frac{1}{k!} D^{n+k} j_qG( R\bar{f}_p(\cdot)) \star_\gamma j_q(R\bar{f}_q(\cdot)-R\bar{f}_p(\cdot))^{\otimes k} \star \tilde{f}_q^{\star_\gamma n}\Big|_\kappa \\
+&  \Big| \sum_{n=0}^{[\gamma/\zeta]}\sum_{k=0}^{[\gamma/\zeta]-n}  \trr \frac{1}{n!} \frac{1}{k!} D^{n+k} j_qG( R\bar{f}_p(\cdot)) \star_\gamma j_q(R\bar{f}_q(\cdot)-R\bar{f}_p(\cdot))^{\otimes k} \star \tilde{f}_q^{\star_\gamma n}\\
&\qquad -\sum_{n=0}^{[\gamma/\zeta]} \trr \frac{1}{n!} j_q D^n G( R\bar{f}_p(\cdot)) \star_\gamma \Gamma_{q,p}\tilde{f}_p^{\star_\gamma n}\Big|_\kappa 
\end{align*}
First we bound the former term
\begin{align*}
&\big|\sum_{n=0}^{[\gamma/\zeta]} \trr \frac{1}{n!} j_q D^n G( R\bar{f}_q(\cdot)) \star \tilde{f}_q^{\star_\gamma n}\\
&\qquad
-\sum_{n=0}^{[\gamma/\zeta]}\sum_{k=0}^{[\gamma/\zeta]-n} \trr \frac{1}{n!} \frac{1}{k!} D^{n+k} j_qG(R\bar{f}_p(\cdot)) \star_\gamma j_q(R\bar{f}_q(\cdot)-R\bar{f}_p(\cdot))^{\otimes k} \star_\gamma \tilde{f}_q^{\star_\gamma n}\big|_\kappa .
\end{align*}
Since the abstract tensor products agree with the canonical product on jets, it can be rewritten and bounded as follows
\begin{align*}
&\Big|\trr \sum_{n=0}^{[\gamma/\zeta]} \frac{1}{n!} j_q \left( D^n G( R\bar{f}_q(\cdot)) - \sum_{k=0}^{[\gamma/\zeta]-n}   \frac{1}{k!} D^{n+k}G( R\bar{f}_p(\cdot)) (R\bar{f}_p(\cdot)-R\bar{f}_q(\cdot))^{\otimes k} \right) \star_\gamma \tilde{f}_q^{\star_\gamma n}\Big|_\kappa\\
&\lesssim\sum_{\alpha+\beta= \kappa} \sum_{n=0}^{[\gamma/\zeta]}\left|\frac{1}{n!} j_q \left( D^n G( R\bar{f}_q(\cdot)) - \sum_{k=0}^{[\gamma/\zeta]-n} \frac{1}{k!} D^{n+k}G( R\bar{f}_p(\cdot)) (R\bar{f}_p(\cdot)-R\bar{f}_q(\cdot))^{\otimes k} \right) \right|_\alpha | \tilde{f}_q^{\star_\gamma n}|_\beta\\
&\lesssim \sum_{n=0}^{[\gamma/\zeta]}  \sum_{\alpha+\beta= \kappa} | \bar{f}_q-R\bar{f}_p(q) |^{([\gamma/\zeta]-n+1- \alpha)\vee 0} | \tilde{f}_q^{\star_\gamma n}|_\beta  \\
&= \sum_{n=0}^{[\gamma/\zeta]}  \sum_{\alpha+\beta= \kappa} | \bar{f}_q-\Gamma_{q,p}\bar{f}_p |_0^{([\gamma/\zeta]-n+1- \alpha)\vee 0} | \tilde{f}_q^{\star_\gamma n}|_\beta  \ .
\end{align*}
where we used that
$$\Big| j_q\Big( D^n G \big( R\bar{f}_q(\cdot)\big)- \sum_{k\leq m} \frac{1}{k!} D^{n+k} G\big( R\bar{f}_p(\cdot)\big)(R\bar{f}_q(\cdot)-R\bar{f}_p(\cdot))^{\otimes k}\Big) \Big|_l \lesssim | \bar{f}_q-R\bar{f}_p(q) |^{(m+1-l)\vee 0}  \ ,$$
for $l\leq m+1$ in the second inequality, which follows from  Lemma~\ref{em version Taylor}.
Next we note for any non-vanishing summand above we have 
\begin{align*}
[\gamma/\zeta]-n +1 -\alpha &
= \frac{1}{\zeta}\left(\zeta([\gamma/\zeta] +1) -\zeta n - \zeta \kappa+ \zeta\beta)\right) \\
&\geq \frac{1}{\zeta}\left(\gamma -\beta - \zeta \kappa-\zeta\beta \right)\\
&\geq \frac{1}{\zeta}\left(\gamma - \kappa    +\kappa -\beta - \zeta \kappa+\zeta\beta \right)\\
&=\frac{1}{\zeta}\left(\gamma - \kappa    +(1-\zeta)(\kappa-\beta) \right)\\
&\geq\frac{\gamma - \kappa  }{\zeta} 
\end{align*}
where in the first line we used $\alpha+\beta= \kappa$, in the second line we used $[\gamma/\zeta] +1\geq \gamma/\zeta$ and $\beta \geq \zeta n$ and in the last line we used that $\zeta\leq 1$ and $\kappa\geq\beta$.
Thus, we conclude
\begin{align*}
&  \sum_{n=0}^{[\gamma/\zeta]}  \sum_{\alpha+\beta= \kappa} | \bar{f}_q-\Gamma_{q,p}\bar{f}_p |_0^{([\gamma/\zeta]-n+1- \alpha)\vee 0} | \tilde{f}_q^{\star_\gamma n}|_\beta  \\
& \lesssim \sum_{n=0}^{[\gamma/\zeta]}  \sum_{\alpha+\beta= \kappa} d_\fraks(p,q)^{\zeta( [\gamma/\zeta]-n+1- \alpha)} \\
 & \lesssim  d_\fraks(p,q)^{\gamma- \kappa } \ ,
\end{align*}
which concludes the estimate on the first term.
Next we bound the second term
\begin{align*}
&\big| \sum_{n=0}^{[\gamma/\zeta]}\sum_{k=0}^{[\gamma/\zeta]-n}  \trr \frac{1}{n!} \frac{1}{k!} D^{n+k} j_qG( R\bar{f}_p(\cdot)) \star_\gamma j_q(R\bar{f}_q(\cdot)-R\bar{f}_p(\cdot))^{\otimes k} \star_\gamma \tilde{f}_q^{\star_\gamma n} \\
&\qquad-\sum_{n=0}^{[\gamma/\zeta]} \trr \frac{1}{n!} j_q D^n G( R\bar{f}_p(\cdot)) \star_\gamma \Gamma_{q,p}\tilde{f}_p^{\star_\gamma n}\big|_\kappa \\
&=\big|  \sum_{n=0}^{[\gamma/\zeta]}\sum_{k=0}^{[\gamma/\zeta]-n} \trr \frac{1}{n!}\frac{1}{k!}\big(  D^{n+k} j_qG( R\bar{f}_p(\cdot)) \star_\gamma (\bar{f}_q-\Gamma_{q,p}\bar{f}_p)^{\star_\gamma k} \star_\gamma \tilde{f}_q^{\star_\gamma n}\big) \\
&\qquad-\sum_{n=0}^{[\gamma/\zeta]} \trr \frac{1}{n!} j_q D^n G( R\bar{f}_p(\cdot)) \star_\gamma \Gamma_{q,p}\tilde{f}_p^{\star_\gamma n}\big|_\kappa \\
&=:R_{\kappa,p,q}
\end{align*}
We set $D_{n,pq}= Q_{<\gamma}\Gamma_{q,p}\tilde{f}_p^{\star_\gamma n}- (f_q -\Gamma_{q,p}\bar{f}_p)^{\star_\gamma n}$ and find
\begin{align*}
&\sum_{n=0}^{[\gamma/\zeta]}  \trr \frac{1}{n!} j_q D^nG( R\bar{f}_p(\cdot)) \star_\gamma \Gamma_{q,p}\tilde{f}_p^{\star_\gamma n} \\
=& \sum_{n=0}^{[\gamma/\zeta]}  \trr \frac{1}{n!} j_q D^nG( R\bar{f}_p(\cdot)) \star_\gamma (\bar{f}_q -\Gamma_{q,p}\bar{f}_p +\tilde{f}_q)^{\star_\gamma n} \\
&+ \sum_{n=0}^{[\gamma/\zeta]}  \trr \frac{1}{n!} j_q D^n G( R\bar{f}_p(\cdot)) \star_\gamma D_{n,pq} \\
=&\sum_{n=0}^{[\gamma/\zeta]}  \sum_{k\leq n}\frac{1}{n!}  {n\choose k}\trr  j_q D^n G( R\bar{f}_p(\cdot)) \star_\gamma (\bar{f}_q -\Gamma_{q,p}\bar{f}_p)^{\otimes (n-k)} \star_\gamma \tilde{f}_q^{\star_\gamma k}\\
&+ \sum_{n=0}^{[\gamma/\zeta]}  \trr \frac{1}{n!} j_q D^nG( R\bar{f}_p(\cdot)) \star_\gamma D_{n,pq} \\
=&\sum_{n=0}^{[\gamma/\zeta]}  \sum_{k\leq n}\frac{1}{(n-k)!}  \frac{1}{k!} \trr \ j_q D^nG( R\bar{f}_p(\cdot))\star_\gamma (\bar{f}_q -\Gamma_{q,p}\bar{f}_p)^{\otimes (n-k)} \star_\gamma \tilde{f}_q^{\star_\gamma k} \\
&+ \sum_{n=0}^{[\gamma/\zeta]}  \trr \frac{1}{n!} j_q D^n G( R\bar{f}_p(\cdot)) \star_\gamma D_{n,pq} \ .
\end{align*}
Thus, we find 
\begin{align*}
R_{\kappa,p,q}&= \left| \sum_{n=0}^{[\gamma/\zeta]}  \trr \frac{1}{n!} j_q D^n G(R\bar{f}_p(\cdot)) \star_\gamma D_{n,pq}  \right|_\kappa \leq  \sum_{n=0}^{[\gamma/\zeta]}  \sum_{\alpha+\beta= \kappa} | j_q D^n G( R\bar{f}_p(\cdot))|_\alpha | D_{n,pq} |_\beta\\
&\lesssim  \sum_{n=0}^{[\gamma/\zeta]}  \sum_{\alpha+\beta= \kappa}  | D_{n,pq} |_\beta\lesssim  d_\fraks (p,q)^{\gamma-\kappa}
\end{align*}
since 
$$ |D_{n,pq}|_\beta  = |(\Gamma_{q,p}f_p-\Gamma_{q,p}\bar{f}_p)^{\star_\gamma n} - (f_q- \Gamma_{q,p}\bar{f}_p)^{\star_\gamma n}|_\beta \lesssim  d_\fraks (p,q)^{\gamma-\beta}\ , $$
where we used the general algebraic identity 
$a^n-b^n= (a-b)(a^{n-1}+ a^{n-2}b + \ldots+ ab^{n-2} + b^{n-1})$ for the last inequality,
This concludes the estimate on the second term.
Lastly, to see \eqref{nonlinear difference bound1} one proceeds as in \cite{Hai14} and sets $h= f-g$ such that 
\begin{align*}
G_\gamma(f)-G_\gamma (g) &= \int_0^1 \partial_t G_\gamma(g+th) dt\\
&=\int_0^1 
\sum_n \frac{1}{n!}\trr\big(j_p^E(D^{n+1}G)(\bar{g}_p+t\bar{h}_p ) \star_\gamma (\tilde {g}_p+t\tilde{h}_p)^{\star_\gamma n} \star_\gamma h \big) dt
\end{align*}
which can be bounded along the same argument as above.
\end{proof}

\begin{lemma}\label{em version Taylor}
Given $H\in \mathcal{C}^\infty(E\ltimes F)$, for $l,m\in \mathbb{N}$ such that $l\leq m+1$ the bound
$$\Big| j_q\Big( H \big( R e_q(\cdot)\big)- \sum_{k\leq m} \frac{1}{k!} D^k H\big( R e_p(\cdot)\big)(R e_q(\cdot)-R e_p(\cdot))^{\otimes k}\Big) \Big|_l \lesssim | e_q-Re_p(q) |^{m+1-l} \ $$
holds uniformly over compacts $K\subset E$ and $e_p\in E_p\cap K$ and $e_q\in E_q\cap K$. 
\end{lemma}
\begin{proof}
Let $F(\cdot):= H \big(R e_q(\cdot)\big)- \sum_{k\leq m} \frac{1}{k!} D^k H\big(R e_p(\cdot)\big)(R e_q(\cdot)-R e_p(\cdot))^{\otimes k}$.
It follows from the usual Taylor formula that $|F(q)|\lesssim |R e_q(q)-R e_p(q)|^{m+1} $. Let $\partial$ be a basis vector field in an exponential chart at $q$. 
We find by slight abuse of notation (suppressing the pushforward onto the chart)
\begin{align*}
(\partial F)(q)=& (\partial_1 H) (q,  e_q)- \sum_{k\leq m} \frac{1}{k!} D^k (\partial_1H)\big(q, R e_p(q)\big)(e_q-R e_p(q))^{\otimes k}\\
&+\sum_{k\leq m} \frac{1}{k!} D^{k+1} H\big(q, R e_p(q)\big)(e_q-R e_p(q))^{\otimes k}\otimes \nabla_\partial R e_p(q)\\
&-\sum_{0<k\leq m} \frac{1}{(k-1)!} D^{k} H\big(q, R e_p(q)\big)(e_q-R e_p(q))^{\otimes k-1}\otimes \nabla_\partial R e_p(q)\\
=& (\partial_1 H) (q,  e_q)- \sum_{k\leq m} \frac{1}{k!} D^k (\partial_1H)\big(q, R e_p(q)\big)(e_q-R e_p(q))^{\otimes k}\\
&+\frac{1}{m!} D^{m+1} H\big(q, R e_p(q)\big)(e_q-R e_p(q))^{\otimes m}\otimes \nabla_\partial R e_p(q)
\end{align*}
and thus $| (\partial F)(q)|\lesssim |R e_q(q)-R e_p(q)|^{m}$.
Differentiating further one obtains the general claim.
\end{proof}

\section{Integration Against Singular Kernels}\label{Section Singular Kernels}
In this section, which is modelled after \cite[Section 5]{Hai14}, we show how to integrate modelled distributions against singular kernels. The construction, though similar, comes with some twists, which in particular explain the necessity to give up the lower triangular structure present in \cite{Hai14}. Furthermore, in contrast to previous works, we drop the assumption that the kernels annihilate ``polynomials'' under the ``mild'' condition that $\beta\notin \mathbb{N}$. The reason for doing so is twofold. First, in the setting of general manifolds, it is not obvious how to canonically change an integral kernel, say the one associated to the inverse of the Laplace operator,  to satisfy this assumption. Second, even if one is able to do so, this only simplifies the theory in very minor ways. A drawback of this setting is, that it introduces some redundancy in the ``information'' contained in the model.\footnote{Let us give two related ways to interpret the somewhat cryptic comment about `` information''.  \begin{itemize}
\item Fix the minimal regularity structure to solve an SPDE, say the $\phi^4_3$ equation, in the setting of \cite{Hai14}, i.e.\ when $K$ annihilates polynomials. Let $Z$ be the BPHZ model, for the (rough!) noise. One typically expects the reconstruction operator $\mathcal{R}_Z$ to be injective.
This does not hold, if $K$ does not annihilate polynomials.
\item In the setting of \cite{Hai14}, one expects the analogue of Lemma~\ref{relating}, to hold for the model corresponding to the $\phi^4_3$ equation with rough noise. One does not expect this here.
\end{itemize}  }
Lastly, some modifications arise due to the fact that we work over general vector bundles and do not limit ourselves to trivial bundles.
\subsection{Singular kernels}
Throughout this section fix two vector bundles  $(E,\nabla^E, \langle\cdot, \cdot\rangle_E)$ and  $(F,\nabla^F, \langle\cdot, \cdot\rangle_F)$ over $M$. Assume that $K\in C^\infty(F\hotimes E^* \setminus \pi^{-1}\triangle)$, where $\triangle\subset M\times M$ denotes the diagonal, is a locally integrable kernel. Throughout this article we shall use the following notation: If $T\in \mathcal{D}'(E)$, we write 
$K(T)\in \mathcal{D}'(F)$ for the distribution given by 
\begin{equation}\label{eq:kernel_acting_on_distribution}
\phi\mapsto T\left(\int K^{\star}(p,\cdot)\phi(p) \, \dVol_p\right) \ , 
\end{equation}
where $K^{\star}(p,q)\in (E\hotimes F^*)|_{q,p}$ is the adjoint of $K(p,q)\in (F\hotimes E^*)|_{p,q} $ and $K^{\star}(p,q)\phi(p)\in E|_q$ denotes the natural pairing between $(E\hotimes F^*)|_{q,p}$ and $F|_p$.
We assume that the kernel satisfies the following assumption.

\begin{assumption}\label{Assumption on Kernel}
One can write $K= \sum_{n\in \mathbb{N}} K_n$ where
\begin{enumerate}
\item $K_n\in \cC^\infty (F\hotimes E^*)$ is supported on $\{(p,q)\  | \ d_\fraks(p,q)<2^{-n}\}$.
\item\label{item:support_in_exp_chart} For all $p\in M$ and $n\in \mathbb{N}$ such that $\min_{d_\fraks(p,q)\leq 1} r_q  \leq 2^{-n} $, one has
$$K_n(p,\cdot)=0 \text{ and } K_n(\cdot,p)=0.$$
(Recall that $r_p>0$ denotes the injectivity radius of the exponential map $\exp_p$.)
\item\label{item:upper_bound}
 $|j_{p,q} K_n|_m\lesssim_m 2^{n(|\mathfrak{s}|-\beta + m)}$ uniformly over $(p,q)\in M\times M$ (equipped with the canonical scaling) and $n\in \mathbb{N}$.
\item\label{item:integragration_against_polynomials} Lastly, assume that for any fixed admissible realisation $R$ of $JF$ and for any two multi-indices $m,l\in \mathbb{N}$, we have
$$ \left|\int_M  Q_l j_q ( K^\star_n (p,\cdot) R(v_q)(p) ) \, \dVol_p\right|_l   \lesssim C |v_q|_m 2^{-\beta n}$$
uniformly over $p,q\in M$ such that $d(p,q)<1$, $v_q\in (JF^*)_m$ and $n\in \mathbb{N}$, where $(K^\star_n (p,\cdot) R(v_q)(p) )$ used the canonical pairing $(E\otimes F^*) \otimes F\to E^*$. 
\end{enumerate}
\end{assumption}

Furthermore we make the following standing assumption throughout this section.
\begin{assumption}\label{Assumption on reg. and model}
Assume that we are working with a regularity structure ensemble $\mathcal{T}=(\{E^\mathfrak{t}\}_{\mathfrak{t}\in \mathfrak{L}},\{{T}^\mathfrak{t}\}_{\mathfrak{t}\in \mathfrak{L}}, L)$ such that 
\begin{itemize}
\item $\mathfrak{L}$ contains two distinguished elements $\mathfrak{e}, \mathfrak{f}$ such that $E^\mathfrak{e}= E$ and $E^\mathfrak{f}= F$.
\item for some $\delta_0\in (0,+\infty)\setminus \mathbb{N}$, $\bar T^\mathfrak{f}\subset T^\mathfrak{f}$ is an $F$-valued jet regularity structure of precision $\delta_0$.
\item All models $Z$ for $\mathcal{T}$ act on $\bar T^\mathfrak{f}$ as a jet model.
\end{itemize}
\end{assumption}


\begin{definition}\label{def of abstract integration map}
Given a sector $V\subset T^\mathfrak{e}$, a bundle morphism ${I}: V \to T^\mathfrak{f}$  is called abstract integration map of order $\beta$ and precision $\delta_0$, if the following properties are satisfied:
\begin{enumerate}
\item For each $\alpha\in A,$ $\delta \in \triangle$ such that $V_{\alpha,\delta}\neq \{0\}$ one has $\alpha+\beta\notin \mathbb{N}$ and $\delta+\beta\notin \mathbb{N}$.
\item For $(\alpha,\delta)\in \left((A\cap [-\beta,\infty))\times \triangle \right)\cup \big(A \times (\triangle\cap (0,\infty) ) \big)$, we have
$$I_p: (V_p)_{\alpha,\delta} \mapsto (T^{\mathfrak{f}}_p)_{{(\alpha+\beta),(\delta+\beta)\wedge\delta_0}}\cap (T^{\mathfrak{f}}_p)_{<\delta_0,:}\ , $$
while for $\alpha\in A\cap (-\infty,-\beta)$ 
$$I_p: (V_p)_{\alpha,+\infty} \mapsto (T^{\mathfrak{f}}_p)_{{\alpha+\beta, + \infty}}\ . $$
\item\label{Commutation cond} $\Gamma_{q,p} I_{p} \tau_p - I_q \Gamma_{q,p}\tau_p \in \bar{T}^{\mathfrak{f}}_q$ for each $\tau_p\in V_{<\delta_0-\beta,:}$ and $\Gamma_{q,p}\in L_{q,p}$.
\end{enumerate}
\end{definition}

\begin{remark}\label{rem:complicated grading}
The separate condition for $\alpha\in A\cap (-\infty,-\beta)$, $\delta= +\infty$ encodes additional constraints on models. These can be enforced, since the two operators $J$ and $E$ introduced below then satisfy for $\tau_p\in T^\frakt_{\alpha,\delta}$ that $J_p\tau_p= E_q\tau_p=0$ for any $p,q\in M$.
This is advantageous (but not necessary) in order to establish convergence of renormalised models in Section~\ref{sec:concrete applications}. 
Furthermore, the second property also implies that $I_p \tau_p= 0$ for $\tau_p\in T^{\mathfrak{e}}_{>\delta-\beta, :}$. 
\end{remark}
\begin{remark}
On first sight, the conditions $\alpha+\beta\notin \mathbb{N}$ and $\delta+\beta\notin \mathbb{N}$ might seem restrictive, but
note that if a kernel $G$ satisfies Assumption~\ref{Assumption on Kernel} for some $\beta>0$, then it automatically satisfies the same assumption for any $\bar{\beta}\in (0,\beta)$.
\end{remark}

Given a model $Z=(\Pi,\Gamma)$, we introduce as in the flat setting for $p\in M$ the linear map $J_p: T^{\mathfrak{e}}_p \to \bar{T}^{\mathfrak{f}}_p$ which for $\tau_p\in T^{\mathfrak{e}}_p$  is given by
$$ J_p \tau_p = \sum_{n} Q_{< (\alpha+\beta)\wedge \delta_0} j_p K_n\big(\Pi_p Q_{<\delta_0-\beta} \tau_p \big)\ . $$
\begin{lemma}\label{J lemma}
The map $J_p\tau$ is well defined. Furthermore, whenever $\Pi_p\tau\in \cC^\infty(E)$ the identity $J_p\tau_p=Q_{<\alpha+\beta} j_p \int K(\cdot,z)\Pi_p\tau (z) \, \dVol_z$ holds.
\end{lemma}
\begin{proof}
This follows by going to an exponential $\fraks$-chart at $p$ and arguing similarly to the flat case. Then by Remark~\ref{norm remark}
$$\big|j_p K_n\big(\Pi_p \tau_p \big)\big|_k \leq \sup_{\deg(\partial_{i_1},\ldots,\partial_{i_m})=k} | (\Pi_p \tau_p \big(\nabla^m_{\partial_{i_1},\ldots,\partial_{i_m}, 1} K^\star_n(p,\cdot)\big) |\lesssim \|\Pi\|_{B_1(p)}2^{-n(\beta +\alpha-k) }\ ,$$
which is summable for $k< \alpha+\beta $ . 
\end{proof}

As in \cite{Hai14}, we define what it means for a model to realise an abstract integration map.
\begin{definition}
Let $I:V\to T^{\mathfrak{f}}$ be an abstract integration map of precision $\delta_0$. We say a model $Z$ realises $K$ for $I$, if for all $\alpha<\delta_0-\beta$ and $\tau_p\in V_\alpha$ 
\begin{equation}\label{above}
\Pi_p I_p\tau_p = K(\Pi_p \tau_p) - \Pi_pJ_p\tau_p \ .
\end{equation}
\end{definition}

Recall that in the flat case, for admissible models one has the identity 
\begin{equation}\label{flat identity}
(I +J_x)\Gamma_{x,y}\tau= \Gamma_{x,y}(I +J_y)\tau \ .
\end{equation}
Unfortunately, this fails in our setting since the re-expansion maps $\Gamma$ are not precise enough on large scales (i.e.\ we only impose that $\Pi_p\Gamma_{p,q}-\Pi_q$ is small for very localised test functions, but not when testing against an element $\phi_p^1\in \mathfrak{B}_p^{r,1}$). Even the natural analogue 
$$|(I +J_x)\Gamma_{x,y}\tau-\Gamma_{x,y}(I +J_y)\tau|_\alpha\lesssim d(x,y)^{\delta-\alpha}$$ 
fails in our setting.
The next proposition makes this precise.

\begin{prop}\label{prop errorterm}
In the setting above, assuming that the model $Z$ realises the abstract integration map $I$, for every compact set $K$ 
\begin{align*}
\Pi_q \big(\Gamma_{q,p}(I_p+J_p)\tau_p-(I_q+J_q)\Gamma_{q,p}\tau_p\big)(\phi_q^\lambda) =& K(\Pi_p\tau_p-\Pi_qQ_{<\delta_0-\beta} \Gamma_{q,p}\tau_p\big)(\phi_q^\lambda)\\
&+\mathcal{O}_Z(\lambda^{(\delta+\beta)\wedge \delta_0})|\tau|\ ,
\end{align*}
holds uniformly over $p,q\in K$, and $\tau_p\in V_{\alpha,\delta}$ for $\alpha<\delta_0-\beta$ and $\phi_q^\lambda\in \mathfrak{B}_p^{r,\lambda}$.
\end{prop}
\begin{proof}
In the identity
\begin{align*}
\Pi_q \big(\Gamma_{q,p}(I_p+J_p)\tau_p-(I_q+J_q)\Gamma_{q,p}\tau_p\big)(\phi_q^\lambda)  &= 
\big(\Pi_q \big(\Gamma_{q,p}(I_p+J_p)\tau_p-\Pi_p (I_p+J_p)\tau_p\big)(\phi_q^\lambda)\\
&+ \big(\Pi_p (I_p+J_p)\tau_p -\Pi_q (I_q+J_q)\Gamma_{q,p}\tau_p\big)(\phi_q^\lambda)
\end{align*}
the first term is bounded by $\lambda^{(\delta+\beta)\wedge \delta_0}$ by definition of a model, while the second term is equal to
$K(\Pi_p\tau_p-\Pi_qQ_{<\delta_0-\beta} \Gamma_{q,p}\tau_p\big)(\phi_q^\lambda)\ .$
\end{proof}
\begin{remark}
Note that in the case $\alpha<-\beta$ and $\delta=+\infty$ one actually has $$\Pi_q \big(\Gamma_{q,p}(I+J_p)\tau_p-(I+J_q)\Gamma_{q,p}\tau_p\big)(\phi_q^\lambda)=0\ .$$
\end{remark}

At this point, it is natural to introduce the operator\footnote{{Note that there is no ambiguity in notation, we use $E$ and $E|_p$ to denote the vector bundle and its fiber and $E_p$ for the operator.}}
$E_q: T^{\mathfrak{e}}\to \bar{T}^{\mathfrak{f}}_q$ which for $\tau_p\in T^{\mathfrak{e}}_{\alpha,\delta}$ and $\alpha<\delta_0-\beta$ is defined as
\begin{equation}\label{def E}
{E}_q \tau_p := \sum_n Q_{<(\delta+\beta)\wedge \delta_0 } j_q K_n(\Pi_p \tau_p -\Pi_q Q_{<\delta_0-\beta } \Gamma_{q,p} \tau_p) \ .
\end{equation}
and extended to $\tau_p\in T^{\mathfrak{e}}_{\geq,:}$ as ${E}_q \tau_p=0$.
To see that this is well defined, one notes that the sum in (\ref{def E}) converges absolutely by the same argument as in the proof of Lemma~\ref{J lemma}.
Recall, that we have fixed an admissible realisation $R$ of $JF$ in Assumption~\ref{Assumption on Kernel}. We introduce for $\alpha \in \RR$ the following quantity
$$ K^{\alpha}_{n,p,q}(z):= K^\star_n(q,z)- R(Q_{<\alpha+\beta }j_p K^\star_n(\cdot,z) )(q) \in (E\hotimes F^*)_{p,q} \simeq E_q\hotimes F^*_{p}.$$
where we also use $R: JF^* \to \cC^\infty (F^*) $ to denote the dual admissible realisation of $R: JF \to \cC^\infty (F) $ defined using the metric induced isomorphism $F^*\simeq F$ and Definition~\ref{Pushforward of admissible realisations}, as well as its canonical extension to an admissible realisation of 
$JF^*\otimes E_z$
\begin{lemma}\label{lemma compensating error term}
For any admissible model and $\tau_p\in V_{<\delta_0-\beta;\delta}$
$$|\Gamma_{q,p}(I_{p} +J_{p})\tau_p-(I_q +J_q)\Gamma_{q,p}\tau_p- E_q \tau_p|_k\lesssim |\tau| d_\fraks(p,q)^{(\delta+\beta)\wedge \delta_0 -k}\   $$
locally uniformly in $p,q$. 
\end{lemma}
\begin{proof}
By Condition~\ref{Commutation cond} in the definition of an abstract integration map, Definition~\ref{def of abstract integration map},
$$\Gamma_{q,p}(I_{p} +J_{p})\tau_p-(I_q +J_q)\Gamma_{q,p}\tau_p- E_q \tau_p\in \bar{T}^{\mathfrak{f}}_{<\delta_0 ,:} \ .$$
Thus, by Lemma~\ref{bounding abstract pol}, it suffices to obtain the bound 
$$|\Pi_q \big(\Gamma_{q,p}(I_{p} +J_{p})\tau_p-(I_q +J_q)\Gamma_{q,p}\tau_p- E_q \tau_p)(\phi_q^\lambda)|\lesssim \lambda^{(\delta+\beta)\wedge \delta_0} $$
for $d_\fraks(p,q)\leq \lambda$. Indeed
\begin{align*}
\Pi_q& \big(\Gamma_{q,p}(I_{p} +J_{p})\tau_p-(I_q +J_q)\Gamma_{q,p}\tau_p- E_q \tau_p)(\phi_q^\lambda) 
\\
&= K(\Pi_p\tau_p-\Pi_q Q_{<\delta_0-\beta}\Gamma_{q,p}\tau_p)(\phi_q^\lambda) - \Pi_q E_q \tau_p(\phi_q^\lambda)+\mathcal{O}_Z(\lambda^{(\delta+\beta)\wedge\delta_0}) \\
&= \sum_n K_n(\Pi_p\tau_p-\Pi_q Q_{<\delta_0-\beta}\Gamma_{q,p}\tau_p)(\phi_q^\lambda) \\
&\quad - \Pi_q Q_{<(\delta+\beta)\wedge \delta_0}\big(j_q K_n(\Pi_p\tau_p -\Pi_q Q_{<\delta_0-\beta}\Gamma_{q,p} \tau_p)\big)(\phi_q^\lambda) +\mathcal{O}_Z(\lambda^{(\delta+\beta)\wedge\delta_0}) \\
&=\int_M \sum_n  \big(\Pi_p\tau_p-\Pi_q Q_{<\delta_0-\beta}\Gamma_{q,p}\tau_p\big)\big(K^{\delta\wedge (\delta_0-\beta)}_{n,q,r}\big)  \phi_q^\lambda(r) \, \dVol_g(r)  +\mathcal{O}(\lambda^\delta)\\
&\lesssim \|Z\|_{K} |\tau| (\lambda^{(\delta+\beta)\wedge\delta_0}+\mathcal{O}_Z(\lambda^{(\delta+\beta)\wedge\delta_0})\ ,
\end{align*}
where to go from the first to the second line we used Proposition~\ref{prop errorterm} and  the last inequality follows from Lemma~\ref{lemma bound on error}.
\begin{remark}
Note that in Lemma~\ref{lemma compensating error term} the condition $\alpha<\delta_0-\beta$ for $\tau_p\in V_{\alpha;\delta}$ is actually necessary.
\end{remark}

\end{proof}
\subsection{Analytic bounds involving $K^{\alpha}_{n,p,q}$}

In this section we collect bounds involving the function
$$ K^{\alpha}_{n,p,q}(z):=  K^\star_n(q,z)- R(Q_{\alpha+\beta}j_p  K^\star_n(\cdot,z) )(q)\  ,$$
and for given $\fraks$-vector fields $W=(W_1,\ldots,W_k)$,
$$K^{W, \alpha}_{n,p,q}(z)= (\nabla_{W_1}...\nabla_{W_k})_1K^\star_n(q,z)- \nabla_{W_1}...\nabla_{W_k}R(Q_{\alpha+\beta}j_p  K^\star_n(\cdot,z) )(q)\ \ ,$$
where the subscript $1$ in $(\nabla_{W_1}...\nabla_{W_k})_1K_n(q,z)$ signifies that the differential operators acts on the first variable and it is understood that
$(\nabla_{W_1}...\nabla_{W_k})f= \nabla_{W_1}\Big(...\big(\nabla_{W_{k-1}}(\nabla_{W_k}f)\big)...\Big)$
 for $f\in \cC^\infty(E)$.

\begin{remark}
Note that, if $W$ consists of coordinate vector fields, these functions are natural analogues of the functions $K^{k, \alpha}_{n,xy}$ in \cite{Hai14}.
\end{remark}
\begin{remark}
Note that $K^{W, \alpha}_{n,p,q}$ is supported on $\left(B_{2^{-n}}(q)\cup B_{2^{-n}}(p)\right)\cap B_{r_q}(q)\cap B_{r_p}(p)$ by Point~\ref{item:support_in_exp_chart} of Assumption~\ref{Assumption kernel decomposition}.
\end{remark}

\begin{lemma}[Hai14, Lemma 5.18]\label{technical Lemma for Kernels}
Let $W=(W_1,\ldots,W_k)$ be a tuple of $\fraks$-vector fields. For $\tau_q\in T_{\alpha,\delta}$, for $\xi, \rho >0$ such that $\xi+\beta<\rho$ and $p,q,n$ such that $d_{\fraks}(p,q)<2^{-n}$ the bound
\begin{align*}
|\Pi_q\tau_q (K^{W,\xi}_{n,p,q})|&\lesssim_{W,C}  \|Z\|_{B_1(p)}|\tau_q|  \\
\end{align*}
holds uniformly over $n$ and models $Z$, satisfying $\|Z\|_{B_1(p)}\leq C$. The finite set $\partial A_{\xi}\subset\mathbb{N}^d$ has the property that $|l|_\fraks \geq \xi + \beta$ for every $l\in \partial A_\xi$.
For a second model, $\tilde{Z}=(\tilde{\Pi},\tilde{\Gamma})$ satisfying $\|\tilde Z\|_{B_1(p)}\leq C$, we have the following bound
\begin{align*}
|\big(\Pi_q-&\tilde{\Pi}_q\tau_q\big) (K^{W,\xi}_{n,p,q})| \lesssim_{W,C} \|Z;\tilde{Z}\|_{B_1(p)}|\tau_q|\\
&\times \Big( \sum_{l\in \delta A_{\xi}} 2^{(|l|_\fraks+\deg_{\fraks}(W)-\alpha-\beta)n } d(p,q)^{|l|_\fraks} +   \sum_{m<\xi+\beta} d(p,q)^{\rho-\deg_{\fraks}(W)} 2^{-n(\alpha+\beta-m)}\Big) \ .
\end{align*}
Furthermore, the same bounds hold if $K^{W,\xi}_{n,p,q}$ is replaced by $K^{W,\xi}_{n,q,p}$. 
\end{lemma}

\begin{proof}
By linearity we can assume $|\tau|=1$. First observe that $\supp (K^{W,\xi}_{n,p,q})\subset B_{r_p}(p) \ , $
by Point~\ref{item:support_in_exp_chart} of Assumption~\ref{Assumption on Kernel}. Thus, let $(B_{r_p},\phi)$ be an $\fraks$-exponential chart around $p$, and denote by $\tilde{Z}= (\tilde{\Pi},\tilde{\Gamma})$ the pushforward regularity structure and denote by $\tilde{K}$ the pulled back kernel given by $\tilde{K} (x,y)= K(\phi^{-1}(x),\phi^{-1}(y))$. Writing $x_p = \phi(p)$ and $y_q = \phi(q)$, we have (using Lemma~\ref{lemma gradedness exponential chart}), that
\begin{align*}
K^\xi_{n,p,q}(\phi^{-1}(z))&= K_n(q,\phi^{-1}(z))- R(Q_{<\xi+\beta}j_p K_n(\cdot,\phi^{-1}(z)) )(q) \\
&= \tilde{K}_n(y_q,z)- \bar R (Q_{<\xi+\beta}j_{x_p} \tilde{K_n}(\cdot,z) )(y_q)\ ,
\end{align*}
where $\bar R$ is the admissible realisation obtained by pushforward under the exponential map from $R$, see Remark~\ref{Pushforward of admissible realisations}. Similarly, denoting by $\tilde{W}=(\tilde W_1,\ldots,\tilde{W}_k)$ the pushforward of $W=(W_1,\ldots,W_k)$ and by $\tilde{\nabla}$ the pullback of $\nabla$, as well as setting $\tilde{K}^\xi_{n,x_p,y_q}(z):= K^\xi_{n,p,q}(\phi^{-1}(z))$ one finds that
$$\widetilde {K^{W,\xi}_{n,x_p,y_q}}:= K^{W,\xi}_{n,p,q}(\phi^{-1}(z))= \tilde{\nabla}_{\tilde W_1}\ldots \tilde{\nabla}_{\tilde W_k}\tilde{K}^\xi_{n,x_p,y_q}(z)$$

Therefore, denoting by $\bar{Z}=(\bar \Pi, \bar \Gamma)$ the push-forward of $Z$, 
one has
\begin{align}\label{rtseq}
\Pi_q\tau_q (K^{W,\xi}_{n,p,q})&=\bar \Pi_{y_q}\tau_{q}(\widetilde {K^{W,\xi}_{n,x_p,y_q}}) 
\end{align}
and we shall bound the right-hand side of (\ref{rtseq}). 

We start with the case when there are no vector fields $W$ present, i.e $k=0$. By the definition of an admissible realisation we have
\begin{align}
\tilde{K}^\xi_{n,x_p,y_q}(z):=&
\tilde{K}_n(y_q,z)- \bar R (Q_{<\xi+\beta}j_{x_p} \tilde{K}_n(\cdot,z) )(y_q)\\
=& \tilde{K}_n(y_q,z)- \sum_{|l|_\fraks<\xi+\beta} \frac{(y_q-x_p)^l}{l!} D_1^{l} \tilde{K}_n (x_p,z) \nonumber \\
&+ \sum_{m < \xi+\beta}  \Big(\sum_{|l|_{\fraks}=m}\frac{(y_q-x_p)^l }{l!} D_1^{l} \tilde{K}_n (x_p,z) - \bar R (Q_{m}j_{x_p} \tilde{K}_n(\cdot,z) )(y_q)\Big)\nonumber \\
=& \tilde{K}_n(y_q,z)- \sum_{|l|<\xi+\beta} \frac{(y_q-x_p)^l}{l!} D_1^{l} \tilde{K}_n (x_p,z) \label{fist summand} \\
&+ \sum_{m<\xi+\beta} \mathcal{O}(|x_p-y_q|^\rho )2^{n(m- \beta)}\phi^{m,2^{-n}}_{x_p}(z)\ , \label{second summand}
\end{align}
where 
$\phi^{m,2^{-n}}_{x_p}\in B^{r,2^{-n}}_{x_p}$ due to Assumption~\ref{Assumption on Kernel} and $\rho>0$ is arbitrary.
We treat the two summands above separately. 

For the term (\ref{fist summand}), we shall argue similarly to \cite[Lemma 5.18]{Hai14}, but additionally using Lemma~\ref{dislocated testfunction lemma}. From the formula \cite[Proposition A.1]{Hai14}, we have
$$ \tilde K_n(y,z)- \sum_{|l|_{\fraks}<\xi+\beta} \frac{(y-x)^l}{l!} D_1^{l}  \tilde K_n (x,z) = \sum_{l\in \delta A_{\xi} } \int D_1^{l}  \tilde K_n(y+h,z)Q^l (x-y, dh)\ $$
and thus we can write 
\begin{align*}
\bar \Pi_y &\tau_y \Big(\tilde K_n(y,\cdot)- \sum_{|l|_{\fraks}<\xi+\beta} \frac{(y-x)^l}{l!} D_1^{l} \tilde K_n (x,\cdot)\Big) = \sum_{l\in \delta A_{\xi}} \int \bar \Pi_y \tau_y \big(D_1^{l}\tilde K_n  (y+h,\cdot)\big)Q^l(x-y, dh) \ .
\end{align*}
Since the integral in $h$ only varies over points satisfying $|h|_{\fraks}<|x-y|_{\fraks}=d_{\fraks}(p,q)$, and since by assumption $d(p,q)< 2^{-n}$, we obtain by Lemma~\ref{dislocated testfunction lemma}
\begin{align*}
&\big|\bar \Pi_y \tau_y \big(\tilde K_n (y,\cdot)- \sum_{|l |_{\fraks}<\xi+\beta} \frac{(y-x)^l}{l!} D_1^{l} \tilde K_n (x,\cdot)\big)\big| \\
&\lesssim \big(\|\bar \Pi\|_{B_{2^{-n}}(q)}\|\bar \Gamma\|_{B_{2^{-n}}(q)}+\llbracket \bar{Z}  \rrbracket_{B_{2^{-n}}(q)} \big)|\tau_q|
\sum_{l\in \delta A_{\xi}}
2^{(|l|_\fraks-\alpha-\beta)n }
\int Q^l(x-y, dh)\\
&\lesssim \big(\|\bar \Pi\|_{B_{2^{-n+1}}(p)}\|\bar \Gamma\|_{B_{2^{-n+1}}(p)}+ \llbracket \bar Z  \rrbracket_{B_{2^{-n+1}}(p)} \big) |\tau_q|\sum_{l\in \delta A_{\xi}} 2^{(|l|_\fraks-\alpha-\beta)n } |x-y|_{\fraks}^{|l|_{\fraks}} \\
&\lesssim_C \|Z\|_{B_1(p)} \sum_{l\in \delta A_{\xi}} 2^{(|l|_\fraks-\alpha-\beta)n } d_\fraks (p,q)^{|l|_\fraks} 
\end{align*}
where in the first inequality we have used the fact that there exists $c>0$ independent of $n$ and $y$ such that
$c\frac{2^{n\beta}}{2^{|l|_{\fraks}n}}D^l \tilde K_n (y,z) \in B_y^{r,2^{-n}}$ and in the last inequality we used Lemma~\ref{lemma push forward}.
To treat the term (\ref{second summand}), simply note that
\begin{align*}
2^{n( m- \beta)}|\bar \Pi_y\tau_q (\phi^{m,2^{-n}}_{x_p})|&\lesssim_C
2^{n( m- \beta)} 2^{-n\alpha}\|Z\|_{B_1(p)} \ ,
\end{align*}
where we used that 
\begin{align*}
 |\bar \Pi_y\tau_q (\phi^{m,2^{-n}}_{x_p})| &\leq |\bar \Pi_{x_p}\bar{\Gamma}_{x_p,y_q}\tau_q (\phi^{m,2^{-n}}_{x_p})|
 +|\bar \Pi_y\tau_q-\bar \Pi_{x_p}\bar \Gamma_{x_p,y_q}\tau_q (\phi^{m,2^{-n}}_{x_p})|\\
 &\leq \sum_{\zeta\in A} |x_p-y_q|^{(\alpha-\zeta)\vee 0} 2^{-\zeta n} \|\bar \Pi\|_{B_{2^{-n+1}(y)}}\|\bar \Gamma\|_{B_{2^{-n+1}(y)}}\\
 &\quad +2^{-\delta n} \llbracket \bar{Z}  \rrbracket_{B_{2^{-n+1}(y)}}\\
 &\lesssim_C  \|Z\|_{B_1(p)} \big( 2^{-\delta n} +\sum_{\zeta\in A} d(p,q)^{(\alpha-\zeta)\vee 0} 2^{-\zeta n}\big)\\
& \lesssim  \|Z\|_{B_1(p)}  2^{-\alpha n} ,
\end{align*}
where we applied Lemma~\ref{lemma push forward} in the second last inequality and used that $d_\fraks(p,q)\leq 2^{-n}$ in the last one. This completes the case in the absence of vector fields $W$.

We turn to the case $W=(W_1,\ldots,W_k)$ with $k\geq 1$ . In this case one finds a similar decomposition as given by (\ref{fist summand}) and (\ref{second summand}), namely 
\begin{align*}
\widetilde {K^{W,\xi}_{n,x_p,y_q}}(z) =& \tilde{\nabla}_{\tilde{W}_1}...\tilde{\nabla}_{\tilde{W}_k}  \left(\tilde{K_n}(\cdot,z)- \sum_{|l|<\xi+\beta} \frac{(\ \cdot\ -x_p)^l}{l!} D_1^{l} \tilde{K_n} (x_p,z)\right)(y_q)  \\
&+ \sum_{m<\xi+\beta} \mathcal{O}(|x_p-y_q|^{\rho-\deg_\fraks W} )2^{n(m- \beta)}\phi^{m,2^{-n}}_{x_p}(z)\ ,
\end{align*}
where again $\phi^{m,2^{-n}}_{x_p}\in B^{r,2^{-n}}_{x_p}$. Proceeding as above one obtains the claim.
\end{proof}

\begin{lemma}[Hai14, Lemma 5.19]\label{lemma well definednes of I}
For $\tau_q\in T^{\mathfrak{e}}_{\alpha,:}$ satisfying $\alpha+\beta\notin \mathbb{N}$ one has the bound 
\begin{equation}
\sum_{n\in\mathbb{N}} |\int\Pi_q\tau_q (K^{\alpha}_{n,p,q})\phi^\lambda_{q}(p) d \Vol_p |\lesssim_C \|Z\|_{B_1(p)} |\tau| \lambda^{\alpha+\beta} 
\end{equation}
uniformly over $q$ and $\phi_q^{\lambda}\in B^{r,\lambda}_q$, $\lambda\leq 1$ and models $Z$ satisfying $\|Z\|_{B_1(p)}<C$.
For a second model $\tilde{Z}=(\tilde{\Pi},\tilde{\Gamma})$
one has 
\begin{align*}
\sum_{n\in\mathbb{N}}& |\int(\Pi_q-\tilde{\Pi}_q)\tau_q (K^{\alpha}_{n,p,q})\phi^\lambda_{q}(p) d \Vol_p | \lesssim_C  \|Z;\tilde{Z}\|_{B_1(p)} |\tau| \lambda^{\alpha+\beta} \ .
\end{align*}
\end{lemma}

\begin{proof}
By linearity assume $|\tau|=1$. For fixed $\lambda>0$ , we define $n_\lambda\in \mathbb{N}$ as the number satisfying $\lambda\in (2^{-n_\lambda-1},2^{-n_\lambda}]$ and 
argue separately in the two cases, $n<n_\lambda$ and $n \geq n_\lambda$, starting
 with the former.
 
Since on $n<n_\lambda$ we can use Lemma~\ref{technical Lemma for Kernels} and that $d_{\fraks}(p,q)<\lambda< 2^{-n}$ to bound
\begin{align*}
\sum_{n\leq n_\lambda} &|\int\Pi_q\tau_q (K^{\alpha}_{n,p,q})\phi^\lambda_{q}(p) d \Vol_p | \\
&\lesssim \| Z \|_{B_1(p)}
 \sum_{n\leq n_\lambda}\Big( \sum_{l\in \delta A_{\alpha}} 2^{(|l|_{\fraks}-\alpha-\beta)n }  \int |\phi^\lambda_{q}(p)| d(p,q)^{|l|_{\fraks}} \, \dVol_p  \\ 
 &\qquad \qquad \qquad \qquad +\sum_{m<\alpha+\beta}2^{-n(\alpha+\beta-m)} \int |\phi^\lambda_{q}(p)| d(p,q)^\rho d\Vol_p )\Big)\\
&\lesssim \| Z \|_{B_1(p)}
\Big( \sum_{l\in \delta A_{\alpha}}   \lambda^{|l|_{\fraks}}\sum_{n\leq n_\lambda}2^{(|l|_{\fraks}-\alpha-\beta)n }    + 
   \sum_{|l|_{\fraks}<\alpha+\beta} \lambda^\rho  \sum_{n\leq n_\lambda}2^{-n(\alpha+\beta-m)}  \Big)\\
&\lesssim \| Z \|_{B_1(p)}
\Big( \sum_{l\in \delta A_{\alpha}}   \lambda^{|l|_{\fraks}} 2^{(|l|_{\fraks}-\alpha-\beta)n_\lambda }    + 
 \lambda^\rho  \Big)\\
 			&\lesssim \| Z \|_{B_1(p)}
(  \lambda^{(\alpha+\beta)}    + 
 \lambda^\delta )\ ,
\end{align*}
where we used that $|l|_{\fraks}>\alpha+\beta$ for all $l\in A_{\alpha}$ .

For the case $n>n_\lambda$ one can argue exactly as in \cite[Lemma 5.19]{Hai14}, since this argument does not involve the maps $\Gamma$. One finds 
$$ |\int\Pi_q\tau_q (K^{\alpha}_{n,p,q})\phi^\lambda_{q}(p) d \Vol_p |\lesssim \|\Pi \|_{B_1(p)} \sum_{\delta>0} \lambda^{\alpha+\beta -\kappa} 2^{-\kappa n}\ ,$$
where the sum runs over some finite set of strictly positive indices $\kappa$ and therefore the desired bound follows after summing over $n>n_\lambda$.
\end{proof}

\begin{remark}
Note that the second part of the proof of Lemma~\ref{lemma well definednes of I} is where the bounds in Point~\ref{item:integragration_against_polynomials} of Assumption~\ref{Assumption on Kernel} are used. Furthermore, the assumption that $\alpha+\beta\notin \mathbb{N}$ was used several times in order to rule out the appearance of logarithmic terms.
\end{remark}
Very similarly to Lemma~\ref{lemma well definednes of I} one obtains the following.
\begin{lemma}\label{lemma bound on error}
If $\delta+\beta \notin \mathbb{N}$, then for $\tau_q\in T^{\mathfrak{e}}_{<\delta_0-\beta, \delta}$ one has for every $C>0$ the bound 
\begin{equation*}
\sum_{n\in\mathbb{N}} |\int(\big(\Pi_p\tau_p-\Pi_q Q_{<\delta_0-\beta}\Gamma_{q,p}\tau_p\big)\big(K^{\delta\wedge (\delta_0-\beta)}_{n,q,r}\big)\phi^\lambda_{q}(p) d Vol_p |\lesssim_C \|Z\|_{B_1(p)}|\tau_p| \lambda^{(\delta+\beta)\wedge\delta_0} 
\end{equation*}
uniformly over $\phi_q^{\lambda}\in B_q^{r,\lambda}$ with $\lambda\leq 1$ and models satisfying $\|Z\|_{B_1(p)}<C$ .
A similar bound holds for difference of models.
\end{lemma}

Note that in our setting the map $J_{q,p}:T^{\mathfrak{e}}_{<\delta_0-\beta,:}\to \bar{T}^{\mathfrak{f}}$ defined by 
$$J_{q,p}\tau_p:= \big(J_q\Gamma_{q,p}\tau_p+ E_q \tau_p -\Gamma_{q,p}J_{p}\big)\tau_p=\big(J_qQ_{<\delta-\beta}\Gamma_{q,p}\tau_p+ E_q \tau_p -\Gamma_{q,p}J_{p}\big)\tau_p $$ plays the role of $\big(J(y)\Gamma_{y,x}- \Gamma_{y,x}J(x)\big) \tau$ in \cite{Hai14}. Next we derive a bound analogous to \cite[Lemma 5.21]{Hai14}. 

\begin{lemma}\label{lemma J_qp}
Let $\alpha<\delta_0-\beta$ such that $\alpha +\beta \notin \mathbb{N}$, then, in the setting above for every $\tau\in T^{\mathfrak{e}}_{\alpha,\delta}$ and every integer $m<\delta_0$ one has the bound 
$$|J_{p,q} \tau_q|_m \lesssim_C {\|Z\|_{B_1(p)}} |\tau| d(p,q)^{(\alpha+\beta - m)\vee 0},$$ 
uniformly over $p,q\in M$ such that with $d(p,q)<1$ and models $Z$ satisfying $\|Z\|_{B_1(p)}<C$. Furthermore, for two models the analogue bound difference bound holds.
\end{lemma}

\begin{proof}

Throughout this proof we fix an exponential $\fraks$-chart at $p$ and denote by $\partial_i$ the coordinate vector fields associated to it. 
We shall denote by $\boldsymbol{\partial}$ a tuple $(\partial_{i_1},\ldots,\partial_{i_k})$ and write $\nabla_{\boldsymbol{\partial}}$ instead of $\nabla_{\partial_{i_1}}...\nabla_{\partial_{i_k}}$. Recall that by definition
\begin{align*}
J_{p,q}\tau_q&= J_p\Gamma_{p,q}\tau_q+ E_p \tau_q -\Gamma_{p,q}J_{q}\tau_q\;.
\end{align*}
First we treat the simpler case $\alpha+\beta\leq m <(\delta+\beta)\wedge \delta_0$ by bounding each of the terms $|J_p\Gamma_{p,q}\tau_q|_m$, $|\Gamma_{p,q}J_{q}\tau_q|_m$ and 
$|E_p \tau_q|$ separately.
The bounds on the former two follow directly from Lemma~\ref{J lemma} the estimate on $E_q$ follows from Remark~\ref{norm remark} combined with the estimate
\begin{align*}
|E_p\tau_q|_m &\leq \sum_n \sup_{\deg(\boldsymbol{\partial}) =m} |(\Pi_q\tau_q -\Pi_p \Gamma_{p,q} \tau_q+ \Pi_p Q_{\geq\delta_0-\beta}\Gamma_{p,q})\big( \nabla_{\boldsymbol{\partial},1} K_n(p,\cdot) \big) |\\
&\lesssim  |\tau_p|\sum_n \big(2^{-n(\delta+\beta-m)}+2^{-n(\delta_0-m)}\big)\\
&\lesssim  |\tau_p|\sum_n 2^{-n((\delta+\beta)\wedge\delta_0-m)}
\end{align*}
Now we turn to the case $m<\alpha+\beta$. 
For $p,q\in M$ such that $d_\fraks(p,q)<1$, we set $n_{p,q}\in \mathbb{N}$ as the number satisfying $d_\fraks(p,q)\in (2^{-n_{p,q}-1},2^{-n_{p,q}}]$ and 
argue separately in the two cases $n<n_{p,q}$ and $n \geq n_{p,q}$, starting with the former.
Observe that for $\tau_p\in T^{\mathfrak{e}}_p$ one has $$Q_m J^n_p \tau_p= Q_m\sum_{\zeta<\delta_0-\beta}  J^n_p Q_\zeta \tau_p = Q_m \sum_{m-\beta <\zeta <\delta_0-\beta} J^n_p Q_\zeta \tau_p \ ,$$
and therefore 
\begin{align*}
&Q_m J^n_{p,q}\tau_q \\
&= Q_m \big(J^n_p\Gamma_{p,q}\tau_q+ E^n_p \tau_q -\Gamma_{p,q}J^n_{q}\tau_q\big) \\
&=  Q_m \big(\sum_{{m-\beta <\zeta <\delta_0-\beta}}J^n_p Q_\zeta \Gamma_{p,q}\tau_q+ E^n_p \tau_q -\Gamma_{p,q}J^n_{q}\tau_q\big)\\
&=Q_{m}\big(\sum_{m-\beta <\zeta< \delta_0-\beta}  j_p K_n\big(\Pi_p Q_\zeta  \Gamma_{p,q}\tau_q) 
+j_p K_n(\Pi_q\tau_q -\Pi_pQ_{\delta_0-\beta} \Gamma_{p,q} \tau_q)
-j_p R J^n_{q}\tau_q\big) \ .
\end{align*}  
It follows from Remark~\ref{norm remark} that it suffices to bound 
\begin{align*}
|J^n_{p,q}\tau_q|_m 
& \lesssim \sup_{\deg(\boldsymbol{\partial}) =m} \big| A_{\boldsymbol{\partial}} -\nabla_{\boldsymbol{\partial}} R (J^n_{q}\tau_q)(p) \big|
\end{align*}
where
$$A_{\boldsymbol{\partial}}:= \sum_{m-\beta<\zeta< \delta_0-\beta} \Pi_p Q_\zeta  \Gamma_{p,q}\tau_q(\nabla_{\boldsymbol{\partial},1}K_n(p,\cdot))
+ (\Pi_q\tau_q -\Pi_pQ_{<\delta_0-\beta} \Gamma_{p,q} \tau_q)(\nabla_{\boldsymbol{\partial},1} K_n(p,\cdot))\ .$$
%
Noting that 
\begin{align*}
A_{\boldsymbol{\partial}}=& - \sum_{\zeta\leq m-\beta} \Pi_p Q_\zeta  \Gamma_{p,q}\tau_q(\nabla_{\boldsymbol{\partial},1} K_n(p,\cdot))
+\Pi_pQ_{<\delta_0-\beta} \Gamma_{p,q}\tau_q(\nabla_{\boldsymbol{\partial},1} K_n(p,\cdot))\\
& + (\Pi_q\tau_q -\Pi_p Q_{<\delta_0-\beta} \Gamma_{p,q} \tau_q)(\nabla_{\boldsymbol{\partial},1} K_n(p,\cdot))\\
=&
- \sum_{\zeta\leq m-\beta} \Pi_p Q_\zeta  \Gamma_{p,q}\tau_q(\nabla_{\boldsymbol{\partial},1} K_n(p,\cdot))
+ \Pi_q\tau_q (\nabla_{\boldsymbol{\partial},1} K_n(p,\cdot)) \\
=&
- \sum_{\zeta< m-\beta} \Pi_p Q_\zeta  \Gamma_{p,q}\tau_q(\nabla_{\boldsymbol{\partial},1} K_n(p,\cdot))
+ \Pi_q\tau_q (\nabla_{\boldsymbol{\partial},1} K_n(p,\cdot)) \ ,
\end{align*}
where we used that $\zeta+\beta\notin \mathbb{N}$ in the last line, we find 
\begin{align*}\label{eq}
|J^n_{p,q}\tau_q|_m &\lesssim  \sup_{\deg(\boldsymbol{\partial}) =m} \big|A_{\boldsymbol{\partial}}- \nabla_{\boldsymbol{\partial},1} R (J^n_{q}\tau_q)(p) \big|\nonumber\\
&\lesssim  \sup_{\deg{\boldsymbol{\partial}} = m}\sum_{\zeta< m-\beta} |\Pi_p Q_\zeta  \Gamma_{p,q}\tau_q(\nabla_{\boldsymbol{\partial},1} K_n(p,\cdot))| +
|\Pi_q\tau_q( K^{\alpha, {\boldsymbol{\partial}}}_{n,q,p})| \ .
\end{align*}
%
%
%
%
%
To bound $\sum_{n=0}^{n_{d(p,q)}} |J^n_{p,q}\tau_q|_m$, we thus estimate using Lemma~\ref{technical Lemma for Kernels} with $\rho> \alpha+\beta$
\begin{align*}
&\sum_{n=0}^{n_{d(p,q)}} \sup_{\deg(\boldsymbol{\partial}) =m}|\Pi_q\tau_q(K^{\alpha, {\boldsymbol{\partial}}}_{n,q,p})|\\
&\lesssim \sum_{n=0}^{n_{d(p,q)}}  \Big( \sum_{l\in \delta A_{\alpha}} 2^{(|l|_\fraks+m-\alpha-\beta)n } d_\fraks(p,q)^{|l|_\fraks} 
+   \sum_{m<\alpha+\beta} d_\fraks(p,q)^{\rho-m} 2^{-n(\alpha+\beta-m)}\Big)\\
&\lesssim d_\fraks(p,q)^{\alpha+\beta - m}+ d_\fraks(p,q)^{\rho-m} \\
&\lesssim d_\fraks(p,q)^{\alpha+\beta - m} \ .
\end{align*}
Next, since $|\Pi_p Q_\zeta  \Gamma_{p,q}\tau_q(\nabla_{\boldsymbol{\partial},1} K_n(p, \cdot)|\lesssim d_\fraks(p,q)^{(\alpha-\zeta)\vee 0}2^{n(m-\beta-\zeta)}$ when $\deg(\boldsymbol{\partial})=m$, we bound 
\begin{align*}
\sup_{\deg(\boldsymbol{\partial}) =m} \sum_{n=0}^{n_{d(p,q)}}\sum_{\zeta< m-\beta}|\Pi_p Q_\zeta  \Gamma_{p,q}\tau_q(\nabla_{\boldsymbol{\partial},1} K_n(p, \cdot)| 
& \lesssim \sum_{n=0}^{n_{d(p,q)}}\sum_{\zeta<m-\beta} d_{\fraks}(p,q)^{(\alpha-\zeta)\vee 0}2^{n(m-\beta-\zeta)} \\
& \lesssim \sum_{\zeta<m-\beta} d_{\fraks}(p,q)^{(\alpha-\zeta)\vee 0}2^{n_{d(p,q)}(m-\beta-\zeta)} \\
&\lesssim \sum_{\zeta<m-\beta}  d_{\fraks}(p,q)^{(\alpha-\zeta)\vee 0}d_{\fraks}(p,q)^{\beta+\zeta-m}\\
&\lesssim d_{\fraks}(p,q)^{\alpha+\beta-m}\ .
\end{align*}
%
This concludes the bound on $\sum_{n=0}^{n_{d(p,q)}} |J^n_{p,q}\tau_q|_m$. 

We turn to $n \geq n_{d(p,q)}$. We have
\begin{align*}
\big | J^n_{p,q}\tau_q \big|_m &\leq\sup_{\deg(\boldsymbol{\partial}) = m} \big| \sum_{m-\beta<\zeta< \delta-\beta} \Pi_p Q_\zeta  \Gamma_{p,q}\tau_q\nabla_{\boldsymbol{\partial},1}K_n(p,\cdot))\big|\\
&\quad +\sup_{\deg(\boldsymbol{\partial}) = m} \big |(\Pi_q\tau_q -\Pi_p \Gamma_{p,q} Q_{<\delta-\beta}\tau_q)(\nabla_{\boldsymbol{\partial},1} K_n(p,\cdot)\big| \\
&\quad+ \big|\Gamma_{pq}J_q^n \tau_q \big| \ ,
\end{align*}
where we shall bound each term  separately.
\begin{enumerate}
\item For the first summand we bound for $m-\beta<\zeta <\delta-\beta$
\begin{align*}
\sum_{n\geq n_{d(p,q)}} |\Pi_p Q_\zeta  \Gamma_{p,q}\tau_q(\nabla_{\boldsymbol{\partial},1} K_n(p,\cdot))|&\lesssim \sum_{n\geq n_{d(p,q)}} d_{\fraks}(p,q)^{(\alpha-\zeta)\vee 0} 2^{n(m-\beta-\zeta)}\\
& \lesssim d_{\fraks}(p,q)^{(\alpha-\zeta)\vee 0} 2^{n_{d(p,q)}(m-\beta-\zeta)}\\
&\lesssim d_{\fraks}(p,q)^{{\alpha-\zeta}\vee 0} d(p,q)^{\beta+\zeta-m}\\
&\lesssim d_{\fraks}(p,q)^{(\alpha+\beta-m)\vee 0} \ .
\end{align*}
\item
We bound the second term using the fact that $2^{-n}\lesssim d_{\fraks}(p,q)$ and $m\leq \alpha+\beta$
\begin{align*}
\sum_{n\geq n_{d(p,q)}}| (\Pi_q\tau_q -\Pi_p Q_{<\delta_0-\beta}\Gamma_{p,q} \tau_q)(\nabla_{\boldsymbol{\partial},1} K_n(p,\cdot)| &\lesssim \sum_{n\geq n_{d(p,q)}}  2^{n(m-\beta)}2^{-n((\delta_0-\beta)\wedge \delta)}\\
&\lesssim 
d_{\fraks}(p,q)^{(\delta+\beta)\wedge\delta_0-m}
\end{align*}
which is obviously bounded by $d_{\fraks}(p,q)^{\alpha+\beta-m}$ .
\item For the last term we have 
\begin{align*}
\sum_{n\geq n_{d(p,q)}}  \big|\Gamma_{pq}J_q^n \tau_q \big|_m& 
\lesssim \sum_{n\geq n_{d(p,q)}} \sum_{l\in \mathbb{N}, l< \alpha+\beta} d_{\fraks}(p,q)^{(l-m)\vee 0} |J_q^n \tau_q|_l \\
&\lesssim \sum_{n\geq n_{d(p,q)}} \sum_{l\in \mathbb{N}, l< \alpha+\beta} d_{\fraks}(p,q)^{(l-m)\vee 0} 2^{-n(\alpha+\beta -l)} \\
&\lesssim  \sum_{l\in \mathbb{N}, l< \alpha+\beta} d_{\fraks}(p,q)^{(l-m)\vee 0}\sum_{n\geq n_{d(p,q)}} 2^{-n(\alpha+\beta -l)} \\
&\lesssim  \sum_{l\in \mathbb{N}, l< \alpha+\beta} d_{\fraks}(p,q)^{(l-m)\vee 0} d_{\fraks}(p,q)^{\alpha+\beta -l} \\
& \lesssim d_{\fraks}(p,q)^{\alpha+\beta-m} \ .
\end{align*}
\end{enumerate}
\end{proof}

The following theorem is almost a word for word adaptation of \cite[Theorem 5.14]{Hai14}, but the proof is slightly easier, since our definitions are more relaxed.
\begin{theorem}\label{thm:extension theorem}
Let $\mathcal{T}=(\{E^{\mathfrak{t}}\}_{t\in \mathfrak{L}},\{T^{\mathfrak{t}}\}_{t\in \mathfrak{L}}, L)$ be a regularity structure ensemble and
let $Z = (\Pi, \Gamma)$ be a model, such that Assumption~\ref{Assumption on reg. and model} is satisfied.
Let $\beta\geq 0$ and let $V\subset T$ be a sector with the property that for every $(\alpha,\delta)\in A\times \triangle$ with $V_{(\alpha,\delta)}\neq \{0\}$ one has $\alpha+\beta \notin \mathbb{N}$ and $(\delta+\beta)\wedge \delta_0 \notin \mathbb{N}$. Furthermore, let $W\subset V$ be a  subsector of $V$ and let $K$ be a kernel satisfying Assumption~\ref{Assumption on Kernel}. 
Let $I: W\to T^{\mathfrak{f}}$ be an abstract integration map of order $\beta$ and precision $\delta_0$ realising $K$ for $I$. 
Then, there exists a regularity structure ensemble $\hat{\mathcal{T}}$ containing $\mathcal{T}$, a model $\hat{Z}=(\hat{\Pi},\hat{\Gamma})$ for $\hat{\mathcal{T}}_\mathfrak{L}$ extending $(\Pi, \Gamma)$, and an abstract integration map $\hat{I}$ of order $\beta$ and precision $\delta$ acting on $\hat{V}= \iota V$ such that:
\begin{itemize}
\item The model $\hat{Z}$ realizes $K$ for $\hat{I}$.
\item The map $\hat{I}$ extends $I$, i.e.\ $\hat{I} \iota \tau = \iota I \tau$ for all $\tau\in W$ .
\end{itemize}
Furthermore, the map $Z\to \hat{Z}$ is locally bounded and Lipschitz continuous as in \cite{Hai14}.
\end{theorem}
\begin{remark}
Note that the extension of the regularity structure ensemble really only affects the $\mathfrak{f}$ component of the ensemble, i.e.\
$\hat{T}^\mathfrak{t}=T^\mathfrak{t}$ for all $\mathfrak{t}\in \mathfrak{L}\setminus \{\mathfrak{f} \}$.
.\end{remark}
\begin{proof}
For $\alpha \in A \cap (-\infty,\delta_0- \beta), \ \delta \in \triangle$ denote by $\bar{W}_{\alpha, \delta}$ a complement, such that that
$$V_{\alpha,\delta}= W_{\alpha, \delta} \oplus \bar{W}_{\alpha, \delta}$$ and set
$\bar W= \bigoplus_{\alpha, \delta} \bar{W}_{\alpha, \delta}$. We further denote by $\bar{W}^+$ the complement such that 
$V= W \oplus \bar{W} \oplus \bar{W}^+$ and observe that $\bar{W}^+= Q_{\geq \delta_0-\beta, :}\bar{W}^+$ .

To construct the regularity structure $\hat{\mathcal{T}}$, we set
\begin{align*}
\hat A &= A \cup \left\{\alpha +\beta\in \mathbb{R}\ :\ \alpha<\delta_0-\beta,\  \bar{W}_{\alpha, :}\neq \{0\} \right\} , \\
\hat \triangle &= \triangle \cup \left\{(\delta +\beta)\wedge \delta_0\in \mathbb{R}\ :\ \bar{W}_{ :\delta}\neq \{0\} \right\}
\end{align*}
and for $(\alpha,\delta)\in \hat{A}\times \hat{\triangle}$ set
\begin{align*}
\hat{T}^{\mathfrak{f}}_{\alpha,\delta}= 
\begin{cases}
T^{\mathfrak{f}}_{\alpha,\delta}\oplus \bar W_{\alpha-\beta, \delta-\beta} &\text{if } \alpha,\delta< \delta_0, \\
T^{\mathfrak{f}}_{\alpha,\delta} \oplus \bigoplus_{\delta_0-\beta\leq\delta'} \bar W_{\alpha-\beta, \delta'} &\text{if } 0\leq\alpha< \delta_0, \ \delta=\delta_0, \\
T^{\mathfrak{f}}_{\alpha,\delta} \oplus \bigoplus_{\delta_0-\beta\leq\delta'<+\infty} \bar W_{\alpha-\beta, \delta'} &\text{if } \alpha<0, \ \delta=\delta_0, \\
T^{\mathfrak{f}}_{\alpha,\delta} \oplus \bar W_{\alpha-\beta, \delta} &\text{if } \alpha< 0, \ \delta=\infty, \\
T^{\mathfrak{f}}_{\alpha,\delta} \oplus 0 &\text{if } \alpha\geq \delta_0-\beta \text{ or } \delta>  \delta_0-\beta \ , \\
\end{cases}
\end{align*}

which is motivated by Definition~\ref{def of abstract integration map}, see also Remark~\ref{rem:complicated grading}.
Accordingly we write elements $\tau_p\in \hat{T}^{\mathfrak{f}}_{\alpha,\delta}$ as pairs $\tau_p= (\tau_p^0, \tau_p^1)$, where $\tau_p^0\in T^{\mathfrak{f}}$ and $\tau_p^1\in \bar{W}$.
Next, we extend the space $L$ appropriately. Note that due to the second condition in the definition of an abstract integration map, Definition~\ref{def of abstract integration map}, we can not simply take the space of all linear maps satisfying the conditions of a regularity structure. Instead we define, similar to \cite{Hai14},
$$\hat L_{p,q}:= L_{p,q}\times M_{p,q}^{\delta_0}$$ where $M_{p,q}^{\delta_0}$ consists of all linear maps from $\bar W_p$ to $(\bar{T}^{\mathfrak{f}}_q)_{<\delta_0}$. 
An element $(\Gamma_{p,q}, M_{p,q})\in \hat L_{p,q}$  acts on $\tau_q=(\tau_q^1,\  \tau^2_q)\in \hat{T}^{\mathfrak{f}}_{\alpha,\delta}$ as 
$$
(\Gamma_{p,q}, M_{p,q})\tau_q = (\Gamma_{p,q}\tau_q^1 + I_p(Q_{W} \Gamma_{p,q} \tau_q^2) + M_{p,q} \tau_q^2, Q_{\bar W}\Gamma_{p,q}\tau_q^2)
$$ 
Similar to \cite[Equation~(5.22)]{Hai14} one can check that $L_{p,q}\circ L_{q,r}\subset L_{p,r}$. 

To extend the map $I$ we note that for $\tau_p \in V_p $ we have a unique decomposition $\tau_p = \tau_p^0 + \tau_p^1+ \tau_p^2$, where $\tau_p^0 \in W$, $\tau_p^1\in \bar W_p$ and $\tau_p^2\in W^+$. We set
$$\hat I_p(\tau_p, 0):= (I_p \tau_p^0, \tau_p^1)\ .$$
We have to check that the conditions in Definition~\ref{def of abstract integration map} of an abstract integration map are satisfied. The only non-trivial fact to check is Condition~\ref{Commutation cond} for which by linearity it suffices to check it on $\hat{I}_q \tau_q = (0,\tau_q)\in \hat T^{\mathfrak{f}}_q$ for $\tau_q \in \bar{W}_q$. Indeed, by definition of $\hat \Gamma_{p,q}= (\Gamma_{p,q}, M_{p,q})\in \hat L_{p,q}$
\begin{align*}
\hat I_p\hat \Gamma_{p.q}\tau_p -\hat \Gamma_{p,q} \hat I_{q} \tau_q &= \hat I_p (\Gamma_{p,q}, M_{p,q})\big(\tau_q, 0\big) -\hat \Gamma_{p,q}\big(0,\tau_q\big)\\
&= \hat I_p \big(Q_W \Gamma_{p.q}\tau_p +Q_{\bar W} \Gamma_{p.q}\tau_p , \  0\big) \\
&\qquad- \big( I_pQ_{W} \Gamma_{p,q} \tau_q + M_{p,q} \tau_q,\  Q_{\bar W} \Gamma_{p,q}\tau_q\big )\\
&=\big(I_p Q_W \Gamma_{p.q}\tau_p  , \  Q_{\bar W} \Gamma_{p.q}\tau_p\big) \\
&\qquad - \big( I_pQ_{W} \Gamma_{p,q} \tau_q + M_{p,q} \tau_q,\  Q_{\bar W} \Gamma_{p,q}\tau_q\big )\\
&=\big( - M_{p,q} \tau_q, 0\big)\ .
\end{align*}
 
Next, we turn to the extension of the model $Z=(\Pi, \Gamma)$ for $\mathcal{T}$ to a model $\hat Z=(\hat \Pi, \hat \Gamma )$ for $\hat{\mathcal{T}}$. For $\tau_q=(\tau_q^1,\  \tau^2_q)$ we define
$$\hat \Pi_q \tau_q = \Pi_q \tau_q^1 + K(\Pi_q \tau_q^2) - \Pi_q J_q \tau_q^2\ .$$
This is well defined due to Lemma~\ref{lemma well definednes of I}.
Next, we define 
$$\hat \Gamma_{p,q} = (\Gamma_{p,q}, M_{p,q})\ ,  \text{ where } M_{p,q}\tau_q^2 := J_p \Gamma_{p,q}\tau_q^2 - \Gamma_{p,q}J_q \tau_q^2 + E_p \tau_q\ =  J_{pq}\tau_q^2\ .$$
We check the second analytic condition in the definition of a model. By linearity and the fact that $\Gamma$ satisfies the this bound, we only need to check it on elements of the form $\tau_q= (0, \tau_q^2)\in \hat T^{\mathfrak{f}}_{\alpha_n+\beta, \delta_n}$.
Indeed, 
\begin{align*}
\hat \Gamma_{p,q}\tau_q &= \big( I_p Q_W\Gamma_{p,q} \tau_q^2 + M_{p,q}\tau_q^2 , \ Q_{\bar{W}} \Gamma_{p,q} \tau_q^2\big)
\end{align*}
and we decompose $\hat \Gamma_{p,q}\tau_q = Q_{\bar{T}^\mathfrak{f}} \hat \Gamma_{p,q}\tau_q +
(\id -Q_{\bar{T}^\mathfrak{f}}) \hat \Gamma_{p,q}\tau_q$. Thus
\begin{align*}
|(\id -Q_{\bar{T}^\mathfrak{f}})\hat \Gamma_{p,q}\tau_q|_\zeta &\leq |I_p Q_W\Gamma_{p,q} \tau^2_q|_\zeta+ |Q_{\bar{W}} \Gamma_{p,q} \tau_q^2|_{\zeta-\beta} \\
&= |Q_W\Gamma_{p,q} \tau^2_q|_{\zeta-\beta}+ |Q_{\bar{W}} \Gamma_{p,q} \tau_q^2|_{\zeta-\beta} \\
&\lesssim |\Gamma_{p,q} \tau^2_q|_{\zeta-\beta} \\
&\lesssim d(p,q)^{(\alpha_n+ \beta-\zeta)\vee 0 }\ .
\end{align*}
The bound on $Q_{\bar{T}^\mathfrak{f}} \hat \Gamma_{p,q}\tau_q = M_{p,q}\tau_q^2 =J_{p,q}\tau_q^2$ follows from Lemma~\ref{lemma J_qp}. 

Lastly, we check the third condition in the definition of a model. Again suffices by linearity to check it on elements of the form $\tau_q= (0, \tau_q^2)\in \hat T^{\mathfrak{f}}_{\alpha, \delta}$ . We assume $\delta<+\infty$, the case $\delta=+\infty$ being simpler.
Indeed, we have
\begin{align*}
\hat\Pi_p & \hat \Gamma_{p,q}\tau_q -\hat \Pi_q\tau_q = \hat\Pi_p \big( I_p Q_W\Gamma_{p,q} \tau_q^2 + M_{p,q}\tau_q^2 , \ Q_{\bar{W}} \Gamma_{p,q} \tau_q^2\big)-\hat \Pi_q \big(0, \tau^2_q)\\
=& \Pi_p I_p Q_W\Gamma_{p,q}\tau^2_q+ \Pi_p (J_p \Gamma_{p,q}\tau_q^2 - \Gamma_{p,q}J_q \tau_q^2 + E_p \tau_q^2) + K(\Pi_p Q_{\bar W}\Gamma_{p,q} \tau_q^2) \\
&- \Pi_p J_p Q_{\bar W}\Gamma_{p,q} \tau_q^2
 - (K(\Pi_q \tau_q^2) - \Pi_q J_q \tau_q^2)\\
=&K(\Pi_p \Gamma_{p,q} \tau_q^2) - \Pi_p J_p \Gamma_{p,q} \tau_q^2+\Pi_p (J_p \Gamma_{p,q}\tau_q^2 - \Gamma_{p,q}J_q \tau_q^2 + E_p \tau_q^2)\\
&-(K(\Pi_q \tau_q^2) - \Pi_q J_q \tau_q^2)\\
=&K(\Pi_p \Gamma_{p,q} \tau_q^2) -\Pi_p \Gamma_{p,q}J_q \tau_q^2 + \Pi_pE_p \tau_q^2\\
&-K(\Pi_q \tau_q^2) + \Pi_q J_q \tau_q^2\\
=&\underbrace{ K(\Pi_p \Gamma_{p,q} \tau_q^2-\Pi_q \tau_q^2) - \Pi_pE_p \tau_q^2}_{T_1}
 + \underbrace{\Pi_q J_q \tau_q^2- \Pi_p\Gamma_{p,q}J_q \tau_q^2}_{T_2}\ .
\end{align*}
Thus, by the same computation as in the proof of Lemma~\ref{lemma compensating error term}
\begin{align*}
|T_1(\phi^\lambda_p)|= \left|\int_M \sum_n (\Pi_p\tau_p-\Pi_q \Gamma_{q,p}\tau_p\big)\big(K^{n,\delta\wedge(\delta_0-\beta)}_{r,s}\big)(\phi_q^\lambda)(s) ds \right|\lesssim\lambda^{(\delta+\beta)\wedge \delta_0}  \,
\end{align*}
where the last inequality follows from Lemma~\ref{lemma bound on error} if $\delta<+\infty$ and even vanishes if $\delta=\infty$.
The bound 
$|T_2(\phi^\lambda_p)|\lesssim \lambda^{\delta_0}$ follows from the fact that the jet model is indeed a model, Proposition~\ref{proposition on polynomial model} and if $\alpha+\beta<0$ vanishes.
\end{proof}

\begin{remark}
Note that if a model is constructed using the Extension theorem above, the term bounded in Lemma~\ref{lemma compensating error term}  actually vanishes.
\end{remark}

\subsection{Schauder estimates}

We introduce for $\gamma<\delta_0-\beta$ the operator $\mathcal{N}_\gamma$ which maps $f \in \mathcal{D}^\gamma(T^\mathfrak{e})$ to the continuous section of $\bar T^{\mathfrak{f}}$
$$(\mathcal{N}_\gamma f)(p) = \sum_{n} Q_{< \gamma+\beta} j_p K\big(\mathcal{R}f-\Pi_p \tau_p\big)  \ .$$
It follows from the Reconstruction Theorem~\ref{reconstruction theorem} and the scaling properties of $K_n$ in Assumption~\ref{Assumption on Kernel} that this sum converges absolutely by a similar argument as in the proof of Lemma~\ref{J lemma}.

Finally, we define $\mathcal{K}_\gamma$ the lift of $K$ to $\mathcal{D}^\gamma(V)$
$$\mathcal{K}_\gamma f(p) = I_pf(p) + J_pf(p) + (\mathcal N_\gamma f)(p)\ .$$

\begin{theorem}\label{schauder theorem}
Let $\mathcal{T}=(\{E^{\mathfrak{t}}\}_{t\in \mathfrak{L}},\{T^{\mathfrak{t}}\}_{t\in \mathfrak{L}}, L)$ be a regularity structure ensemble and $Z=(\Pi, \Gamma)$ a model such that Assumption~\ref{Assumption on reg. and model} is satisfied. Let $V\subset T^\mathfrak{e}$ be a sector, $K$ be $\beta$ regularising kernel satisfying Assumption~\ref{Assumption on Kernel} and $I:V\to T^\mathfrak{f}$ an abstract integration map such that $Z$ realises $K$. Then, for any $\gamma\in \mathbb{R}$ such that $\gamma+\beta\notin \mathbb{N}$ and $0<\gamma<\delta_0-\beta$ it holds that $\mathcal{K}_\gamma$ maps $\mathcal{D}^\gamma (V)$ to $\mathcal{D}^{\gamma+\beta}(T^\mathfrak{f})$ and the following identity holds
\begin{equation}\label{Konvolution Identity}
K \mathcal{R} f = \mathcal{R}\mathcal K_\gamma f \ 
\end{equation}
for all $f\in\mathcal{D}^\gamma (V)$. 
Furthermore, the following bound holds
$$ \interleave\mathcal K_\gamma f ; \mathcal{K}_\gamma \bar{f}\interleave_{\gamma +\beta ; K}\lesssim_C \interleave f;\bar{f}\interleave_{\gamma ; \bar K } + \|Z;\bar{Z}\|_{\bar{K}}
$$
uniformly over models $Z,\bar{Z}$ and modelled distributions $f\in\mathcal{D}_Z^\gamma (V), \ \bar{f}\in \mathcal{D}_{\bar{Z}}^\gamma (V)$ satisfying $\|Z\|_{\bar{K}},$ $\|\bar{Z}\|_{\bar{K}}, \ \|f\|_{\mathcal{D}^\gamma(\bar{K})}$ and $\|\bar{f}\|_{\mathcal{D}^\gamma(\bar{K})}  <C$ .
\end{theorem}
\begin{proof}
First we note that 
\begin{align*}
\|\mathcal{K}_\gamma f\|_{\gamma+\beta, K} & \lesssim \|f\|_{\gamma, K},
& \|\mathcal{K}_\gamma f-\bar{ \mathcal{K}}_\gamma \bar f\|_{\gamma+\beta, K}&\lesssim \|f-\bar{f} \|_{\gamma, K} + \|\Pi -\bar \Pi\|_{\delta, K},
\end{align*}
follows easily as in \cite{Hai14}. 
Similarly,
for $\zeta\in A\setminus \mathbb{N}$ such that $\zeta<\gamma+\beta$ it follows by direct inspection that
\begin{equation}\label{eq:quantitative_shauder}
 |\mathcal{K}_\gamma f(p) - \Gamma_{p,q}\mathcal{K}_\gamma f(q)|_\zeta \lesssim d(p,q)^{\gamma+\beta-\zeta} 
 \end{equation}
and 
\begin{align*}
|\mathcal{K}_\gamma f(p) - \Gamma_{p,q}\mathcal{K}_\gamma f(q)-\big( \bar{K}_\gamma \bar f (p) - &\bar{\Gamma}_{p,q}\bar{ \mathcal{K}}_\gamma \bar f(q)\big)|_\zeta\\
& \lesssim d(p,q)^{\gamma+\beta-\zeta}\big(\| f;\bar{f}\|_{\gamma,\bar K} + \|Z-\bar{Z}\|_{\gamma+\beta, \bar{K}} \big) \ .
\end{align*}

Next, we establish \eqref{eq:quantitative_shauder} for $\zeta\in \mathbb{N}$. Let $k:=\zeta<\gamma+\beta$. It follows from Lemma~\ref{lemma compensating error term} that
\begin{align*}
&\big(\mathcal{K}_\gamma f(p)\big)_k  - \big(\Gamma_{p,q}\mathcal{K}_\gamma f(q)\big)_k \\
&= \big(J_p f(p) + (\mathcal{N}_\gamma f)(p)\big)_k \\
&\qquad-\big(J_p\Gamma_{p,q}f(q)+ E_p f(q)+ \Gamma_{p,q}(\mathcal{N}_\gamma f)(q)\big)_k + \mathcal{O}(d(p,q)^{\gamma+\beta-k})\\
&= \big(J_p (f(p)-Q_{<\gamma}\Gamma_{p,q}f(q))\big)_k + \big(\mathcal{N}_\gamma f)(p)\big)_k - \big(\Gamma_{p,q}(\mathcal{N}_\gamma f)(q)\big)_k \\
&\qquad - \big(E_p f(q)\big)_k-\big(J_p Q_{\geq\gamma}\Gamma_{p,q}f(q))_k +\mathcal{O}(d(p,q)^{\gamma+\beta-k})\\
\end{align*}

We decompose $J_p=\sum_{n=1}^\infty J_p^{(n)}$ , $ E_p=\sum_{n=1}^\infty E_p^{(n)}$ , $ \mathcal{N}_p=\sum_{n=1}^\infty \mathcal{N}_p^{(n)}$ by replacing $K$ by $K_n$. As in previous proofs, we let $n_{d(p,q)}\in \mathbb{N}$ be such that $d(p,q)\in (2^{-n_{d(p,q)}-1},2^{-n_{d(p,q)}} ]$ and proceed by looking at the cases $n\geq n_{d(p,q)}$ and $n< n_{d(p,q)}$ separately. 

For $n> n_{d(p,q)}$ one proceeds as in \cite{Hai14} to obtain
\begin{enumerate}
\item $\sum_{n> n_{d(p,q)}} \big|\mathcal{N}^{n}_\gamma f(p)\big|_k \lesssim d_\fraks(p,q)^{\gamma+\beta -k}$
\item $\sum_{n> n_{d(p,q)}} \big|J^{n}_p \big(f(p)-Q_{<\gamma}\Gamma_{p,q}f(q)\big)\big|_k \lesssim d_\fraks(p,q)^{\gamma+\beta -k}$
\item $\sum_{n> n_{d(p,q)}} \big|\Gamma_{p,q}(\mathcal{N}_\gamma f)(q)\big|_k \lesssim d_\fraks(p,q)^{\gamma+\beta -k} \ .$ 
\end{enumerate}
The remaining terms to be checked are $ \big(E_p f(q)\big)_k \ \text{and} \ \big(J_p Q_{\geq\gamma}\Gamma_{p,q}f(q)\big)_k \ .$
For the former we have by Remark~\ref{norm remark} and using the notation $\boldsymbol{\partial}$  and  $\nabla_{\boldsymbol{\partial}}$ as in the proof of Lemma~\ref{lemma J_qp}
\begin{align*}
\sum_{n> n_{d(p,q)}}|E^{(n)}_p f(q)|_k& 
\leq \sum_{n> n_{d(p,q)}}\sup_{\deg (\boldsymbol{\partial})=k }|(\Pi_pf (p)-\Pi_p Q_{<\delta_0-\beta,:} \Gamma_{p,q}f(q))(\nabla_{\boldsymbol{\partial}, 1} K_n(p,\cdot))|\\
&\lesssim \sum_{n> n_{d(p,q)}} 2^{-n(\gamma+\beta-k)}\lesssim  d_\fraks(p,q)^{\gamma+\beta -k}\ . 
\end{align*}
and similarly
\begin{align*}
\sum_{n> n_{d(p,q)}}|J^{n}_p Q_{\geq\gamma}\Gamma_{p,q}f(q)|_k 
&\leq \sum_{n> n_{d(p,q)}}\sup_{\deg (\boldsymbol{\partial})=k } | \Pi_p Q_{\geq\gamma}\Gamma_{p,q}f(q)(\nabla_{\boldsymbol{\partial}, 1}K_n(p,\cdot)) |\\
&\lesssim \sum_{n> n_{d(p,q)}}2^{-n(\gamma + \beta -k)}\lesssim d_\fraks(p,q)^{\gamma+\beta -k} \ .
\end{align*}

Now we turn to the case $n\geq n_{d(p,q)}$.
Similarly to the large scale bound in the proof of \cite[Theorem~5.12]{Hai14}, we define
\begin{align*}
\mathcal{T}_1^k &:= - Q_{k} \big((\mathcal{N}^{(n)}_\gamma f)(p)+ J^{(n)}_p f(p)\big)\\
\mathcal{T}_2^k &:=   Q_{k}\big(J^{(n)}_p \Gamma_{p,q}f(q)+ E^{(n)}_pf(q)+\Gamma_{p,q}(\mathcal{N}^{(n)}_\gamma f)(q) \big)
\end{align*} 
and note that 
$$-\big(\mathcal{T}_1^k+ \mathcal{T}_2^k\big) = Q_k\big(J^{(n)}_p (f(p)-\Gamma_{p,q}f(q)) + \mathcal{N}^{(n)}_\gamma f)(p) - \Gamma_{p,q}(\mathcal{N}^{(n)}_\gamma f)(q) - E^{(n)}_p f(q) \big) \ .$$
Inspecting the definitions of the terms, we obtain the identities
\begin{align*}
\mathcal{T}_1^k =& Q_k j_p K_n\big( \sum_{\zeta\leq k-\beta} \Pi_p Q_\zeta f(p)-\mathcal{R}f\big)\ ,\\
\mathcal{T}_2^k =& Q_k j_p K_n\big(\sum_{k-\beta<\zeta<{\delta_0-\beta}}\big( \Pi_pQ_\zeta\Gamma_{p,q} f(q)\big)\\
+& Q_k j_p K_n\big(\Pi_q f(q) -\Pi_p Q_{<\delta-\beta}\Gamma_{p,q} f(q))\big)\\
-& Q_k j_pR\big(Q_{<\gamma+\beta} j_q K_n(\Pi_q f(q)-\mathcal{R}f)
\big)\ .
\end{align*}
We can rewrite 
\begin{align*}
\mathcal{T}_2^k &=
 Q_k j_p K_n\big( \Pi_p Q_{<\delta_0-\beta} \Gamma_{p,q} f(q)\big) -  Q_k j_p K_n\big( \sum_{\zeta\leq k-\beta} \Pi_pQ_\zeta\Gamma_{p,q} f(q)\big)\\
&+ Q_k j_p K_n\big(\Pi_q f(q) -\Pi_p Q_{<\delta-\beta} \Gamma_{p,q} f(q))\big)\\
&- Q_k j_pR\big(Q_{<\gamma+\beta} j_q K_n(\Pi_q f(q)-\mathcal{R}f)\\
&=  Q_k j_p K_n\big(\Pi_q f(q)\big)-  Q_k j_p K_n\big( \sum_{\zeta\leq k-\beta} \Pi_pQ_\zeta\Gamma_{p,q} f(q)\big)\\
&- Q_k j_pR\big(Q_{<\gamma+\beta} j_q K_n(\Pi_q f(q)-\mathcal{R}f)\big) \ . \\
\end{align*}
Thus
\begin{align*}
 \mathcal{T}_1^k+ \mathcal{T}_2^k&= Q_k j_p K_n\big(\Pi_q f(q)-\mathcal{R}f\big) -Q_k j_p R\big(Q_{<\gamma+\beta} j_q K_n(\Pi_q f(q)-\mathcal{R}f)\\
 &+Q_k j_p K\big( \sum_{\zeta\leq k-\beta} \Pi_p Q_\zeta (f(p)- \Gamma_{p,q} f(q))\big) 
\end{align*}
and therefore
\begin{align}\label{two terms to bound}
 |\mathcal{T}_1^k+ \mathcal{T}_2^k|_k &\leq \sup_{\deg (\boldsymbol{\partial}) =k} \big| \big(\Pi_q f(q)-\mathcal{R}f\big)(K^{\boldsymbol{\partial},\gamma}_{n,p,q})\big|\\
 &+\sup_{\deg (\boldsymbol{\partial}) =k} \sum_{\zeta\leq k-\beta} |\Pi_p Q_\zeta (f(p)-\Gamma_{p,q} f(q))(\nabla_{\boldsymbol{\partial}} K_n(p,\cdot))| \ . \nonumber
\end{align}
We now bound these two summands separately. 
\item Since $d(p,q)<2^{-n}$ it follows as in Lemma~\ref{technical Lemma for Kernels} that for $\rho>\gamma+\beta$,
\begin{align}\label{something to reuse}
|(\Pi_q f(q)-\mathcal{R}f\big)(K^{\boldsymbol{\partial},\gamma}_{n,p,q})|&\lesssim  \sum_{\tilde{l}\in \delta A_{\gamma}} 2^{(|\tilde{l}|_\fraks+\deg (\boldsymbol{\partial})-\gamma-\beta)n } d(p,q)^{|\tilde{l}|_\fraks} \\
&+   \sum_{m<\gamma+\beta} d_\fraks(p,q)^{\rho- \deg (\boldsymbol{\partial})} 2^{-n(\gamma+\beta-m}) \nonumber
\end{align}
and summing over $n<n_{d(p,q)}$ 
we obtain for $\deg (\boldsymbol{\partial}) =k$
\begin{align*}
\sum_{n<n_{d(p,q)}}|(\Pi_q f(q)-\mathcal{R}f\big)(K^{\boldsymbol{\partial},\gamma}_{n,p,q})|
\lesssim d(p,q)^{\gamma+\beta -\deg (\boldsymbol{\partial})} +d(p,q)^{\rho-\deg (\boldsymbol{\partial})}\lesssim d(p,q)^{\gamma+\beta -k} \ .
\end{align*}
\item For the second term in \eqref{two terms to bound}, note that $\zeta+\beta\leq k$ actually means $\zeta+\beta< k$ since $\zeta+\beta\notin \mathbb{N}$. This yields
\begin{align*}
\sum_{n<n_{d(p,q)}} |\Pi_p Q_\zeta (f(p)-\Gamma_{p,q} f(q))(\nabla_{\boldsymbol{\partial}}  K_n(p,\cdot))|&\lesssim \sum_{n<n_{d(p,q)}} d(p,q)^{(\gamma-\zeta)\vee 0}2^{(k-\beta)n}2^{-n\zeta}\\
&\lesssim d(p,q)^{\gamma-\zeta}2^{(k-\beta-\zeta)n_{d(p,q)}}\\
&\lesssim  d(p,q)^{\gamma+\beta-k} \ .
\end{align*}

It remains to show the the identity \eqref{Konvolution Identity}.
As in the last part of the proof of \cite[Theorem~5.12]{Hai14}, one sees mutatis mutandis that 
$$(\Pi_p \mathcal{K}f(p)-K(\mathcal{R}f))(\phi^\lambda_p)= \sum_n \int(\Pi_p f(p)-\mathcal{R}f)(K^{\gamma}_{n,q,p})\phi^\lambda_p(q) \, \dVol_q \ .$$
One bounds $$\sum_{n\leq n_\lambda} \int(\Pi_p f(p)-\mathcal{R}f)(K^{\gamma}_{n,q,p})\phi^\lambda_p(q) \, \dVol_q\lesssim \lambda^{\gamma+\beta}$$ by the same argument as in Lemma~\ref{lemma well definednes of I} .
To obtain 
$$\sum_{n > n_\lambda} \int(\Pi_p f(p)-\mathcal{R}f)(K^{\gamma}_{n,q,p})\phi^\lambda_p(q) \, \dVol_q\lesssim \lambda^{\gamma+\beta}$$
one decomposes
$$ K^{\gamma}_{n,p,q}(z):= K_n(q,z)- \sum_{k<\gamma+\beta} R(Q_k j_p K_n(\cdot,z) )(q)$$
and bounds each term separately.
\end{proof}

\begin{remark}
The identity 
\begin{align*}  
|\mathcal{T}_1^k+ \mathcal{T}_2^k|_k &\leq \sup_{\deg (\boldsymbol{\partial}) =k}  \Big| \big(\Pi_q f(q)-\mathcal{R}f\big)(K^{\boldsymbol{\partial},\gamma}_{n,p,q}) \\
 +&\sum_{\zeta\leq k-\beta} \Pi_p Q_\zeta (f(p)-\Gamma_{p,q} f(q))(\nabla_{\boldsymbol{\partial}} K_n(p,\cdot)\Big|
 \end{align*}
 is the only place
 where we crucially need the term $E_p\tau_q$ and where replacing $E_p\tau_q$ by $Q_{<\alpha+\beta}E_p\tau_q$ would be insufficient. If we were to replace 
$E_p\tau_q$ by $Q_{<\alpha+\beta}E_p\tau_q$ we would instead obtain
 \begin{align*}
 |\mathcal{T}_1^k+ \mathcal{T}_2^k|&\leq \sup_{\deg (\boldsymbol{\partial}) =k}  \Big| \big(\Pi_q f(q)-\mathcal{R}f\big)(K^{\boldsymbol{\partial},\gamma}_{n,p,q})\\
 +&\sum_{\zeta\leq |k|-\beta} \Pi_p Q_\zeta (f(p)-\Gamma_{p,q} f(q))(\nabla_{\boldsymbol{\partial}} K_n(p,\cdot))\\
 &-\sum_{\zeta\leq k-\beta} (\Pi_q -\Pi_p \Gamma_{p,q}) Q_\zeta f(q)(\nabla_{\boldsymbol{\partial}} K_n(p,\cdot))\Big|
\end{align*}
where the last term in general destroys the needed estimate (think of $n=1$). 

\end{remark}
\begin{remark}\label{philosophical remark}
The fact that we cannot truncate $E_p\tau_q$ is also consistent with the following interpretation: Let $\tau_p\in T^{\mathfrak{e}}$ and assume that $$q\mapsto  f(q)=\Gamma_{q,p}\tau_p \ $$
is the modelled distribution in $\mathcal{D}^\delta$. (Note that this is not necessarily the case with our definitions.) 
Then,
\begin{align*}
\mathcal{K}_\delta f(p)&= I_p f(p) + J_p f(p) + (\mathcal{N}_\delta  f)(p)\\
&= I_p \tau_p + J_p\tau_p\ ,
\end{align*}
while  
\begin{align*}
\mathcal{K}_\delta f(q)&= I_q f(q) + J_q f(q) + (\mathcal{N}_\delta  f)(q)\\
&= I_q \Gamma_{q,p}\tau_p + J_q \Gamma_{q,p}\tau_p+ E_q\tau_p\  .
\end{align*}
\end{remark}

\section{Singular Modelled Distributions}\label{section singular modelled distributions}
In this section we shall be working with modelled distributions defined on $\mathbb{R}\setminus \{0\} \times \bar{M}$, which belong locally to $\mathcal{D}^\gamma$ away from $\{t=0\}:=\{0\}\times \bar{M}$ and derive results analogous to those in \cite[Section 6]{Hai14}. 
We assume that $\bar{M}$ is equipped with a scaling $\bar{\fraks}$ as in Section~\ref{section geometric setting} and extend the scaling in the canonical way to a scaling $\fraks= (\fraks_0, \bar{\fraks})$ on $\mathbb{R} \times \bar{M}$, where $\fraks_0\in \mathbb{N}$.  We write for $z_i=(t_i,p_i)\in \mathbb{R}\times \bar{M}$ 
$$d_{ \fraks}(z_1,z_2):= |t_1-t_2|^{\frac{1}{\fraks_0}} + d_{\bar{\fraks}}(p_1,p_2) \ .$$
Then, we introduce 
for $z=(t,p)\in \mathbb{R}\times \bar{M}$,
$$
|t|_{\mathfrak{s}_0}:= |t|^{\frac{1}{\mathfrak{s}_0}}, \qquad |z|_{\mathfrak{s}_0}:=|t|_{\mathfrak{s}_0}\ , 
$$
and for any compact set ${K}\subset  \mathbb{R}\times M$
$$
{K}_\mathfrak{t}:= \left\{ (z_1, z_2)\in  ({K}\setminus \{t=0\})^2 \ | z_1\neq z_2 \text{ and } d_\fraks(z_1, z_0)\leq |z_1|_{\mathfrak{s}_0}\wedge |z_1|_{\mathfrak{s}_0} \right\} \ .
$$
\begin{definition}
Fix a regularity structure ensemble $\mathcal{T}=(\{E^{\mathfrak{t}}\}_{t\in \mathfrak{L}},\{T^{\mathfrak{t}}\}_{t\in \mathfrak{L}}, L)$ and a model $Z=(\Pi,\Gamma)$ on $\mathbb{R}\times M$. For $0\leq\eta<\gamma$
we define $\mathcal{D}^{\gamma,\eta}_\mathfrak{t}$ to consist of all sections $f$ of $T^{\frakt,\gamma}$ such that for each compact set ${K}\subset \mathbb{R}\times M$ the following semi-norm is finite $$\interleave f \interleave_{\gamma,\eta;\mathfrak{t}(K)}:= \| f\|_{\gamma, \eta; K}
+\sup_{(z_1,z_2)\in K_\mathfrak{t}, } \sup_{\alpha <\gamma} \frac{|f(z_1)-\Gamma_{z_1,z_2}f (z_2)|_{\alpha}}{d(z_1,z_2)^{\gamma-\alpha}(|z_1|_{\mathfrak{s}_0}\wedge|z_2|_{\mathfrak{s}_0})^{\eta- \gamma}  }\ ,$$
where 
$$ \| f\|_{\gamma, \eta; K}:= \sup_{(t,p)\in K\setminus \{t=0\}} \sup_{\alpha<\gamma} \frac{|f(t,p)|_\alpha}{|t|_{\mathfrak{s}_0}^{(\eta-\alpha) \wedge 0}} $$
For $f\in\mathcal{D}^{\gamma,\eta}_\mathfrak{t}$, we furthermore introduce the following quantity 
$$\talloblong f \talloblong_{\gamma,\eta ;K} :=\sup_{(t,p)\in K} \sup_{\alpha<\gamma}\frac{|f(t,p)|_\alpha}{|t|_{\mathfrak{t}}^{\eta-\alpha}} \ . $$
Lastly, for  a second model $\tilde Z = (\tilde \Pi, \tilde{\Gamma})$ and $\tilde{f}\in \mathcal{D}_{\tilde{Z}}^\gamma$ we define the distance
\begin{align*}
\interleave f ; \tilde{f} \interleave_{\gamma,\eta;K}:=& \|f-\tilde{f}\|_{\gamma, \eta ; K}+\sup_{(z_1,z_2)\in K_\mathfrak{t}, } \sup_{\alpha <\gamma} \frac{|f(z_1)-\Gamma_{z_1,z_2}f (z_2)-(\tilde{f}(z_1)-\tilde\Gamma_{z_1,z_2}\tilde{f} (z_2))|_{\alpha}}{d(z_1,z_2)^{\gamma-\alpha}(|z_1|_{\mathfrak{s}_0}\wedge|z_2|_{\mathfrak{s}_0})^{\eta- \gamma}  }.
\end{align*}
\end{definition}
The next lemma differs from \cite[Lemma 6.5]{Hai14} due to the fact that the maps $\Gamma$ are not lower triangular and since for $p\in M$, $t,t' \in \mathbb{R}$ we do not identify $T_{(t',p)}$ with $T_{(t,p)}$ (and therefore addition between elements in these spaces is not well defined). This forces us to impose a stronger continuity condition on $f$ at $\{t=0\}$ .
\begin{lemma}\label{two norms lemma exact case}
Let $\eta<\gamma$. Let $I:=[a,b]\subset[-1,1]$ be an interval containing zero, and $K\subset \bar{M}$ compact. Assume that $f\in \mathcal{D}^{\gamma,\eta}_\mathfrak{t}$ extends continuously to all of $I\times K$ and and that there exists $\epsilon>0$ such that for all $\alpha<\eta $ one has $ |f(t,p)|\lesssim |t|^\epsilon$ uniformly over $I\times K$ . Then one has the bounds
$$\talloblong f \talloblong_{\gamma,\eta;I\times K}\lesssim \interleave f \interleave_{\gamma,\eta;I\times K}$$
and
$$\talloblong f-\bar f \talloblong_{\gamma,\eta;I\times K}\lesssim \interleave f;\bar f \interleave_{\gamma,\eta;I\times K}+ \|\Gamma -\bar{\Gamma}\|_{I\times K} \big( \interleave f \interleave_{\gamma,\eta;I\times K}+\interleave f \interleave_{\gamma,\eta;I\times K} \big)\ ,$$
where the implicit constant only depends on $\|\Gamma\|_{I\times K}+\|\bar\Gamma \|_{I\times K}$ .
\end{lemma}
\begin{proof}
For the purpose of the proof, it is useful to define the quantity 
$$\interleave \Gamma \interleave_{ I\times K}:= \sup_{\alpha\in A} \sup_{t, t'\in I} \sup_{p\in K} \sup_{\tau_{(t,p)}\in T^{\frakt}_{:,:}} \frac{|\Gamma_{(t',p),(t,p)} \tau_{p,t} |_\alpha}{|t-t'|_{s_0}}\ $$
and observe that there exists a constant $C>0$ such that $\interleave \Gamma \interleave_{ I\times K}\leq C \| \Gamma\|_{I\times K}$.
It is also useful to define $N\in \mathbb{R}$ be such that $\interleave \Gamma \interleave_{I\times K}< N^\epsilon\wedge \min_{\alpha,\beta \in A, \alpha\neq \beta} N^{|\alpha-\beta|}$.

Let $\beta\in A$ be such that the bound $|f (t,p)|_\zeta \lesssim |t|_{\mathfrak{s}_1}^{\eta-\zeta}$ holds for all $\zeta>\beta$. 
We argue that then the same bound holds for $\zeta = \beta$ as well. We can assume that $\beta<\eta$ since otherwise the inequality follows from the definition of $\interleave \cdot \interleave_{\mathcal{D}^{\gamma,\eta}_\mathfrak{t}(I\times K)}$ .

For notational convenience we write for fixed $p\in {M}$: $$\Gamma_{n}^{N,\beta}:=\Gamma_{(N^{-n}t,p),(N^{-(n+1)}t,p)}Q_{\leq \beta}  \ \text{ and } f(s):= f(s,p) \ ,$$
as well as 
$$\triangle f_n^{N,\beta} :=f(N^{-n}t)- \Gamma_{n}^{N,\beta}f(N^{-(n+1)}t) \ . $$
Note that then it follows from the H\"older type assumption of $f$ close to $\{t=0\}$ and the choice of $N$ that
$$Q_\beta \Gamma_{0}^{N,\beta} \circ\ldots\circ \Gamma_{n}^{N,\beta} f(N^{-(n+1)})\to 0 $$
as $n\to \infty$. Indeed we have for $t\leq 1$
$$
|\Gamma_{0}^{N,\beta} \circ\ldots\circ \Gamma_{n}^{N,\beta} f(tN^{-(n+1)})| \leq \interleave \Gamma \interleave_{ I\times K}^{n+1} N^{-\epsilon(n+1)}
$$
which is summable.
Using this and the convention that for $n=0$ $$\Gamma_{0}^{N,\beta} \circ\ldots\circ \Gamma_{n-1}^{N,\beta} :=\id \ ,$$  one finds
\begin{align*}
f(t) &= f(t) - \sum_{n=0}^\infty \Gamma_{0}^{N,\beta}\circ\ldots\circ \Gamma_{n}^{N,\beta} \big(f(N^{-(n+1)}t)-f(N^{-(n+1)}t)\big) \\
&= \sum_{n=0}^{\infty}  \Gamma_{0}^{N,\beta} \circ\ldots\circ \Gamma_{n-1}^{N,\beta} \big( f(N^{-n}t)- \Gamma_{n}^{N,\beta}f(N^{-(n+1)}t)\big) \\
&= \sum_{n=0}^{\infty}  \Gamma_{0}^{N,\beta} \circ\ldots\circ \Gamma_{n-1}^{N,\beta} (\triangle f_n^{N,\beta}) \ .
\end{align*}
Thus
\begin{align*}
|f(t,p)|_\beta & \leq \sum_{n=0}^{\infty} \big|  \Gamma_{0}^{N,\beta} \circ\ldots\circ \Gamma_{n-1}^{N,\beta} (\triangle f_n^{N,\beta}) \big|_\beta\\
&\leq \sum_{n=0}^{\infty} \sum_{\alpha\leq \beta} \interleave\Gamma \interleave_{ I\times K}^{n}|\triangle f_n^{N,\beta} |_\alpha
\end{align*}
Next observe that 
\begin{align*}
|\triangle f_n^{N,\beta} |_\alpha 
&\leq |f(N^{-n}t)- \Gamma_{(N^{-n}t,p),(N^{-(n+1)}t,p) }f(N^{-(n+1)}t) |_\alpha \\
&\qquad + |\Gamma_{(N^{-n}t,p),(N^{-(n+1)}t,p) } Q_{>\beta} f(N^{-(n+1)}t)|_\alpha\\
&\lesssim  (N^{-n}|t|_{\fraks_0})^{\eta-\alpha} + \sum_{\zeta>\beta} C(N^{-n}|t|)^{(\zeta-\alpha)\vee 0} |f(N^{-(n+1)}t)|_\zeta\\
&\lesssim (N^{-n}|t|_{\fraks_0})^{\eta-\alpha} + \sum_{\zeta>\beta} (N^{-n}|t|_{\fraks_0})^{\zeta-\alpha} (N^{-n}|t|_{\fraks_0})^{\eta-\zeta}\\
&\lesssim (N^{-n}|t|_{\fraks_0})^{\eta-\alpha}  \ ,
\end{align*}
where we used the fact that that $\alpha\leq \beta < \zeta$ and the assumption on $\beta$ in the second last line.
Therefore 
\begin{align*}
|f(t,p)|_\beta  & \lesssim |t|^{\eta-\beta} \sum_{n}\big( \sum_{\alpha\leq \beta} \interleave\Gamma \interleave_{ I\times K}^{n} N^{-n(\eta-\alpha)} \big) \\
& \lesssim |t|^{\eta-\beta} \sum_{n} \interleave\Gamma \interleave_{ I\times K}^{n} N^{-n(\eta-\beta)} \ .
\end{align*}
\end{proof}

\begin{remark}
It is not difficult to see, that argument of the proof generalises to setting of \cite[Remark 6.1]{Hai14}, i.e.\ where $\{t=0\}$ is replaced by a more complicated submanifold.
\end{remark}

Exactly as in \cite{Hai14}, one obtains the following interpolation inequality.
\begin{lemma}
Let $\gamma>0$ and $\kappa\in (0,1)$, as well as $f,\bar{f}$ as in Lemma~\ref{two norms lemma exact case}. Then for every compact set $K$ the following holds

$$
 \interleave f,\bar{f} \interleave_{(1-\kappa)\gamma,\eta;I\times K}\lesssim \talloblong f - \bar{f}\talloblong_{\gamma,\eta;I\times K}^\kappa\big(\talloblong f \talloblong_{\mathcal{D}^{\gamma,\eta}_\mathfrak{t}(I\times K)}+\talloblong \bar f \talloblong_{\gamma,\eta ;I\times K}\big)^{1-\kappa} \,
$$
where the implicit constant only depends on $\|\Gamma\|_{I\times K}+\|\bar\Gamma \|_{I\times K}$ .
\end{lemma}


\begin{prop}[Hai14, Prop. 6.9]\label{reconstruction for singular modelled distribution}
Let $f\in \mathcal{D}^{\gamma,\eta}(V)$ for some sector $V\subset T^\frakt$ of regularity $\alpha\leq 0$, some $\gamma>0$ and $\eta\leq\gamma\leq \delta$. Then, provided that $\alpha\wedge \eta>-\mathfrak{s}_0$, there exists a unique distribution $\mathcal{R}f\in \mathcal{C}^{\alpha\wedge \eta}(E^\frakt)$ which agrees with the reconstruction operator away from $\{t=0\}$. 
Furthermore one has the bound
\begin{equation}\label{reconstruction for singular modelled dist}
\|\mathcal{R}f- \bar{ \mathcal{R}} \bar f \|_{\mathcal{C}^{\alpha\wedge\eta}( I \times K)}\lesssim \interleave f,\bar{f} \interleave_{\gamma,\eta; I\times \bar K} + \|Z, \bar{Z}\|_{\gamma, \bar K} \  .
\end{equation}
\end{prop}
\begin{proof}
The following results can also be obtained mutatis mutandis as in \cite{Hai14} by going to appropriate $\fraks$-charts.
\end{proof}

\subsection{Products and nonlinearites}
Then next lemma is the analogue of Lemma~\ref{lemma product}, the proof of which is a straightforward adaptation of the proof of \cite[Prop.~6.12]{Hai14}. 
\begin{lemma}\label{lemma product singular}
Let $V^{\mathfrak{t}_1}\subset T^{\mathfrak{t}_1}, {V}^{\frakt_2}\subset T^{\frakt_2}$ be two sectors of regularity $\alpha_1$ resp. $\alpha_2$ and let $\star: V^{\mathfrak{t}_1}\times  V^{\mathfrak{t}_2}\to T^{\mathfrak{t}_3}$ be an abstract tensor product. Let $Z$ be a model and $\gamma_1,\gamma_2>0$ and $\eta_1<\gamma_1$, $\eta_2<\gamma_2$. Assuming the product $\star$ is $\gamma$-precise for $\gamma= (\gamma_1+\alpha_2) \wedge (\gamma_2+\alpha_1)$ and setting $\eta= (\eta_1+\alpha_2)\wedge(\eta_2+\alpha_1)\wedge(\eta_1+\eta_2)$, the map 
$$ \mathcal{D}^{\gamma_1,\eta_1}_{\mathfrak{t}}(V^{\frakt_1})\times \mathcal{D}^{\gamma_2\eta_2}_{\mathfrak{t}}(V^{\frakt_2}) \to \mathcal{D}^{\gamma,\eta}_{\mathfrak{t}}(T^{\frakt_3}), \quad (f_1,f_2)\mapsto f_1\star_\gamma f_2$$
is well defined, continuous and the bound 
$$\interleave f_1\star_\gamma f_2 \interleave_{\gamma,\eta; K}\lesssim  \interleave f_1 \interleave_{\gamma_1,\eta_1; K} \interleave f_2 \interleave_{\gamma_2,\eta_2; K}(1+ \|\Gamma\|_{K})$$
holds.
If furthermore $\tilde Z = (\tilde \Pi, \tilde{\Gamma})$ is a second model and $\tilde{f}_1\in \mathcal{D}_{\tilde{Z}}^{\gamma_1}(V^{\frakt_1})$, $\tilde{f}_2\in \mathcal{D}_{\tilde{Z}}^{\gamma_2}(V^{\frakt_2})$ one has the bound 
$$ \interleave  f_1\star_\gamma f_2 ; \tilde{f}_1\star_\gamma \tilde{f}_2  \interleave_{\gamma,\eta ;K} \lesssim_C \interleave f_1 ; \tilde{f}_1 \interleave_{\gamma_1,\eta_1 ;K} +\interleave f_2 ;  \tilde{f}_2 \interleave_{\gamma_2,\eta_2 ;K} + \|\Gamma-\tilde{\Gamma}\|_{ K})$$
uniformly over  $f_i, \tilde{f}_i$ satisfying $\interleave f_i \interleave_{\gamma_i,\eta_i; K} , \ \interleave \tilde{f}_i \interleave_{\gamma_i,\eta_i; K} <C$ and models satisfying \linebreak 
$\|\Gamma\|_{ K},\  \|\tilde{\Gamma}\|_{ K}<C$.
\end{lemma}
Since the lift of general non-linearities differs from the constructions in the flat setting, we provide a proof of the next Proposition, even though it is very close to the proof of Proposition~\ref{theorem composition}.
\begin{theorem}\label{theorem composition singular}
Let $\mathcal{T}$ be a regularity structure ensemble, $\{V^{\mathfrak{l}}\}_{\frakl \in \bar{\mathfrak{L}}}$ as in Assumption~\ref{Assumption Nonlinearites}, $Z=(\Pi,\Gamma)$ a model for $\mathcal{T}$ and
 $G\in \mathcal{C}^\infty(E\ltimes F)$. For $\eta\in[ 0,\gamma] $ one finds that for $f\in \mathcal{D}^{\gamma,\eta}(\bar{V})$ 
$$G_\gamma(f)(p):= \sum_n \frac{1}{n!}\trr\big(j_p^E(D^nG)(\bar f_p ) \star_\gamma \tilde{f}_p^{\star_\gamma n}\big)\in \mathcal{D}^{\gamma,\eta}(U^0) \ .$$ 
Furthermore, the map
$f\mapsto G_\gamma(f)$ is locally Lipschitz continuous in the sense that for every $C>0$,
\begin{equation}\label{nonlinear difference bound11}
\interleave G_\gamma(f)-G_\gamma(g)\interleave_{\gamma,\eta ; K}\lesssim \interleave f-g\interleave_{\gamma,\eta ; K}
\end{equation}
uniformly over $f,g$ satisfying $\interleave f \interleave_{\gamma,\eta ; K}, \interleave g\interleave_{\gamma,\eta ; K} \leq C$.
\end{theorem}

The next corollary follows as in Remark~\ref{rem:upgrade_trick}.
\begin{corollary}
In the setting of Theorem~\ref{theorem composition singular}, if $\tilde Z = (\tilde \Pi, \tilde{\Gamma})$ is a second model that agrees with $Z$ on jets and $M$ is compact then one has the bound
\begin{equation}
\interleave G_\gamma(f);G_\gamma(\tilde{g}) \interleave_{\gamma, \eta ; K} \lesssim \interleave f;\tilde{g}\interleave_{\gamma, \eta ; K}+ \|Z;\tilde{Z}\|_K \ ,
\end{equation}
uniformly over $f,\tilde{g}$ satisfying $\interleave f\interleave_{\gamma , \eta ; K}, \interleave \tilde{g} \interleave_{\gamma, \eta ; K} \leq C$
and models satisfying $\|Z\|_K$, $\|\tilde{Z}\|_K\leq C$, i.e.\ the map
$$\mathcal{D}^{\gamma,\eta}(\bar{V})f \to \mathcal{D}^{\gamma,\eta}(U^0), \qquad  f \mapsto G_\gamma(f)(p)$$
is strongly locally Lipschitz continuous.
\end{corollary}

\begin{proof}[Proof of Theorem~\ref{theorem composition singular}]

Let $\zeta$ be as in Assumption~\ref{Assumption Nonlinearites}, we again only treat the more involved case $\gamma>\zeta$.

In order to bound $\|G_\gamma(f)\|_{\gamma, \eta ; K}$ note that
 $$|G_\gamma(f)(p)|_\alpha \leq  \sum_n^{[\gamma/\zeta]} |j_p^E(D^nG)(\bar f_p )|_l  |\tilde{f}_p^{\star_\gamma n}|_{\alpha-l}\ .$$
Since  $|j_p^E(D^nG)(\bar f_p )|_l\lesssim |\bar{f}_p|\lesssim |t|_{\fraks_0}^{\eta\wedge 0}$ and 
$$|\tilde{f}_p^{\star_\gamma n}|_{\alpha-l}\lesssim \sum_{l_1+\ldots+l_n= \alpha-l}  |\tilde{f}_p|_{l_n}\lesssim |t|_{\fraks_0}^{(\eta-l_1)\wedge 0}\cdots|t|_{\fraks_0}^{(\eta-l_n)\wedge 0} \ ,$$
it thus suffices to check that 
$$ (\eta \wedge 0) + \left( (\eta-l_1)\wedge0 \right) +\ldots+\left( (\eta-l_n)\wedge 0\right)\geq (\eta-\alpha) \wedge 0 \ .$$
If all terms on the left vanish, we are done. If not, note that
$(\eta \wedge 0)=0$ and thus there must be a summand contributing $\eta-l_i$, for the rest of the summands note that
$\left( (\eta-l_j)\wedge 0\right)\geq -l_j$. Thus, we find that 
\begin{align*}
(\eta \wedge 0)  + \left( (\eta-l_1)\wedge 0 \right) +\ldots+\left( (\eta-l_n)\wedge 0\right)
\geq \eta -\alpha + l  
\geq (\eta-\alpha) \wedge 0 \ ,
\end{align*}
and conclude that $\|G_\gamma(f)\|_{\gamma, \eta ; K}\lesssim \|f\|_{\gamma, \eta ; K} $ .

In order to bound $ \interleave G_\gamma(f) \interleave_{\gamma, \eta ; K}$
we find as in the proof of Theorem~\ref{theorem composition} that for $\kappa<\gamma$
$$Q_\kappa\left(\Gamma_{q,p}G_\gamma(f)(p)- \sum_{n=0}^{[\gamma/\zeta]}\trr \frac{1}{n!} \Gamma_{q,p} j_p D^n G( R\bar{f}_p (\cdot)) \star_\gamma \Gamma_{q,p}\tilde{f}_p^{\star_\gamma n}\right)=0$$ and 
\begin{align*}
&\left| \sum_{n=0}^{[\gamma/\zeta]} \trr \frac{1}{n!} j_q D^n G( R\bar{f}_p(\cdot)) \star_\gamma \Gamma_{q,p}\tilde{f}_p^{\star_\gamma n}   - \sum_{n=0}^{[\gamma/\zeta]} \trr \frac{1}{n!} \Gamma_{q,p} j_p D^n G( R\bar{f}_p (\cdot)) \star_\gamma \Gamma_{q,p}\tilde{f}_p^{\star_\gamma n} \right|_\kappa \\
&\lesssim \sum_{n=0}^{[\gamma/\zeta]}  \sum_{\alpha+\beta= \kappa} \left| Q_{<\gamma-\beta} \big(j_q D^n G( R\bar{f}_p(\cdot))  - \Gamma_{q,p} j_p D^n G( R\bar{f}_p (\cdot)\big) \right|_\alpha |\Gamma_{q,p}\tilde{f}_p^{\star_\gamma n}  |_\beta \\
&\leq \sum_{n=0}^{[\gamma/\zeta]}  \sum_{\alpha+\beta= \kappa} d_\fraks (p,q)^{\gamma-\beta-\alpha} |\bar{f}_q|_0 
\sum_\chi |\tilde{f}_p^{\star_\gamma n}  |_\chi \\
&\lesssim \sum_{n=0}^{[\gamma/\zeta]}   d_\fraks (p,q)^{\gamma-\kappa} 
\sum_\zeta (|p|_{\mathfrak{s}_0}\wedge|q|_{\mathfrak{s}_0})^{\eta-\chi\wedge 0} \\
&\lesssim d_\fraks(p,q)^{\gamma-\kappa} (|p|_{\mathfrak{s}_0}\wedge|q|_{\mathfrak{s}_0})^{\eta- \gamma} 
\end{align*}
As previously we find
\begin{align*}
 &\big|G(f)(q) -\sum_{n=0}^{[\gamma/\zeta]}\trr \frac{1}{n!} j_q D^n G( R\bar{f}_p(\cdot)) \star_\gamma \Gamma_{q,p}\tilde{f}_p^{\star_\gamma n}\big|_\kappa\\
\leq & \big|\sum_{n=0}^{[\gamma/\zeta]} \trr \frac{1}{n!} j_q D^n G( R\bar{f}_q(\cdot)) \star \tilde{f}_q^{\star_\gamma n}\\
&\qquad-\sum_{n=0}^{[\gamma/\zeta]}\sum_{k=0}^{[\gamma/\zeta]-n} \trr \frac{1}{n!} \frac{1}{k!} D^{n+k} j_qG( R\bar{f}_p(\cdot)) \star_\gamma j_q(R\bar{f}_q(\cdot)-R\bar{f}_p(\cdot))^{\otimes k} \star \tilde{f}_q^{\star_\gamma n}\big|_\kappa \\
+& \big| \sum_{n=0}^{[\gamma/\zeta]}\sum_{k=0}^{[\gamma/\zeta]-n}  \trr \frac{1}{n!} \frac{1}{k!} D^{n+k} j_qG( R\bar{f}_p(\cdot)) \star_\gamma j_q(R\bar{f}_q(\cdot)-R\bar{f}_p(\cdot))^{\otimes k} \star \tilde{f}_q^{\star_\gamma n}\\
&\qquad -\sum_{n=0}^{[\gamma/\zeta]} \trr \frac{1}{n!} j_q D^n G( R\bar{f}_p(\cdot)) \star_\gamma \Gamma_{q,p}\tilde{f}_p^{\star_\gamma n}\big|_\kappa 
\end{align*}
As previously, the former summand can be bounded by
$$
\sum_{n=0}^{[\gamma/\zeta]}  \sum_{\alpha+\beta= \kappa} | \bar{f}_q-\Gamma_{q,p}\bar{f}_p |_0^{([\gamma/\zeta]-n+1- \alpha)} | \tilde{f}_q^{\star_\gamma n}|_\beta  \ ,
$$
where for each non-vanishing term
\begin{equation}\label{eq:citable1}
[\gamma/\zeta]-n+1- \alpha\geq \frac{\gamma - \kappa  }{\zeta}.
\end{equation} 
We note that
\begin{equation}\label{eq:citable2}
| \bar{f}_q-\Gamma_{q,p}\bar{f}_p |_0\lesssim d_{\fraks}(p,q)^\zeta (|p|_{\mathfrak{s}_0}\wedge|q|_{\mathfrak{s}_0})^{(\eta- \zeta)\wedge 0}
\end{equation}
We treat the two cases $\eta\geq\zeta$ and $\eta<\zeta$ separately, starting with the former.
Here we find using the crude bound $ | \tilde{f}_q^{\star_\gamma n}|_\beta \lesssim (|p|_{\mathfrak{s}_0}\wedge|q|_{\mathfrak{s}_0})^{(\eta- \beta)\wedge 0}\lesssim (|p|_{\mathfrak{s}_0}\wedge|q|_{\mathfrak{s}_0})^{\eta- \gamma} \ $ that
\begin{align*}
& \sum_{n=0}^{[\gamma/\zeta]}  \sum_{\alpha+\beta= \kappa} | \bar{f}_q-\Gamma_{q,p}\bar{f}_p |_0^{[\gamma/\zeta]-n+1- \alpha} | \tilde{f}_q^{\star_\gamma n}|_\beta  \\
& \lesssim \sum_{n=0}^{[\gamma/\zeta]}  \sum_{\alpha+\beta= \kappa} d_\fraks(p,q)^{\zeta( [\gamma/\zeta]-n+1- \alpha)}  
| \tilde{f}_q^{\star_\gamma n}|_\beta \\
& \lesssim \sum_{n=0}^{[\gamma/\zeta]}  \sum_{\alpha+\beta= \kappa} d_\fraks(p,q)^{\gamma- \kappa }
| \tilde{f}_q^{\star_\gamma n}|_\beta   \\
 & \lesssim  d_\fraks(p,q)^{\gamma- \kappa }(|p|_{\mathfrak{s}_0}\wedge|q|_{\mathfrak{s}_0})^{\eta-\gamma} ,
\end{align*}
where in the first inequality we used \eqref{eq:citable2} with the fact that $(\eta- \zeta)\wedge 0$ and in the second in \eqref{eq:citable1}.
We turn to the case $\eta<\zeta$ where using \eqref{eq:citable2} we find
\begin{align*}
& \sum_{n=0}^{[\gamma/\zeta]}  \sum_{\alpha+\beta= \kappa} | \bar{f}_q-\Gamma_{q,p}\bar{f}_p |_0^{[\gamma/\zeta]-n+1- \alpha} | \tilde{f}_q^{\star_\gamma n}|_\beta  \\
&\lesssim \sum_{n=0}^{[\gamma/\zeta]}  \sum_{\alpha+\beta= \kappa} d_\fraks(p,q)^{\zeta( [\gamma/\zeta]-n+1- \alpha)}  
(|p|_{\mathfrak{s}_0}\wedge|q|_{\mathfrak{s}_0})^{(\eta- \zeta)([\gamma/\zeta]-n+1- \alpha)} 
| \tilde{f}_q^{\star_\gamma n}|_\beta   \\
& =\sum_{n=0}^{[\gamma/\zeta]}  \sum_{\alpha+\beta= \kappa} \left( \frac{d_\fraks(p,q)}{|p|_{\mathfrak{s}_0}\wedge|q|_{\mathfrak{s}_0}}\right) ^{\zeta( [\gamma/\zeta]-n+1- \alpha)}  
(|p|_{\mathfrak{s}_0}\wedge|q|_{\mathfrak{s}_0})^{\eta([\gamma/\zeta]-n+1- \alpha)} 
| \tilde{f}_q^{\star_\gamma n}|_\beta   \\
&\leq \sum_{n=0}^{[\gamma/\zeta]}  \sum_{\alpha+\beta= \kappa} \left( \frac{d_\fraks(p,q)}{|p|_{\mathfrak{s}_0}\wedge|q|_{\mathfrak{s}_0}}\right) ^{\gamma- \kappa}  
(|p|_{\mathfrak{s}_0}\wedge|q|_{\mathfrak{s}_0})^{\eta([\gamma/\zeta]-n+1- \alpha)} 
| \tilde{f}_q^{\star_\gamma n}|_\beta   \\
& \leq\sum_{n=0}^{[\gamma/\zeta]}  \sum_{\alpha+\beta= \kappa} d_\fraks(p,q)^{\gamma- \kappa}  
(|p|_{\mathfrak{s}_0}\wedge|q|_{\mathfrak{s}_0})^{-\gamma+ \kappa} 
| \tilde{f}_q^{\star_\gamma n}|_\beta \ ,  \\
\end{align*}
where in the second last inequality we used \eqref{eq:citable1} together with the fact that $1\geq \frac{ d_\fraks(p,q)}{|p|_{\mathfrak{s}_0}\wedge|q|_{\mathfrak{s}_0}}$ and in the last inequality we again used \eqref{eq:citable1}. Next, note that 
$$ | \tilde{f}_q^{\star_\gamma n}|_\beta \lesssim (|p|_{\mathfrak{s}_0}\wedge|q|_{\mathfrak{s}_0})^{(\eta- \beta)\wedge 0}\ ,$$
thus we conclude 
\begin{align*}
&\sum_{n=0}^{[\gamma/\zeta]}  \sum_{\alpha+\beta= \kappa} d_\fraks(p,q)^{\gamma- \kappa}  
(|p|_{\mathfrak{s}_0}\wedge|q|_{\mathfrak{s}_0})^{-\gamma+ \kappa} 
| \tilde{f}_q^{\star_\gamma n}|_\beta   \\
& \leq\sum_{n=0}^{[\gamma/\zeta]}  \sum_{\alpha+\beta= \kappa} d_\fraks(p,q)^{\gamma- \kappa}  
(|p|_{\mathfrak{s}_0}\wedge|q|_{\mathfrak{s}_0})^{-\gamma+ \kappa} 
(|p|_{\mathfrak{s}_0}\wedge|q|_{\mathfrak{s}_0})^{\eta-\beta}    \\
&\lesssim d_\fraks(p,q)^{\gamma- \kappa}  
(|p|_{\mathfrak{s}_0}\wedge|q|_{\mathfrak{s}_0})^{\eta-\gamma}  \ .
\end{align*}
Next we bound the second term 
$R_{\kappa,p,q}$ for which we find exactly as in the proof of Theorem~\ref{theorem composition}
\begin{align*}
R_{\kappa,p,q}&= \left| \sum_{n=0}^{[\gamma/\zeta]}  \trr \frac{1}{n!} j_q D^n G(R\bar{f}_p(\cdot)) \star_\gamma D_{n,pq}  \right|_\kappa 
\leq  \sum_{n=0}^{[\gamma/\zeta]}  \sum_{\alpha+\beta= \kappa} | j_q D^n G( R\bar{f}_p(\cdot))|_\alpha | D_{n,pq} |_\beta\\
&\lesssim  \sum_{n=0}^{[\gamma/\zeta]}  \sum_{\alpha+\beta= \kappa}  | D_{n,pq} |_\beta
\end{align*}
where we recall that $D_{n,pq}:= Q_{<\gamma}\Gamma_{q,p}\tilde{f}_p^{\star_\gamma n}- (f_q -\Gamma_{q,p}\bar{f}_p)^{\star_\gamma n}$
and note that $$ |D_{n,pq}|_\beta \lesssim d_\fraks (p,q)^{\gamma-\beta} (|p|_{\mathfrak{s}_0}\wedge|q|_{\mathfrak{s}_0})^{\eta-\gamma} \ . $$
This concludes the estimate of the second term. Lastly, to see \eqref{nonlinear difference bound11} one again writes $h= f-g$ and bounds
\begin{align*}
G_\gamma(f)-G_\gamma (g) =\int_0^1 
\sum_n \frac{1}{n!}\trr\big(j_p^E(D^{n+1}G)(\bar{g}_p+t\bar{h}_p ) \star_\gamma (\tilde {g}_p+t\tilde{h}_p)^{\star_\gamma n} \star_\gamma h \big) dt \ .
\end{align*}
\end{proof}

\subsection{Schauder estimate}
The next proposition is the analogue of \cite[Prop. 6.16]{Hai14}. 
\begin{prop}\label{schauder for singular}
In the setting of Theorem~\ref{schauder theorem} let $f\in \mathcal{D}^{\gamma,\eta}(V)$ for $\eta<\gamma$, where $\eta\wedge\alpha>-\mathfrak{s}_0$ and $\eta+\beta\notin \mathbb{N}$, then
$$\mathcal{K}_\gamma f \in \mathcal{D}^{\gamma+\beta, (\eta\wedge\alpha)+\beta} .$$
Furthermore, let $I\subset [-1,1] $ be an interval, then the following bound holds
$$
 \interleave \mathcal{K}_\gamma f ; \bar{\mathcal{K}}_\gamma \bar{f} \interleave_{\gamma+\beta, (\eta\wedge\alpha)+\beta, I\times K} \lesssim \interleave f; \bar f \interleave_{\gamma, \eta, I\times K} + \|Z;\bar Z\|_{(I+[-1,1])\times \bar K} \ ,
$$
where the implicit constant depends on 
$\|Z\|_{(I+[-1,1])\times \bar K},$ $\|\bar{Z}\|_{(I+[-1,1])\times \bar K}, \ \interleave f\interleave_{\gamma, \eta, I\times K} $ and $\interleave\bar{f}\interleave_{\gamma, \eta, I\times K} $.
\end{prop}
\begin{remark}
As in \cite{Hai14} we observe that the condition $\alpha\wedge\eta >-\mathfrak{s}_0$ is only required in order to have a unique reconstruction operator.  In situations where this fails but one needs to make a choice of reconstruction operator $\mathcal{R}$ satisfying the bound (\ref{reconstruction for singular modelled dist}).
\end{remark}
\begin{proof}
The following follows directly as in \cite{Hai14}.
\begin{itemize}
\item The section $\mathcal{N}f$ is well defined for $f\in \mathcal{D}^{\gamma,\eta}(V)$ (outside the set $\{0\} \times M$ ) .
\item The necessary estimates on  $|\mathcal{K}_\gamma f(p)|_\alpha$ , 
$|\mathcal{K}_\gamma f(p)-\bar{\mathcal{K}}_\gamma \bar{f} (p)|_\alpha$ , \linebreak
$|\mathcal{K}_\gamma f(p) -\Gamma_{p,q} \mathcal{K}_\gamma f(q)|_\alpha$, and 
$|\mathcal{K}_\gamma f(p)-\bar{\mathcal{K}}_\gamma \bar{f}   -\Gamma_{p,q} \mathcal{K}_\gamma f(q)+\bar{\Gamma}_{p,q} \bar{\mathcal{K}}_\gamma \bar f(q)|_\alpha$ for
 $\alpha\in A\setminus \mathbb{N}$ .
\item The necessary bound on $|\mathcal{K}_\gamma f(p)|_k$ and $|\mathcal{K}_\gamma f(p)-\bar{\mathcal{K}}_\gamma \bar{f}(p)|_k$ for $k\in \mathbb{N}$ .
\end{itemize}
The place where the proof in our setting varies from the one in \cite{Hai14} is in estimating $|\mathcal{K}_\gamma f (p)-\Gamma_{p,q} \mathcal{K}_\gamma f(q)|_k$  and $|\mathcal{K}_\gamma f(p)-\bar{\mathcal{K}}_\gamma \bar{f}(p)   -\Gamma_{p,q} \mathcal{K}_\gamma f(q)+\bar \Gamma_{p,q} \bar{\mathcal{K}}_\gamma \bar f(q)|_k$ for $k\in \mathbb{N}$. We only estimate the former, one proceeds analogously for the latter.

To estimate $|\mathcal{K}^{(n)}_\gamma f(p) -\Gamma_{p,q} \mathcal{K}^{(n)}_\gamma f(q)|_k$ for $p,q \in K_{\mathfrak{t}}$ we proceed differently in the three cases
\begin{enumerate}
\item $2^{-n} \leq d_\fraks(p,q)$,
\item $2^{-n}\in [d_\fraks(p,q) ,\frac{1}{2} |p|_{\fraks_0}\wedge |q|_{\fraks_0}]$,
\item $2^{-n}\geq \frac{1}{2} |p|_{\fraks_0}\wedge |q|_{\fraks_0}$.
\end{enumerate}

In the case $2^{-n} \leq d_\fraks(p,q)$ one proceeds exactly as in the proof of Theorem~\ref{schauder theorem} to obtain
$$\sum_{n:  2^{-n} \leq d_\fraks(p,q)} |\mathcal{K}^{(n)}_\gamma f(p) -\Gamma_{p,q} \mathcal{K}^{(n)}_\gamma f(q)|_k \lesssim d_\fraks(p,q)^{\gamma+\beta -k} (|p|_{\fraks_0}\wedge |q|_{\fraks_0})^{\eta-\gamma} \ .$$

In the case $2^{-n}\in [d_\fraks(p,q) ,\frac{1}{2} |p|_{\fraks_0}\wedge |q|_{\fraks_0}]$, one proceeds as in the second step of the proof of Theorem~\ref{schauder theorem} and it suffices to bound 
\begin{equation}\label{some term}
\sup_{\deg (\boldsymbol{\partial}) =k} \big| \big(\Pi_q f(q)-\mathcal{R}f\big)(K^{\boldsymbol{\partial},\gamma}_{n,p,q})\big|
\end{equation}
and
\begin{equation}\label{some term2}
\sup_{\deg (\boldsymbol{\partial}) =k} \sum_{\zeta\leq k-\beta} |\Pi_p Q_\zeta (f(p)-\Gamma_{p,q} f(q))(\nabla_{\boldsymbol{\partial}} K_n(p,\cdot))| \ ,
\end{equation}
see Equation~(\ref{two terms to bound}). 
To bound \eqref{some term}, we proceed as in (\ref{something to reuse}) but using Lemma~\ref{local reconstruction theorem} to obtain
$$\sum_{d(p,q)\leq 2^{-n} \leq \frac{|p|_{\fraks_0}\wedge |q|_{\fraks_0}}{2 }}|(\Pi_q f(q)-\mathcal{R}f\big)(K^{\boldsymbol{\partial},\gamma}_{n,p,q})|
\lesssim (|p|_{\fraks_0}\wedge |q|_{\fraks_0})^{\eta-\gamma} (d(p,q)^{\gamma+\beta -k} +d(p,q)^{\rho-k})\ $$
for $\rho>\gamma+\beta$. 
For the second term \eqref{some term2} one obtains the desired bound directly.
We turn to the last case $2^{-n}\geq \frac{1}{2} |p|_{\fraks_0}\wedge |q|_{\fraks_0}$. Here we proceed as in the previous step, except for the term (\ref{some term}), which we estimate as
$$\big| \big(\Pi_q f(q)-\mathcal{R}f\big)(K^{\boldsymbol{\partial},\gamma}_{n,p,q})\big|\leq \big| \Pi_q f(q)(K^{\boldsymbol{\partial},\gamma}_{n,p,q})\big|+\big| \mathcal{R}f (K^{\boldsymbol{\partial},\gamma}_{n,p,q})\big|\ .$$
Lemma~\ref{technical Lemma for Kernels} implies that
\begin{align*}
\big| \Pi_q f(q)(K^{\boldsymbol{\partial},\gamma}_{n,p,q})\big|&\lesssim \sum_{\tilde{l}\in \delta A_{\gamma}} 2^{(|\tilde{l}|_\fraks+\deg(\boldsymbol{\partial})-\eta\wedge\alpha-\beta)n } d(p,q)^{|\tilde{l}|_\fraks} \\
&\quad +   \sum_{m<\gamma+\beta} d_\fraks(p,q)^{\rho- \deg(\boldsymbol{\partial})} 2^{-n(\eta\wedge\alpha+\beta-m)} \ .
\end{align*}
For $\tilde{l}\in \delta A_{\gamma}$ and $\deg(\boldsymbol{\partial})=k$ one finds
\begin{align*}
\sum_{n\ :\  2^{-n}\geq \frac{1}{2} |p|_{\fraks_0}\wedge |q|_{\fraks_0}} &2^{(|\tilde{l}|_\fraks+k-\eta\wedge\alpha-\beta)n } d_\fraks(p,q)^{|\tilde{l}|_\fraks} \\
&\lesssim (|p|_{\fraks_0}\wedge |q|_{\fraks_0})^{-(|\tilde{l}|_\fraks+k-\eta\wedge\alpha-\beta)}d(p,q)^{|\tilde{l}|_\fraks} \\
&=(|p|_{\fraks_0}\wedge |q|_{\fraks_0})^{(\eta\wedge\alpha)+\beta-k} \big(\frac{d_\fraks(p,q)}{|p|_{\fraks_0}\wedge |q|_{\fraks_0}} \big)^{|\tilde{l}|_\fraks}\\
&\lesssim(|p|_{\fraks_0}\wedge |q|_{\fraks_0})^{(\eta\wedge\alpha)+\beta-k} \big(\frac{d_\fraks(p,q)}{|p|_{\fraks_0}\wedge |q|_{\fraks_0}} \big)^{\gamma+\beta-k}\\
&=(|p|_{\fraks_0}\wedge |q|_{\fraks_0})^{\eta\wedge\alpha-\gamma} d_\fraks(p,q)^{\gamma+\beta-k}\ ,
\end{align*}
where in the last inequality we used that $d_\fraks(p,q)\leq |p|_{\fraks_0}\wedge |q|_{\fraks_0}$ and $|\tilde{l}|_{\fraks}>\gamma+\beta-k$.
For $m<\gamma+\beta$ and $\rho> \gamma+\beta$ one has
\begin{align*}
\sum_{n: 2^{-n}\geq \frac{1}{2} |p|_{\fraks_0}\wedge |q|_{\fraks_0}} d_\fraks(p,q)^{\rho- k} 2^{-n(\eta\wedge\alpha+\beta-m)}
\lesssim d_\fraks(p,q)^{\gamma+\beta- k}(|p|_{\fraks_0}\wedge |q|_{\fraks_0})^{(\eta\wedge\alpha-\gamma)} 
\end{align*}
and thus for $\deg(\boldsymbol{\partial})=k$
$$\sum_{n: 2^{-n}\geq \frac{1}{2} |p|_{\fraks_0}\wedge |q|_{\fraks_0}}\big| \Pi_q f(q)(K^{\boldsymbol{\partial},\gamma}_{n,p,q})\big|\lesssim  (|p|_{\fraks_0}\wedge |q|_{\fraks_0})^{\eta\wedge\alpha-\gamma} d(p,q)^{\gamma+\beta-k} \ .$$
The same way one bounds $\big| \mathcal{R}f (K^{\boldsymbol{\partial},\gamma}_{n,p,q})\big|$ and therefore concludes that
$$\sum_{n: 2^{-n}\geq \frac{1}{2} |p|_{\fraks_0}\wedge |q|_{\fraks_0}} \big| \big(\Pi_q f(q)-\mathcal{R}f\big)(K^{l,\gamma}_{n,p,q})\big|\lesssim (|p|_{\fraks_0}\wedge |q|_{\fraks_0})^{\eta\wedge\alpha-\gamma} d(p,q)^{\gamma+\beta-k} \ .$$
\end{proof}

\section{Solutions to Semi-linear Equations on Manifolds}\label{section solution to semilinear...}
In this section we provide the remaining ingredients in order to generalise the fixed point theorem in \cite[Section~7.3]{Hai14} to the geometric setting of this article. From now on we work under the following assumption.
\begin{assumption}
We assume that the manifold $\bar{M}$ is compact.
\end{assumption} 
Since the equations we shall consider will always be in integral form, we proceed by working with kernels. A discussion about the existence of such kernels for heat type differential operators can be found in Section~\ref{section:excoursion}.

\begin{assumption}\label{Assumption kernel decomposition}
Let $G\in \cC^\infty(F\hotimes E^* \setminus \pi^{-1}\triangle)$ be an integral kernel and $\beta > 0$. We assume 
there exist 
$K\in \cC^\infty(F\hotimes E^* \setminus \pi^{-1}\triangle)$ and $K_{-1}\in \cC^\infty(F\hotimes E^*)$ such that
$G=K+ K_{-1}$ 
and the following properties are satisfied.
\begin{itemize}
\item $K$ satisfies Assumption~\ref{Assumption on Kernel} for $\beta$,
\item both, $K$ and $K_{-1}$ are non-anticipative, i.e.\ 
$$t<s \Rightarrow K((t,p), (s,q)) = 0 \ \text{ and } K_{-1}((t,p), (s,q)) = 0 \text{ for all } p,q\in M\ . $$
\end{itemize}
\end{assumption}

\subsection{Short-time behaviour}
We define $M_T= [-1, T]\times M$ and as in \cite[Section 7.1]{Hai14} the operator $$\mathbf{R}^+:\mathcal{D}^{\gamma,{\eta} }_\mathfrak{t}\to \mathcal{D}^{\gamma,{\eta} }_\mathfrak{t}$$
given by multiplication with the indicator function of $\mathbb{R}_+$. 
The following is the analogue of \cite[Theorem 7.1]{Hai14} in our setting.
\begin{theorem}\label{theorem short time}
Let $K$ be a Kernel satisfying Assumption~\ref{Assumption kernel decomposition} for some $\beta>0$. Assume furthermore that the regularity structure $\mathcal{T}$ comes with an integration map $I$ of order $\beta$ and precision $\delta_0$ acting on a sector $V$ of regularity $\alpha>-\mathfrak{s}_0$ and assume that the models $Z=(\Pi,\Gamma)$ and $\bar{Z}=(\bar\Pi,\bar\Gamma)$ realise $K$ for $I$ on $V$.
Let $\gamma<\delta_0-\beta$  

Then, for every $T\in (0,1]$ uniformly over models $Z, \bar{Z}$ realising $K$ for $I$ on $V$ and $f\in \mathcal{D}^{\gamma,\eta}(V)$, 
 $\bar{f}\in \bar{\mathcal{D}}^{\gamma,\eta}(V)$ satisfying 
 $\interleave  f\interleave_{\gamma,\eta ; M_T}, \interleave  \bar f\interleave_{\gamma,\eta; M_T} \leq C$ :
\begin{equation}\label{equatin short 1}
\interleave  \mathcal{K}_\gamma  \mathbf{R}^+ f\interleave_{{\gamma+\beta,\bar{\eta} } ;M_T }
\lesssim_C T^{\frac{\kappa}{\mathfrak{s}_0}} \interleave  f\interleave_{\gamma,\eta ; M_T}
\end{equation}
\begin{equation}\label{equation short 2}
\interleave  \mathcal{K}_\gamma \mathbf{R}^+ f; \bar{\mathcal{K}}_\gamma \mathbf{R}^+ \bar{f} \interleave_{\gamma+\beta,\bar{\eta} ;M_T}
\lesssim_C T^{\frac{\kappa}{\mathfrak{s}_0}} \big(\interleave  f; \bar{f}\interleave_{\gamma,\eta ; M_T}+\|Z, \bar Z\|_{M_{T+1}}\big),
\end{equation}
where $\bar{\eta}=(\eta\wedge\alpha) +\beta -\kappa$ and $C>0$ depends on $\|Z\|_{M_{T+1}}, \|\bar{Z}\|_{M_{T+1}}\leq C$ in the first inequality and additionally $ \interleave  f\interleave_{\gamma,\eta ; M_T}$, $ \interleave  \bar{f}\interleave_{\gamma,\eta ; M_T}\leq C$ in the second inequality.
\end{theorem}

\begin{proof}
For $\zeta\in A$ such that $\zeta\geq (\alpha\wedge \eta)+\beta$ the estimates on $Q_{\zeta}(\mathcal{K}_\gamma \mathbf{R}^+ f)$ follow as in the proof of \cite[Theorem 7.1]{Hai14} using Proposition~\ref{schauder for singular}, the local version of the Reconstruction theorem in Proposition~\ref{local reconstruction theorem} and non-anticipativity of the kernel, together with the definition of the semi-norm $\interleave\cdot \interleave_{\mathcal{D}^{\gamma,\eta}_\mathfrak{t}}$ .

Thus, it remains to show the estimate for $k\in \mathbb{N}$ such that $k<(\alpha\wedge \eta)+\beta$. Observes that
$$ Q_k (\mathcal{K}_\gamma  \mathbf{R}^+ f)(z)= \sum_n Q_k j_z K_n(\mathcal{R} \mathbf{R}^+ f) ,$$
which follows by unravelling the definition of $J_z (\mathbf{R}^+ f)$ and $\mathcal{N}(\mathbf{R}^+ f)(z)$.
Since the operator 
$$ K: \mathcal{C}^{\alpha\wedge\eta} \to \mathcal{C}^{(\alpha\wedge\eta)+\beta} , \qquad F\mapsto K F= \sum_n K_n(F)$$
 is well defined and since $\mathcal{K}_\gamma  \mathbf{R}^+ f(t,p)=0$ for $t<0$ (as $K$ is assumed non-anticipative), one observes that $\mathcal{K}_\gamma  \mathbf{R}^+ f$ satisfies the assumptions of Lemma~\ref{two norms lemma exact case} for $\epsilon \in (0,(\alpha\wedge\eta)+\beta )$. Therefore 
$$\talloblong \mathcal{K}_\gamma \mathbf{R}^+ f \talloblong_{{\gamma,(\eta\wedge\alpha) +\beta} ; M_T}\lesssim \interleave \mathcal{K}_\gamma \mathbf{R}^+ f \interleave_{{\gamma,(\eta\wedge\alpha) +\beta} ; M_T} ,$$
which together with Proposition~\ref{schauder for singular} implies (\ref{equatin short 1}) . One obtains (\ref{equation short 2}) in a similar way.
\end{proof}
In Assumption~\ref{Assumption kernel decomposition} we decomposed $G= K+K_{-1}$, where $K_{-1}$ is a smooth kernel. 
For any $\alpha,\gamma >0$ and $\eta>-\fraks_{0}$, we extend $K_{-1}$ to an operator on modelled distributions $$\mathcal{K}_{-1} :\mathcal{D}^{\alpha, \eta}\mapsto \mathcal{D}^{\gamma}(\bar{T}^{\mathfrak{f}}) \ \text{ as } \ f\mapsto Q_{<\gamma} j K_{-1}(\mathcal{R}f) \ .$$
The next lemma follows as \cite[Lemma 7.3]{Hai14}.
\begin{lemma}
Let $K_{-1}$ be as in Assumption~\ref{Assumption kernel decomposition}, then in the setting of Theorem~\ref{theorem short time} the following estimates hold
\begin{equation}
\interleave  \mathcal{K}_{-1}  \mathbf{R}^+ f\interleave_{\mathcal{D}^{\gamma+\beta,\bar{\eta} }_\mathfrak{t}(M_T)}
\lesssim_C T\interleave  f\interleave_{\mathcal{D}^{\gamma,\eta}_\mathfrak{t}(M_T)}
\end{equation}
\begin{equation}
\interleave  \mathcal{K}_{-1} \mathbf{R}^+ f; \mathcal{K}_{-1} \mathbf{R}^+ \bar{f} \interleave_{\mathcal{D}^{\gamma+\beta,\bar{\eta} }_\mathfrak{t}(M_T)}
\lesssim_C T \big(\interleave  f; \bar{f}\interleave_{\mathcal{D}^{\gamma,\eta}_\mathfrak{t}(M_T)} +\|Z, \bar Z\|_{M_{T+1}}\big) \ .
\end{equation}
\end{lemma}

\subsection{Initial condition}
In order to treat initial condition, we denote by $\bar{E}$ the vector bundle over $\bar{M}$ obtained by restriction $E|_{\{0\}\times \bar{M}}$. 

\begin{lemma}
For $\eta>0$ such that $\beta+\eta\notin \mathbb{N}$ and
 $u_0\in \mathcal{C}^\eta(\bar{E})$ and $G\in \cC^\infty(F\hotimes E^* \setminus \pi^{-1}\triangle)$ satisfying Assumption~\ref{Assumption kernel decomposition} 
We set for $t\neq 0$ 
$$G(u_0)(t,p)= \int G((t,p),(0,q)) u_0(q) \, \dVol_q\ ,$$
where we used suggestive but informal notation if $\eta\leq 0$.
Then, for any $\gamma>(\eta+\beta-\fraks_0) \vee 0$ 
\begin{equation}\label{equation inition}
Q_{<\gamma}j_{(\cdot )} G(u_0)\in \mathcal{D}_\mathfrak{t}^{\gamma, (\eta+\beta- \mathfrak{s}_0)\wedge 0 }\ .
\end{equation}
\end{lemma} 

\begin{proof}
Observe that by Assumption $G=K+K_{-1}$, where $K_{-1}$ is a smooth kernel and $K= \sum_n K_n$ satisfies the properties in Assumption~\ref{Assumption on Kernel}. Thus it is easy to see that $G(u_0)$ is smooth away from the origin and quantitative bounds on $K_{-1}(u_0)$ are obvious.
Thus we focus on the term $K(u_0)$.
For $t>0$ let $n_t\in \mathbb{N}$ be such that $|t|_\fraks\in (2^{-n_t-1}, 2^{-n_t})$, then due to the support constraints on $K_n$
$$ K(u_0) =\sum_{n=0}^\infty v (K_n((t,p)(0,\cdot)) = \sum_{n=0}^{n_t} v (K_n((t,p)(0,\cdot)) $$ and thus since $\eta+\beta \notin \mathbb{N}$ and $\fraks_0\in \mathbb{N}$
$$|K(u_0)(t,p) |\lesssim \sum_{n=0}^{n_t} 2^{-n\eta+n(\fraks_0-\beta)} = \sum_{n=0}^{n_t} 2^{n(\fraks_0-\beta-\eta)}\lesssim |t|_\fraks^{(\eta +\beta-\fraks_0)\wedge 0} \ .$$
The bound 
$|j_{(t,p)} K(u_0)|_m\lesssim |t|_\fraks^{(\eta +\beta-\fraks_0-m)\wedge 0}$ follows similarly.  
This establishes the bound $\| G(u_0) \|_{\gamma, \eta+\beta-\mathfrak{s}_0; M_T}<+\infty$.

The bound $\interleave G(u_0) \interleave _{\gamma, \eta+\beta-\mathfrak{s}_0; M_T}<+\infty$ follows by going to an exponential $\fraks$-chart and applying the Taylor estimate \cite[Proposition A.1]{Hai14} or equivalently applying \cite[Corollary~69]{DDD19} (or strictly speaking the slight extension thereof to the setting of non-trivial scalings.) 
\end{proof}
\begin{remark}
In applications kernels usually satisfy $\beta=\fraks_0$, see for example Section~\ref{section:excoursion} as well as \cite[Section~7.2]{Hai14}. 
\end{remark}

\subsection{Fixed point theorem}
Finally, we have generalised all necessary results of \cite{Hai14} in order to extend the general fixed point theorem \cite[Theorem 7.8]{Hai14} to our geometric setting. 
The proof therein adapts \textit{mutatis mutandis}. 

\begin{theorem}\label{thm:abstract_fixed point}
Let $V, \bar{V}$ be two sectors of regularity $\zeta, \bar{\zeta}$ and precision $\delta_0$ of a regularity structure ensemble $\mathcal{T}= (\{E^\mathfrak{t}\}_{\mathfrak{t}\in \mathfrak{L}},\{T^\mathfrak{t}\}_{\mathfrak{t}\in \mathfrak{L}}, L )$ and let 
$F: V \to \bar{V}$ be a local non-linearity such that for $f\in \mathcal{D}^{\gamma,\eta}(V)$ one has $F(f)\in\mathcal{D}^{\bar{\gamma},\bar{\eta}}(\bar{V})$.
Assume further that we are given a $\beta$ regularising abstract integration map $I:V\to \bar{V}$ of precision $\delta_0>0$.

If $\eta<(\bar{\eta}\wedge\bar{\zeta})+\beta$, $\gamma<\bar{\gamma}+\beta$, $(\bar{\eta}\wedge\bar{\zeta})>-\beta$ and $\delta_0>\gamma$ as well as $F$ is strongly locally Lipschitz, then for every $v\in \mathcal{D}^{\gamma,\eta}$ and every model $Z=(\Pi,\Gamma)$ for $\mathcal{T}$ that realises $K$ for $I$, there exists $T>0$ such that the fixed point equation
$$U= (\mathcal{K}+ \mathcal{K}_{-1} ) \mathbf{R}^+ F(U) + v ,$$
admits a unique solutions $U\in \mathcal{D}^{\gamma,\eta}$ on $M_T$. Furthermore the solution map 
$(v,Z)\mapsto U$ is locally jointly Lipschitz continuous. 
\end{theorem}
\begin{remark}
Note that the assumption $\bar{\eta}\wedge\bar{\zeta}>-\beta$ is imposed in order to apply Proposition~\ref{schauder for singular}. This can sometimes be relaxed by splitting the solution into two parts $U= \tilde{U}+(\mathcal{K}+ \mathcal{K}_{-1} )\tilde{v}$, where $\tilde{v}\in \mathcal{D}^\infty$ and studying the fixed point problem
$$\tilde{U}=  (\mathcal{K}+ \mathcal{K}_{-1} ) \mathbf{R}^+ F( \tilde{U}+(\mathcal{K}+ \mathcal{K}_{-1} )\tilde{v} ) +(v -(\mathcal{K}+ \mathcal{K}_{-1} )\tilde{v}) \ .$$
This can be seen as applying a lifted version of the ``Da Prato Debussche'' trick, c.f.\ \cite[Rem.~7.9]{Hai14} as well as \cite[Section~5.5]{BCCH20}.
\end{remark}

\section{Excursion on Differential Operators, Greens Functions and Heat Kernels}\label{section:excoursion}
Recall that in this article we are concerned with developing a solution theory for singular SPDEs, typically\footnote{Parabolic systems, c.f.\ \cite{Fri08}, also fall within the scope of the theory developed here.} of the form 
\begin{equation}
\partial_t u + \mathcal{L}u = \sum_{i=0}^m G_i(u, \nabla u,\ldots, \nabla^n u)\xi_i\ ,
\end{equation}
see \eqref{equation to solve}. In this section, we present a (non-exhaustive) class of differential operators $\mathcal{L}\in \mathfrak{Dif}_{2k}(E,E)$ to which the theory applies. 
In particular we recall what it means for $\mathcal{L}$ to be uniformly elliptic. 
\subsection{Trivial scaling}
First, let us assume that the scaling on $\bar{M}$ is the trivial scaling $\bar{s}=(1)$, c.f.\ Definition~\ref{def scaling on manifold}. 
We denote by $\pi_{T^*M}: T^*M\to M$ the canonical projection.
The \textit{symbol} of a $k$-th order differential operator $\mathcal{A}: \mathcal{C}^\infty (E)\to \mathcal{C}^\infty (F)$ is a section 
$ \sigma_k(\mathcal{A}) $ of the bundle $\pi_{T^*M}^*(L(E,F))$
determined as follows. For $v_p\in T_p^*M$ and $e_p\in E|_p$ let $g\in \mathcal{C}^\infty (M)$ be such that 
$f(p)=0$ and $v_p= dg(p)$ let $e\in  \mathcal{C}^\infty (E)$ be a section $e(p)=e_p$, then,  
$$\sigma_k(\mathcal{A})(v_p)(e_p):= \frac{1}{k!} \mathcal{A}(g^{k}e)(p) \in F|_p\ .$$
We denote by $\mathfrak{Sym}_k(E,F):= \left\{\sigma_k(\mathcal{A})\ : \mathcal{A}\in \mathfrak{Dif}_k(E,F)\right\}  $ and one can observe that 
$\sigma_k(\mathcal{A})$ is actually a polynomial on $T^*M$ and can thus canonically be identified with a section of $(T^*M)^{\otimes_s k}\otimes L(E,F)$.
\begin{remark}
The symbol of a differential operator captures the behaviour of the differential operator at leading order. This can for example be seen by observing that the following sequence of vector spaces is exact, c.f.\ \cite[Chapter~3, Section~IV]{Pal65},
\[ \begin{tikzcd}
0  \arrow{r} & \mathfrak{Dif}_{k-1}(E,F) \arrow{r} & \mathfrak{Dif}_{k}(E,F) \arrow{r}{\sigma} &\mathfrak{Sym}_{k}(E,F)  \arrow{r} & 0 \ ,
\end{tikzcd}
\]
where the second arrow denotes the inclusion map.
\end{remark}

Let us recall some important classes of differential operators characterised by their symbol:
\begin{enumerate}
\item A differential operator $\mathcal{A}\in \mathfrak{Dif}_{k}(E,F)$ is called \textit{elliptic} if for every $v_p\in T^*M\setminus\{0\}$ the map 
$\sigma_k(\mathcal{A})(v_p)\in L(E|_p,F|_p)$ is invertible. 
\item A differential operator $\mathfrak{L}\in \mathfrak{Dif}_{2k}(E,E)$ is called \textit{strongly elliptic}
if 
$(-1)^k \sigma_{2k}(\mathfrak{L})(v_p)$ is positive definite whenever $|v_p|\neq 0$. 
\item $\mathfrak{L}\in \mathfrak{Dif}_{2k}(E,E)$ is called \textit{uniformly elliptic}\footnote{Note that if $M$ is compact, every strongly elliptic operator is uniformly elliptic since then $\left\{v\in T^*M \ : \ |v|=1 \right\}$ is compact.}  if there exists $\epsilon>0$, such that 
$$(-1)^k \sigma_{2k}(\mathfrak{L})(v_p) \geq \epsilon  |v_p|^{2k} \id_{E_p}\ .$$
\end{enumerate}
\begin{remark}
Note that, since the symbol is a polynomial map on $T^*M$, the notion of strong ellipticity only makes sense for differential operators of even order.
\end{remark}
It is well understood that whenever 
$\mathcal{L}\in \mathfrak{Dif}_{2k}(E,E)$ is uniformly elliptic 
and $\bar{M}$ is compact, the operator 
$$
\partial_t + \mathcal{L}
$$
admits a unique fundamental solution $G$ satisfying Assumption~\ref{Assumption kernel decomposition} on $M= \mathbb{R}\times \bar{M}$ equipped with the scaling $\fraks=(2k,1)$ and $\beta= 2k$. We refer e.g.\ to \cite{Ejd94}, \cite{Fri08}, \cite{BG13} and \cite{Dav95} where this is proved in Euclidean space (for more general operators), but the construction can be generalised to compact manifolds along the lines of \cite{Kot16}, \cite{Alb17}, \cite{BGV03}.
In particular, since $\mathcal{L}$ does not depend on $t\in \mathbb{R}$ the fundamental solution is of the form 
$G((t,p),(s,q))= G_{t-s}(p,q)$.

Recall that our vector bundles are always equipped with a connection. In this case a particularly important uniformly elliptic operator is given by the connection Laplacian $\triangle^E\in \mathfrak{Dif}_{2}(E,E)$
\begin{equation}\label{eq:connection_laplacian}
\triangle^E: \cC^\infty(E)\to \cC^\infty(E), \qquad f\mapsto  - \tr_g (\nabla^{(T^*M\otimes E)} \nabla^E f)\ ,
\end{equation}
where $\tr_g: \cC^\infty(T^*M\otimes T^*M \otimes E)\to \cC^\infty(E)$ denotes the contraction with the metric on $T^*M$.
It is straightforward to check that the symbol of this operator is given by 
\begin{equation}\label{eq:laplace_symbol}
\sigma_2(\triangle^E)(v_p)= -|v_p|^2 \id_{E_p}.
\end{equation}

\begin{remark}\label{rem:generalised_laplacian}
An important class of second order uniformly elliptic operators on $E$ are given by generalised Laplacians, i.e.\ operators 
whose symbol symbol is given by
$ - |v_p|^2 \id_{E_p} \ .$
Thus, Equation~\ref{eq:laplace_symbol} states that the operator $\triangle^E$ is a generalised Laplacian. It turns out that for any generalised Laplacian $\triangle$ on $E$ there exists a connection $\nabla^E$ on $E$ such that $\triangle-\triangle^E\in \mathfrak{Dif}_0(E,E)\simeq L(E,E)$, c.f.\ \cite[Prop.~2.5]{BGV03}.
\end{remark}

\begin{remark}
The framework developed in this article also allows to study equations, where the differential operator is time dependent. This is for example of interest when the underlying manifold evolves in time, c.f.\ \cite{EHS12}. 
\end{remark}
While understanding the qualitative behaviour of a heat kernels is sufficient to apply the theory developed so far in this article, in order to construct the appropriate renormalisation procedure it is often necessary to have a detailed understanding of its behaviour near the diagonal. The next theorem, \cite[Theorem~2.30]{BGV03}, captures this for the Laplacian $\triangle^E$.
\begin{theorem}\label{thm:heat_kernel_expansion}
Let $\bar{M}$ be compact and the vector bundle $\bar{E}\to \bar{M}$ equipped with a metric and connection. Let $E = \pi^* \bar{E}$ be the pullback bundle under the submersion
$$ \pi: M= \mathbb{R}\times\bar{M} \to \bar{M}, \qquad (t,p)\mapsto p\ .$$
Let $G_t $ be the heat kernel associated to the Laplacian $\triangle^{\bar{E}}$, there exist smooth sections $\Phi_i\in \cC^\infty(\bar{E}\hotimes \bar{E}^*)$ supported on $\{(p,q)\in M\times M \ | d(p,q)<\frac{r_p}{2}  \}$ such that the heat kernel $G_t(p,q)$ satisfies for $t>0$
$$\big\|\partial^k_t \big(G_t(p,q) -  \frac{1}{(4\pi t)^{-d/2} }e^{-d(p,q)/4t} \sum_i^N t^i \Phi_i(p,q)\big)\big\|_{C^l(\bar M\times \bar M)} \lesssim t^{N-d/2-l/2-k} \ .$$
Furthermore, the leading term $\Phi_0$ is given by $\Phi_0(p,q)e_q= e_q (p)$, where for $e_q\in E_q$ the section $e_q(\cdot)$ is given by parallel transport along the geodesic from $q$.
\end{theorem}
\begin{remark}
The sections $\Phi_i\in \cC^\infty(\bar{E}\hotimes \bar{E}^*)$ can be explicitly calculated, c.f.\ the proof of \cite[Theorem~2.26]{BGV03}. In particular one finds that $\Phi_1(p,p)= \frac{1}{3} s_{\bar{M}}(p)\id_{E|_p}$, where $s_{\bar{M}}(p)$ denotes the scalar curvature of $\bar{M}$, see \cite[Example~2.27]{BGV03}.
\end{remark}
\begin{remark}
Results similar to Theorem~\ref{thm:heat_kernel_expansion} also hold for more general elliptic operators, c.f.\ \cite{Tin82} and \cite[Section~2.7]{BGV03}.
\end{remark}
\begin{remark}
Fractional Laplacians,\footnote{Note that we work with the sign convention that makes the Laplacian a positive operator.} say scalar valued, $\triangle^{s}$ for $s\in (0,1)$, fall slightly outside of the scope of Assumption~\ref{Assumption kernel decomposition} since the regularity condition in general fails on $\{t=0\}$. We refer to \cite{BH21} on how to accommodate such kernels.
\end{remark}

\begin{remark}\label{rem:change_scaling}
Note that if $\mathbb{R}\times M$ is equipped with a scaling $\fraks$ and $G$ satisfies Assumption~\ref{Assumption kernel decomposition} for $\beta$, then, for every $k\in \mathbb{N}$ $G$ satisfies the same assumption for ${\beta}':=k\beta$ if we equip $\mathbb{R}\times M$ with the scaling $\fraks' = k\fraks$ for the new decomposition given by the kernels $K'_n=\sum_{i=0}^{k-1} K_{kn+i}$.
Indeed, it is straightforward to check all assumptions.
\end{remark}

\subsection{Non-trivial scaling}\label{subsection:non-trivial_scalings}
Suppose $\bar{M}= M^1\times\ldots\times M^n$, $E^i\to M^i$ are vector bundles and let $E:=\hat{\bigotimes}_{i=1}^n E^i $. 
For $\mathcal{L}_i\in \mathfrak{Dif}_{2k_i}(E^i,E^i)$ let
$\mathcal{L} \in \mathfrak{Dif}(E,E)$ be the differential operator uniquely characterised by the fact that for all  $u_i\in \cC^\infty(E_i)$
\begin{equation}\label{class of differential operators}
 \mathcal{L} (u_1 \hotimes\ldots\hotimes u_n)= \sum_{i=1}^n  u_1\hotimes\ldots\hotimes\mathcal{L}_i u_i\hotimes\ldots \hotimes u_n\ . 
 \end{equation}
Note that it follows directly form Definition~\ref{definition:differential_operator} and Proposition~\ref{prop relation to cotangent} that $\mathcal{L} $ is well defined. 
Now, it is straightforward to check that if $G^i_t(\cdot,\cdot)$ is a fundamental solution of $\partial_t+\mathcal{L}_i$, then 
$$G_t(p,q)= \hat{\bigotimes}_{i=1}^n G^i_t(p_i, q_i)$$
is a fundamental solution of $\partial_t+\mathcal{L}$. 

Next, let $k$ be the least common multiple of $k_1,\ldots,k_n$ and set $\fraks_0= 2k$ and $\fraks_i: =k/k_i$.
If we choose the scaling ${\mathfrak{s}}=(\fraks_0,\fraks_1,\ldots,\fraks_n)$ on $M=\mathbb{R} \times M^1\times\ldots\times M^n$, then one sees that the differential operator $\mathcal{L}$ has scaled order $2k$. 
For what follows we work with the (equivalent) distance
$$d_\mathfrak{s}(z,\bar{z}):=\sup_{i\in \{1,\ldots,n\}} d^i(x_i,\bar{x}_i)^{\frac{1}{\mathfrak{s}_i}} \vee |t-\bar{t}|^{{\frac{1}{\mathfrak{s}_0}}}\ ,$$
on $M= \mathbb{R}\times  M^1\times\ldots\times M^n$, where $z=(t,x_1,\ldots,x_n)$ and $\bar{z}=(\bar t ,\bar{x}_1,\ldots,\bar{x}_n)$.
\begin{prop}
In the setting above, if each $G^i$ satisfies Assumption~\ref{Assumption kernel decomposition} for $\beta=2k_i$ on $\RR\times M^i$ equipped with the scaling $(2k_i,1)$, then 
 $G$ satisfies Assumption~\ref{Assumption kernel decomposition} with $\beta=2k$ on $\RR\times \bar{M}$ equipped with the scaling $\mathfrak{s}$.
\end{prop}

\begin{proof}
First, in view of Remark~\ref{rem:change_scaling} we assume that each $G^i$ satisfies the Assumption~\ref{Assumption kernel decomposition} on 
$\mathbb{R}\times M^i$ for the scaling $(2k, k/k_i)$ and $\beta= 2k$. 
We show the claim for $n=2$, the general case then follows by repeated application of the same argument. Furthermore, since the non-anticipativity property follows directly, it only remains to check Assumption~\ref{Assumption on Kernel}.
For $K^1 = \sum_{k\geq 0} K^1_k$ and $K^2 = \sum_{k\geq 0} K^2_k$ as in Assumption~\ref{Assumption on Kernel}, we write
\begin{equation}\label{three terms}
K_m = K^1_m K^2_m + K_m^{>}+ K_m^{<}\  ,
\end{equation}
where 
$K_m^{>}(z,\bar{z}) = \sum_{n>m} K^1_n(z,\bar{z}) K^2_m(z,\bar{z}) $ and $K_m^{<}(z,\bar{z})=\sum_{n>m} K^1_m(z,\bar{z}) K^2_n(z,\bar{z})$.
We shall freely use the notation $z=(t,x_1,x_2)$ and $\bar{z}=(\bar{t}, \bar{x}_1, \bar{x}_2)$, where $x_1, \bar{x}_1\in M^1$ and $x_2, \bar{x}_2\in M^2$.
Clearly the first summand in \eqref{three terms} satisfies the desired bounds uniformly in $m$.
We turn to $K_m^{>}(z,\bar{z})$ and observe that for each $n\geq m$ the support of $K^1_n K^2_m$ is contained in
$$\left\{ \left( (t,x_1,x_2), (\bar{t}, \bar{x}_1,\bar{x}_2 ) \right)\in M\times M\ : \max_{i\in \{1,2\}} d(x_i, \bar{x}_i)^{1/\fraks_i}  < 2^{-m} \ , |t-\bar{t}|^{1/\fraks_0}<2^{-n}   \right\} \ .$$
Note that furthermore, since $K_n$ is non-anticipative for any $L\geq 0$ we have 
$|K_m^2 (z,\bar{z})|\lesssim_L |t-\bar{t}|^L 2^{m (\fraks_0+\fraks_2-\beta_2 + L\fraks_0)}$ and thus 
\begin{align*}
|K_m^{>}(z,\bar{z})| &\lesssim \sum_{n>m} |K_{n}^1(z, \bar{z})| |t-\bar{t}|^L 2^{m (\fraks_0+\fraks_2-\beta_2 + L\fraks_0)} \\
 &\lesssim \sum_{n>m} 2^{n (\fraks_0+\fraks_1-\beta_1 )} 2^{-nL\fraks_0} 2^{m (\fraks_0+\fraks_2-\beta_2 + L\fraks_0)} 
 \end{align*}
and therefore choosing $L\fraks_0>\fraks_0 + \fraks_1-\beta_1$ we find  $$|K_m^{>}(z,\bar{z})|\lesssim 2^{m(2\fraks_0+\fraks_1 + \fraks_2 -\beta_1 -\beta_2)}= 2^{m(|\fraks| -\beta)} \ , $$
since $\fraks_0=\beta= \beta_1=\beta_2$ and $|\fraks|= \fraks_0+\fraks_1+\fraks_2$.
After taking derivatives in exponential charts, one obtains along the same lines that
$$ |j_{z,\bar{z}} K_m^{>}|_l\lesssim 2^{m(|\fraks| -\beta-l)}\ ,$$
which is Item~\ref{item:upper_bound} of Assumption~\ref{Assumption on Kernel}.
By symmetry the same estimate holds for $K_m^{<}(z,\bar{z})$. 

Lastly, Item~\ref{item:integragration_against_polynomials} of Assumption~\ref{Assumption on Kernel} is checked easily and since we are working on compact manifolds, it is easy to accommodate Item~\ref{item:support_in_exp_chart} of Assumption~\ref{Assumption on Kernel} by removing some additional smooth part of the kernel $K$ and adding it to $K_{-1}$.
\end{proof}

\section{Symmetric Sets and Vector Bundles}\label{section symmetric sets and Vector bundles}

The notion of symmetric set introduced in \cite{CCHS22} allows one to circumvent choosing a particular basis when constructing a vector valued regularity structure. This was used in \cite{CCHS22} to be able to work with Lie algebra-valued noises while still working with scalar valued models
 without fixing an (arbitrary) basis of the Lie algebra. Here, symmetric sets allow us to replace ``polynomial decorations'' by jets which do not come with a canonical basis (c.f.\ Section~\ref{basic facts on jets}) as well as to encode symmetries of differential operators in terms of symmetries of certain trees. In this section we recall the relevant notions from \cite{CCHS22}.

For this section, we fix an underlying set of types $\mathfrak{S}$. For two finite typed sets $T^1$ and $T^2$ denote by $\Iso(T^1, T^2)$ the set of all type preserving bijections $T^1\to T^2$ and by $\Vec(\Iso (T^1, T^2))$ the free vector space generated by $\Iso(T^1, T^2)$.
First, we recall that symmetric sets are connected groupoids in the category of typed sets. 
Explicitly, they and the corresponding notion of morphism can be described as follows.
\begin{definition}\cite[Def~5.3]{CCHS22}
A symetric set $\symset$ consists of an index set $A_\symset$ and a triple 
$$\symset = \big(\{T^a_{\symset}\}_{a\in A_{s}}, \ \{\mathfrak{t}_\symset^a\}_{a\in A_\symset},\ \{\Gamma^{a,b}_\symset\}_{a,b\in A_\symset} \big)\ ,$$
where $(T^a_\symset, t^a_\symset)$ are finite typed sets and $\Gamma^{a,b}_\symset\subset \Iso(T^b_\symset, T^a_\symset)$ a non-empty set satisfying for $a,b,c\in A_\symset$
\begin{equs}
\gamma \in \Gamma_\symset^{a,b} \quad&\Rightarrow \quad \gamma^{-1} \in \Gamma_\symset^{b,a}\;,\\
\gamma \in \Gamma_\symset^{a,b}\;,\quad
\bar \gamma \in \Gamma_\symset^{b,c}\quad&\Rightarrow \quad
\gamma \circ \bar \gamma \in \Gamma_\symset^{a,c}\;.
\end{equs}
\end{definition}

\begin{definition}\cite[Def~5.6 \& Rem~5.11]{CCHS22}
A \textit{morphism} $\Phi \in \Hom(\symset, \bar{\symset})$ between two symmetric sets $\symset$ and $\bar{\symset}$ is a two parameter family
$$
A_\symset \times A_{\bar \symset} \ni (a,\bar a) \mapsto \Phi_{\bar{a}, a} \in \Vec \big( \Iso(T_\symset^a, T_{\bar \symset}^{\bar a})\big)\;,
$$
which is invariant in the sense that, for any 
$\gamma_{a,b} \in \Gamma_\symset^{a,b}$ and
$\bar \gamma_{\bar a,\bar b} \in \Gamma_{\bar \symset}^{\bar a,\bar b}$, one 
has the identity
\begin{equ}[e:Phi-aa-gamma-ab]
\Phi_{\bar a,a}\circ \gamma_{a,b} = \bar \gamma_{\bar a,\bar b} \circ \Phi_{\bar b,b}\;.
\end{equ}
Composition is given by 
\begin{equ}
(\bar \Phi\circ \Phi)_{\bbar a, a} = \bar \Phi_{\bbar a, \bar a} \circ \Phi_{\bar a, a}\;,
\end{equ}
for any \textit{fixed} choice of $\bar a$ (no summation).
\end{definition}
Note that the above definition of morphisms is equivalent to \cite[Def 5.6]{CCHS22} by \cite[Rem 5.11]{CCHS22}.

\begin{definition}\cite[Def 5.22]{CCHS22}\label{def:typed_struct}
Denote by $\SSet$ the category of symmetric sets and define $\TStruc$ to be the category obtained by freely adjoining 
countable products to $\SSet$.
We write $\SSet_\Lab$ and $\TStruc_\Lab$ when we want to emphasise the dependence of this category on 
the underlying label set $\Lab$.
For a description of morphisms in $\TStruc_\Lab$ see \cite[Rem.~5.23]{CCHS22}.
\end{definition}
\begin{remark}\label{rem:disj_union_symsets}
Fix some symmetric sets $\{\symset_{i}\}_{i\in I}$ for some finite index set $I$, where \linebreak
$\symset_i=\big(\{T^a_{i}\}_{a\in A_{i}}, \ \{\mathfrak{t}_{i}^a\}_{a\in A_{i}},\ \{\Gamma^{a,b}_{i}\}_{a,b\in A_{i}} \big) $.
 We then define the symmetric set $\symset= \otimes_{i\in I} \symset_i $ where 
\begin{itemize}
\item $A_\symset= \prod_{i\in I} A_{i}$
\item for each $a\in A_\symset$, $T^a_{\symset}= \bigsqcup_{i\in I} T^a_{i}$ is given by the disjoint union,
\item $\mathfrak{t}_{\symset}^a: T^a_{\symset}\to \Lab$ is the unique map such that $\mathfrak{t}_{\symset}^a|_{T^a_{i}} = \mathfrak{t}_{i}^{a_i}$
\item for each $a,b\in A_{\symset}$ 
$$ \Gamma^{a,b}_{\symset}= \prod_{i\in I} \Gamma^{a_i,b_i}_{i}\;,$$
elements of which act on $T^a_{\symset}$ in the canonical way.
\end{itemize} 
We call $\symset= \otimes_{i\in I} \symset_i $ the \textit{tensor product} of the symmetric sets $\{\symset_{i}\}_{i\in I}$. This product turns $\SSet$ into a monoidal category, c.f.\ \cite[Remark~5.12]{CCHS22}.
\end{remark}

%

For manifold $M$, let $\VecB_M$ be the category of finite dimensional smooth vector bundles over $M$.
A vector bundle assignment $W=(W^{\frakt})_{\frakt\in \mathfrak{S}}$ allows to define for a symmetric set $\symset$ a vector bundle $W^{\symset}$ in exactly the same way as it is done in \cite{CCHS22} for vector spaces:
\begin{enumerate}
\item For $T$ a typed set, let $$W^{\otimes T}:= \bigotimes_{x\in T} W^{\mathfrak{t}(x)}\ $$
and interpret any $\psi\in \Iso (T,\bar{T})$ as the vector bundle morphism $W^{\otimes T}\to W^{\otimes \bar{T}}$ characterised by 
\begin{equation}\label{eq:iso acting on vect}
W_p^{\otimes T} \ni w_p=\otimes_{x\in T} w^x_p\ \mapsto \psi \cdot w_p = \otimes_{y\in \bar T } w^{\psi^{-1}(y)}_p \ .
\end{equation}
\item Define \begin{equ}[e:def-V-tensor-symset]
W^{\otimes \symset} = \Big\{w \in \prod_{a\in A_\symset} W^{\otimes T_\symset^a}
\,:\, w^{(a)} = \gamma_{a,b} \cdot w^{(b)}\quad \forall a,b\in A_\symset\;,\; \forall \gamma_{a,b} \in \Gamma_{\symset}^{a,b}\Big\}\;.
\end{equ}
For $a \in A_\symset$, denote by $\pi_{\symset,a}$ the symmetrisation map $W^{\otimes T_\symset^a} \to W^{\otimes \symset}$ given by
\begin{equ}\label{eq:translating_on_section}
(\pi_{\symset,a} w)^{(b)} = {\frac{1}{|\Gamma_\symset^{b,a}|}} \sum_{\gamma \in \Gamma_\symset^{b,a}} \gamma \cdot w\;,
\end{equ}
These maps $(\pi_{\symset,a})_{a\in A_{\symset}}$ have the property that
\begin{equ}\label{e:covpi}
\pi_{\symset,a} \circ \gamma_{a,b} = \pi_{\symset,b}\;,\quad \forall a,b \in A_\symset\;,\quad \forall \gamma_{a,b} \in \Gamma_\symset^{a,b}\;.
\end{equ}
and they are left inverses to the natural inclusions 
\begin{equ}\label{eq iota}
 \iota_{\symset,a} \colon W^{\otimes \symset} \to W^{\otimes T_\symset^a}\ , \quad (w^{(b)})_{b\in A_\symset} \mapsto w^{(a)}\ .
 \end{equ}
\end{enumerate}

We now fix two symmetric sets $\symset$ and $\bar \symset$.
Given $\Phi \in \Hom(\symset,\bar{\symset})$, it 
naturally defines a linear map $F^{a,\bar{a}}_\Phi\colon W^{\otimes T_\symset^a} \to W^{{\otimes T_{\bar{\symset}}^{\bar{a}}}}$ by\footnote{Here we  use the following extension of \eqref{eq:iso acting on vect}: For $i\in \{1,\ldots,n\}$, let $\lambda_{i}\in \mathbb{R}$ and $\psi_i\in \Iso (T,\bar{T})$, then
 $\bar{\psi}=\sum_i \lambda_{i}\psi_i\in \Vec(\Iso (T,\bar{T}))$ induces the map 
$$
W^{\otimes T}\to W^{\otimes \bar{T}}, \qquad w_p\ \mapsto\bar{\psi}\cdot w_p = \sum_i \lambda_{i}(\psi_i \cdot w_p) \  .
$$}

\begin{equ}[e:defFPhi]
F^{a,\bar{a}}_\Phi w = \Phi_{\bar{a}, a} \cdot w \;,
\end{equ}
where $\Phi_{\bar{a}, a}$ denotes the representative of $\Phi$ in $\Vec(\Iso(T_\symset^a, T_{\bar \symset}^{\bar a}))$. 
One sees that
for $w = (w^{(a)})_{a \in A_\symset} \in W^{\otimes \symset}$, $\Phi = (\Phi_a)_{a \in A_\symset} \in \Hom(\symset,\bar \symset)$, 
as well as $\gamma_{a,b} \in \Gamma_{\symset}^{a,b}$ and $\bar{\gamma}_{\bar b,\bar a} \in \Gamma_{\bar \symset}^{\bar a , \bar b }$, we have the identity
\begin{equ}
F_{\Phi}^{b,\bar b}  w^{(b)} =\Phi_{\bar b, b} \cdot w^{(b)}  = \bar{\gamma}_{\bar b,\bar a} \Phi_{\bar a, a} \gamma_{a,b}  \cdot w^{(b)} = \bar{\gamma}_{\bar b,\bar a} \Phi_{\bar a, a} \cdot w^{(a)}\;,
\end{equ}
and thus 
\begin{equ}[e:def-F_Phi]
F_\Phi: W^{\otimes \symset} \to W^{\otimes \bar \symset}, \qquad F_\Phi w = \pi_{\bar \symset,\bar a} F_{\Phi}^{a,\bar a} \iota_{\symset,a} w\;,
\end{equ}
is well defined and independent of the choices of $a$ and $\bar a$.


\begin{remark}\label{rem:can_morph}
Note that in the setting of Remark~\ref{rem:disj_union_symsets} writing $\symset =\otimes_{i\in I} \symset_i$, there is a canonical isomorphism 
\begin{equation}\label{loceq}
 \bigotimes_{i\in I} \Func_W(\symset_i)= \bigotimes_{i\in I} W^{\otimes \symset_i} \to W^{\otimes \symset} =\Func_W(\symset)\ .
\end{equation}
Observe the different meanings of the $\otimes$ in $\symset =\otimes_{i\in I} \symset_i$ and in \eqref{loceq}.

In particular, as in \cite[Lemma 5.18]{CCHS22} this defines a monoidal functor $\Func_W$ mapping $\symset$ to $W^{\otimes \symset}$ and  
$\Phi$ to $F_\Phi$ between the category $\TStruc$ of symmetric sets and the category $\VecB_M$.
\end{remark}

\subsection{Some canonical morphisms}
In this section we introduce certain canonical morphisms of symmetric sets which will be useful in the sequel.
The following is a special case of \cite[Rem.~5.13]{CCHS22}.
\begin{definition}\label{def:morphisms_for_colouring}
Suppose
for the same index set $A$ we are given two symmetric sets
 $\symset=\big(\{T^a\}_{a\in A}, \ \{\mathfrak{t}_\symset^a\}_{a\in A},\ \{\Gamma^{a,b}_\symset\}_{a,b\in A} \big)\ $ 
and $\bar{\symset}=\big(\{T^a \}_{a\in A}, \ \{\mathfrak{t}_{\bar{\symset}}^a\}_{a\in A},\ \{\Gamma^{a,b}_{\bar \symset }\}_{a,b\in A} \big)\ $,
such that furthermore $\mathfrak{t}^a_{\symset}=\mathfrak{t}^a_{\bar{\symset}}$ for all $a\in A$.

We define $S\in \Hom (\symset, \bar{\symset})$ as 
$$S_{\bar{a},a}
:= \frac{1}{|\Gamma_{\bar{\symset}}^{\bar{a},\bar{a}}| |\Gamma_{\symset}^{\bar{a},a}| }
\sum_{ \bar{\gamma}_{\bar{a},\bar{a}}\in \Gamma_{\bar{\symset}}^{\bar{a},\bar{a}}, \gamma_{\bar{a},a}\in \Gamma_{\symset}^{\bar{a},a} }
 \bar{\gamma}_{\bar{a},\bar{a}}\ \circ \gamma_{\bar{a},a}  
= \frac{1}{|\Gamma_{\bar{\symset}}^{\bar{a}, a}| |\Gamma_{\symset}^{a,a}| }
\sum_{ \bar{\gamma}_{\bar{a},a}\in \Gamma_{\bar{\symset}}^{\bar{a},a}, \gamma_{a,a}\in \Gamma_{\symset}^{a,a}}
 \bar{\gamma}_{\bar{a},a}\ \circ \gamma_{a,a}  \ . 
 $$
 Sometimes we shall write $S^{\symset, \bar{\symset}}$ for emphasis.
\end{definition}
In applications we shall refer to this morphism when speaking about a construction \textit{extending} to symmetric sets, see for example the discussion after \eqref{eq:colouring}.
The proof of the next lemma is a simple exercise.
\begin{lemma}\label{lem:symset_composition}
Let $\mathring{\symset}$, $\bar \symset$, $\symset_1$ and $\symset_2$ be symmetric sets over the same index set as in Definition~\ref{def:morphisms_for_colouring} and furthermore
suppose $\mathring{\Gamma }\subset \Gamma^{(1)}\subset \bar{\Gamma}$ and $\mathring{\Gamma }\subset \Gamma^{(2)} \subset \bar{\Gamma}$. Then, 
the following diagrams commute
\[\begin{tikzcd}
\symset_1 \arrow{r}{S^{\symset_1,\symset_2}} \arrow[swap]{dr}{ S^{\symset_1 ,\bar{\symset}}}  &\symset_2 \arrow{d}{S^{\symset_2 ,\bar{\symset}}}\\
&\bar{\symset}\ ,
\end{tikzcd} \qquad
\begin{tikzcd}
\symset_2 \arrow{r}{S^{\symset_2,\mathring{\symset}}} \arrow[swap]{dr}{S^{{\symset}_2,\symset_1}}  &\mathring{\symset}\arrow{d}{{S^{\mathring{\symset},\symset_1}}}\\
&\symset_1 \ .
\end{tikzcd}  
\] 
\end{lemma}
\begin{remark}\label{smaller symmetry group}
The following special cases of Definition~\ref{def:morphisms_for_colouring} are of particular interest.
\begin{enumerate}
\item In the case $\Gamma^{a,b}_\symset\supset \Gamma^{a,b}_{\bar{\symset}}$,
we write $\mathcal{I}:=S^{\symset, \bar \symset}\in \Hom (\symset, \bar{\symset})$ for this morphism and find the identity 
$$\mathcal{I}_{\bar{a},a}= \frac{1}{|\Gamma_{\symset}^{\bar{a},a}|}\sum_{\gamma_{\bar{a},a}\in \Gamma_{\symset}^{\bar{a},a}} \gamma_{\bar{a},a} \ .$$
\item If $\Gamma^{a,b}_\symset\subset \Gamma^{a,b}_{\bar{\symset}}$, let 
$\Pi:= S^{\symset, \bar \symset}\in \Hom (\symset, \bar{\symset})$. Then one can write $$\Pi_{\bar{a},a} = \frac{1}{|\Gamma_{\bar{\symset}}^{\bar{a}, a}|}\sum_{\bar{\gamma}_{\bar{a},a}\in \Gamma_{\bar{\symset}}^{\bar{a},a}} \bar{\gamma}_{\bar{a},a} \ .$$
\end{enumerate}
\end{remark}

\begin{remark}\label{rem:symmetrisation_injection inv}
In the setting of Definition~\ref{def:morphisms_for_colouring}, if $\Gamma^{a,b}_\symset\supset \Gamma^{a,b}_{\bar{\symset}}$ and using the notation  
$\Pi\in \Hom(\bar{\symset},\symset)$ and $\mathcal{I}\in \Hom({\symset, \bar{\symset}} )$ from the previous remark, one finds that
$$\Func_W \Pi \circ \Func_W \mathcal{I} =\Func_W S^{\symset,\symset}=  \id_{W^{\otimes {\symset}}} \ ,$$
where the first equality follows from Lemma~\ref{lem:symset_composition}.
\end{remark}

\begin{remark}
For a symmetric set $\symset=\big(T, \ \mathfrak{t},\ \Gamma_{\symset} \big)$ where $A=A_\symset$ consists of one element, we write $\mathring{\symset}$ for the symmetric set $\mathring{\symset}=\big(T, \ \mathfrak{t},\ \bar{\Gamma} \big)$, 
where $\bar{\Gamma}$ consists only of the identity map.
Under the canonical isomorphism $V^{\otimes T}\simeq V^{\otimes \mathring{\symset}}$ one sees for $\Pi\in \Hom (\mathring{\symset},\symset)$ and $\mathcal{I}\in \Hom(\symset, \mathring{\symset})$ as in Remark~\ref{smaller symmetry group} that $\Func_W \Pi$ agrees with the symmetrisation map in \eqref{eq:translating_on_section} and $\Func_W \mathcal{I}$ agrees with the inclusion \eqref{eq iota}.
\end{remark}


\subsection{Direct sum decomposition}
Next we recall the notion of type decomposition, see \cite[Def 5.25]{CCHS22}, with the only slight change that we allow for countable decompositions of vector bundles.
\begin{definition}\label{def:label-decompose}
Let $\mathcal{P}(A)$ denote the powerset of a set $A$.\label{powerset page ref}
Given two distinct countable sets of labels $\Lab$ and $\bar \Lab$ as well as a 
map $\proj \colon \Lab \to \mathcal{P}(\bar \Lab) \setminus \{\emptyset\}$,
such that $\{ \proj(\mft): \mft \in \Lab\}$ is a partition of $\bar\Lab$, we call $\bar{\Lab}$ a type decomposition of $\Lab$ under $\proj$.

If we are also given space assignments $(W_{\mft})_{\mft \in \Lab}$ and $(\bar{W}_{\mfl})_{\mfl \in \bar{\Lab}}$ with the property that
\begin{equ}\label{eq: type vector space decomp}
W_\mft = \bigoplus_{\mfl \in \proj(\mft)} \bar W_\mfl \qquad\textnormal{ for every }\mft \in \mfL\;,
\end{equ} 
then we say that $(\bar{W}_{\mfl})_{\mfl \in \bar{\Lab}}$ is a decomposition of $(W_{\mft})_{\mft \in \Lab}$. 
For $\mfl \in \proj(\mft)$, we write $\PP_{\mfl} \colon W_{\mft} \to \bar W_{\mfl}$
for the projection induced by \eqref{eq: type vector space decomp}.
\end{definition}

Now we recall how a type decompositions allow us to decompose symmetric sets and associated vector bundles:
\begin{enumerate}
\item Fix $\Lab$, $\bar{\Lab}$, $\proj$ as well as  $(W_{\mft})_{\mft \in \Lab}$, $(\bar{W}_{\mfl})_{\mfl \in \bar{\Lab}}$ as in Definition~\ref{def:label-decompose} and, given a set $B$ and functions $\mft \colon B \to \Lab$,  $\mfl \colon B \to \bar\Lab$, use the shorthand notation $\mfl \tto \mft$ to signify that $\mfl(p) \in \proj(\mft(p))$ for every $p \in B$.
\item For a symmetric set
$\symset$ with label set $\Lab$ and $a \in A_\symset$
we write $\hat L_\symset^a = \{\mfl\colon T_\symset^a \to \bar \Lab\,:\, \mfl \tto \mft_\symset^a\}$
and we consider on $\hat L_\symset = \bigcup_{a\in A_\symset} \hat L_\symset^a$ the equivalence 
relation $\sim$ given by
\begin{equ}[e:defsim]
\hat L_\symset^a\ni \mfl \sim \bar \mfl \in \hat L_\symset^b \quad\Leftrightarrow\quad \exists \gamma_{b,a} \in \Gamma_\symset^{b,a}\,:\,  
\mfl = \bar\mfl \circ \gamma_{b,a}\;.
\end{equ}
Note that each equivalence class contains finitely many elements, but there might be a countable number of equivalence classes which we denote by 
$$L_\symset = \hat L_\symset / {\sim}\; .$$
\item Given an equivalence class $Y \in L_\symset$ and $a \in A_{\symset}$, we define $Y_a = Y \cap \hat L_\symset^a$ (which we note is non-empty due to the connectedness of $\Gamma_\symset$).
We then define a symmetric set
$\symset_Y$ by
\begin{equs}[decomp]
A_{\symset_Y} &= \{(a, \mfl)\,:\, a \in A_\symset\;,\quad \mfl \in Y_a\}\;,\qquad
T_{\symset_Y}^{(a,\mfl)} = T_\symset^a\;,\qquad
\mft_{\symset_Y}^{(a,\mfl)} = \mfl\;,\\
\Gamma_{\symset_Y}^{(a,\mfl),(b,\bar \mfl)} &= \{\gamma \in \Gamma_{\symset}^{a,b}\,:\, \bar \mfl = \mfl \circ \gamma\}\;.
\end{equs}
\item With these notations at hand, we can define a functor 
$\proj^*$ from $\SSet_\Lab$ to $\TStruc_{\bar\Lab}$ as follows. 
Given any $\symset\in \ob(\SSet_\Lab)$, we define
\begin{equ}[e:linkpi3]
\proj^* \symset = \bigoplus_{Y \in L_\symset} \symset_Y \in \ob\big(\TStruc_{\bar \Lab} \big)\;.
\end{equ}
We refer the reader to \cite[Eq.~5.17]{CCHS22} to see how $\proj^*$ acts on morphisms.
\end{enumerate}

One of our main interests in the functor $\proj^*$ is that it performs the direct sum decompositions
at the level of partially symmetric tensor products of the spaces $V_\mft$, which is formulated in the following proposition.
Here, when working with infinite dimensional vector bundles, the grading is crucially needed.
\begin{prop}\cite[Prop 5.29]{CCHS22}\label{prop:nat_transform}
One has $\Func_{\bar W} \circ \proj^* = \Func_{W}$, modulo natural transformation.
\end{prop}

\begin{remark}\label{remark tensor product}
Note that the whole discussion on symmetric sets adapts ad verbatim, when using the tensor product $\hotimes$ introduced in Section~\ref{section hat tensorproduct} instead of $\otimes$ on vector bundles, one only has to replace the category $\VecB_M$ of vector bundles over $M$ by the ``larger" category $\VecB$ of vector bundles over arbitrary base manifolds.
\end{remark}

\section{Trees for Regularity Structures}\label{section sets of trees for...}
\subsection{Indexing vector bundles}\label{section indexing vector bundels}
We shall construct a regularity structure ensemble associated to a subcritical rule. For this we start by fixing an index set of types $\mathfrak{L}$.\footnote{The set $\mathfrak{L}$ plays a different role in this section than it did in the definition of a regularity structure ensemble, see Definition~\ref{definition regularity structure}.}
We shall assume that $\mathfrak{L}$ is equipped with an involution without fixed points
$$*:\mathfrak{L}\to \mathfrak{L}, \ \mathfrak{l}\mapsto \mathfrak{l}^* \ ,$$
which we extend to act on $\mathbb{N}^{\mathfrak{L}}$ in the canonical way by mapping $\sigma\in \NN^\mathfrak{L}$ to $\sigma^*\in \NN^\mathfrak{L}$ characterized by $\sigma^*(\mathfrak{l})= \sigma (\mathfrak{l}^*)$. We define the map $$\text{red}_*: \mathbb{N}^{\mathfrak{L}}\to \mathbb{N}^{\mathfrak{L}}\ ,$$ where 
$\text{red}_*(\sigma)$ is given by the minimal element of $\mathbb{N}^{\mathfrak{L}}$ such that $[\sigma]=[\text{red}_*(\sigma)]$ in $\frac{\mathbb{Z}[\mathfrak{L}]}{\ker(\id-*)}$,
or equivalently $\big(\text{red}_*(\sigma)\big)_\mfl = \sigma_\mfl - (\sigma_\mfl \wedge \sigma_{\mfl^*})$.

From now on we assume that all vector bundle assignments $\{ V^\mathfrak{l}\}_{\mathfrak{l}\in \mathfrak{L}}$ satisfy 
\begin{equation}\label{eq:assigment_dual}
(V^\mathfrak{l})^*= V^{\mathfrak{l}^*}\ ,
\end{equation}
 where $(V^\mathfrak{l})^*$ denotes the dual bundle of $V^\mfl$.
We define
for $\sigma\in \mathbb{N}^\mathfrak{L}$ 
$$V^{\sigma}:= \bigotimes_{\mfl \in \mathfrak{L}} (V^\mfl)^{\otimes_s \sigma_{\mfl}} \ $$
and denote by 
$$\odot_s: V^{\sigma}\times V^{\bar \sigma} \to V^{\sigma+\bar \sigma}$$
the symmetrized tensor product induced by the canonical symmetric tensor product $W^{\otimes_s n}\otimes W^{\otimes_s m}\to W^{\otimes_s n+m}$ for any vector bundle $W$ and $n,m\in \mathbb{N}$.
Observe that $\odot$ endows $\bigoplus_{\sigma\in \mathbb{N}^{\mathfrak{L}}} V^{\sigma}$ with the structure of a commutative algebra bundle.\footnote{This space can canonically be equipped with the structure of a Hopf algebra bundle (in the same way as the tensor algebra over a vector space), but we shall not use this fact.}
Note that for any $\sigma\in \mathbb{N}^\mathfrak{L}$ one then has a natural (and unique) 
partial trace operator 
$$\Tr: V^{\sigma}\to V^{\text{red}_*\sigma} $$
which by polarisation is characterised by its action on elementary tensors 
$$\bigotimes_{l\in \mathfrak{L}} v_{\mathfrak{l}}^{\otimes \sigma_{\mathfrak{l}}} \mapsto 
\bigotimes_{l\in \mathfrak{L} } (v_{\mathfrak{l}}, v_{\mathfrak{l}^*})^{\frac{\sigma_\mathfrak{l}\wedge \sigma_{\mathfrak{l}^*}}{2}} v_{\mathfrak{l}}^{\otimes \text{red}(\sigma)_{\mathfrak{l}}} \ ,$$
which is well defined since each factor $(v_{\mathfrak{l}}, v_{\mathfrak{l}^*})^{\frac{\sigma_\mathfrak{l}\wedge \sigma_{\mathfrak{l}^*}}{2}}$ appears exactly twice in this tensor product.
\begin{remark}\label{rem canoncial product on sections}
Recall the basic fact that for any $\sigma, \bar \sigma\in \mathbb{N}^\mathfrak{L}$ the product $\odot$ induces a canonical pointwise product 
$$\odot: \cC^\infty(V^\sigma)\times \cC^\infty(V^{\bar{\sigma}} )\to \cC^\infty(V^{\sigma+ \bar{\sigma}})\ .$$
\end{remark}
\begin{remark}\label{remark incompatibilit product}
The partial trace is not compatible with the algebra structure introduced above,
in the sense that in general $\Tr (v_1 \odot \Tr (v_2 \odot v_3)) \neq \Tr (\Tr (v_1 \odot v_2) \odot v_3)$ for $v_i\in V^{\sigma_i}$.
\end{remark}

\begin{remark}
Let us emphasise the advantage of allowing for possible redundancy in the vector bundle assignment for models, i.e.\ $\mathfrak{e},\mathfrak{f}\in \mathfrak{L}$ such that $V^\mathfrak{e}=V^\mathfrak{f}$. It is such redundancies that allow one to assume that all operations in the same index are symmetric, 
allowing us to work with combinatorial rather than ordered trees\footnote{This in turn has the advantage that one can reuse notions and ideas from existing works such as \cite{BHZ19} and \cite{BCCH20}.} and encode traces unambiguously.
Let us illustrate this with two examples: Let $u,v$ be $TM$-valued solutions to a system of equations and let $A$ be a section of $(TM^*)^{\otimes 2}$, then the trees appearing to describe $A(u,v)$ can be taken to be 
combinatorial, as we are indexing the two ``components'' of $(TM^*)^{\otimes 2}=T^*M\otimes T^*M$ differently and thus contractions are unique. 
Now suppose $B\in (TM^*)^{\otimes 3}$ and the right-hand side of the equation contains the term $B(u,v,v)$. Then we can assume that $B$ is symmetric in the last two components without changing the equation and the same logic holds. 
See also Section~\ref{section:phi43} where this is implemented for a concrete example.
\end{remark}

\subsection{Construction of underlying trees}\label{section typed rooted trees}
We work with typed rooted trees $(T,\mathfrak{e})$, where $T=(N_T,E_T)$ is a rooted tree and $\mathfrak{e}:  E_T \to \mathcal{E}$  is a type map.\footnote{Here we shall take set $\mathcal{E}$ as given. In Section~\ref{section application to spdes} it is explained how it can be related to $\mathfrak{L}$ when working with SPDEs.} We shall simply write $T$ instead of $(T,\mathfrak{e})$. We assume that the set of types $\mathcal{E}$ can be disjointly decomposed as
$$\mathcal{E}=\mathcal{E}_+\cup \mathcal{E}_0 \cup \mathcal{E}_- \ .$$ 
The set $\mathcal{E}_+$ will encode kernels, $\mathcal{E}_0$ will be used to encode jets and one can think of the edges as encoding kernels given by Dirac distributions on vector bundles. The set $\mathcal{E}_-$ will encode noises. We assume we are given a map
\begin{equation}\label{eq iota E}
\iota: \mathcal{E}_+\cup \mathcal{E}_0\to \mathcal{E}_0 \ 
\end{equation}
such that $\iota|_{\mathcal{E}_+}$ is an injection and $\iota|_{\mathcal{E}_0}= \id_{{\mathcal{E}_0}}$.
%
%
We also fix indexing maps
\begin{itemize}
\item $\ind_+ : \mathcal{E}_+\to \mathfrak{L}\times\mathfrak{L},\  \varepsilon \mapsto (\varepsilon_+, \varepsilon_-)$
\item $\ind_- : \mathcal{E}_-\to \mathfrak{L},\  \varepsilon \mapsto \varepsilon_-$
\item $\ind_0 : \mathcal{E}_0\to \mathbb{N}^\mathfrak{L},\  \varepsilon \mapsto \varepsilon_- \ ,$
\end{itemize}
which we assume, after identifying the elements of $\mathfrak{L}$ with the canonical generators of $\mathbb{N}^\mathfrak{L}$, to satisfy the identity $\ind_+(\varepsilon)_-= \ind_0(\iota (\varepsilon))$ for any $\varepsilon\in \mathcal{E}_+$. 
We shall omit the subscript and write $\ind$ for all three maps, or even just $\varepsilon_-$, resp. $\varepsilon_+$ for $\varepsilon\in \mathcal{E}$. 
We also assume we are given a map $|\cdot |: \mathcal{E}\to \mathbb{R}$
with the property that $\eps \in \mathcal{E}_{\sign |\eps|}$ (with the convention that $\sign 0 = 0$). 
We denote by $\bar{\mathfrak{U}}$ the set of all typed rooted trees and to each $T\in \bar{ \mathfrak{U}}$ we assign the homogeneity 
$$|T|:= \sum_{e \in E_T} |\mathfrak{e}(e)|\ .$$
We also call $\mathcal{E}$ the set of edge types and introduce
node types
$$\mathcal{N}= \hat{\mathcal{P}}\mathcal{E} =\bigcup_{k\geq 0} [\mathcal{E}]^k$$
i.e.\ $\mathcal{N}$ consists of all finite (unordered) multi-sets%
\footnote{There is a canonical identification $\hat{\mathcal{P}}\mathcal{E}\simeq \mathbb{N}^\mathcal{E}$ under which $[\mathcal{E}]^k$ consists of those $\nu\in \mathbb{N}^\mathcal{E}$ satisfying $$\sum_{\varepsilon\in \mathcal{E}} \nu(\varepsilon)=k \ .$$ }
whose elements belong to $\mathcal{E}$.
We define for $\nu\in \mathcal{N}$
$$\ind (\nu)= \re \big(\sum_{\varepsilon\in \nu} \varepsilon_-\big)\in \mathbb{N}^\mathfrak{L} \ .$$
For any set $A$ we denote by $\mathcal{P}(A)$ its powerset.
\begin{definition}\label{def normal rule}
A rule is a map $R: \mathcal{E}\to \mathcal{P}(\mathcal{N})\setminus \{\emptyset\}$ satisfying the following conditions 
\begin{enumerate}
\item $R(\varepsilon)= \{()\}$ for every $\varepsilon\in \mathcal{E}_-\cup \mathcal{E}_0$
\item for each $\mathfrak{k}\in \mathcal{E}_+$ and $\nu\in R(\mathfrak{k})$ one has
$$\ind (\nu)= \mathfrak{k}_+ \ .$$
\end{enumerate}
The rule is said to be normal if it additionally satisfies the following property
\begin{itemize}
\item if $A\in R(\varepsilon)$ and $A= B\sqcup C$ for some non-empty multi-subset $B\sqsubset A$ such that $B\in \hat{\mathcal{P}}(\mathcal{E}_+)$,  one has $B^0 \sqcup C\in R(\varepsilon)$, where $B^0$ is obtained by replacing each edge of type $\tilde{\mathfrak{k}}\in \mathcal{E}_+$ by one of type $\iota(\tilde{\mathfrak{k}}) \in\mathcal{E}_0$.
\end{itemize}
Given any rule $\mathring{R}$, the normalisation $\bar{R}$ of $\mathring{R}$ is the minimal normal rule containing $\mathring{R}$ in the sense that
$$\mathring{R}(\varepsilon)\subset \bar{R}(\varepsilon)$$
for all $\varepsilon\in \mathcal{E}$.

We call a rule $R$ \textit{equation-like} if for each $\mathfrak{e}\in \mathcal{E}$ the following holds: For every $A\in R(\mathfrak{e})$
there exists $\tilde{A}= B\sqcup C$, where $B\in \hat{\mathcal{P}}(\mathcal{E}_+)$ and $C$ contains at most one edge of type $\mathcal{E}_0$, such that $A= B^0\sqcup C$.\footnote{Where as previously $B^0$ is as in Definition~\ref{def normal rule}, i.e.\ obtained by replacing each edge of type $\varepsilon\in \mathcal{E}_+$ in $B$ by one of type $\iota(\varepsilon) \in\mathcal{E}_0$.}
\end{definition}
\begin{remark} Let us make the following comments.
\begin{itemize}
\item The first condition in the definition of a rule is motivated by the fact that only the $\mathcal{E}_+$ type edges encode operators, while $\mathcal{E}_-$ and $\mathcal{E}_0$ encode functions\slash distributions.
\item The second condition in the definition of a rule is a consistency condition needed to obtain meaningful expressions later. It guarantees distributions take values in the correct vector bundles and traces are well defined.
\item Our notion of a rule being normal replaces the second condition of being normal in \cite {BHZ19}.\footnote{Namely, for every $\varepsilon\in \mathcal{E}$ and multi-sets $A,B$ such that $A\subset B\in R(\varepsilon)$ it holds that $A\in R(\varepsilon)$.} In both cases this property is needed to close the algebraic structure with respect to the operations needed for positive renormalisation. Our setting will have the property that the only products that arise are the ones that are already present in the equation.
\item Remark~\ref{rem:eq_like_rules} below explains why the notion of equation-like rules is useful. If a rule $R$ is equation-like, its normalisation is equation-like as well.
\end{itemize}
\end{remark}

Now define subcriticality similarly to \cite[Definition 5.13]{BHZ19}.
\begin{definition}
A rule $R$ is called subcritical if there exists a map $\reg: \mathcal{E} \to \RR$ satisfying $\reg|_{\mathcal{E}_0}=0$, such that for each $\mathfrak{k}\in \mathcal{E}_+$
$$\reg(\mathfrak{k})< |\mathfrak{k}| + \inf_{N\in R(\mathfrak{k})} \sum_{\varepsilon\in \nu}\reg (\varepsilon) \ .$$
\end{definition}

\begin{remark}\label{rem:eq_like_rules}
Subcriticality does not translate into the fact that for each $r\in \mathbb{R}$ the set $\mathfrak{T}_{<r}:=\{ T \in \mathfrak{T} | |T|<r \}$ is finite, since it imposes no restriction on the possible number of $\mathcal{E}_0$ type edges attached to each node. For equation-like rules however this holds by a straightforward adaptation of the proof of \cite[Proposition 5.15]{BHZ19}.
\end{remark}

\begin{remark}\label{remark partial order}
Recall that for any rooted tree $T$, the set $N_T\cup E_T$ has a natural partial order in which the root is the unique minimal element. For an edge $e\in E_T$ we write $e_+,e_- \in N_T$ for the nodes such that $\{e_+,e_-\}=e$ and $e_-<e_+$ \ . 
\end{remark}

From now on fix a subcritical rule $R$. This defines ${\mathfrak{T}}\subset \mathfrak{U}$, the collection of trees which (in the language of \cite{BHZ19}) \textit{conform} to $R$, that is ${\mathfrak{T}}$ consists of all trees $T$, such that for each edge $e\in E_T$ the set $E^{\mathtt{in}}_e=\{\bar{e}\in E_T \ | \ \bar{e}_-= e_+ \text{ for some } e\in E_T\}$ satisfies 
$$\bigsqcup_{\bar{e}\in E^{\mathtt{in}}_e} \mathfrak{e}(e_i) \in  R(\mathfrak{e}(e)) \ . $$
For a node type $\nu\in \mathcal{N}$, let $\mathfrak{T}^\nu\subset \mathfrak{T}$ consist of all trees where the root has node type $\nu$ and for $\sigma\in \mathbb{N}^\mathfrak{L}$ set 
$$\mathfrak{T}^\sigma= \bigcup_{ \ind(\nu)= \sigma} \mathfrak{T}^\nu \ .$$

\begin{remark}
The set ${\mathfrak{T}}$ will play the role of the universe of trees, which we use for essentially all our constructions, while $\mathfrak{T}^\sigma$ will be used to construct objects related to distributions with values in $V^\sigma$.
\end{remark}
\begin{remark}\label{rem recursive construction}
We make the observation that all trees in $\mathfrak{T}$ can be constructed recursively. The base case is given by trees consisting of only one edge which is necessarily of type in $\mathcal{E}_0\cup \mathcal{E}_-$ by the definition of $\mathfrak{T}$. From now on we denote the tree consisting of one edge by $\Xi^\varepsilon$ if $\varepsilon \in \mathcal{E}_-$ and by $\delta_{\varepsilon}$ if $\varepsilon \in \mathcal{E}_0$. 
The rest of the trees are obtained by applying the following two operations repeatedly.
\begin{itemize}
\item The \textit{tree product}: For $T=(N_T,E_T), T'=(N_{T'}, E_{T'})$ let 
$$T \cdot T':= (N_T\cup N_{T'}\big/ {\sim} \ , E_T \cup E_{T'} ) $$
where $\sim$ identifies the root nodes of $T$ and $T'$.
One easily checks that the tree product is associative and commutative and we simply write $T_1\cdots T_n$ for the tree product of $n\in \NN$ trees. 
\item Attaching an edge of type $\mathfrak{k}\in \mathcal{E}$ to the root of a tree $T=(N_T,E_T)$: Let $\mathcal{I}_\mathfrak{k}=(\{e_+,e_-\}, \{e\})$ by the tree consisting of only one edge which is of type $\mathfrak{k}$, then
$$\mathcal{I}_\mathfrak{k} (T):= (N_T\cup e \big/ {\sim} \ , E_T\cup \{e\})$$
where $\sim$ identifies the root of $T$ with the node $e_+$.
\end{itemize}

\end{remark}

\subsection{Negative subtrees}
For $T\in \mathfrak{T}$ and $E\subset E_T$ define $T_{E}$ to be the minimal subtree of $T$ containing these edges. 
Let  $p$ be a partition of $E$. We say $p$ \textit{generates a forest in }$T$, if the trees $\{T_A : A\in p\}$ are mutually vertex disjoint and denote by $F_{T,p}$ the corresponding subforest of $T$. We write $P(T,E)$ for all partitions of $E$ which generate a forest. 

For a fixed tree $T$ we define 
\begin{itemize}
\item $E^-_T:=\{ e\in E_T\ | \mathfrak{e}(e)\in \mathcal{E}_- \}$,
\item $\mathfrak{T}^-_T:= \{ T_{E}\subset T\  |  \ E\subset E^-_T \}  $,
\item $\mathfrak{T}^-_{T,\rho}= \{ T_{E}\subset T\  |  \ E\subset E^-_T ,\ \rho_T\in T_{E} \}  $,
\item $\mathfrak{F}^-_{T}:= \{ F_{T,p}\subset T\  |  \ E\subset E^-_T, p\in P(T,E)  \}  $.
\end{itemize}
as well as\begin{equation}\label{eq def neg trees}
\mathfrak{T}_-:= \bigcup_{T\in \mathfrak{T}} \mathfrak{T}^-_T \Big/ {\sim} \, 
\end{equation}
where we write $A_1\sim A_2$ for subtrees $A_1\subset T_1\in \mathfrak{T}$ and $A_2\subset T_2\in \mathfrak{T}$, if $A_1$ and $A_2$ are isomorphic as trees.
Denote by $\mathfrak{F}_-$ the free monoid generated by $\mathfrak{T}_-$ under the forest product. 
\begin{remark}
Note that the maximal edges of every tree in $\mathfrak{T}_-$ are necessarily of type in $\mathcal{E}_-$ and thus by definition of a normal rule, a tree $T\in \mathfrak{T}_-$ never contains an edge of type in $\mathcal{E}_0$. 
\end{remark}

Note that elements of $\mathfrak{T}_-$ in general do \textit{not} conform to the rule used to define $\mathfrak{T}$. 
Indeed they usually do not conform to any normal rule, since the second point of Definition~\ref{def normal rule} is not satisfied. We quantify this by introducing for $T=(N_T,E_T)$ 
$$\dif_T:\ N_T\to \mathbb{N}^\mathfrak{L}$$ as 
\begin{equation}\label{eq def dif}
\dif_T (n)=\begin{cases}
\re\big(\sum_{e\in E_T : \, e=\{n,e_-\}} \mathfrak{e}(e)_+- \sum_{e\in E_T : \, e=\{e_+,n\}} \mathfrak{e}(e)_-  \big) &\text{if }n\text{ is not a leaf,}\\
0 & \text{else.} 
\end{cases}
\end{equation}

\begin{remark}
Observe that if $n$ is the root of $T$, the first sum in (\ref{eq def dif}) is empty. If $n$ is neither the root nor a leaf it contains exactly one summand. We also observe that the second case in (\ref{eq def dif}) will not play a major role.
\end{remark}
\begin{remark}
Note that for any tree  $T\in {\mathfrak{T}}$ and any node $n\in N_T$ which is not the root, one has $\dif_T(n)=0$. 
\end{remark}
\subsection{Operators associated to trees}
Let us recall the notion of a multi-linear differential operator. Let $S$ be a finite set and let $\{W^s\}_{s\in S}$ and $W$ be vector bundles over $M$. A map $$\mathcal{A}: \prod_{s\in S} \mathcal{C}^\infty (W^s) \to \mathcal{C}^\infty(W)$$ is called \textit{multi-linear differential operator} of order $k$ if it factors through the $k$-jet bundle via a multi-linear bundle morphism, i.e.\ there exists
$T_\mathcal{A}\in L(\otimes_{s\in S} J^k W^s, W)$, acting on $\prod_{s\in S}J^k W^s$ in the canonical way, such that the following diagram commutes
\[\begin{tikzcd}
\prod_{s\in S} \mathcal{C}^\infty (W^s)  \arrow{r}{\mathcal{A}} \arrow[swap]{d}{j^k_{\cdot}\times\ldots\times j^k_{\cdot}}  &\mathcal{C}^\infty (W) \\
 \prod_{s\in S} \mathcal{C}^\infty( J^k W^s) \arrow{ur}{T_\mathcal{A}} &  \ .
\end{tikzcd}  
\] 
Let $\gamma$ be a permutation of the set $S$ such that $W^s= W^{\gamma(s)}$ for each $s\in S$. The operator $\mathcal{A}$ is called $\gamma$-\textit{invariant} if for all $f_s\in \mathcal{C}(W^s)$ one has
$$\mathcal{A} (\prod_{s\in S} f_s)=\mathcal{A} (\prod_{s\in S} f_{\gamma(s)}) \ .$$
This is of course equivalent to the tensor $T_\mathcal{A}$ being $\gamma$-symmetric.

Later we shall renormalise equations by associating to each $T\in \mathfrak{T}_-$ such a differential operator, but for now we introduce only the relevant spaces of operators. 
For $\sigma\in \mathbb{N}^\mathfrak{L}$, $\alpha\in \RR$ and a symmetric set $\symset= (S, i, \Gamma)$ where the type set is given by $\mathbb{N}^\mathfrak{L}$, i.e.\ $i: S\to \mathbb{N}^\mathfrak{L}$, and the index set $A_\symset$ consists only of one element, c.f.\ Remark~\ref{rem trivial index set},
we define 
\begin{equation}\label{space Dif}
\mathfrak{Dif}_\alpha\big(\symset,\sigma\big) 
\end{equation}
as the set of all multilinear differential operators of order $\max\{n\in \mathbb{N}\cup \{-\infty\} \ : n\leq-\alpha\}$ mapping $\prod_{s\in S} \mathcal{C}^\infty (V^{i(s)})$ to $\mathcal{C}^\infty(V^\sigma)$  which are $\gamma$ invariant for all $\gamma \in \Gamma$.
We use the convention that $f\mapsto 0$ is the only differential operator of order $-\infty$.

\begin{remark}\label{rem pairing differential operator}
For any operator $\mathcal{A}\in \mathfrak{Dif}_\alpha\big(\symset,\sigma\big)$ mapping $\prod_{s\in S} \mathcal{C}^\infty (V^{i(s)}) \to \mathcal{C}^\infty(V^\sigma)$ and any $\kappa \in \mathbb{N}^{\mathfrak{L}}$ we define a new multi-linear differential operator 
$$\mathcal{A}^\kappa: \prod_{s\in S} \mathcal{C}^\infty (V^{i(s)})\times \mathcal{C}^\infty (V^\kappa) \to \mathcal{C}^\infty(V^{\red_*(\sigma+\kappa)})$$
given on $f_s\in \mathcal{C}(V^s)$, $g \in \mathcal{C}(V^\kappa)$ by 
$${\mathcal{A}}^\kappa (\prod_{s\in S} f_s, g)= \Tr \big({\mathcal{A}} (\prod_{s\in S} f_s) \odot g )\big) \ .$$
We shall from now on abuse notation and simply write $\mathcal{A}$ for the operator ${\mathcal{A}}^\kappa$ as well as
$$\langle (\prod_{s\in S} f_s) \times g, \mathcal{A} \rangle := {\mathcal{A}}^\kappa (\prod_{s\in S} f_s, g).$$
\end{remark}

\begin{remark}\label{rem trivial index set}
The fact that in (\ref{space Dif}) we restricted ourselves to the case when the index set $A_\symset$ consists only of one element was purely for notational convenience and there is an obvious extension to the case $|A_\symset|>0$. 
Indeed, arguing as in \cite[Remark 5.30]{CCHS22}, one can usually work with concrete trees instead of an equivalence class of trees. We shall do so in the next definition and often throughout the text.
\end{remark}

\begin{definition}
For a tree $T$, we set 
\begin{equation}\label{def DifT}
\mathfrak{Dif}_T:= \mathfrak{Dif}_{|T|}\big( \symset_T, \ind(\rho_T)\big) \  
\end{equation}
where
$\symset_T=(S_T, i_T, \Gamma_T)$ is given by
\begin{itemize}
\item $S_T= N_T\setminus (\{\rho_T\}\cup L_T)$ ,
\item $i_T= \dif_T$ ,
\item $\Gamma_T$ consists of all tree symmetries restricted to $S_T$.
\end{itemize}
\end{definition}

Defining the type set $\mathfrak{S}=\big\{ (|T|,\symset_T, \ind \rho_T) : T\in \mathfrak{T}_-\big\}$ consisting of triples composed of the degree of a tree $T$, the isomorphism class of the symmetric set $\symset_T$ and the index of the root $\rho_T$, we introduce the vector space assignment
$$\mathfrak{D}=  \big\{\mathfrak{Dif}_{\alpha}\big( \symset,\frakt\big) \big\}_{(\alpha,\symset, \frakt)\in \mathfrak{S}} \ .$$
\begin{definition}
For a forest  $F\in \mathfrak{F}_-$ define a new symmetric set with type set $\mathfrak{S}$
$$\lll F \rrr=(\Con (F), \langle \cdot \rangle, \Per_F ) $$
where 
\begin{itemize}
\item $\Con(F)$ denotes the connected components of $F$,
\item the type map is given by restricting the map 
$$\langle \cdot \rangle: \mathfrak{T}_-\to \mathfrak{S},\qquad  T\mapsto (|T|,\symset_T, \ind \rho_T) $$
 to $\Con(F)$.
\item $\Per_F$ consists of the permutations of isomorphic elements of $\Con(F)$.
\end{itemize}
We define 
\begin{equation}\label{def DifF}
\mathfrak{Dif}_F:= \mathfrak{D}^{{\otimes} \lll F \rrr}
\end{equation}
where ${\otimes}$ is the algebraic\footnote{
The choice of tensor product is arbitrary, since we shall only work with finite products and linear combinations.
}  tensor product of vector spaces.
For later use we also define the family of vector spaces $$\mathfrak{Dif}:=\{\mathfrak{Dif}_F \}_{F\in\mathfrak{F}_-}\ .$$
\end{definition}
\begin{remark}
Observe that for a tree $T$ (which is a forest with only one connected component) the canonical isomorphism 
$$\mathfrak{D}^{{\otimes} \lll T \rrr}\simeq \mathfrak{Dif}_{|T|}\big( \symset_T, \ind(\rho_T)\big)$$
makes the notation $\mathfrak{Dif}_T$ in (\ref{def DifT}) and (\ref{def DifF}) unambiguous.
\end{remark}

\begin{remark}
These definitions imply the following.
\begin{itemize}
\item If a forest $F$ has a connected component $T$ such that $|T|> 0$ , then $\mathfrak{Dif}_F\simeq \{0\}$ .
\item If $F=\emptyset$, then $\mathfrak{Dif}_{F}\simeq \mathbb{R}$ (acting by scalar multiplication). This follows from the fact that $\mathfrak{Dif}_{F}$ is given by the image of the empty symmetric set under the functor $\Func_{\mathfrak{D}}$.
\end{itemize}
\end{remark}

\begin{remark}{\label{remark symmetrisation}}
The forest product turns $\mathfrak{F}_-$ into a commutative unital monoid, with unit given by the empty forest.

Using the notions of Remark~\ref{smaller symmetry group} 
this product induces the canonical morphism $$\Pi \in \Hom(\lll F_1 \rrr \otimes \lll F_2 \rrr , \lll F_1 F_2 \rrr )\ ,$$
which in turn induces the symmetrisation 
\begin{equation}\label{symmetrisation dif}
\Func_\mathfrak{D}\Pi:\mathfrak{Dif}_{F_1}\otimes \mathfrak{Dif}_{F_2} \to \mathfrak{Dif}_{F_1F_2} \ ,
\end{equation}
where we implicitly used the isomorphism $\mathfrak{Dif}_{F_1}\otimes \mathfrak{Dif}_{F_2} \simeq \mathfrak{Dif}^{ \lll {F_1} \rrr \otimes \lll {F_2}\rrr }$,  see Remark~\ref{rem:can_morph}.
Similarly, the canonical morphism $\mathcal{I} \in \Hom(\lll F_1 F_2 \rrr , \lll F_1 \rrr \otimes \lll F_2 \rrr )$ induces the canonical injection
\begin{equation}\label{inj dif}
\Func_\mathfrak{D}\mathcal{I}: \mathfrak{Dif}_{F_1F_2}\to \mathfrak{Dif}_{F_1}\otimes \mathfrak{Dif}_{F_2} .
\end{equation}

It is straightforward to check that \eqref{symmetrisation dif} is associative and \eqref{inj dif} is coassociative, since the underlying morphisms of symmetric sets are.
Lastly, it follows from Remark~\ref{rem:symmetrisation_injection inv} that the former is a left inverse of the latter.

\end{remark}

\begin{remark}\label{rem gluing construction first version}
The following construction will be very useful in the sequel. Let $\tau\in \mathfrak{T}_-$, and $G=\{G_n\}_{n\in N_\tau}$ be a family of trees where each $G_n \in  {\mathfrak{T}}\cup\{\bullet\}$, here $\bullet$ denotes the tree consisting of only one node. We define a new tree $\mathring{\tau}(G)$ as 
$$\mathring{\tau}(G)= \tau \cup \bigcup_{n\in N_\tau} G_n \Big / {\sim}$$ where $\sim$ identifies the root of each $G_n$ with the node $n\in\tau$. 
Sometimes a similar construction, but working with planted trees\footnote{A planted tree is a rooted tree whose root has exactly one incident edge.} exclusively is useful. In this case we write $G=\{G_{n,i}\}_{n\in N_\tau\times \mathbb{N}},$ where for each $n\in N_\tau$, for finitely many $i\in \mathbb{N}$ the tree $G_{n,i}$ is a planted tree and $G_{n,i}= \bullet$ for all other $i\in \mathbb{N}$. In this case set
$$\mathring{\tau}(G)= \tau \cup \bigcup_{n\in N_\tau, i\in \NN} G_{n,i} \Big / {\sim},$$
where $\sim$ identifies the root of every $G_{n,i}$ with the node $n\in\tau$ for each $n\in N_\tau$. 
\end{remark}

\subsection{Positive subtrees}\label{section positive subtrees}
For a tree $T$ and any edge $e\in E_T$, we denote by $T_{\geq e}\subset T$ the subtree $T_{\geq e}=(N_{\geq e}, E_{\geq e})$, where\footnote{Recall that we work with the canonial partial order on $N_T\cup E_T$, which has the root as minimal element.} $N_{\geq e}=\{ n\in N_T \ | \ n > e\}\cup e $ and $E_{\geq e}:=\{e'\in E_T \ | \ e'\geq e \}$ . Similarly, for a set of edges $E\subset E_T$ with the property that no two edges in $E$ are comparable, let $T_{\geq E}$ be the subforest of $T$ with connected components 
$$\Con(T_{\geq E})=\{T_{\geq e} \ : \ {e\in E} \}\ .$$ 
Let $\mathfrak{T}^+_T:=   \{T_{\geq e} \ : { e\in E_T: \mathfrak{e}(e)\in \mathcal{E}_+\cup\mathcal{E}_0}\} $ and write $C_+(T)$ for the set of all subsets $E\subset E_T$ satisfying the following properties: 
\begin{itemize}
\item no two edges in $E$ are comparable,
\item $\mathfrak{e}(e)\in \mathcal{E}_+\cup\mathcal{E}_0 $ for all $e\in E$,
\item $E$ contains all edges in $E_T$ of type $\mathcal{E}_0$ 
\end{itemize}
and define the set of subforests 
$$\mathfrak{F}^+_{T}:= \{F_{\geq E}\subset T \ | \ E\in C_+(T) \}\  $$
as well as
$$\mathfrak{T}_+:= \bigcup_{T\in \mathfrak{T}} \mathfrak{T}^+_T\ \Big/{\sim} \ ,$$
where $\sim$ is as in (\ref{eq def neg trees}).
For $\varepsilon\in \mathcal{E}_+\cup\mathcal{E}_0$ we set $\mathfrak{T}_+^{\varepsilon}\subset\mathfrak{T}_+$ to consist of those trees, for which the edge incident to the root is of type $\varepsilon$. 
\begin{remark}
The set $\mathfrak{T}_+$ will in general contain trees of negative degree, the plus sign is meant to signify that these trees will encode positive renormalisation. The set $C_+(T)\subset E_T$ consists of the subsets of edges of $T$ that can be ``cut''.
\end{remark}
The following definition will be useful in the sequel.
\begin{definition}\label{def sector regularity}
For $T\in \mathfrak{T}$ define
$$\underline{\alpha}(T):= \min \{ |T\setminus F| \ : F\in\mathfrak{F}_T^+ \}$$
where $T\setminus F$ denotes the tree obtained as follows. For $E\in C_+(T)$ such that $F=T_{\geq E}$ define $T\setminus F$ to be the tree obtained by replacing for each $e\in E$ the subtree $T_{\geq e}$ by a new edge of type $\iota(\mathfrak{e}(e))\in \mathcal{E}_0$, where $\iota$ is the injection from (\ref{eq iota E}).
\end{definition}

\subsection{Second homogeneity of trees}
\label{sec:homTree}

We introduce a``homogeneity''  of trees $\| T\|$.
Given $\delta_0>0$, we define a map $\| \cdot \|:\mathfrak{T}\to \mathbb{R}\cup \{+\infty\}$ recursively as follows
\begin{enumerate}
\item\label{fir} If $T$ consists of one edge of type $\mathcal{E}_0$, set $\|T\|= \delta_0$. 
\item\label{sec} If $T$ consists of one edge of type $\mathcal{E}_-$, set $\|T\|= +\infty$.
\item\label{thi} For $\mathcal{I}_{\frakk}T\in \mathfrak{T}$, set 
$\|\mathcal{I}_{\frakk}T\|=\begin{cases}
+\infty &\text{ if } \|T\|=\infty,\  |\frakk | + |T|<0\ , \\
 (|\frakk | + \|T\|)\wedge \delta_0 & \text{ else.}
\end{cases} 
$
\item\label{fou} Lastly, with the convention $\|\bullet\|=\underline{\alpha}(\bullet)= 0$, define
$$
\|T\|= \min_{\tau\ :\ \mathring{\tau}(G)=T} \big( |\tau| + \|G_{m,j}\| + \sum_{(n,i)\neq (m,j)}\underline{\alpha}(G_{m,j})  \big)$$
where the minimum is taken over $\tau\in \mathfrak{T}_-\cup \{\bullet\}$ and $G=\{G_{n,i}\}_{(n,i)\in N_\tau\times \mathbb{N}}$ as in the second half of Remark~\ref{rem gluing construction first version} satisfying $\mathring{\tau}(G)=T$, as well as over $(m,j)\in N_\tau\times \mathbb{N}$.
\end{enumerate}
Sometimes, to emphasise the dependence on $\delta_0$, we shall write $\| \cdot \|_{\delta_0}$ instead of $\| \cdot \|$.

\begin{lemma}
The function $\| \cdot \|_{\delta_0}: \mathfrak{T}\to \mathbb{R}\cup \{\infty\}$ is well defined and uniquely determined. Furthermore, for any fixed $T\in \mathfrak{T}$
\begin{equation}\label{localeq}
\lim_{\delta_0\to \infty} \|T\|_{\delta_0}= +\infty \ .
\end{equation}
\end{lemma}
\begin{proof}
The proof proceeds inductively on trees using the recursive construction from Remark~\ref{rem recursive construction}.
\begin{itemize}
\item The first two points (\ref{fir}) and (\ref{sec}) determine the value for the induction hypothesis, namely the base case when the tree consists of exactly one edge, and (\ref{localeq}) is clearly satisfied. 
\item Note that if a tree is planted $\mathcal{I}_{\frakk}T$, the third point determines the value of the function and it satisfies (\ref{localeq}) assuming $\|T\|_{\delta_0}$ is already determined and satisfies (\ref{localeq}) . Note that in this case (\ref{fou}) is an empty condition, since the only viable decomposition such that $\mathcal{I}_{\frakk}T= \mathring{\tau}(\{G_{n,i}\})$ is one where $\tau=\bullet$.
\item Lastly, if the tree is not planted, (\ref{fou}) determines its value uniquely. 
While the minimum looks like it is taken over an infinite set, for each tree $T$ there exists an $M_T\in \mathbb{N}$ (for example take $M$ to be the maximal degree of its nodes), such that it suffices to take the minimum over $\tau\in \mathfrak{T}_-\cup \{\bullet\}$ and $G=\{G_{n,i}\}_{(n,i)\in N_\tau\times \{1,\ldots,M_T\} }$. 
Thus, $(\ref{localeq})$ is satisfied since it holds for each term over which the minimum is taken.
\end{itemize}
\end{proof}

\begin{remark}
This second notion of homogeneity of trees will be used to construct the set $\triangle$ appearing in the definition of a regularity structure and to determine the precision of elements in the regularity structure. 
Condition~\ref{fir} enforces that jets always have precision $\delta_0$. Sometimes it is convenient and possible to relax this assumption, c.f.\ Remark~\ref{rem:nonstandard_polynomial}. 
\end{remark}
\subsection{Coloured trees and forests}\label{section notation for coloured trees}
In this section we first introduce a general notion of (multi) coloured trees. 
\begin{definition}
Let $C$ denote any
set, whose elements we call colours.\footnote{Recall that $ \hat{\mathcal{P}}(C)$ denotes the set of multisets of $C$ and  that we use $\sqsubset$ for the inclusion of multi-sets.} Given a tree $T$, a colouring is a map ${\hat{T}}: N_T\cup E_T\to \hat{\mathcal{P}}(C)$ such that, for each $c\in \hat{\mathcal{P}}(C)$ the set $\hat{T}_{c}:=\{x\in N_T\cup E_T\,:\, c\sqsubset \hat T(x) \}$ is a subforest of $T$.
A colouring of a forest is defined analogously.

A pair $(T,\hat{T})$, where $T$ is a (possibly typed) tree and $\hat{T}$ is a colouring on it, is called a coloured tree and similarly for forests.
\end{definition}

For a finite index set $I$ and $\{c_i\}_{i\in I}\in \hat{\mathcal{P}}(C)$, we denote by $\bigvee_{i\in I} c_i$ minimal element (with respect to the partial order $\sqsubset$ ) of $\hat{\mathcal{P}}(C)$ such that 
$c_i\sqsubset \bigvee_{i\in I} c_i$ for all $i\in I$.\footnote{Note that under the canonical bijection $\hat{\mathcal{P}}(C)\sim\mathbb{N}^C$ this agrees with the component-wise maximum function.}
 
\begin{remark}\label{trees construction}
Suppose $\bar{\mathtt{T}}$ is a family of coloured trees with some colour set $C$ and underlying set of trees $\mathtt{T}$. 
We define
\begin{equ}
\dep(T,\hat T) = \bigvee_{x\in N_T\cup E_T} \hat T(x)\;,
\end{equ}
and set for $c \in \hat{\mathcal{P}}(C)$
$$\mathtt{T}^{(c)}=\big\{ (T,\hat{T})\in \bar{\mathtt{T}} \ : \  \dep(T,\hat T)\sqsubset c \big\}\ .$$
Furthermore, for $c \in C$, we shall use the convention
$\mathtt{T}^{(c)} = \mathtt{T}^{([c])}$.
\end{remark}

\begin{remark}\label{colouring operations}
For a coloured tree $(T,\hat{T})$ and $c\in \hat{\mathcal{P}}(C)$ we introduce the following notation for readability
$$(T_c,\hat{T}):=(\hat{T}_c, \hat{T}|_{\hat{T}_c})\ .$$
Furthermore, given some coloured tree $T=( T, {\hat T})$ and a subforest $F\subset T$ such that $\hat{T}|_{F}\sqsubset c$,
 For the former
 we define a new colouring 
\[
    \hat T \sqcup_c F = \begin{cases}
    c & \text{on } F, \\
   \hat T & \text{else. } \\
\end{cases}
   	\]
   	In the particular case that $\hat{T}|_{T\setminus F}$ is the empty colouring, we simply write 
   	$\sqcup_c F$ instead of $\hat{T}\sqcup_c F$ for this new colouring.
\end{remark}
Lastly, we extend everything in this section in the canonical way to forests.
If we were to omit the condition $\hat{T}|_{F}\sqsubset c$, above, 
it might happen that ${\hat T \sqcup_c F}: N_T\cup E_T\to \hat{\mathcal{P}}(C)$ fails to satisfy for each $c'\in \hat{\mathcal{P}}(C)$ that the set $\hat{T}_{c'}:=\{x\in N_T\cup E_T\,:\, c'\sqsubset \hat T(x) \}$ is a subforest of $T$. The latter is part of the definition of a colouring.

Consider the following example: The tree $T$ consists of one edge and two nodes and has constant non-empty colour. If $F$ consist of the two nodes and $c= \emptyset$, then the new colouring $\hat{T}|_{F}\sqsubset c$ would have the property that only one edge and no nodes have non-empty colour.

%

\subsection{Coloured trees for regularity structures}\label{construction and notation for coloured trees}
In this section we apply the colouring constructions of the previous section with the set of colours $\{-,+\}$ to the trees and forests $\mathfrak{T}_-$, $\mathfrak{T}$, $\mathfrak{T}_+$ and $\mathfrak{F}_-$, $\mathfrak{F}_+$. 

We denote by $\bar{\mathfrak{T}}$, resp. $\bar{\mathfrak{T}}_+$ the set of coloured trees $(T,\hat{T})$ satisfying the following properties:
\begin{enumerate}
\item $T\in \mathfrak{T}$, resp. $T\in \mathfrak{T}_+$
\item For every $n\geq 1$, each connected component of $\hat T_{[-]^n}$ is an element of $\mathfrak{F}^-_T$
\item For every $n\geq 1$, each connected component of $\hat T_{[+]^n}$ is an element of $\mathfrak{F}^+_T$
\item If $\hat T_{[+]}\neq \emptyset$, it contains all edges of type $\mathcal{E}_0$. 
\item\label{last condition above} For every $n\in \mathbb{N}$, each connected component $S$ of $\hat T_{[-]}$ satisfies either $S\subset\hat{T}_{[+]^n}$ or $S\cap\hat{T}_{[+]^n}=\emptyset$. 
\end{enumerate}

We denote by $\bar{\mathfrak{T}}_-$ the set coloured trees $(T,\hat{T})$ satisfying the following properties:
\begin{enumerate}
\item $T\in \mathfrak{T}_-$,
\item $\hat T_{[+]}= \emptyset$,
\item For every $n\in \mathbb{N}$, each connected component of $\hat T_{[-]^n}$ is an element of $\mathfrak{T}_-$ .
\end{enumerate}

%

\begin{remark}
Here, we shall actually never use the full sets $\bar{\mathfrak{T}}$ and $\bar{\mathfrak{T}}_+$ but only specific subsets thereof.
The reason we still introduce them, is that they can be used to write down explicit (non-recursive) formulae for renormalised models similarly to \cite{CH16}.
\end{remark}

\section{Models and the Renormalisation Group}\label{section regularity structure, models,...}
Finally, we turn to the task of constructing in the following order:
\begin{itemize}
\item structure spaces $\mathcal{T}^\sigma$ for the regularity structure ensemble,
\item a space of pre-models $\pmb{\Pi}$,
\item a renormalisation group $\mathfrak{G}_-$ acting on pre-models,
\item finally, a regularity structure ensemble and model $Z=(\Pi, \Gamma)$ built from certain pre-models $\pmb{\Pi}$.
\end{itemize}
The sets of trees $\mathfrak{T}$, $\mathfrak{T}_+$, $\mathfrak{T}_-$ and their coloured variants will serve as building blocks. 
\begin{assumption}\label{assumption:truncation}
We are working with a finite edge set $\mathcal{E}$. The rule $R$ from which the trees are constructed is subcritical, normal and equation-like. Furthermore, we fix a $\delta_0>0$.
\end{assumption}

\begin{remark}
Let us though mention the following features of this assumption:
\begin{itemize}
\item
Sub-criticality is needed for the definition of the renormalisation group below.
\item The rule being equation like has the advantage that one obtains finite families of vector bundles. This is strictly speaking not necessary, but  simplifies exposition.
\item Normality and a choice of $\delta_0>0$ is only needed in order to define the maps $\Gamma_{p,q}$ later. The role of $\delta_0$ will be the same as in Section~\ref{sec:homTree}.
\end{itemize}
\end{remark}

\subsection{Symmetric sets and Vector bundles from trees}\label{sec:symmetrics_sets and vect. from trees}
Recall from Section~\ref{section indexing vector bundels} that we extend any vector bundle assignment $(V^\frakt)_{\frakt\in \mfL}$ to a vector bundle assignment $(V^\sigma)_{\sigma\in \mathbb{N}^\mfL}$. For now we work with the set of types $\Lab=\NN^{\mfL}$ as labels for symmetric sets. Throughout this section we shall work with the induced vector bundle assignment $(W^\sigma)_{\sigma\in \mathfrak{S}}$ given by 
\begin{equation}\label{eq:jet_assignment1}
W^\sigma=(JV^\sigma)_{<\delta_0}\ .
\end{equation}
We assign to each $T\in \mathfrak{U}$ (equivalence class of a (coloured) labelled rooted tree) a symmetric set $\symset=\langle T \rangle$, where 
the index set $A_\symset$ consists of the elements of the equivalence class $T$ itself,\footnote{As in \cite{CCHS22} the set of nodes can be fixed to be a subset of the natural numbers, to make this well defined.} For each $t\in T$, the set 
$$ T^t_\symset:= \{e\in E_t \ : \ \mathfrak{e}(e)\in \mathcal{E}_0 \}  $$ 
is given by edges of type in $\mathcal{E}_0$, the maps $\Gamma_{\symset}$ in the definition of symmetric sets consist of tree isomorphisms restricted to the sets $T^t_\symset$ and the type map is given by $\mathfrak{t}^t_\symset (e):= \mathfrak{e}(e)_-$ . 

\begin{definition}
For for a set of trees $\mathfrak{U} \subset \bar{\mathfrak{T}}$ or $\mathfrak{U}  \subset \bar{\mathfrak{T}}_+$ 
let 
$$(\mathfrak{U})_{<\delta_0}:= \left\{ T\in \mathfrak{U} \ : \ |T'|<\delta_0 \text{ for every subtree } T'\subset T \right\}$$
and set
$$\pmb{\mathcal{U}}= \bigoplus_{T\in (\mathfrak{U})_{<\delta_0}} \langle T \rangle\ , \quad
{\mathcal{U}}= \Func_{ W} \pmb{\mathcal{U}} = \bigoplus_{T\in (\mathfrak{U})_{<\delta_0}} \Func_{ W}\langle T \rangle \ .$$
\end{definition}
For example for $\nu \in \mathcal{N}$, $\varepsilon\in {{\mathcal{E}_+\cup \mathcal{E}_0}}$ and $\mathfrak{U}\in \{ \mathfrak{T}^{\nu},\mathfrak{T}^{\nu,(-) }, \mathfrak{T}^\varepsilon_+\}$, we 
denote the thus constructed spaces by ${\mathcal{T}}^\nu, {\mathcal{T}}^{\nu,(-) }, {\mathcal{T}}^\varepsilon_+$.
With these definitions note that for
 $\sigma \in \NN^\mathfrak{L}$, 
$$ \mathcal{T}^{\sigma}= \bigoplus_{\nu\ : \ \ind(\nu)=\sigma}\mathcal{T}^\nu \ , \qquad
 \mathcal{T}^{\sigma, (-)}= \bigoplus_{\nu\ : \ \ind(\nu)=\sigma}\mathcal{T}^{\nu, (-)} \ .
 $$
 In particular all these spaces are finite dimensional by Assumption~\ref{assumption:truncation}.
 \begin{definition}
Finally, set
$$\mathcal{T}:= \{ \mathcal{T}^{\sigma}\}_{\sigma\in\Lab} \ ,\quad \mathcal{T}^{(-)}:= \{ \mathcal{T}^{\sigma,(-)}\}_{\sigma\in \Lab}\ , \quad \text{ and } \quad\mathcal{T}_+:= \{ \mathcal{T}_+^{\varepsilon}\}_{{\varepsilon\in \mathcal{E}_+\cup \mathcal{E}_0}}\ .$$
\end{definition}

\begin{remark}
If a trees $T$ contains no edges of type $\mathcal{E}_0$, then $\Func_{W}\langle T \rangle $ is a trivial one dimensional vector bundle and we often single out a non-vanishing section which we declare to be canonical. 
This will be most frequently done for trees of the form $T=\Xi^{\varepsilon}$ for some $\varepsilon \in \mathcal{E}_-$. We shall use a subscript to denote $\Xi^{\varepsilon}_p \in \Func_{W}\langle \Xi^{\varepsilon} \rangle$, this canonical section evaluated at $p\in M$.
\end{remark}
\begin{remark}
For $\varepsilon \in \mathcal{E}_0$ we write $\delta^{\varepsilon}$ for the tree consisting of exactly one edge of type $\varepsilon$. We set ${J}^{\varepsilon}=\Func_{W} \langle \delta^{\varepsilon} \rangle $ and note it is canonically isomorphic to $(JV^{\varepsilon_-})_{<\delta_0}$. Furthermore, we write
 $  J:=\{ J^{\varepsilon} \}_{\varepsilon \in \mathcal{E}_0}$.
\end{remark}

\begin{definition}\label{def prod and adding edge}
The tree product $\mathcal{M}$ as well as the operation $\mathcal{I}_{\mathfrak{k}}$ extend to morphisms of symmetric sets and thus define corresponding bundle morphisms via the map $\Func_W$. We keep the convention of abbreviating $\Func_W\mathcal{M}$ by ``$\;\cdot\,$'', as well as simply writing $\mathcal{I}_\frakk$ for the bundle morphism
$$\Func_W\mathcal{I}_\mathfrak{k}: \Func_W \langle T \rangle \to {\mathcal{T}}^\frakk \ , $$
whenever $T\in \mathfrak{T}^{(-)}$ is such that $\mathcal{I}_{\mathfrak{k}} ({T})\in ( \mathfrak{T}^{(-)})_{<\delta_0}$.
\end{definition}
Note that we implicitly used the morphism from Definition~\ref{def:morphisms_for_colouring} in the above definition.
\subsection{Grading} 
 In this section we introduce the appropriate bi-grading on the vector bundles ${\mathcal{T}}^\nu$ for $\nu\in \mathcal{N}$.
First, define the new set of types $\bar \Lab= \Lab\times \NN_{<\delta}$ and the type decomposition 
\begin{equation}\label{eq:jet_type_decomposition}
\bar \Lab=\Lab\times \NN_{<\delta_0}  \ni(\sigma,n) \tto \sigma \in \Lab 
\end{equation}
 as well as the space decomposition 
$$(JV^\sigma)_{<\delta_0}= \bigoplus_{n=0}^{\lfloor \delta_0 \rfloor} (JV^\sigma)_{n} \quad \text{ i.e.\ } \quad \bar W^{(\sigma,n)}=  (JV^\sigma)_{n} \ .$$
For $\pmb{\tau} \in L_{\langle T \rangle}$, where $L_{\langle T \rangle}$ was introduced above (\ref{decomp}),  and $\symset_{\pmb{\tau}}$ the corresponding symmetric set, let
$$|\pmb{\tau}|=|T|+ \sum_{e\in T^a_{\symset_{\pmb{\tau}}}} n(e),$$
which is clearly independent of $a\in A_{\symset_{\pmb{\tau}}}$ .
Thus for $\mathfrak{U}\subset \bar{\mathfrak{T}}$ one can set $$A_{\mathfrak{U}, \delta_0}= \{ (|\pmb{\tau} |, \| T\|_{\delta_0}) \colon \pmb{\tau}\in Y_{\langle T \rangle}, T\in \mathfrak{U} \} \ .$$ 
Since
$$\Func_{W}(\langle T \rangle)=  \Func_{\bar W} \circ \proj^*\langle T \rangle=\Func_{\bar W} \bigoplus_{\pmb{\tau}\in Y_{\langle T \rangle}} \symset_{\pmb{\tau}}$$
we find that $\mathcal{U}= \bigoplus_{(\alpha,\delta)\in A_{\mathfrak{U},\delta_0 }} \mathcal{U}_{\alpha,\delta}$ is a bi-graded vector bundle.

Finally, we set $A$ to consist of all $\alpha \in \mathbb{R}$ such that there exists a pair $(\alpha, \delta)\in A_{\mathfrak{T}, \delta_0}$ and similarly $\triangle$ to consist of $\delta \in \mathbb{R}$ such that there exists $(\alpha, \delta)\in A_{\mathfrak{T},\delta_0}$.

%

\begin{remark}
Note that for the induced decomposition $J^\varepsilon = \bigoplus_{\alpha,\delta} (J^\varepsilon)_{\alpha,\delta}$ one has $\delta=\delta_0$ for all summands. As already mentioned, sometimes one can ``by hand'' increase this value to $+\infty$ for some subspace, c.f.\ Remark~\ref{rem:nonstandard_polynomial}.
\end{remark}

\subsection{Construction of the maps $\pmb{\Pi}$}

We fix kernel assignments $ \{K_{\mathfrak{k}}\}_{\mathfrak{k}\in \mathcal{E}_+}$, where $K_{\mathfrak{k}}$ is a kernel from $V^{\mathfrak{k}_+}$ to $V^{\mathfrak{k}_-}$ that satisfies Assumption~\ref{Assumption on Kernel} with $\beta = |\mathfrak{k}|$ for each $\mathfrak{k}\in \mathcal{E}_+$ . Furthermore, fix an admissible realisation $R$ of $J$, i.e.\ admissible realisations of $JV^{{\sigma}}$ for each $\sigma\in \mathbb{N}^{\mathfrak{L}}$.

\begin{definition}\label{def admissible bfPi}
We call a map $\pmb{\Pi} : \mathcal{T}^{(-)}\to L(\mathfrak{Dif}, \mathcal{C}^\infty (V))$
admissible, if the following properties are satisfied:
\begin{enumerate}
\item For each $\sigma\in \Lab$ the restriction of $\pmb \Pi$ yields a linear map 
$$ \mathcal{T}^{\sigma,(-)} \to L(\mathfrak{Dif}, \cC^\infty(V^{\sigma}))$$
and for each $\tau\in \Func_{W}(\langle (T,\hat{T}) \rangle)\subset \mathcal{T}^{(-)}$ one has $\pmb{\Pi}\tau\in  L\big(\mathfrak{Dif}_{T_-} ,\  \cC^\infty(V^{\ind(\rho_T)})\big)$.
\item For each $\tau\in J$,  $\pmb{\Pi}\tau= R\tau$.
\item For each $\mathfrak{k}\in \mathcal{E}_+$, $\pmb{\Pi}\mathcal{I}_\mathfrak{k} \tau = K_{\mathfrak{k}}(\pmb{\Pi} \tau)$.
\end{enumerate}
\end{definition}
\begin{remark}
Note that since $\mathcal{T}^{(-)}$ is a finite family of vector bundles, we can write\linebreak $\pmb{\Pi}\in L\left(\mathcal{T}^{(-)}, \bar{L}(\mathfrak{Dif}, \cC^{\infty}(V) )\right)$ where $\mathcal{T}^{(-)}=\{\Func_{ W}\langle T \rangle \}_{T\in \mathfrak{T}^{(-)}}  $ and 
$\bar{L}(\mathfrak{Dif}, \cC^{\infty}(V))=\linebreak \{L(\mathfrak{Dif}_{T_-} ,\  \cC^\infty(V^{\ind(\rho_T)})\}_{T\in \mathfrak{T}^{(-)}} $.
\end{remark}

\begin{remark} 
Due to the canonical isomorphisms $L(\mathbb{R}, \cC^\infty(V^\sigma))\simeq \cC^\infty(V^\sigma)$ and $\mathfrak{Dif}_{\emptyset}\simeq\mathbb{R}$ combined with the fact that $\hat{T}_-=\emptyset$ for each $T\in \mathfrak{T}$, we can naturally see $\pmb{\Pi} |_{\mathcal{T}}$ as an element of $L(\mathcal{T},\mathcal{C}^\infty
(V))$, i.e.\ for each $\sigma \in \Lab$
$$\pmb{\Pi} |_{\mathcal{T}^\sigma}: \mathcal{T}^\sigma \to  \cC^\infty(V^\sigma) \ .$$ 
It turns out that $\pmb{\Pi} |_{\mathcal{T}}$ encodes all the information required to define a model. The reason we defined $\pmb{\Pi}$ on $\mathcal{T}^{(-)}$ and not just on $\mathcal{T}$ is to be able to define the action of the renormalisation group. 
This is related to what is done in \cite{BHZ19} by taking the completion of the rule. 
\end{remark}

\begin{remark}\label{remark attaching trees to negative trees}
We describe a construction of coloured trees, building on the one in Remark~\ref{rem gluing construction first version} for uncoloured trees. Denote by $\mathring{\mathfrak{T}}^{(-)}\subset \mathfrak{T}^{(-)}$ trees with uncoloured roots.
\begin{enumerate}
\item Let $\tau\in \mathfrak{T}_-$, and $G=\{G_n\}_{n\in N_\tau}$ be a family of trees where each $G_n \in  \mathring{\mathfrak{T}}^{(-)}\cup\{\bullet\}$, here $\bullet$ denotes the tree consisting of only one (uncoloured) node. We define a new coloured tree $\tau(G)=(\mathring{\tau}(G), \hat{\tau}(G))$ as follows:
\begin{itemize}
\item The tree $\mathring{\tau}(G)$ is as in Remark~\ref{rem gluing construction first version}, i.e.
$$\mathring{\tau}(G)= \tau \cup \bigcup_{n\in N_\tau} G_n \Big / {\sim}$$
where $\sim$ identifies the root of each $G_n$ with the node $n\in N_\tau$. 
\item The colouring is given by
\[
    \hat{\tau}(G) = \left\{ \begin{array}{llr}
   		[-] , & \text{on } & \tau\subset \mathring{\tau}(G) \\
        \hat G_n, & \text{on } & G_n\setminus \rho_{G_n}\subset \mathring{\tau}(G_n)\\
        \end{array}\right\} \ .
  \]
\end{itemize} 
\item We define for each $\tau\in \mathfrak{T}_-$ a partial map $$\tau: (\mathring{\mathfrak{T}}^{(-)}\cup \{\bullet\})^{N_\tau}\nrightarrow \mathfrak{T}^{(-)}\ $$ where $G=\{G_n\}_{n\in N_\tau}$ is in the domain $\mathfrak{U}_\tau$ of this map if $\tau (G)\in (\mathfrak{T}^{(-)})_{<\delta_0}$. 
Note that this partial map extends canonically to symmetric sets by first considering the tensor product
 symmetric set $\symset_G= \otimes_{n\in N_\tau}\langle G_n\rangle $ as in Remark~\ref{rem:disj_union_symsets} and considering the canonical extension of the the map $\tau$ to $\symset_G$ as in Remark~\ref{smaller symmetry group}.
\item 
We define a map
 \begin{equation}\label{eqeq}
\tau\in L\Big(  \Big\{\bigotimes_{n\in N_\tau} \Func_W \langle G_n\rangle\Big\}_{G\in \mathfrak{U}_\tau}, \mathcal{T}^{(-)} \Big) \ ,
 \end{equation}
which by abuse of notation is denoted by the same symbol. 
It acts on the space $\bigotimes_{n\in N_\tau} \Func_W \langle G_n\rangle$
as the following composition
 \[ \begin{tikzcd}
\bigotimes_{n\in N_\tau} \Func_W \langle G_n\rangle
\arrow{r}
& \Func_W \symset_G  \arrow{r}{\Func_W\tau} 
& \mathcal{T}^{\ind(\mathfrak{n}(\rho_\tau)), (-)}\ ,
\end{tikzcd}
\] 
where the first arrow denotes the canonical isomorphism, c.f.\ Remark~\ref{rem:can_morph}.
\end{enumerate}
Given a tree with coloured root we can decompose it.
\begin{enumerate}
\item Let $T\in \mathfrak{T}^{(-)}\setminus \mathring{\mathfrak{T}}^{(-)}$, observe that the connected component of the root in $T_{-}$ is the unique
$\tau\in \mathfrak{T}_-$ such that for some $G=\{G_n\}_{n\in N_\tau}\in \mathring{\mathfrak{T}}^{(-)}\cup \{\bullet\}$ one has 
$$T= \tau(G) \ .$$ We write $T\setminus\tau$ for the sub forest $\bigcup_n G_n \subset T$ and also use the notation $(T\setminus\tau)_{\geq n} := G_n$.
Here $\tau$ is the connected component of the root in $T_{-}$.
\item\label{last item} Note that for any $\eta\in \Func_{ W}\langle T \rangle $  , where $T\in \mathfrak{T}^{(-)}\setminus \mathring{\mathfrak{T}}^{(-)}$ there exist $\tau\in \mathfrak{T}_-$ and $G=\{G_n\}_n\in \mathfrak{U}_\tau$, as well as $\eta^n\in \Func_{W} \langle G_n\rangle$ such that 
$$ \eta = \tau (\mathcal{H} ) \ ,$$
where we write $\mathcal{H}=\{\eta^n\}_{n\in N_{\tau}}$. (Note that this decomposition is far from unique.)
\end{enumerate}
\end{remark}
%

\begin{definition}
We call $\pmb{\Pi}$ an admissible lift\footnote{Note that $\varepsilon\in \mathcal{E}_-$ is \textit{not} a regularisation scale.} of $\{\xi^{\varepsilon}\}_{\varepsilon\in \mathcal{E}_-}$,
where each $\xi^{\varepsilon}\in \mathcal{C}^\infty(V^{{\varepsilon}_{-}})$, if the map $\pmb{\Pi}$ is admissible and furthermore satisfies for all $ \varepsilon\in \mathcal{E}_-$, $$\pmb{\Pi}(\Xi^{\varepsilon})= \xi^{\varepsilon}\ .$$
We call the map $\pmb{\Pi}$ the canonical lift of $\{\xi^{\varepsilon}\}_{\varepsilon\in \mathcal{E}_-}$,
if additionally the following conditions are satisfied.
\begin{itemize}
\item For $\eta^{1},\ldots,\eta^n \in \mathring{\mathcal{T}}^{(-)}$ such that $\eta^{i}\in \Func_{W}(T^i)$ for planted $T^i\in \mathring{\mathfrak{T}}^{(-)}$ and writing $T=T^1\cdots T^n\in \mathfrak{T}^{(-)}$, 
one has
\begin{equation}\label{eq:canonical_bfPi1}
\pmb{\Pi}\big(\prod_{i=1}^n \eta^i \big) = \Tr\big( {\bigodot_{i=1}^n }  \pmb{\Pi}(\eta^i) \big) \  ,
\end{equation}
where the former product $\prod$ was introduced in Definition~\ref{def prod and adding edge} and the latter is the product canonically induced by $\bigodot$ from Remark~\ref{rem canoncial product on sections} 
$$\bigotimes_{i=1}^n L\big(\mathfrak{Dif}_{T^i_-} ,\  \cC^\infty(V^{\ind (\rho_{T^i})} )\big)\to L\big(\mathfrak{Dif}_{{T}_-} ,\  \cC^\infty(V^{\sum_i\ind (\rho_{T^i})})\big) \ ,$$
which is well defined using the canonical identification $T_- \simeq \bigcup_i (T^i)_- $.
\item 
 For every $\eta\in \mathcal{T}^{(-)}$ of the form $\eta = \tau_0 (\mathcal{H}) $ as in Point~\ref{last item} of Remark~\ref{remark attaching trees to negative trees} and $\mathcal{A}\in \mathfrak{Dif}_{T_-}$ of the form $\mathcal{A}= \otimes_{\tau \in \Con(T_-)} \mathcal{A}_\tau $, c.f.\ Remark~\ref{remark symmetrisation}, one has 
\begin{equation}\label{coloured root canonical bfpi}
\pmb{\Pi} \eta (\mathcal{A})= \mathcal{A}_{\tau_0}\Big( \prod_{n\in N_{\tau_0}} \pmb{\Pi}\eta_n \big(\otimes_{\tau_j \in \Con(T_n)_-} \mathcal{A}_{\tau_j}\big)(\cdot_{n})\Big) \ ,  
\end{equation}
where we canonically set $\pmb{\Pi}{\bullet}=1$ and use $\cdot_n$ to emphasise the variable corresponding to $n\in N_{\tau_0}$ and where 
$\mathcal{A}_{\tau_0}\in \mathfrak{Dif}_T$ acts on
$$\prod_{N_{\tau_0}\setminus (\rho_{\tau_0}\cup L_{\tau_0})} \cC^\infty(V^{\dif_{\tau_0}(n)} )\times \cC^\infty\left(V^ \sigma\right)$$ for 
$\sigma=\ind (\rho_{G_{\rho_{\tau_0}}}) $ as described in Remark~\ref{rem pairing differential operator}. 
\end{itemize}
\end{definition}

\begin{remark}
Note that the restriction of the canonical lift $\pmb{\Pi}$ of smooth noises to an element of $L(\mathcal{T}, \cC^\infty(V))$ is not quite multiplicative with respect to the products on the domain and target, c.f.\ Definition~\ref{def prod and adding edge} and Remark~\ref{rem canoncial product on sections} due to the operator $\Tr$ in \eqref{eq:canonical_bfPi1}, c.f.\ Remark~\ref{remark incompatibilit product}.
\end{remark}

The following definition, while not crucial for the arguments, will significantly simplify notation and is closely related to preparation maps and recursive renormalisation as introduced in \cite{Bru18} and used in the recent work \cite{BB21} for related purposes.
\begin{definition}\label{def:mathringbfPi}
For any admissible map $\pmb{\Pi} : \mathcal{T}^{(-)}\to L(\mathfrak{Dif}, \mathcal{C}^\infty (V))$, we define an induced map 
$\mathring{\pmb{\Pi}} : \mathcal{T}\to \mathcal{C}^\infty (V)$ as the unique map satisfying the following properties:
\begin{itemize}
\item for each planted $T \in \mathfrak{T}$, one has $$\mathring{\pmb{\Pi}}|_{\Func_{ W}\langle T \rangle} ={\pmb{\Pi}}|_{\Func_{ W}\langle T \rangle}$$
\item For $\eta^{1},\ldots,\eta^n \in \mathring{\mathcal{T}}$ such that $\eta^{i}\in \Func_{V}(T^i)$ where each $T^i$ is planted,
one has
$$ \mathring{\pmb{\Pi}}\big(\prod_{i=1}^n \eta^i \big) = \Tr\big( \bigodot_{i=1}^n  \mathring{ \pmb{\Pi}}(\eta^i) \big) \ .$$
\end{itemize}
\end{definition}

\subsection{A renormalisation group}\label{sec:renormalisation group}

In this section we shall construct a renormalisation group.
We define $\mathfrak{G}_{-}$ to consist of all maps $g:  {\mathfrak{F}}_-^{(-)} \to L (\mathfrak{Dif},\mathfrak{Dif})$ satisfying the following properties:
\begin{itemize}
\item for each $(F,\hat{F})\in{\mathfrak{F}}_-^{(-)}$,  $g(F,\hat{F}) \in L(\mathfrak{Dif}_{F_-}, \mathfrak{Dif}_F) ,$
\item $g(\emptyset)=1$,
\item if $(T,\hat{T})\in \mathfrak{T}_-^{(-)}$ is such that $\hat{T}_-= T$, then $g(T)= \text{id}$
\item $g$ is multiplicative with respect to the forest product on ${\mathfrak{F}}_-^{(-)}$ in the sense that the following diagram commutes for any $(F^1,\hat{F}^1), (F^2,\hat{F}^2)\in \mathfrak{F}_-^{(-)}$ and $(F,\hat{F})=(F^1,\hat{F}^1)\cdot (F^2,\hat{F}^2)$

\begin{equation}\label{diagram multiplicativity of ren group}
\begin{tikzcd} 
\mathfrak{Dif}_{\hat{F}_-^1}\otimes \mathfrak{Dif}_{\hat{F}_-^2} \arrow{r}{g(F^1,\hat{F}^1)\otimes g(F^2,\hat{F}^2)} &[6em] \mathfrak{Dif}_{\hat{F}^1}\otimes\mathfrak{Dif}_{\hat{F}^2} \arrow{d}\\
\mathfrak{Dif}_{\hat{F}_-}\arrow{u}   \arrow{r}{g(F,\hat{F})}  & \mathfrak{Dif}_{{F}}  \ ,
\end{tikzcd}  
\end{equation}
where the upwards pointing arrow is the canonical injection $\Func_\mathfrak{D} \mathcal{I}$ and the downwards pointing arrows is the canonical symmetrisation map $\Func_\mathfrak{D}\Pi$ from Remark~\ref{remark symmetrisation}.
\end{itemize}

\begin{remark}
For $(F,\hat{F}) \in \mathfrak{F}^{(-)}$ write $\Vec(F,\hat{F})$ for the one dimensional vector space\linebreak generated by $(F,\hat{F})$ and let 
$\mathcal{F}^{(-)}= \{ \Vec(F,\hat{F})\}_{(F,\hat{F}) \in \mathfrak{F}^{(-)}}$ as well as
$\bar{L}(\mathfrak{Dif}, \mathfrak{Dif}) =\linebreak \{ L(\mathfrak{Dif}_{F_-}, \mathfrak{Dif}_F ) \}_{(F,\hat{F}) \in \mathfrak{F}^{(-)}}$.
Then, we can interpret $\mathfrak{G}_- \subset L\big(\mathcal{F}^{(-)}, \bar{L}(\mathfrak{Dif}, \mathfrak{Dif} ) \big)$ . 
\end{remark}
Next we define the group structure on $\mathfrak{G}_-$.
\begin{definition}
For $f,g \in \mathfrak{G}_-$ and $(T,\hat{T})\in {\mathfrak{T}}_-^{(-)}$, set
\begin{align*}
(f\star g) (T,\hat{T})&= \sum_{F\in \mathfrak{F}^-_T, \hat{T}_-\subset F}  f(T,\sqcup_- F)\circ g(F, \hat{T}) ,
\end{align*}
where $\sqcup_- F$ was introduced in Remark~\ref{colouring operations}.
\end{definition}

%

\begin{prop}\label{Prop Assiciativity}
Let $f,g, h \in \mathfrak{G}_-$, then $f\star (g\star h)= (f \star g)\star h$, 
\end{prop}
\begin{proof}
Indeed, we find for $f,g,h\in \mathfrak{G}_-$ and $(T,\hat{T})\in \mathfrak{T}_-^{(-)}$
\begin{align*}
\big(f\star (g\star h)\big) (T, \hat{T}) &= \sum_{\tilde{F}\in \mathfrak{F}_T^{-}, \hat{T}_-\subset \tilde{F}}
f (T, \sqcup_- \tilde{F})\circ  (g\star h)(\tilde{F},\hat{T}) \\
&= \sum_{F,\ \tilde{F}\in \mathfrak{F}_T^{-}: \hat{T}_-\subset F\subset \tilde{F}} f(T, \sqcup_- \tilde{F}) \circ g(\tilde{F}, \sqcup_-{F})\circ h(F,\hat{T})\\
&=  \sum_{{F}\in \mathfrak{F}_T^{-}:\hat{T}\subset {F}} (f\star g) (T, \sqcup_- {F}) \circ h( {F}, \hat{T})\\
&=\big((f\star g)\star h\big) (T, \hat{T})
\end{align*}
\end{proof}

\begin{prop}\label{prop group structure}
Let $e\in \mathfrak{G}_-$ be the unique element satisfying $e( \emptyset)=1$ and
$$ e(T,\hat{T})=\begin{cases} 
\id_{\mathfrak{Dif}_T} & \text{if } T_-=T \ ,\\
0 & \text{else.} 
\end{cases}$$
Then $(\mathfrak{G}_-,\ \star, e)$ forms a group with $e$ as the unit.

\end{prop}
Before proceeding with the proof, we introduce for a coloured tree $(T,\hat{T})\in \mathfrak{T}_-^{(-)}$ and subforests $F_1,F_2\subset T$ such that
$\hat{T}_-\subset F_1$ and $\hat{T}_-\subset F_2$ the notation 
$$F_1\ll_{(T,\hat{T})} F_2$$ to mean 
\begin{itemize}
\item $F_1\subset F_2$ and $F_1\neq F_2$
\item for each connected component of $A$ of $F_1$ there exists a connected component $B$ of $F_2$ such that
$A\subset B$ and $A=B$ only if $B$ is a connected component of $\hat{T}_-$.
\end{itemize}

We also introduce the notion of \textit{potential depth} $m$ of a tree $(T,\hat{T})\in \mathfrak{T}_-^{(-)}$.
If $\hat T_{[-]}= T$ we set $m=0$. Else $m$ is the maximal number, such that there exists a  colouring ${\tilde{T}}$ of $T$ with colour set $C=\{-\}$, such that $(T, \tilde{T})\in \bar{\mathfrak{T}}_- $ satisfies the following properties:
\begin{itemize}
\item $\hat T_{-}=\tilde T_{[-]^m}$ (which could be the empty forest),
\item for each $n\in \mathbb{N}$, $\tilde T_{[-]^{n}}\in \mathfrak{F}^-_T$ ,
\item for each $0<n< m$ it holds that $\tilde{T}_{[-]^n}\neq \emptyset$, 
\item for each $0<n\leq m$ it holds that $\tilde T_{[-]^{n}} \ll_{(T,\hat{T})} \tilde{T}_{[-]^{n-1}}$ .
\end{itemize}

\begin{remark}
While we will not need this here, the notion of potential depth naturally extends to forests. Then it is natural to set the potential depth of the empty forest to be $0$ and the potential depth of a non-empty forest to be the maximal potential depth of its connected components. 
\end{remark}

\begin{remark}
Let $(T, \hat{T})\in  \mathfrak{T}_-^{(-)}$ be a coloured tree s.t.\ $\hat{T}_{[-]}\neq \emptyset$ and which has potential depth $k$. Then $(T, 0)$ has potential depth $>k$.
\end{remark}

\begin{remark}\label{remark potential depth forest interval}
Note that for any coloured $(T, \hat{T})\in \mathfrak{T}_-^{(-)}$ tree such that $\hat{T}_-\neq T$, we have the following disjoint union
$$ \{F\in \mathcal{F}_T^- : \hat{T}_-\subset F\} = \{\hat{T}_-\} \cup \{T\} \cup \{F\in \mathcal{F}_T^- : \hat{T}_-\ll_{(T,\hat{T})} F\ll_{(T,\hat{T})} T\}$$ 
where the last term is empty if and only if $(T, \hat{T})$ has potential depth $1$. 
\end{remark}

\begin{remark}\label{remark potential depth}
Let $(T,\hat T)$ have potential depth $n>1$ and $F\in \mathfrak{F}^-_T$ be such that $\hat{T}_-\ll_{(T,\hat{T})} F\ll T$ (which implies that $F\neq\hat{T}_-$ and $F\neq T$), then the tree $(T,\sqcup_- F)$ has potential less than $n$. 
\end{remark}

\begin{proof}[Proof of Proposition~\ref{prop group structure}]
It is clear that the map $e$ is a unit, as one checks directly. Associativity of $\star$ was observed in the previous proposition. 
Now we define for any $g\in \mathfrak{G}_-$ a new element $\mathcal{A}g\in\mathfrak{G}_-$.
\begin{itemize}
\item Define $\mathcal{A}g$ for the empty forest, as well as trees of potential depth $0$ as imposed in the definition of $\mathfrak{G}_-$.
\item If $(T,\hat T)$ has potential depth $n\geq 1$, define recursively
$$\mathcal{A}g(T,\hat T)= - g(T,\hat T) - \sum_{F\in \mathfrak{F}^-_T, T_-\ll_{(T,\hat{T})} F\ll_{(T,\hat{T})} T} \mathcal{A}g (T,\sqcup_- F)\circ g(F,\hat{T}) \ $$
which is well defined since
\begin{enumerate}
\item if $n=1$ the sum is empty by Remark~\ref{remark potential depth forest interval},
\item if $n>1$ by Remark~\ref{remark potential depth} each coloured tree $(T,\sqcup_- F)$ in the sum above has potential less than $n$.
\end{enumerate}
\item Extend this definition multiplicatively in the sense of (\ref{diagram multiplicativity of ren group}), which fully defines $\mathcal{A}g\in \mathfrak{G}_-$.
\end{itemize}
Next we check that $\mathcal{A}g$ defined as above is indeed an inverse, clearly we have $(\mathcal{A}g \star g) (T, [-])= e(T,[-])$ and for other trees 
\begin{align*}
\mathcal{A}g \star g (T, \hat{T}) &= \sum_{F\in \mathfrak{F}^-_T, \hat{T}_-\subset F}  \mathcal{A}g(T,\sqcup_- F)\circ g(F, \hat{T})\\
&=\mathcal{A}g (T, \hat{T})  + g (T, \hat{T}) +  \sum_{F\in \mathfrak{F}^-_{(T,\hat{T})}, \hat{T}_-\ll_{(T,\hat{T})} F\ll_T}  \mathcal{A}g(T,\sqcup_- F)\circ g(F, \hat{T})\\
&=0 \ 
\end{align*}
as required. Finally $\mathcal{A}g$ is a right inverse by standard arguments about groups.
\end{proof}
From now on for $g\in \mathfrak{G}_-$ and $(F,0)\in \mathfrak{F}_-^{(-)}$ we shall abbreviate $g(F,0)$ by $g(F)$.
\begin{remark}
In this remark we exhibit a subgroup $\mathring{\mathfrak{G}}_-$ of ${\mathfrak{G}}_-$, which is closer to the renormalisation group of \cite{BHZ19}, but still allows for ``renormalisation functions''.
We define $\mathring{\mathfrak{G}}_-$ to consist of all elements $g\in \mathfrak{G}_{-}$, such that $g(T,\hat{T})=0$ if $|\hat{T}_-|<|T|$ .
If we define $|(T,\hat{T})|_-= |T|-|\hat{T}_-|$, we see that this is equivalent to $g(T,\hat{T})=0$ whenever $|(T,\hat{T})|_->0$.
Since for any $F\in \mathfrak{F}^-_T$ such that $\hat{T}_- \subset F$
one has 
$$|(T,\hat{T})|_- = |(T,\sqcup_- F)|_- + |(F, \hat{T}) |_- \ ,$$
it easily follows that $\mathring{\mathfrak{G}}_-$ is a subgroup.
\end{remark}

\subsection{Action of the renormalisation group}
Recall the map $\hat{T}\mapsto \sqcup_- F$ from Remark~\ref{colouring operations}, it to acts on $[-]$-coloured trees $(T,\hat{T})$ such that $\hat{T}_-\subset F$
by 
\begin{equation}\label{eq:colouring}
 (T,\hat{T}) \mapsto (T,\hat{T})_{\sqcup_- F }:=  (T, \hat{T} \sqcup_- F) \ . 
 \end{equation}
The corresponding canonical morphism $$S_{\sqcup_- F}:=S^{\langle (T,\hat{T}) \rangle, \langle (T,\hat{T} \sqcup_- F) \rangle} \in \Hom(\langle (T,\hat{T}) \rangle, \langle (T,\hat{T} \sqcup_- F) \rangle)$$ from Definition~\ref{def:morphisms_for_colouring}, which we call the \textit{extension} of \eqref{eq:colouring} to symmetric sets, in turn induces the map $
\Func_W S_{\sqcup_- F} $.
For notational convenience we suppress the functor and for $\tau \in \Func_{W}\langle (T,\hat{T})\rangle$ simply write 
$\tau_{\sqcup_- F}$ for $\Func_W S_{\sqcup_- F}\, (\tau)$.

\begin{definition}
For $g\in \mathfrak{G}_-$ and $\tau \in \Func_{W}\langle (T,\hat{T})\rangle \subset \mathcal{T}^{(-)}$ set
$$ \pmb{\Pi}^g \tau:= \sum_{F\in \mathfrak{F}^-_T, \hat{T}_-\subset F} \pmb{\Pi}(\tau_{\sqcup_- F})\circ g(F,\hat{T})\ .$$
\end{definition}

\begin{prop}
Let $g\in {\mathfrak{G}}_-$.
\begin{enumerate}
\item If $\pmb{\Pi}$ is admissible, $\pmb{\Pi}^g$ is admissible.
\item If $\pmb{\Pi}$ is an admissible lift of smooth noises $\{\xi^{\varepsilon}\}_{\varepsilon\in \mathcal{E}_-}$, $\pmb{\Pi}^g$ is an admissible lift of the smooth noises $(\xi^\varepsilon)^g$, where
$$ (\xi^\varepsilon)^g= \xi^{\varepsilon}- g(\Xi^\varepsilon) \ .$$
\end{enumerate}
\end{prop}
\begin{proof}
For the first claim, the first two conditions of Definition~\ref{def admissible bfPi} are clearly still satisfied. The third follows from the fact that for any planted tree $\mathcal{I}_{\mathfrak{k}}T$, all elements  $F\in \mathfrak{F}^-_{\mathcal{I}_{\mathfrak{k}}T}$ satisfy $F\subset T\subset\mathcal{I}_{\mathfrak{k}}T$.
For the second claim, note that 
$$\pmb{\Pi}^g (\Xi^\varepsilon,0) = \pmb{\Pi}(\Xi^\varepsilon, 0)+\pmb{\Pi}(\Xi^\varepsilon, [-]) \circ g(\Xi^\varepsilon, 0)= \xi^{\varepsilon}- g(\Xi^\varepsilon)\ . $$
\end{proof}

One sees that the properties of a group action are satisfied.
\begin{prop}
For an admissible map $\pmb{\Pi}$ the following identities hold
\begin{itemize}
\item $(\pmb{\Pi}^{f})^{g}=\pmb{\Pi}^{f \star g}$ for each $f,g\in{\mathfrak{G}}_-$,
\item $\pmb{\Pi}^e= \pmb{\Pi}$ .
\end{itemize}
\end{prop}
\begin{proof}
We check the first point, the second is seen easily as well. First note that it follows from the first identity in Lemma~\ref{lem:symset_composition}, that for 
$\tau\in \Func_W \langle T \rangle$ and $F,\tilde{F}\in \mathfrak{F}^-_T$ such that  $\hat{T}_-\subset F\subset \tilde{F} $ one has the identity
$$(\tau_{\sqcup_- F})_{\sqcup_- \tilde{F}} = \tau_{\sqcup_- \tilde{F}} \ .$$
Therefore, 
\begin{align*}
(\pmb{\Pi}^{f})^{g}\tau =& \sum_{F\in \mathfrak{F}^-_T, \hat{T}_-\subset F} \pmb{\Pi}^f(\tau_{\sqcup_- F})\circ g(F,\hat{T}) \\
=& \sum_{F\in \mathfrak{F}^-_T, \hat{T}_-\subset F}  \sum_{\tilde{F}\in \mathfrak{F}^-_T, F\subset \tilde{F}} 
\pmb{\Pi}\big( (\tau_{\sqcup_- F})_{\sqcup_- \tilde{F}} \big) \circ f(\tilde{F},\sqcup_- {F}) \circ g(F,\hat{T}) \\
=& \sum_{F,\tilde{F}\in \mathfrak{F}^-_T, \hat{T}_-\subset F\subset \tilde{F}}   \pmb{\Pi}(\tau_{\sqcup_- \tilde{F}}) \circ f(\tilde{F},\sqcup_- {F}) \circ g(F,\hat{T}) \\
=& \sum_{\tilde{F}\in \mathfrak{F}^-_T, \hat{T}_-\subset \tilde{F}}   \pmb{\Pi}( \tau_{\sqcup_- \tilde{F}}) \circ (f\star g)(\tilde{F},\hat{T}) \\
=&(\pmb{\Pi}^{f\star g})\tau \ .
\end{align*}
\end{proof}
\begin{prop}
Let $\pmb{\Pi}$ denote the canonical lift of smooth noises $\{\xi^{\varepsilon}\}_{\varepsilon\in \mathcal{E}_-}$ and $g\in \mathfrak{G}_-$ such that $g(\Xi^\varepsilon)=0$ for all $\varepsilon\in \mathcal{E}_-$.
For planted trees $T^1,\ldots,T^n\in \mathfrak{T}$ and $\eta^i \in \Func_W \langle T^i\rangle$, such that $\eta =\eta^1\cdots \eta^n\in \mathcal{T}$ one has
\begin{equation}\label{eq:bfPi_renormalisation_formula}
\pmb{\Pi}^g \eta= \mathring{\pmb{\Pi}}^g \eta+ \sum_{{\tau}\in \mathfrak{T}_T^- : \rho_T= \rho_\tau} 
g({\tau}) \Big(\prod_{n\in N_{{\tau}}} \mathring{\pmb{\Pi}}^g\eta^{\tau,n}(\cdot_{n})\Big) \,  
\end{equation}
where for each ${\tau}$ denote by $\mathcal{H}^{{\tau}}$ a family $\mathcal{H}^{{\tau}}=\{\eta^{{\tau},n}\}_{n\in N_{{\tau}}}$ such that 
$\eta= {\tau}(\mathcal{H}^{{\tau}}) $ c.f.\ Remark~\ref{remark attaching trees to negative trees}.

If we further decompose $\eta^{\tau,n}=\prod_i\eta^{\tau,n}_i$ where each $\eta^{\tau,n}_i$ is planted, one can rewrite \eqref{eq:bfPi_renormalisation_formula} as 
$$\pmb{\Pi}^g \eta= \Tr \bigodot_i (\pmb{\Pi}^g \eta^i)+ \sum_{{\tau}\in \mathfrak{T}_T^- : \rho_T\in \rho_\tau} 
g({\tau}) \Big(\prod_{n\in N_{{\tau}}} \Tr \bigodot_i \pmb{\Pi}^g\eta^{\tau,n}_i(\cdot_{n})\Big) \ .  $$
\end{prop}
\begin{remark}
A similar but slightly more involved formula holds, if we remove the condition $g(\Xi^\varepsilon)=0$. 
\end{remark}

\begin{proof}
Note that the second formula follows directly from \eqref{eq:bfPi_renormalisation_formula} and Definition~\ref{def:mathringbfPi}. 
To proof \eqref{eq:bfPi_renormalisation_formula} we proceed inductively, c.f.\ Remark~\ref{rem recursive construction}. Indeed the claim holds for the sectors containing only jets or only a noise.
\begin{itemize}
\item Suppose the claim holds for $\tau\in \Func_W \langle T\rangle\subset \mathcal{T}$, then it clearly holds for $\mathcal{I}_{\mathfrak{k}} (\tau)$ since $\big\{\tau\in \mathfrak{T}_{\mathcal{I}_{\mathfrak{k}} (T)}^-\ :\  \rho_{\mathcal{I}_{\mathfrak{k}} (T)}= \rho_\tau\big\} = \emptyset $ and thus
$$\sum_{F\in \mathfrak{F}^-_{\mathcal{I}_{\mathfrak{k}}(T)}}  (\mathcal{I}_{\mathfrak{k}}\tau)_{\sqcup_- F}= \mathcal{I}_{\mathfrak{k}} \big(\sum_{F\in \mathfrak{F}^-_T}  \tau_{\sqcup_- F}\big) \ $$
which implies that
\begin{align*}
\pmb{\Pi}^g  (\mathcal{I}_{\mathfrak{k}}\tau) 
&= \sum_{F\in \mathfrak{F}^-_{\mathcal{I}_{\mathfrak{k}}(T)}}  \pmb{\Pi}(\mathcal{I}_{\mathfrak{k}}\tau)_{\sqcup_- F} \circ g(F)\\
 &= K_\frakk\big(\sum_{F\in \mathfrak{F}^-_T} \pmb{\Pi} \big( \tau_{\sqcup_- F}\big) \circ g(F)\big)\\
  &= K_\frakk \big(\pmb{\Pi}^g \tau\big) \\
  &= \mathring{\pmb{\Pi}}^g \mathcal{I}_{\mathfrak{k}}\tau \ .
  \end{align*}
\item Suppose the claim holds for planted $\eta^1,\ldots,\eta^m$ as in the proposition and $\eta=\eta^1\cdots\eta^m$ . Then 
\begin{align*}
\sum_{F\in \mathfrak{F}^-_T}  \eta_{\sqcup_- F} &= \sum_{F\in \mathfrak{F}^-_T: \rho\notin F}  \eta_{\sqcup_- F} +\sum_{F\in \mathfrak{F}^-_T: \rho\in F}  \eta_{\sqcup_- F}  \\
&= \sum_{F_i\in \mathfrak{F}^-_{T^i}: \rho_i\notin F^i}  \eta^1_{\sqcup_- F^1}\cdots\eta^m_{\sqcup_- F^n} 
+\sum_{{\tau}\in \mathfrak{T}_T^- : \rho_T\in {\tau}} {\tau} \big( \big\{ \sum_{F^{n}\in \mathfrak{F}^-_{G^{{\tau},n}}: \rho_{G^{{\tau},n}}\notin G^{{\tau},n}}  \eta^{{\tau},n}_{\sqcup_- F^n}\big\}_{n\in N_{{\tau}}} \big)  \\
\end{align*}
where $\mathcal{H}^{{\tau}}=\{\eta^{{\tau},n}\}_{n\in N_{{\tau}}}$ and each $\eta^{{\tau,n}}\in \Func_W (G^{{\tau},n})$ is as in the proposition. 
Thus, we find
\begin{align*}
\pmb{\Pi}^g \eta
&=\sum_{F\in \mathfrak{F}^-_T } \pmb{\Pi}(\eta_{\sqcup_- F})\circ g(F)\\
&=\sum_{F^i\in \mathfrak{F}^-_{T^i}: \rho_i\notin F^i}  \pmb{\Pi}(\eta^1_{\sqcup_- F^1}\cdots \eta^m_{\sqcup_- F^m}) \circ g(F^1\cdots F^m )\\
 &+ \sum_{{\tau}\in \mathfrak{T}_T^- : \rho_T= \rho_\tau} 
g({\tau}) \Bigg(\prod_{n\in N_{{\tau}}} \Big(  \sum_{F^{n}\in \mathfrak{F}^-_{G^{{\tau},n}}: \rho_{G^{{\tau},n}}\notin G^{{\tau},n}}  \pmb{\Pi}(\eta^{{\tau},n}_{\sqcup_- F^n})\circ g(F^n) (\cdot_n)\Big)_{n\in N_{{\tau}}} \Bigg) \\
&=\mathring{\pmb{\Pi}}^g \eta + \sum_{{\tau}\in \mathfrak{T}_T^- : \rho_T= \rho_\tau} 
g({\tau}) \Big(\prod_{n\in N_{{\tau}}} \mathring{\pmb{\Pi}}^g\eta^{\tau,n}(\cdot_{n})\Big) \ , 
\end{align*}
where we used that $g(\Xi^\varepsilon)=0$ and the fact that by construction
$$ \mathring{\pmb{\Pi}}^g \eta:= \sum_{F\in \mathfrak{F}^-_T: \rho \notin F} \pmb{\Pi}(\eta_{\sqcup_- F})\circ g(F)\ .$$
\end{itemize}
\end{proof}

\subsection{Construction of models}\label{section construction of models}
In this section we investigate how to obtain a model $(\Pi, \Gamma)$ from an admissible map $\pmb{\Pi}$. The underlying spaces for this construction will be $\mathcal{T}_+:= \{\mathcal{T}_+^\varepsilon\}_{\varepsilon\in \mathcal{E}_+\cup \mathcal{E}_0}$ and $J= \{J^\varepsilon\}_{\varepsilon\in \mathcal{E}_0}$.

\begin{definition}
We define $\mathfrak{G}_{+}=\hat{L}_{\iota}(\mathcal{T}_+, J)$, the bundle whose fibre at $(p,q)\in M\times M$ consists of linear maps $\gamma_{p,q}\colon (\mathcal{T}_+^\varepsilon)_q \to J^{\iota(\varepsilon)}_p$ such that for each $\varepsilon\in \mathcal{E}_+\cup \mathcal{E}_0$, each tree $T\in \mathfrak{T}^{\varepsilon}_+$ and each $\tau_q \in \Func_W\langle T\rangle$ , 
$$\gamma_{p,q}(\tau_q)\in J^{\iota(\varepsilon)}_p\ .$$

We denote by $\mathfrak{G}_{+,\loc}$ the sub-bundle\footnote{Recall that $\diag:M\to M\times M$ is the diagonal map. } of $\diag^*(\mathfrak{G}_+)$, whose fibre at $p$ consists of all maps $\gamma_{p}$ satisfying
\begin{equation}\label{triviality on jets}
\gamma_{p}|_{J_p} = \id_{J_p} \;.
\end{equation}

\end{definition}

\begin{remark}
The following is useful in order to formulate positive renormalisation. Let $ T \in \mathfrak{T}$, we shall build for each $F\in\mathfrak{F}_T^+$ and $\gamma\in \mathfrak{G}_+$ a map 
$$\Func_V \langle T\rangle_q \to \mathcal{T}_p, \quad \tau_q \mapsto \tau\setminus_{\gamma_{p,q}} F \ . $$
\begin{itemize}
\item First we define for $T$ and $F$ as above a new tree $T\setminus F$. Recall from Section~\ref{section positive subtrees} that there is a set $E\in P_+(T)$ such that $F=T_{\geq E}$, and we define $T\setminus F$ to be the tree obtained by replacing for each $e\in E$ the subtree $T_{\geq e}$ by a new edge of type $\iota(\mathfrak{e}(e))$.
\item For a concrete realisation of the tree $T$ write $S\in \Hom (\langle T \rangle, \otimes_{e\in E} \langle T_{\geq e}\rangle)$ for the morphism in Definition~\ref{def:morphisms_for_colouring}. Then we
 define $\tau\setminus_{\gamma_{p,q}} F$ to be the image of $\tau$ under the following composition\footnote{The maps $\pi$ and $\iota$ were introduced in (\ref{eq:translating_on_section}) and (\ref{eq iota E}) respectively.}
\[ \begin{tikzcd}
W^{\otimes \langle T\rangle}  \arrow{r}{\Func_W S}
& \bigotimes_{e\in E} W^{\otimes \langle T_{\geq e}\rangle } \arrow{r}{\gamma^{\otimes |E|}}
& \bigotimes_{e\in E} W^{\iota(e)}\simeq W^{\otimes T\setminus F }   \arrow{r}{\pi_{T\setminus F,\langle T\setminus F\rangle }} 
&[2em] W^{\otimes \langle T\setminus F \rangle} \ .
\end{tikzcd}
\] 
\end{itemize}
\end{remark}

\begin{definition}
We define the action of $\mathfrak{G}_+$ on $\mathcal{T}$ as follows. For $T\in \mathfrak{T}^\sigma$,
$\tau_p \in \Func_W\langle T\rangle \subset \mathcal{T}^{\sigma}$
 and $\gamma\in \mathfrak{G}_{+}$ or $\gamma\in \mathfrak{G}_{+,\loc}$ set
$$ \Gamma_{\gamma} \tau:= \sum_{F\in \mathfrak{F}^+_T} \tau\setminus_{\gamma} F \ .$$
\end{definition}
\begin{remark}\label{remark on Gamma}
Note that for every $\gamma$ the map $\Gamma_\gamma$ is multiplicative by construction, i.e.\ $\Gamma_\gamma(\tau_1\tau_2)=(\Gamma_\gamma\tau_1)(\Gamma_\gamma\tau_2)$. Furthermore for a planted tree $\mathcal{I}_{\mathfrak{k}}\tau\in \Func_W\langle T\rangle $ one has
$$\Gamma_\gamma(\mathcal{I}_{\mathfrak{k}}\tau)= \mathcal{I}_{\mathfrak{k}}\tau\setminus_\gamma T + \mathcal{I}_{\mathfrak{k}}( \Gamma_\gamma \tau)= \gamma(\mathcal{I}_{\mathfrak{k}}\tau )+ \mathcal{I}_{\mathfrak{k}}( \Gamma_\gamma \tau) \ .$$
\end{remark}
\subsubsection{Construction of $\Pi_p$}
Given $\pmb{\Pi}$, we define $\Pi_p = \pmb{\Pi}\circ \Gamma_{\gamma_p}$ where the global section $\gamma=\gamma(\pmb{\Pi})$ of $\mathfrak{G}_{+,\loc}$ is define recursively as follows
\begin{itemize}
\item If $\tau\in \mathcal{T}^\varepsilon_{+,(\alpha, \delta)}$ is such that $\alpha< 0$, then $\gamma(\tau)=0$,
\item if $\tau\in J$, then $\gamma(\tau)=\tau$,
\item every other basis element is of the form $\mathcal{I}_{\frakk}\tau_p \in \mathcal{T}^\varepsilon_{+,(\alpha, \delta)}$, then 
$$\gamma_p(\mathcal{I}_\frakk\tau_p)=-Q_{<\delta_0} j_p (K_{\frakk}\Pi_p Q_{<\delta_0-|\frakk|}\tau)\ .$$
\end{itemize}
\begin{remark}
The definition of 
$\gamma_p(\mathcal{I}\tau)$ is chosen to agree with $- J_p\tau_p$ in the setting of Section~\ref{Section Singular Kernels}.
\end{remark}

\subsubsection{Construction of $\Gamma$}
Given $\pmb{\Pi}$, $\Pi_{(\cdot)}$ and $\gamma_{(\cdot)}$, we define $\Gamma_{p,q} = \Gamma_{\gamma_{p,q}} $ for an appropriately chosen global section $\gamma=\gamma(\pmb{\Pi})\in \mathfrak{G}_+$. This section is defined recursively as follows
\begin{itemize}
\item  for $\tau_q\in J_q$, let $\gamma_{p,q}(\tau_q)=Q_{<\delta_0} j_p(\pmb{\Pi}\tau_q)=Q_{<\delta_0}j_p(R\tau_q)$,
\item  for $\mathcal{I}_{\frakk} \tau_q\in \mathcal{T}_{+(\alpha,\delta)}$ set
\begin{align*}
\gamma_{p,q}(\mathcal{I}\tau_q)
=& -\gamma_p (\mathcal{I}_{\frakk}\Gamma_{p,q}\tau)+ \Gamma_{p,q} \gamma_q (\mathcal{I}_{\frakk}\tau_q) +Q_{<(\delta+\beta)\wedge \delta_0} j_q K_{\frakk}(\Pi_p-\Pi_qQ_{\delta_0-|\frakk|}\Gamma_{q,p})\tau_q \ .
\end{align*}
\end{itemize}

\begin{remark}
Note that $\gamma_{p,q}(\mathcal{I}\tau)$ above is defined so that it agrees with $$ J_{p,q}\tau= J_p\Gamma_{p,q}\tau_q- \Gamma_{p,q}J_q\tau_q +E_q\tau_q$$ in the setting of Section~\ref{Section Singular Kernels}.
\end{remark}
\begin{remark}
Note that, in contrast to the flat setting, the maps $\Gamma_{(\cdot),(\cdot)}$ do \textit{not} satisfy $\Gamma_{p,q}\circ\Gamma_{q,r}= \Gamma_{p,r}$, not even locally.
\end{remark}
\begin{remark}
Observe that $\Gamma_{\gamma_{p}}$ is invertable, while $\Gamma_{\gamma_{p,q}}$ is in general only invertible for $d_\fraks(p,q)$ small enough.
\end{remark}

\begin{prop}\label{prop canonical lift is a a model}
Let $(\Pi, \Gamma)$ be constructed as above from the canonical lift $\pmb{\Pi}$ of smooth noises $\{\xi_{\mathfrak{l}}\}_{l\in \mathfrak{L}_-}$. Then $(\Pi, \Gamma)$ is a model, we shall call it the canonical model for $\{\xi_{\mathfrak{l}}\}_{l\in \mathfrak{L}_-}$.
\end{prop} 
\begin{proof}
We need to check that the bounds from Definition~\ref{def model} hold for
\begin{equation}\label{eq:whattubound}
\Pi_p \tau_p, \qquad
 \Gamma_{q,p}\tau_p,  \qquad (\Pi_p- \Pi_q \Gamma_{q,p})\tau_p \ .
 \end{equation}
Observe that all three estimates hold if $\tau_p\in J_p^\varepsilon$ for every $\varepsilon\in \mathcal{E}_0$ by Proposition~\ref{proposition on polynomial model} or if $\tau_p$ is a noise.
Next, since all three terms are multiplicative with respect to the forest product, it suffices to check that if the desired bound holds for $\tau$, then it holds for 
$\mathcal{I}\tau$. This follows essentially from the considerations in Section~\ref{Section Singular Kernels}, thus we use the notation $J_{p}$ and $J_{p,q}$ from therein for the kernel $K_{\mathfrak{k}}$. 

We have by Remark~\ref{remark on Gamma} 
$$\Pi_p \mathcal{I}_{\frakk}\tau= \pmb{\Pi} \Gamma_{\gamma_{p}} \mathcal{I}_{\frakk}\tau = \pmb{\Pi} \gamma_p(\mathcal{I}_{\frakk}\tau )+ \pmb{\Pi} \mathcal{I}_{\frakk}( \Gamma_{\gamma_p} \tau) = -R J_p\tau_p + K_{\frakk}( \Pi_p \tau_p)$$ which satisfies the desired bound by Lemma~\ref{lemma well definednes of I}.

Again by Remark~\ref{remark on Gamma} 
$$\Gamma_{q,p}\mathcal{I}_{\frakk}\tau_p= \gamma_{q,p}(\mathcal{I}_{\frakk}\tau )+  \mathcal{I}_{\frakk}( \Gamma_{q,p} \tau)= J_{q,p}\tau  + \mathcal{I}_{\frakk}( \Gamma_{q,p} \tau)$$
and the first term satisfies the desired bound by Lemma~\ref{lemma J_qp} and the second term by the definition of an abstract integration map (Definition~\ref{def of abstract integration map}).
For the last bound we find that 
\begin{align*}
&(\Pi_p- \Pi_q \Gamma_{q,p})\mathcal{I}_{\frakk}\tau_p\\ 
&= -R J_p\tau_p + K_{\frakk}( \Pi_p \tau_p)- (R J_{q,p}\tau  + {\Pi}_q\mathcal{I}_{\frakk}( \Gamma_{q,p} \tau_p))\\
&= -R J_p\tau_p + K_{\frakk}( \Pi_p \tau_p)- (R J_{q,p}\tau)  + RJ_q( \Gamma_{q,p} \tau)) - K(\Pi_q \Gamma_{q,p} \tau_p)\\
&= -R J_p\tau_p + K_{\frakk}( \Pi_p \tau_p)- R(J_q\Gamma_{q,p}\tau_p- \Gamma_{q,p}J_p\tau_p +E_q\tau_p)  +R J_q( \Gamma_{q,p} \tau_p)) - K(\Pi_q \Gamma_{q,p} \tau_p)\\
&= -R J_p\tau_p + K_{\frakk}( \Pi_p \tau_p)- R(- \Gamma_{q,p}J_p\tau_p +E_q\tau_p)  - K(\Pi_q \Gamma_{q,p} \tau_p)\\
&=  \big(K_{\frakk}( \Pi_p -\Pi_q \Gamma_{q,p} \tau_p)-RE_q\tau_p\big) - (\Pi_p-\Pi_q\Gamma_{q,p})J_p\tau_p 
\end{align*}
where the first summand satisfies the desired bound by Lemma~\ref{lemma bound on error} and the second summand by the fact that the jet model is a model, i.e.\ Proposition~\ref{proposition on polynomial model}.
\end{proof}

\begin{definition}
We call a model $(\Pi, \Gamma)$ admissible, if it is of the form $\Pi_p= \pmb{\Pi}\circ \Gamma_{\gamma_p(\pmb{\Pi})}$, $\Gamma_{p,q}= \Gamma_{\gamma_{p,q}(\pmb{\Pi})}$ for an admissible map $\pmb{\Pi}$. 
We say that an admissible map $\pmb{\Pi}$ generates a model, if $(\Pi, \Gamma)$ given by $\Pi_p= \pmb{\Pi}\circ \Gamma_{\gamma_p(\pmb{\Pi})}$, $\Gamma_{p,q}= \Gamma_{\gamma_{p,q}(\pmb{\Pi})}$ is a model.
\end{definition}
Note that two different maps $\pmb{\Pi}$ might generate the same model $(\Pi, \Gamma)$. In that sense $\pmb{\Pi}$ contains more information. 
The next proposition is the main Proposition of this section.
\begin{prop}\label{prop orbit generates models}
If $\pmb{\Pi}$ is an element of the orbit under the renormalisation group $\mathfrak{G}$ of the canonical lift  of smooth noises $\{\xi_{\mathfrak{l}}\}_{l\in \mathfrak{L}_-}$, then $\pmb{\Pi}$ generates a model. 
\end{prop}
The following lemmas will be useful in the proof.
\begin{lemma}\label{lemma pi renormalized}
Let $(\Pi, \Gamma)$ be the canonical model (generated by the canonical $\pmb{\Pi}$) and $g\in \mathfrak{G}_-$ such that $g(\Xi^\varepsilon)=0$ for all $\varepsilon\in \mathcal{E}_-$.
For planted trees $T^1,\ldots,T^m\in \mathfrak{T}$ and $\eta^i \in \Func_W \langle T^i\rangle$, such that $\eta =\eta^1\cdots \eta^m\in \mathcal{T}$ one has
\begin{equation}\label{eq:lemma pi renormalsed}
\Pi_p^g \eta = \Tr \bigodot_i (\Pi_p^g \eta^i) + \sum_{{\tau}\in \mathfrak{T}^{-}_T \ : \rho_T\in {\tau}} g({\tau}) \Big( \prod_{n\in N_{{\tau}}} \Tr \big(\bigodot_i\Pi_p^g \eta_i^{{\tau}, n}(\cdot_n)\big)\Big) \ ,
\end{equation}
where the sum runs over all negative subtrees $\tau$ of $T$ containing the root and where 
for each ${\tau}$, $\mathcal{H}^{{\tau}}$ denotes a family $\mathcal{H}^{{\tau}}=\{\eta^{{\tau},n}\}_{n\in N_{{\tau}}}$ such that 
$\eta= {\tau}(\mathcal{H}^{{\tau}}) $ as in Remark~\ref{remark attaching trees to negative trees} and furthermore write $\eta^{\tau,n}=\prod_i\eta^{\tau,n}_i$ where each $\eta^{\tau,n}_i\in \Func_W (G_i^{{\tau}, n})$ is planted.
\end{lemma}

\begin{remark}\label{def:mathringfPi}
It is useful to introduce for a smooth model $Z= (\Pi, \Gamma)$ the maps
$\mathring{{\Pi}}_p : \mathcal{T}\to \mathcal{C}^\infty (V)$ as the unique map satisfying the following properties:
\begin{itemize}
\item for each planted $T \in \mathfrak{T}$, one has $$\mathring{\Pi}_p|_{\Func_{ W}\langle T \rangle} ={\Pi}_p|_{\Func_{ W}\langle T \rangle}$$
\item For $\eta^{1}_p,\ldots,\eta^n_p \in \mathring{\mathcal{T}}$ such that $\eta_p^{i}\in \Func_{V}(T^i)$ where each $T^i$ is planted,
one has
$$ \mathring{{\Pi}}_p\big(\prod_{i=1}^n \eta^i \big) = \Tr\big( {\bigodot_{i=1}^n } \mathring{{\Pi}}_p(\eta^i) \big) \ .$$
\end{itemize}
Then, one can write \eqref{eq:lemma pi renormalsed} as
$$
\Pi_p^g \eta = (\mathring{\Pi}_p^g \eta) + \sum_{{\tau}\in \mathfrak{T}^{-}_T \ : \rho_T\in {\tau}} g({\tau}) \Big( \prod_{n\in N_{{\tau}}} \mathring{\Pi}_p^g \eta^{{\tau}, n}(\cdot_n)\Big) \ .
$$

\end{remark}

\begin{proof}
Throughout this proof we simply write $\gamma_p$ instead of $\gamma_p(\pmb{\Pi}^g)$ in order to make notation lighter.
Writing $T= T^1\cdots T^n$ and for $F\in \mathfrak{F}_T^+$, $F_i= F\cap T^i$ we compute
$$
\Pi_p^g \eta = \pmb{\Pi}^g \Gamma_{\gamma_p} \eta
 = \sum_{F\in \mathfrak{F}^+_{{T}}} \pmb{\Pi}^g( {\eta}\setminus_{\gamma} F)
 = \sum_{F\in \mathfrak{F}^+_{{T}}} \pmb{\Pi}^g\Big( \prod_{i} \eta^{i}  \setminus_{\gamma_p} F_i \Big)
 =  \pmb{\Pi}^g\Big( \prod_{i}  \sum_{F_i\in \mathfrak{F}^+_{{T}_i}}\eta^{i}  \setminus_{\gamma_p} F_i \Big)
$$
Since
$$\pmb{\Pi}^g\Big( \prod_{i} \eta^{i}  \setminus_{\gamma_p} F_i \Big)= \Tr\bigodot_{i} (\pmb{\Pi}^g \eta^{i}  \setminus_{\gamma_p} F_i ) + \sum_{{\tau}\in \mathfrak{T}_{{T}_\setminus F}^- : \rho\in \tau} 
g({\tau})\Bigg( \prod_{n\in N_{{\tau}}} \big(\bigodot_j \pmb{\Pi}^g\eta^{\tau,n}_j\setminus_{\gamma_p}( G^{\tau,n}_j\cap F_i)(\cdot_{n})\Big)\Bigg)\, $$
we have
\begin{align*}
\Pi_p^g \eta &= \sum_{F\in \mathfrak{F}^+_{{T}}} \Tr\bigodot_{i} (\pmb{\Pi}^g \eta^{i}  \setminus_{\gamma_p} F_i ) + \sum_{F\in \mathfrak{F}^+_{{T}}}\sum_{{\tau}\in \mathfrak{T}_{{T}_\setminus F}^- : \rho\in {\tau}}
g({\tau}) \Bigg(\prod_{n\in N_{\tau}} \big(\bigodot_j \pmb{\Pi}^g\eta^{{\tau},n}_j\setminus_{\gamma_p}( G^{{\tau},n}_j\cap F_i)(\cdot_{n})\Big)\Bigg) \\
&=\Tr\bigodot_{i} (\pmb{\Pi}^g \sum_{F_i\in \mathfrak{F}^+_{{T}_i}} \eta^{i}  \setminus_{\gamma_p} F_i )  +\sum_{F\in \mathfrak{F}^+_{{T}}}\sum_{{\tau}\in \mathfrak{T}_{{T}_\setminus F}^- : \rho\in {\tau}} 
g({\tau}) \Bigg(\prod_{n\in N_{{\tau}}} \big(\bigodot_j \pmb{\Pi}^g\eta^{\tau,n}_j\setminus_{\gamma_p}( G^{{\tau},n}_j\cap F_i)(\cdot_{n})\Big)\Bigg) \\
&=  \Tr\bigodot_{i} \pmb{\Pi}^g\Gamma_{\gamma_p} \eta^{i}    +\sum_{F_i\in \mathfrak{F}^+_{{T}_i}} \sum_{{\tau}\in \mathfrak{T}_{{T}_\setminus F}^- : \rho\in {\tau}} 
g({\tau}) \Bigg(\prod_{n\in N_{\tau}} \big(\bigodot_j \pmb{\Pi}^g\eta^{{\tau},n}_j\setminus_{\gamma_p} ( G^{{\tau},n}_j\cap F_i)(\cdot_{n})\Big)\Bigg) \\
&=  \Tr\bigodot_{i} \pmb{\Pi}^g\Gamma_{\gamma_p} \eta^{i}   + \sum_{\tau\in \mathfrak{T}_{{T}}^- : \rho\in \tau} \sum_{F_j\in \mathfrak{F}^+_{{G^{\tau,n}_j}}}
g(\tau) \Bigg( \prod_{n\in N_{\tau}} \big(\bigodot_j \pmb{\Pi}^g\eta^{\tau,n}_j\setminus_{\gamma_p} F_j)(\cdot_{n})\Big)\\
&= \Tr\bigodot_{i} \pmb{\Pi}^g\Gamma_{\gamma_p} \eta^{i}   + \sum_{\tau\in \mathfrak{T}_{{T}}^- : \rho\in \tau} 
g(\tau) \Big(\prod_{n\in N_{\tau}} \big(\bigodot_j \pmb{\Pi}^g \Gamma_{\gamma_p} \eta^{\tau,n}_j\big)(\cdot_{n})\Big) \\
&=  \Tr\bigodot_{i} {\Pi}^g_{p} \eta^{i}    + \sum_{\tau\in \mathfrak{T}_{{T}}^- : \rho\in \tau} 
g(\tau) \Big(\prod_{n\in N_{\tau}} \big(\bigodot_j {\Pi}_p^g\eta^{\tau,n}_j\big) (\cdot_{n})\Big)\;,
\end{align*}
as claimed.
\end{proof}

\begin{remark}\label{extracted formula}
Note that in the previous proof we only used the explicit choice of $\gamma_p$ in the last line. In particular we can extract the following general formula from the proof above. For any $g\in \mathfrak{G}_-$ and $f\in \mathfrak{G}_+$
$$\pmb{\Pi}^g \Gamma_f \eta=  \Tr\bigodot_{i} \pmb{\Pi}^g\Gamma_f \eta^{i} + \sum_{\tau\in \mathfrak{T}_{{T}}^- : \rho\in \tau} 
g(\tau)\Big( \prod_{n\in N_{\tau}} \big(\bigodot_j \pmb{\Pi}^g \Gamma_f \eta^{\tau,n}_j(\cdot_{n})\big)\Big) \\$$
\end{remark}

\begin{lemma}
In the setting and notation of the previous lemma, the following formula holds.

\begin{align*}
\Pi^g_q \Gamma_{q,p}^g \eta &=  \Tr\bigodot_{i} \Pi^g_q \Gamma_{q,p}^g\eta^{i} + \sum_{\tau\in \mathfrak{T}_{{T}}^- : \rho\in \tau} 
g({\tau})\Big( \prod_{n\in N_{{\tau}}} \big(\bigodot_j \Pi^g_q \Gamma_{q,p}^g \eta^{{\tau},n}_j(\cdot_{n})\big)\Big) 
\end{align*}
\end{lemma}

\begin{proof}
Again we we simply write $\gamma_p$ and $\gamma_{p,q}$ instead of $\gamma_p(\Pi^g)$ and $\gamma_{p,q}(\Pi^g)$ in order to make notation lighter.
Note that
\begin{align*}
\Pi^g_q\Gamma^g_{q,p}\tau&= \pmb{\Pi}^g\Gamma_{\gamma_q}\Gamma^g_{q,p}\tau =\pmb{\Pi}^g \Gamma_{\gamma_q} \sum_{\tilde{F}\in \mathfrak{F}^+_T} \tau\setminus_{\gamma_{q,p}} \tilde{F} = 
\pmb{\Pi}^g \sum_{{F}\in \mathfrak{F}^+_{T\setminus \tilde{F}}} \sum_{\tilde{F}\in \mathfrak{F}^+_T} (\tau\setminus_{\gamma_{q,p}} \tilde{F}) \setminus_{\gamma_{p}} F\\
&=\pmb{\Pi}^g \sum_{{F}\in \mathfrak{F}^+_T} \tau\setminus_{\beta_{q,p}} {F} 
= \pmb{\Pi}^g  \Gamma_{\beta_{q,p}}\tau 
\end{align*}
where we define $\beta_{q,p}\in (\mathfrak{G}_+)_{q,p}$ by
$$\beta_{q,p}(\tau)= \sum_{\tilde{F}\in \mathfrak{F}^+_T} \gamma_q (\tau\setminus_{\gamma_{q,p}} \tilde{F}) \ $$
for $\tau\in\Func_W \langle T\rangle_p \subset \mathcal{T}_+$.
Thus we find by Remark~\ref{extracted formula}
\begin{align*}
\Pi^g_q \Gamma_{q,p}^g&=  \Tr\bigodot_{i} \pmb{\Pi}^g\Gamma_{\beta_{q,p}} \tau_{i} + \sum_{\tau\in \mathfrak{T}_{{T}}^- : \rho\in \tau} 
g(\tau)\Big( \prod_{n\in N_{\tau}} \big(\bigodot_j \pmb{\Pi}^g \Gamma_{\beta_{q,p}}  \eta^{\tau,n}_j(\cdot_{n})\big)\Big) \\
&=  \Tr\bigodot_{i} \Pi^g_q \Gamma_{q,p}^g\tau_{i} + \sum_{\tau\in \mathfrak{T}_{{T}}^- : \rho\in \tau} 
g(\tau)\Big( \prod_{n\in N_{\tau}} \big(\bigodot_j \Pi^g_q \Gamma_{q,p}^g \eta^{\tau,n}_j(\cdot_{n})\big)\Big) \ ,
\end{align*}
which completes the proof.
\end{proof}

\begin{remark}\label{rem telescopic sum}
Using a telescopic sum and combining the previous two lemmas on finds that 
\begin{align*}
&(\Pi_p^g -\Pi^g_q \Gamma_{q,p}^g) \eta \\
&= \Tr\sum_{j=1}^n (\Pi_p^g \eta^1)\odot\ldots\odot(\Pi_p^g \eta^{j-1}) \odot (\Pi_p^g\eta^j -\Pi^g_q \Gamma_{q,p}^g\eta^j) \odot( \Pi^g_q \Gamma_{q,p}^g\eta^{j+1}) \odot\ldots\odot( \Pi^g_q \Gamma_{q,p}^g \eta^{m}) \\
&+ \sum_{\text{``telescopic sum''}} \sum_{\tau\in \mathfrak{T}^{-}_T \ : \rho_T\in \tau} g(\tau) \Big( \prod_{n\in N_{\tau}} \Tr \big(\bigodot_i\big(\cdots\big)\eta_i^{\tau, n}(\cdot_n)\big)\Big)
\end{align*}
where $\big(\cdots\big)$ stands for either $\Pi_p^g$, $\Pi_p^g -\Pi^g_q\Gamma_{q,p}^g$, or $\Pi_p^g -\Pi^g_q\Gamma_{q,p}^g$ with $\Pi_p^g -\Pi^g_q\Gamma_{q,p}^g$ appearing exactly once
\end{remark}

\begin{proof}[Proof of Proposition~\ref{prop orbit generates models}]
First observe that any element in the orbit can be written as $\pmb{\Pi}^{f}$ for some $f\in \mathfrak{G}_-$. We proceed by induction as in the proof of Proposition~\ref{prop canonical lift is a a model}. The desired estimates hold for jets and noises. To observe the bound on the terms in \eqref{eq:whattubound} we use the above lemmas and observe the following:
\begin{enumerate}
\item To see that the desired bound on $\Pi_p \tau_p$ holds, use Lemma~\ref{lemma pi renormalized}.
\item The bound on $\Gamma_{q,p}\tau_p$ follows as in Proposition~\ref{prop canonical lift is a a model}.
\item Note that $\|T\|$ was defined in exactly such a way that it imposes a bound on $(\Pi_p- \Pi_q \Gamma_{q,p})\tau_p$ which is satisfied by Remark~\ref{rem telescopic sum}.
\end{enumerate}
\end{proof}

\begin{remark}
Instead of attaching jets one could directly attach smooth sections to $\mathcal{E}_0$ type edges and all of this sections easily adapts.
But besides the minor drawback of having to work with infinite dimensional vector bundles one would not gain much, since while the definition 
of the map $\pmb{\Pi}$ would then be more canonical, still the choice of an admissible realisation of the jet bundle would have to be made for the definition of the maps $\Pi_p$.
\end{remark}

\begin{remark}\label{rem:identifying_under_can product on jets}
Recall from Section~\ref{remark algebraic structure} that is often useful to impose further restrictions on the choice of admissible realisations and from Section~\ref{section:lifting general nonlinearities} to enforce the canonical counterpart on products of jets. The regularity structure ensembles in this section do not enforce any such algebraic constraints. This can be remedied by working with quotient spaces of $\mathcal{T}^{\sigma}/{\sim}$, such that Assumption~\ref{Assumption Nonlinearites} is satisfied. 
That is, whenever a tree contains two or more $\mathcal{E}_0$ type edges incident to the same node, one applies the canonical product $ \star : J^\sigma \otimes J^{\bar{\sigma}}  \to J^{\sigma +\bar{\sigma}}$ composed with the projection $Q_{<\delta_0}$ at that node.

It follows by construction, that by choosing an admissible realisation $R$ of $J$ respecting these constraints, which can be done by the constructions in Section~\ref{remark algebraic structure}, all considerations in this section extend to the quotient space.
A similar procedure is executed in \cite{BCCH20}, see also \cite[Remark~4.35]{Che22}.
\end{remark}

\begin{remark}
Let us emphasise that the constructions of this section are conceptually very close to the ones in \cite{BHZ19}.
One also expects to be able to reformulate many of the constructions here in a framework closer to the one herein, i.e.\ one involving Hopf algebra bundles. Some of the novelties one would need to incorporate on the abstract (algebraic) level are the following 
\begin{itemize}
\item A non-triangular action of the positive coproduct, when acting on jets and negative trees,
\item further ``extending" the extended decorations therein.
\end{itemize}
But once this has been executed one can then turn the trees into vector bundles by working with a direct product of frame bundles with some structure group $G$. Then associating to each tree a symmetric set and to each symmetric set a representation of $G$ one can thus construct the desired vector bundles as associated bundles. By the categorical properties of symmetric sets and the thus constructed vector bundle assignment one then expects to be able to "push through" the algebraic structure on trees to the vector bundles, somewhat analogously to what is done in \cite{CCHS22}.
\end{remark}

\section{Application to SPDEs}\label{section application to spdes}
In this section we first explain, starting with an SPDE to choose a rule to construct then the regularity structure ensemble in order to give a solution theory. Then, in Proposition~\ref{prop:renormalised equation} we give a formula how the renormalisation group affects the concrete equation one wants to solve.

For index sets $\mathfrak{L}_+$ and $\mathfrak{L}_-$ and a vector bundle assignment $\{V^\mathfrak{l}\}_{\frakl \in \mathfrak{L}_+\cup \mathfrak{L}_-}$, we investigate the following system of equations 
\begin{equation}\label{eq main}
\mathcal{L}u^\frakt=F^{\frakt}\big( (u^\frakt)_{\frakt\in \mathfrak{L}_+},(Du^\frakt)_{\frakt\in \mathfrak{L}_+},\ldots,(D^n u^\frakt)_{\frakt\in \mathfrak{L}_+},(\xi^\frakf)_{\frakf\in \mathfrak{L}_-}\big) \qquad \text{for } \frakt \in \mathfrak{L}_+ \ ,
\end{equation}
where in this section we can  $u^\frakt\in \mathcal{D}'(V^\frakt)$ and $\xi^\frakf\in \mathcal{D}'(V^\frakf)$.
Let $\mathring{\mathfrak{L}}=(\mathfrak{L}_+\times \NN )\cup  \mathfrak{L}_-$ and $\tilde{\mathfrak{L}}$ be a disjoint copy of $\mathring{\mathfrak{L}}$, for $\frakf\in\mathring{\mathfrak{L}}$ we write $\tilde{\frakf}$ for the corresponding element of $\tilde{\mathfrak{L}}$. We define $\mathfrak{L}= \mathring{\mathfrak{L}}\cup \tilde{\mathfrak{L}}$ and 
$$*:\mathfrak{L}\to \mathfrak{L}, \ \mathfrak{l}\mapsto \mathfrak{l}^* \ ,$$
to be the unique involution without fixed points satisfying  $\frakf^*=\tilde{\frakf}$ for all $\frakf\in\mathring{\mathfrak{L}}$.
We define the vector bundle assignment $(V^\frakl)_{l\in \mathfrak{L}}$ the be the assignment 
$$ \{V^{(\frakt,n)}\}_{(\frakt,n)\in \mathfrak{L}_+\times \NN}= \{V^\frakt\otimes (T^*M)^{ n}\}_{(\frakt,n)\in \mathfrak{L}_+\times \NN} \quad \text{and} \quad \{V^{\frakf}\}_{\frakf \in \mathfrak{L}_-}$$
extended to $\mathfrak{L}$ so that \eqref{eq:assigment_dual} is satisfied.\footnote{Recall the definition of $(T^*M)^n$ in the presence of a scaling $\fraks$ from Section~\ref{section higher order derivatives} .} 

This puts us in the setting of Section~\ref{section indexing vector bundels}.
and we introduce the edge types
\begin{itemize}
\item $\mathcal{E}_+\subset\mathfrak{L}_+\times \NN$, where $(\frakk,n)\in \mathcal{E}_+$ if the term $(D^n u^\frakk)$ appears in (\ref{eq main})\footnote{Note that this is not quite well defined in the sense that one can always artificially increase this set. But there is generically a canonical minimal choice.},
\item $\mathcal{E}_-=\mathfrak{L}_-$ 
\item $\mathcal{E}_0=\mathcal{E}_0^+\cup \mathcal{E}_0^0$, where $\mathcal{E}_0^+$ is a disjoint copy of $\mathcal{E}_+$
and  $$\mathcal{E}_0^0\subset \{(\frakt,\sigma)\subset \mathfrak{L}_+\times \NN^{\mathcal{E}_+\cup \mathcal{E}_-} \ :\  D^{\sigma} F^\frakt \neq 0 \} \, 
$$
where there is usually a choice of of the exact subset.\footnote{It is made depending on the parameters $\gamma<\delta_0$ suggested by the equation.}
\end{itemize}
It is important that we do \textit{not} identify $\mathcal{E}_0^+$ and $\mathcal{E}_+$! But we define the injection 
\begin{equation}\label{eq:iota}
\iota: \mathcal{E}_+\to \mathcal{E}_0, \quad (\mathfrak{l},n) \mapsto (\mathfrak{l},n) \  .
\end{equation}

\begin{remark}
Note that $\mathcal{E}_0^0$ typically contains an infinite number of elements. 
\end{remark}
We also fix the indexing maps
\begin{itemize}
\item $\ind_+=  (\ind^+_+, \ind_+^-) : \mathcal{E}_+\to \mathfrak{L}\times\mathfrak{L},\ ( \mathfrak{k},n)\mapsto (\mathfrak{k},( \mathfrak{k},n))$,
\item $\ind_- : \mathcal{E}_-\to \mathfrak{L},\  \mathfrak{f} \mapsto \frakf$,
\item $\ind_0 : \mathcal{E}_0\to \mathbb{N}^\mathfrak{L}$ is defined as 
\begin{enumerate}
\item $\ind_0|_{\mathcal{E}^0_0}: (\frakk,\sigma) \mapsto \frakk-\sum_{\varepsilon\in \mathcal{E}_+\cup\mathcal{E}_-} \sigma_{\varepsilon} \varepsilon_- \ ,$
\item $\ind_0|_{\mathcal{E}^+_0} = \ind_+^- \circ \iota^{-1},$ 
\end{enumerate}
\end{itemize}
a homogeneity assignment $|\cdot |: \mathcal{E}\to \mathbb{R}$ and $\delta_0> \gamma>0$, which in general is suggested by the Equation (\ref{eq main}).
This then puts us in the setting of Section~\ref{section typed rooted trees}.
It is now natural to rewrite the equation in the form 
$$\mathfrak{L}u^\frakt=F^{\frakt}\big((u^o)_{o\in \mathcal{E}_+}, (\xi^\frakf)_{\frakf\in \mathcal{E}_-} \big)\  .$$

We introduce the space $\bar{E}= \bigoplus_{\varepsilon\in \mathcal{E}_+\cup \mathcal{E}_-} V^\varepsilon$ and in view of Section~\ref{section nonlinearities} we shall lift the local non-linearity $F^\frakt$ to an element of
\begin{align*}
\mathcal{C}^\infty\big(\bar{E}\ltimes  (JE^{\frakt}\otimes P(\bar{E}))\big)&= \bigoplus_{n\in \NN}\mathcal{C}^\infty\big(\bar{E} \ltimes (JE^{\frakt}\otimes (\bar{E}^*)^n) \big)\\
&= \bigoplus_{\sigma\in \NN^{\mathcal{E}_+\cup \mathcal{E}_-}}\mathcal{C}^\infty\big(\bar{E} \ltimes ( JE^{\frakt}\otimes (V^\sigma)^*) \big)\ .
\end{align*}
Here again, the exact truncation is chosen depending on the parameters $\gamma, \delta_0$.
This motivates the definition of a rule ${R}:\mathcal{E}\to \mathcal{P}(\mathcal{N})\setminus \{\emptyset\}$ as the normalisation of the following naive rule $\mathring{R}$.
\begin{itemize}
\item $\mathring{R}(\varepsilon)={\emptyset}$ for $\varepsilon\in \mathcal{E}_0\cup \mathcal{E}_-$
\item for $(\mathfrak{k},n)\in \mathcal{E}_+$, 
$$\mathring{R}((\mathfrak{k},n))=\big\{ [(\frakk,\sigma) ]\sqcup \sigma \ :\  (\frakk ,\sigma)\in \mathcal{E}^0_0  \big\} \ ,$$
where we interpret $\sigma\in \mathbb{N}^{\mathcal{E}_+}$ as a multi-set and recall the notation $\sqcup$ for the union of multi-sets.
\end{itemize}
Since the rule $\mathring{R}$ is equation-like, $R$ is as well.

\begin{assumption}
We assume that the rule $R$ constructed above is subcritical.
\end{assumption}

We define for $o\in \mathcal{E}_+$ the set 
$$\mathring{\mathfrak{T}}^{o}=\bigcup_{\nu\in R(o)} (\mathfrak{T}^\nu)_{<\gamma - |o|}, 
\quad 
{\mathfrak{T}}^{o}=\{\mathcal{I}_o(T) : T\in \bigcup_{\nu \in R(o)} \mathfrak{T}^\nu\} $$
and make the following observations.
\begin{enumerate}
\item For $o=(k,n)\in \mathcal{E}_+$ and $T\in \mathring{\mathfrak{T}}^o\cup {\mathfrak{T}}^{o}$, the node type $\mathfrak{n}(n)$ of every node $n\in N_T$ which is not maximal  nor minimal (i.e.\ the root) has exactly one incident edge of type in $\mathcal{E}^0_0$.

\item For $o=(\frakk, n)\in \mathcal{E}_+$ and $o'=(\frakk, m)\in \mathcal{E}_+$ 
one has $\mathring{\mathfrak{T}}^{o}= \mathring{\mathfrak{T}}^{o'}$.
\end{enumerate}
\begin{remark}

The trees in ${ \mathfrak{T}}^o\cup \{\delta^{\iota(o)}\} $ are used for the expansion of $\{u^o\}_{o\in \mathcal{E}_+},$ while the trees in $\mathring{\mathfrak{T}}^{o}$ are used for the expansion of the right-hand side of Equation~\ref{eq main} .
\end{remark}
\begin{assumption}\label{assumption polynomial in noise}
We assume $F$ is mulitilinear in the noise variables and the modelled distribution describing $\xi^\frakf$ is given by $p\mapsto \Xi_p$. 
\end{assumption}

\begin{remark}
A weaker natural assumption would be to impose that the lift of the noise $\xi^\frakf$ is given by a modelled distribution taking values in a sector of negative regularity. This, combined with the assumption of sub-criticality actually that the non-linearity is polynomial in the noises. 
\end{remark}

The above assumption tells us that for each $\frakt\in \mathfrak{L}_+$ one has
\begin{equation}\label{eq:non-linearity-decompostion}
F^{\frakt}\big((u^o)_{o\in \mathcal{E}_+}, (\xi^\frakf)_{\frakf\in \mathcal{E}_-} \big)\ = F^{\frakt} \big((u^o)_{o\in \mathcal{E}_+},0 \big)+\sum_{\sigma\in \NN^{\mathcal{E}_-}} D^{\sigma} F^{\frakt}\big((u^o)_{o\in \mathcal{E}_+}, 0 \big) \xi^\sigma
\end{equation}
where we used the shorthand notation $\xi^\sigma= \prod_{\frakf\in \mathcal{E}_-} (\xi^\frakf)^{\sigma_\frakf}$ .

It will be useful to denote for $\frakf\in \mathcal{E}_-$ by $U^\frakf$ the modelled distribution $p\mapsto \Xi_p$ and interpret it as modelled distribution where the component of $U^\frakf$ in $T^\frakf_{0,:}$ is constantly $0$. This allows to write the lift of the non-linearity simply as 
\begin{equation}\label{rhs explicit lift}
F^\frakt(U)(p)= \sum_{m\in\NN^{\mathcal{E}_+\cup \mathcal{E}_- }} j^E_p \big(\frac{D^m F^\frakt}{m!}\big) (Q_{0} U(p)) \cdot \tilde{U}^{ m}(p) \ 
\end{equation}
and to formulate the lift of Equation~\ref{eq main} as
\begin{equation}\label{eq abstract}
U^o= \mathcal{I}_{o} Q_{<\gamma-|o|} F^\frakt\big((U^o)_{o\in \mathcal{E}_+\cup \mathcal{E}_-} \big) + u^{o}  \quad \text{ for } o \in \mathcal{E}_+ \ ,
\end{equation}
where $u^{o}$ is a section of the jet bundle $J^{\iota (o)}_{<\gamma,:}$.
\begin{remark}
It is important to note that the solution depends on the choice of model!
\end{remark}
\begin{prop}
Let $\{U^\frakt \}_{\frakt}$ be the solution of the abstract fixed point problem~\eqref{eq abstract} with respect to the canonical model for smooth noises $\{\xi_{\mathfrak{l}} \}_{l\in \mathfrak{L}_-}$. Then $v^\frakt=\mathcal{R}(U^\frakt)$ solves Equation~\ref{eq main} in the classical sense.
\end{prop}
\begin{proof}
Indeed we have
$$
v^\frakt= \mathcal{R}U^\frakt\\
= \mathcal{R}\mathcal{K}_{\frakt} F^\frakt \big((U^o)_{o\in \mathcal{E}_+\cup \mathcal{E}_-} \big) \\
= K_\frakt \mathcal{R}F^\frakt \big((U^o)_{o\in \mathcal{E}_+\cup \mathcal{E}_-} \big)
$$
and thus 
\begin{align*}
\mathcal{L}^\frakt v^\frakt(p)&= \mathcal{R} F^\frakt \big((U^o)_{o\in \mathcal{E}_+\cup \mathcal{E}_-} \big)(p)\\
&=\Pi_p F^\frakt \big((U^o)_{o\in \mathcal{E}_+\cup \mathcal{E}_-} \big) (p)\\
&= \Pi_p \sum_{\sigma\in\NN^{\mathcal{E}_+\cup \mathcal{E}_- }} j^E_p \big(\frac{D^\sigma F^\frakt}{\sigma!}\big) (Q_0 U(p))   \tilde{U}^{ \sigma} (p) \\
&= \Tr \sum_{\sigma\in\NN^{\mathcal{E}_+\cup \mathcal{E}_- }}  \big(\frac{D^\sigma F^\frakt}{\sigma!}\big) (Q_0 U_0(p)) \odot (\Pi_p\tilde{U}(p))^{\odot \sigma} \\
&= F^\frakt\big( \Pi_pU(p) (p)\big) \\
&= F^\frakt\big((v^o)_{o\in \mathcal{E}_+} \big)(p) 
\end{align*}
where we used Corollary~\ref{prop:reconstruction for continous models} several times.
\end{proof}
%


\subsection{Renormalised equation}
For $\tau\in \mathfrak{T}_-$ and $\frakt\in \mathfrak{L}_+$, define the new local non-linearity $(D_\tau F)^{\frakt}$
acting on families of smooth sections $v=\{v^o\}_{o\in \mathcal{E}_+\cup\mathcal{E}_-}$ where $v^{o}\in C^{\infty}(V^o)$ for each $o \in \mathcal{E}_+\cup \mathcal{E}_-$.
First, we define for $\tau\in \mathfrak{T}_-$ and $n \in N_\tau\setminus L_\tau$, where $n\neq \rho_\tau$
$$(D_{\tau,n}F)^{\frakt} : \prod_{o\in \mathcal{E}_+\cup\mathcal{E}_- } C^\infty (V^o)\to C^\infty (V^{\dif_\tau(n)}) ,
\quad
(D_{\tau,n}F)(v)^{\frakt}=D^{\mathfrak{n}(n)} F^{\mathfrak{e}(n\downarrow)_+} ( v)\ , $$
where $n\downarrow$ denotes the unique edge $e\in E_{\tau}$ such that $e_+=n$. Furthermore, set 
$$(D_{\tau,\rho_\tau}F)^{\frakt} : \prod_{o\in \mathcal{E}_+\cup\mathcal{E}_- } C^\infty (V^o)\to \mathcal{C}^\infty (V^{\red_* (\frakt-\ind(\rho))}) , \quad (D_{\tau,\rho_\tau}F)^{\frakt}(v) =D^{\mathfrak{n}(\rho_\tau)} F^\frakt (v) \ .$$
%
%
%
Finally, we define
\begin{align*}
(D_\tau F)^\frakt: \prod_{o\in \mathcal{E}_+\cup\mathcal{E}_-} \mathcal{C}^{\infty}(V^{o})\to &\mathcal{C}^\infty (V^{\red_* (\frakt-\ind(\rho))})\times \prod_{n\in N_\tau\setminus (L_\tau\cup\{\rho_\tau\})} \mathcal{C}^\infty (V^{\dif_T(n)}) 
\end{align*}
%
%
as 
$$(D_\tau F)^{\frakt}(v)= \prod_{n\in N_\tau \setminus L_\tau }\frac{1}{S_n(\tau)} (D_\tau F)^{\frakt, n}\big(v(\cdot_n) \big)\ ,$$
where $S_n(\tau)$ is defined as follows: If $\tau$ at the node $n\in N_\tau$ has the form
 $$\tau_{\geq n}= \prod_{i=1}^{m}(\Xi_i)^{\alpha_i} \prod_{j=1}^l (\mathcal{I}_{o_i}(T_i))^{\beta_j}$$ where terms are paired in such a way that 
$$ (o_i, T_i)\neq (o_j, T_j) \quad \text{and} \quad \Xi_i\neq \Xi_j$$
whenever $i\neq j$, we set
$$S_n(\tau)= \prod_{i=1}^m \alpha_i ! \cdot \prod_{j=1}^l \beta_j ! \ .$$
\begin{remark}
Note that $S(\tau):=\prod_{n\in N_\tau} S_n(\tau)$ agrees with the corresponding definition in \cite{BCCH20} and $S(\tau)_n$ has the interpretation of ``the number of tree symmetries at the node $n$'', see also Remark~\ref{rem:relating S(tau) to bar tau}
\end{remark}
\begin{prop}\label{prop:renormalised equation}
If $U^g\in \mathcal{D}^\gamma$ solves the abstract fixed point problem for the renormalised model and for $\gamma$ large enough, then $u= \mathcal{R}U^g$ solves the equation 
$$\mathfrak{L}u^\frakt= F^{\frakt}\big((u^o)_{o\in \mathcal{E}_+}, (\xi^\frakf)_{\frakf\in \mathcal{E}_-} \big)  + \sum_{\tau\in \mathfrak{T}_-} \left\langle (D_\tau F)^{\frakt}\big((u^o)_{o\in \mathcal{E}_+}, (\xi^\frakf)_{\frakf\in \mathcal{E}_-} \big),  g(\tau) \right\rangle\ \,  $$ 
where the bracket $\langle \cdot , \cdot \rangle$ denotes the canonical pairing between $\mathfrak{Dif}_\tau$ and 
$$\mathcal{C}^\infty (V^{\red_* (\frakt-\ind(\rho))})\times \prod_{n\in N_\tau\setminus (L_\tau\cup\{\rho\})}\mathcal{C}^\infty (V^{\dif_T(n)})$$ defined in Remark~\ref{rem pairing differential operator}.
\end{prop}

In order to prove this theorem, it will be useful to introduce certain rooted plane trees.
Recall that a rooted plane tree $\bar{T}$, consists of a rooted tree together with an order of the incoming edges of each node, that is for each $n\in N_T$, the set
 $E^{\mathtt{in}}_{T,n}=\{e\in E_T \ | \ e_-= n \}$ is equipped with a strict total order $<_{T,n}$. 
For rooted pane trees $\pT$ we denote by $\lfloor \pT \rfloor$
the rooted tree obtained by forgetting the order on edges.

Next, we fix a strict total order $\prec$ on $\mathcal{E}_0^0 \cup\mathcal{E}_+ \cup \mathcal{E}_-$
such that 
$\varepsilon_0\prec \varepsilon_+ \prec \varepsilon_-$ for every  $\varepsilon_0\in \mathcal{E}_0^0, \ \varepsilon_+\in \mathcal{E}_+, \ \varepsilon_-\in \mathcal{E}_-$.

Lastly, we declare a tree in $\pT$ to conform to the order $\prec$, if 
for each $n\in N_T$ and any two edges $e_1,e_2\in E^{\mathtt{in}}_n$ such that $e_1<_{T,n} e_2$ one has 
\begin{itemize}
\item if $\mathfrak{e}(e_1), \ \mathfrak{e}(e_2) \notin \mathcal{E}_0^+$,
$$\mathfrak{e}(e_1) \prec \mathfrak{e}(e_2) \ ,$$
\item if $\mathfrak{e}(e_1) \notin \mathcal{E}_0^+$ but  $\mathfrak{e}(e_2) \in \mathcal{E}_0^+$
$$\mathfrak{e}(e_1) \prec \iota^{-1}(\mathfrak{e}(e_2)) \ ,$$
and similarly, if the roles of $e_1$ and $e_2$ are reversed,
\item if $\mathfrak{e}(e_1),\ \mathfrak{e}(e_2) \in \mathcal{E}_0^+$
$$\iota^{-1}(\mathfrak{e}(e_1)) \prec \iota^{-1}(\mathfrak{e}(e_2)) \ .$$
\end{itemize}

At this point, let us introduce the following analogues of the objects already introduced for plane trees:
\begin{enumerate}
\item For $o\in \mathcal{E}_+$ denote by $\pTset^o$ the set of plane rooted trees $\pT$ that conform to $\prec$ and such that $\lfloor \pT \rfloor \in \mathfrak{T}^o$. Similarly, we set $\pTset_-$ to consist of plane trees $\bar{\tau}$ which conform to $\prec$ and such that $\lfloor \tau\rfloor \in \mathfrak{T}_-$.
\item For plane rooted trees $\pT_1, \ldots , \pT_n$, denote by $\pT=\pT_1\cdot \ldots \cdot \pT_n$ the plane rooted tree such that
\begin{itemize}
\item 
 One has $\lfloor \pT \rfloor = \lfloor \pT_1 \rfloor \cdots \lfloor \pT_n \rfloor$,
 where the latter tree product was introduced in Remark~\ref{rem recursive construction}.
\item The order of the edges $E^{\mathtt{in}}_{\pT,n}$ for $n\in N_{\bar T}\setminus \rho_{\bar T}$ agrees with the order inherited from the $\pT_1,\ldots, \pT_n$.
\item The order of $E^{\mathtt{in}}_{\pT,\rho}$ is the unique order such for each $i, j \in \NN$ and $e_i \in E_{\pT_i, \rho }\subset E_{\pT, \rho }$  one has $e_i \prec e_j$ whenever $i<j$, and such that the order on $E_{\pT_i, \rho_{\pT} }\subset E_{\pT, \rho_{\pT} }$ is simply inherited from the order in $\pT_i$.
\end{itemize}
We call this the \textit{plane tree product}.
\end{enumerate}
\begin{remark}
Instead of working with a rule on rooted trees together with an order on edges, we could have introduced an (obvious) notion of a plane rule.
\end{remark}

For $\pT\in \pTset$ we define maps $\Upsilon[\pT]: JE\to \Func_W \langle \lfloor \pT \rfloor \rangle$ as follows:
\begin{enumerate}
\item If $\pT= \delta^{\iota{(o)}}$ consists of exactly one edge of type $\iota(o)\in \mathcal{E}_0^+$, then
$$\Upsilon[\pT](u_p)= \Upsilon[\delta^{\iota{(o)}}](u_p)=Q_{>0}u^{\iota(o)}_p\ .$$
\item If $\pT=\delta^{e}$ consists of exactly one edge of type $e=(\frakt,\sigma)\in \mathcal{E}_0^0$ ,  then
$$\Upsilon[\pT](u_p)=j^E_p \big(\frac{D^\sigma F (Q_0 u_p) }{\sigma!}\big)\  .$$
\item If $\pT=\Xi^{\frakf}$ consists of exactly one edge of type $\frakf\in \mathcal{E}_-$, then
 $$\Upsilon[\Xi^{\frakf}](u_p)=\Xi^{\frakf}_p\ .$$
\end{enumerate}
We further set 
\begin{enumerate}
\item[(4)] for $\pT =\pT^1\cdots\pT^m$ set
$\Upsilon[\pT](u_p)=  \Upsilon[\pT^1](u_p)  \cdots  \Upsilon[\pT^m](u_p)\ .$ 
\item[(5)] $\Upsilon[\mathcal{I}_\frakt \pT](u_p)= \mathcal{I}_\frakt\Upsilon[\pT](u_p)$.
\end{enumerate}

The next lemma now follows directly from the definition of the maps $\Upsilon$.
\begin{lemma}
Let $U= \{U^o\}_{o\in \mathcal{E}_+\cup \mathcal{E}_-}$ be the abstract solution of (\ref{eq main}). The solution can be uniquely decomposed as $U^o= u^o + R^o$, where $u^{ o}_p\in J_p V^{\iota (o)}_{<\gamma}$ and $R^o_p\in {\mathcal{T}}^o_p$ and it holds that
\begin{equation}\label{eq:formula for the solution}
U^o= u^o+ Q_{<\gamma}\sum_{\pT \in {\pTset}^o} \Upsilon[T](u) = Q_0 u^o+ Q_{<\gamma}\sum_{\pT \in {{\pTset}}^o \cup \{\delta^{\iota(o)} \}} \Upsilon[T](u)\ .
\end{equation}
\end{lemma}
\begin{remark}
The fact that we work with plane trees accounts for the combinatorial factors usually encountered. A version of this observation was already made in the case of $\Phi^4$ equations in \cite{CMW23}. Similarly, the fact that we do not identify the abstract product on jets with the canonical product also simplifies the expression.
\end{remark}

\begin{remark}\label{rem:relating S(tau) to bar tau}
For a tree $\tau \in \mathfrak{T}_-$, we define
$k_\tau: N_\tau \to \mathbb{N}^{\mathcal{E}_+ \cup \mathcal{E}_{-}}$ to be the function such that for $n\in N_\tau$ and $o\in \mathcal{E}_+ \cup \mathcal{E}_{-}$ the number $k_\tau(n)_o\in \mathbb{N}$ denote the number of edges of type $o$ attached to $n$ from above.
Then one finds that for  $\tau \in \mathfrak{T}_-$ one has
$$\left| \big\{  \bar{\tau} \in \pTset_- \ : \ \lfloor  \bar{\tau} \rfloor = \tau   \big\} \right| = \prod_{n\in N_\tau } \frac{k_\tau (n)!}{S_n(\tau)} \ .$$
\end{remark}

\begin{remark}
We extend Remark~\ref{rem gluing construction first version} and Remark~\ref{remark attaching trees to negative trees} to plane trees (in some way conforming with $\prec$ ) and write $\bar{\tau}(G)$ for a tree obtained from 
$\bar{\tau}\in {\pTset}_-$ and $\bar{G}=\{\bar{G}_n\}_{n\in N_\tau}$ as therein. Note that this definition is not unique, but we make a choice (which turns out to be irrelevant).
Then, it if follows from the fact that $\bar{\tau}$ contains no edges of type in $\mathcal{E}_0$ that for $\tau:= \lfloor \bar{\tau} \rfloor$ 
 $$\Upsilon[\bar{\tau}(\bar{G})]= \tau(\{\Upsilon[\bar{G}_n]\}_{n\in N_\tau}) \ .$$
\end{remark}
%

\begin{proof}[Proof of Proposition~\ref{prop:renormalised equation}]
We shall use the common notation for $m, k \in \mathbb{N}\in \mathbb{N}^{\mathcal{E}_+ \cup \mathcal{E}_{-}}$ and set for $m\geq k$
$$k!:= \prod_{o\in \mathcal{E}_+ \cup \mathcal{E}_{-}} k_o! \ , \qquad {m \choose k}:=  \prod_{o\in \mathcal{E}_+ \cup \mathcal{E}_{-}}{m_o \choose k_o}\ .$$
Furthermore, whenever we have a sum over $\mathbb{N}$ or another infinite set such as $\mathbb{N}^{\mathcal{E}_+ \cup \mathcal{E}_{-}}$ it is understood that the sum is in reality truncated (with the truncation depending on the parameter $\gamma$).

We shall write  $\bar{u}= Q_{0 } U$ and $\tilde{U}= U-\bar u$ one use for $m \in \mathbb{N}^{\mathcal{E}_+ \cup \mathcal{E}_{-}}$ the multi-index notation $\tilde{U}^m:= \prod_{o\in \mathcal{E}_+ \cup \mathcal{E}_{-} } (\tilde{U}^o )^{m_o}$. 
We also fix the convention that whenever we write $\prod_{o\in \mathcal{E}_+ \cup \mathcal{E}_{-} } \bar{T}^{o}$, it is understood that the product is taken in the order respecting $\prec$ . For each multi-index we $m\in \mathbb{N}^{\mathcal{E}_+ \cup \mathcal{E}_{-}}$ we introduce the set 
$$\pTset^{(m)}:= \left\{ \prod_{o\in \mathcal{E}_+ \cup \mathcal{E}_{-}} \prod_{i=1}^{m_o} \pT^{(o, i)} \ : \ \pT^{(o, i)}\in \pTset^{o}\cup \delta{\iota(o)} \right\}$$
Then it follows for $m\in \mathbb{N}^{\mathcal{E}_+ \cup \mathcal{E}_{-}}$ that
$$\Pi^g_p \left( \tilde{U}^m\right) 
= \Pi^g_p \Big( \prod_{o\in \mathcal{E}_+ \cup \mathcal{E}_{-}} \prod_{i=1}^{m_o} \big( \sum_{\pT^{(o,i)} \in \pTset^{o}\cup \delta^{\iota (o) } } \Upsilon [ \pT^{(o,i)} ] \big) \Big) 
= \Pi^g_p \sum_{\pT\in \pTset^{(m)} } \Upsilon [ \pT ]  
$$
and therefore from Lemma~\ref{lemma pi renormalized} that
\begin{align*}
\Pi^g_p \left( \tilde{U}^m\right) &= \Tr  \bigodot_{\substack{o\in \mathcal{E}_+ \cup \mathcal{E}_{-}\\
i\in \{1,\ldots,m_o\}}}
  \Big( \sum_{\pT^{(o,i)} \in \pTset^{o}\cup \delta^{\iota (o) } } \Pi^g_p \Upsilon [ \pT^{(o,i)} ]  \Big)  \\
&\qquad +\sum_{\pT\in \pTset^{(m)} } 
 \sum_{\tau\subset_\rho \pT } g(\bar \tau)\Big( \prod_{n\in {N}_\tau} \mathring{\Pi}^g_p \Upsilon[(\pT \setminus \bar \tau)_{\geq n}] (\cdot_n)\Big) \ ,
\end{align*}
where $\bar \tau\subset_\rho \pT$ means that $\bar \tau\subset \pT$ contains the root of $\pT$.
In particular,
\begin{align*}
\Pi^g_pF^\frakt(U)= & \mathring{\Pi}^g_p\sum_{m\in \mathbb{N}^{\mathcal{E}_+ \cup \mathcal{E}_{-}}}  j^E_p \frac{D^m F^\frakt(\bar{U})}{m!}   \tilde{U}^{  m} \\
&+ \sum_{m\in \mathbb{N}^{\mathcal{E}_+ \cup \mathcal{E}_{-}}} \sum_{\pT\in \pTset^{(m)} } 
 \sum_{\bar \tau\subset_\rho \pT } g(\bar \tau)\big( \prod_{n\in {N}_{\bar \tau}} \mathring{\Pi}^g_p \Upsilon[(\delta^{(\frakt,m)}\cdot \pT \setminus \bar \tau)_{\geq n}] (\cdot_n)\big)
\end{align*}
where we used in the last line, the fact that the differential operator $g(\tau)$ acts multiplicatively at the root, c.f.\ Remark~\ref{rem pairing differential operator}. 
For the first term we find 
\begin{equation}\label{eq:first term renormalised eq}
\mathring{\Pi}^g_p \left(\sum_{m\in \mathbb{N}^{\mathcal{E}_+ \cup \mathcal{E}_{-}}}  j^E_p \frac{D^m F^\frakt (\bar{U})}{m!}   \tilde{U}^{  m} \right)(p) = F^\frakt(u)(p) \ .
\end{equation}
%
Next fix ${\tau}\in \mathfrak{T}_-$ and recall the map $k_{\tau} : N_\tau \to \mathbb{N}^{\mathcal{E}_+ \cup \mathcal{E}_{-}}$ from Remark~\ref{rem:relating S(tau) to bar tau}. 
Looking only at summands such that $\lfloor\bar \tau \rfloor$ is tree isomorphic to $\tau$ we find
\begin{align*}
& \sum_{m\in \mathbb{N}^{\mathcal{E}_+ \cup \mathcal{E}_{-}}}  \sum_{\pT\in \pTset^{(m)} }  \sum_{\bar \tau\subset_\rho \pT } \mathbf{1}_{\lfloor\tau \rfloor \sim  \bar \tau} g(\bar \tau)
\Big( \prod_{n \in N_\tau} \mathring \Pi^g_p (\Upsilon[(\delta^m T\setminus \bar \tau)_{\geq n}] (\cdot_n) \Big)\\
&= \sum_{\sigma: N_\tau \to \mathbb{N}^{\mathcal{E}_+ \cup \mathcal{E}_{-}}} g({\tau})
\Bigg( \prod_{n \in N_{{\tau}} } {\sigma(n)+k_\tau(n) \choose k_\tau(n)} \frac{k_\tau(n)!}{S_n({\tau})} \Big( \sum_{\pT^{(n)} \in \pTset^{(\sigma(n))} }\mathring \Pi^g_p\Upsilon[\delta^{(\sigma(n)+k_\tau(n))} \pT^{(n)}] (\cdot_n) \Big) \Bigg) \ ,
\end{align*}
where we sum over all maps $\sigma: N_\tau \to \NN^{\mathcal{E}_+ \cup
 \mathcal{E}_-}$ (with the finite truncation implicitly understood). Here $\sigma(n)$ for number of edges attached at a node $n\in N_\tau$ which do not belong to $\tau$ and are not of type $\mathcal{E}_0^0$.  To see the line introducing the combinatorial factor ${\sigma(n)+k_\tau(n) \choose k(n)} \frac{k(n)!}{S_n({\tau})}$,
 one counts the number of ways the planted subtree $\bar \tau $ can appear as a subtree. Say, counting from the top nodes downwards\footnote{Alternatively, one can also count from the root up.}, for the top nodes of ${\tau}$ the combinatorial factor is trivial (since $k_\tau (n)=0$). Otherwise, one has to choose $k(n)$ edges of the $\sigma(n)+k(n)$ possible edges, which are part of $\tau$.  Then, it follows as in Remark~\ref{rem:relating S(tau) to bar tau} that one has exactly $\frac{k(n)!}{S_n(\tau)!}$ ways to assign the continuation $\tau$ at $n$. 

Next note that
\begin{align*}
&\sum_{\sigma: N_\tau \to \mathbb{N}^{\mathcal{E}_+ \cup \mathcal{E}_{-}}} g({\tau})
\Bigg( \prod_{n \in N_{{\tau}} } {\sigma(n)+k_\tau(n) \choose k_\tau(n)} \frac{k_\tau(n)!}{S_n({\tau})} \Big( \sum_{\pT^{(n)} \in \pTset^{(\sigma(n))} }\mathring \Pi^g_p\Upsilon[\delta^{(\sigma(n)+k_\tau(n))} \pT^{(n)}] (\cdot_n) \Big) \Bigg) \\
&=\sum_{\sigma: N_\tau \to \mathbb{N}^{\mathcal{E}_+ \cup \mathcal{E}_{-}}} g({\tau})
\Bigg( \prod_{n \in N_{{\tau}} } {\sigma(n)+k_\tau(n) \choose k_\tau(n)} \frac{k_\tau(n)!}{S_n({\tau})} \\
&\qquad\qquad\qquad\qquad\qquad \times \Big( \sum_{\pT^{(n)} \in \pTset^{(\sigma(n))} }\mathring \Pi^g_p j^E_p\frac{D^{\sigma(n)+k_\tau (n)}F(\bar{U})}{(\sigma(n)+k_\tau(n))!}\Upsilon[\pT^{(n)}] (\cdot_n)  \Big) \Bigg)\\
&=\sum_{\sigma: N_\tau \to \NN^{\mathcal{E}_+ \cup \mathcal{E}_{-}}}g({\tau})
\left( \prod_{n \in N_{\tau}} \frac{1}{S_n({\tau})} \sum_{\pT^{(n)} \in \pTset^{(\sigma(n))} }\mathring \Pi^g_p j^E_p\frac{D^{\sigma(n)+k_\tau(n)}F(\bar{U})}{\sigma(n)!}\Upsilon[\pT^{(n)}] (\cdot_n) \right)\\
&=\sum_{\sigma: N_\tau \to \NN^{\mathcal{E}_+ \cup \mathcal{E}_{-}}}g({\tau})
\left( \prod_{n \in {\tau}} \frac{1}{S_n(\bar{\tau})} \mathring \Pi^g_p j^E_p\frac{D^{\sigma(n)+k_\tau (n)}F(\bar{U})}{\sigma(n)!}   \tilde{U}^{  \sigma(n)} (\cdot_n) \right)\\
&=g({\tau})
\left( \prod_{n \in N_{\tau}} \frac{1}{S_n({\tau})} \sum_{m\in \NN^{\mathcal{E}_+ \cup \mathcal{E}_{-}}} \mathring \Pi^g_p j^E_p\frac{D^{m+k_\tau(n)}F(\bar{U})}{m!}   \tilde{U}^{  m} (\cdot_n) \right)\ .
\end{align*}
Finally, we find that 
\begin{align*}
&g({\tau})
\left( \prod_{n \in N_{\tau}} \frac{1}{S_n({\tau})} \sum_{m\in \NN^{\mathcal{E}_+ \cup \mathcal{E}_{-}}}  \mathring \Pi^g_p j^E_p\frac{D^{m+k_\tau(n)}F(\bar{U})}{m!} \tilde{U}^{  m} (\cdot_n) \right)(p)\\
&=g({\tau})
\big( \prod_{n \in N_{\tau}} \frac{1}{S_n(\bar{\tau})} (D^{k_\tau (n)}F)(u)  (\cdot_n) \big)(p)\\
&=\langle \big( D_{\tau} F\big)(u),g(\tau) \rangle (p)\ ,
\end{align*}
where in the first equality we used that 
$$\sum_{m\in \NN^{\mathcal{E}_+ \cup \mathcal{E}_{-}}} \left(\mathring  \Pi^g_p j^E_p \frac{D^{m+k_\tau (n)}F^\frakt(\bar{U})}{m!}   \tilde{U}^{  m}\right) (\cdot_n)- (D^{k(n)}F)(u)(\cdot_n)$$
vanishes at $p\in M$ faster than $-|\tau|$, which is the case whenever $\gamma>0$ is chosen large enough.
This combined with \eqref{eq:first term renormalised eq} completes the proof.

\end{proof}

\begin{remark}
Observe that we did not enforce the products of jets to be the canonical product, but that the results of this section extend to the case when such an identification is made, c.f.\ Remark~\ref{rem:identifying_under_can product on jets}.
\end{remark}

\section{Canonical Renormalisation Procedures}\label{sec:canonical renormalisation}
One observes that for a given equation we have constructed a rather large renormalisation group $\mathfrak{G}_-$. In our setting, compared to the usual setting of translation invariant SPDEs in flat space, one a priori obtains a family of solutions indexed by the (infinite-dimensional) space of sections of certain vector bundles over $M$. 
Fortunately, one can often still single out a canonical, \textit{finite dimensional} family of renormalised solutions. Let us informally describe the corresponding procedure: Typically one is working with a mollification scheme which depends on a scale $\epsilon$. If this mollification scheme is ``sufficiently covariant''  it is possible to decompose the corresponding family of renormalisation differential operators  $$g_\epsilon (T)= \sum_{i} f_i(\epsilon) \cdot g^i(T)  \ ,$$
where 
\begin{itemize}
\item each $f_i: (0,1)\to \mathbb{R}$ typically
\footnote{
It can sometimes be useful to also include $f_i$, which just converge to a finite limit in the above decomposition.
}
converges to $\pm \infty$ as $\epsilon \to 0$.
\item the $g^i(T)$ are differential operators which depend explicitly on the curvature and the differential operator of the linear part of the equation in a local manner  and thus are arguably ``canonical''. 
\end{itemize}
Thus, for such regularisation schemes it is then natural to work with only the finite dimensional subgroup ${\mathfrak{G}}'_-\subset \mathfrak{G}_-$ containing the differential operators $g_\epsilon (T)$. One thus obtains a canonical (finite dimensional) solution theory.

\begin{remark}
We note that, when working with an arbitrary mollification scheme, one in general needs elements of $g_\epsilon \in \mathfrak{G}_-$ which are not contained in ${\mathfrak{G}}'_-$ in order to converge to one of the  ``canonical" solutions. Thus, in this case the canonical solutions are only reached by considering the cosets  ${\mathfrak{G}}'_- g_\epsilon \subset \mathfrak{G}_-$.
\end{remark}

Note that in general it non-trivial to extract the operators $g^i(T)$ described above explicitly. When working for example with generalised Laplacians and white noise this is a rather straightforward consequence of the following: 
\begin{enumerate}
\item Explicit expansions of the heat kernel, c.f.\ Theorem~\ref{thm:heat_kernel_expansion}, which can formally be written as
\begin{equation}\label{eq:heat_kernel_asymptotics}
G_t(p,q)=  \frac{1}{(4\pi t)^{d/2}}e^{-d(p,q)/4t} \sum_{i=0}^N t^i \Phi_i(p,q) +O(t^{N-n/2}) \ ,
\end{equation}
for $0<t\ll 1$.
\item Explicit expansion of the density in exponential coordinates, c.f.\ \cite[Cor.~2.10]{Gra74}. Fur us the following will be sufficient
\begin{equation}\label{eq:volume_asymptotics}
 \dVol_p= \left(1-\frac{1}{6}\sum_{i,j} R_{i,j} x_ix_ j  \right) dx +\mathcal{O}(|x|^3) dx \ .
\end{equation}
\end{enumerate}

\begin{remark}
This explains, why for many equations of interest when working with a heat kernel regularisation, renormalisation constants are sufficient. 
More precisely, when working with a generalised Laplacian and white noise, as long as the most negative tree has homogeneity higher than $-2$, one does not expect any geometric counterterms to appear. 
In Section~\ref{sec:phi34} we treat an equation where this condition is violated and indeed a logarithmic divergence proportional to the scalar curvature $s(x)$ appears.
\end{remark}

\section{Concrete applications}\label{sec:concrete applications}
Finally, we turn to the three examples described in the introduction. Recall that for a vector bundle $E\to M$, $E$-valued (spatial) white noise can be characterised as the unique random distributions $\xi\in \mathcal{D}'(E)$,
such that for any $\psi, \psi' \in \mathcal{D}(E)$, 
\begin{enumerate}
\item $\xi  ( \psi)$ is a centred Gaussian random variable,
\item $E[\xi  ( \psi), \xi  ( \psi')]= \int_M \langle \psi, \psi' \rangle_E \dVol_g$.  
\end{enumerate}
We speak of $E$-valued space-time white noise, when the bundle above is replaced by the pull-back bundle $\pi^*E$, where $\pi: \mathbb{R}\times M \to M$ is the canonical projection.

We shall often writhe $\langle \xi  , \psi \rangle = \xi  ( \psi)$, where the absence of a subscript means the canonical pairing of $\mathcal{D}'(E)\times \mathcal{D}(E)$ or $\mathcal{D}'(\pi^*E)\times \mathcal{D}(\pi^*E)$ respectively.

We shall be working with the following two regularisation procedures. Let $G_t\in \mathcal{C}^\infty (E \hat{\otimes} E^*)$ be the heat kernel of the Laplacian in the equation.
\begin{enumerate}
\item If $\xi\in \mathcal{D}'(E)$ is spatial white noise, set $$\xi_\epsilon (p)= \langle \xi, G_{{\epsilon}^2} (p,\cdot )\rangle  $$
\item If $\xi\in \mathcal{D}'(\pi^*E)$ is space time white noise, furthermore fix 
an additional even, smooth compactly supported function 
$\phi: \mathbb{R}\to \mathbb{R}$ integrating to $1$. 
Writing
 $\phi^\epsilon(t) :=  \frac{1}{\epsilon^2}\phi \left( \frac{t}{\epsilon^2} \right) $,
  we shall set 
\begin{equation}\label{eq:mollified space-time white noise}
 \xi_\epsilon (t,p)= \langle \xi, \phi^\epsilon(t-\cdot) G_{{\epsilon}^2} (p,\cdot )\rangle \ ,
 \end{equation} 
 which can informally be written as $\xi_\epsilon (t,p)= \int_{\mathbb{R}\times M} \phi^\epsilon(t-s) G_{{\epsilon}^2} (p,q )\xi(ds, dq)$. 
\end{enumerate}

Furthermore, fix a cut-off $\kappa: \mathbb{R}\to [0,1]$ such that the following properties are satisfied
\begin{itemize}
 \item $\kappa(t)=0$ for $t<0$ or $t>2$,
 \item $\kappa(t)=1$ for $t\in [0,1)$.
 \item $\kappa|_{\mathbb{R}_+}$ is smooth.
 \end{itemize}

For fixed $N \in \mathbb{N}$ large enough\footnote{This number is conveniently fixed depending on the equation, for the examples discussed here $N=10$ is sufficent.} we denote by $\bar{r}= \frac{\min_{p \in M} r_p}{N}$, where $r_p$ is the injectivity radius at $p\in M$.
We introduce the following notation for the leading terms of the scalar heat kernel $G$, c.f.\ \eqref{eq:heat_kernel_asymptotics}  
\begin{align}\label{eq:diverging parts of kernel}
Z_t(p,q):&=  \frac{1}{(4\pi t)^{d/2}}e^{-d(p,q)/4t} \kappa(d(p,q)/\bar{r}) \kappa(t)\ ,\\
Z^1_t(p,q):& = t \, Z_t(p,q)\Phi_1(p,q) \, 
\end{align}
as well 
$Z^E_t(p,q)=  Z_t (p,q) \tau_{p,q} $
for the $E$-valued kernel, where we recall that $\tau_{p,q}: E|_q\to E|_p$ is the parallel transport map.

Lastly, we use the following notation for the corresponding truncation of the heat kernel in Euclidean space
$$\bar{Z}_t (x)= \frac{1}{(4\pi t)^{d/2}}e^{-|x|^2 /4t} \kappa(t) .$$


\subsection{The g-PAM equation}
For $\epsilon>0$, we define the following counterterms:
\begin{align}
{C_\epsilon}&=\int_{\mathbb{R}} \int_{B_{r}(0)} \bar{Z}_{t-s+\epsilon^2}(x) \bar{Z}_{\epsilon^2}(x) \ dx\  
 ds ,\label{pam counterterm1} \\
{C_\epsilon'}&=\int_{\mathbb{R}^2}	\int_{B_{r}(0)} \partial_i \bar{Z}_{t-s+\epsilon^2}(x ) \partial_i \bar{Z}_{t-s'+\epsilon^2}(x ) dx\ 
ds ds' \  \label{pam counterterm2} 
\end{align}
where the latter is the same for $i\in \{1,2\}$.
In this section we proof the following theorem.
\begin{theorem}\label{thm:g-pam}
Let $M$ be a compact two dimensional Riemannian manifold, $\xi$ be white noise on $M$, and $\xi_\epsilon$ its regularisation as described above.
Furthermore, let $f\in \cC^{\infty}(\RR)$, $A(\cdot, \cdot)$ be a smooth bilinear form on the tangent bundle $TM$ and let $u_0\in \cC^{\alpha}(M) $ for $\alpha>0$.
For $C_\epsilon, C'_\epsilon$ as in \eqref{pam counterterm1}\& \eqref{pam counterterm2} and for $ u_\epsilon$ solving 
\begin{equation}\label{gpam}
\partial_t u_\epsilon + \Delta u_\epsilon = 
A ( \nabla u_\epsilon, \nabla u_\epsilon)- C_\epsilon' f^2(u_\epsilon)(\tr A) +f(u_\epsilon)\left(\xi_\epsilon- C_\epsilon f'(u_\epsilon)\right), \qquad u_\epsilon(0)=u_0 \ ,
\end{equation}
there exists a (random) $T>0$ and $u: [0,T]\times M \to \mathbb{R}$, such that $u_\epsilon \to u$ uniformly as $\epsilon\to 0$ in probability.
\end{theorem}
The proof of the theorem follows along the usual lines, when applying the theory of regularity structures. 
%
First, introduce the following types:
\begin{itemize}
\item kernel types $\mathcal{E}_+= \{\mathcal{I}, \nabla\mathcal{I} \}$,
\item noise types $\mathcal{E}_-= \{ \Xi \}$,
\item jet types
 $\mathcal{E}^0_0 = \mathcal{E}^A_0 \cup \mathcal{E}^f_0 \ ,$ 
where\footnote{
It is notationally convenient to treat the right-hand side as two different non-linearities and introduce the edge types 
$\delta^{f, n}$,  $\delta^{A, n}$. 
Comparing to \eqref{eq:non-linearity-decompostion}, the $\delta^{f, n}$ edges correspond to a derivative in the $\xi $ direction.
} $\mathcal{E}^f_0=\{ \delta^{(f,n)} \ : \ n\in \mathbb{N} \}$ and $\mathcal{E}^A_0= \{ \delta^{(A,n )}\  :\  n\in \mathbb{N} \}$ to encode the non-linearities,
\item further jet types $\mathcal{E}^+_0= \{\delta^{\mathcal{I}}, \delta^{\nabla \mathcal{I}}\}$. We fix the injection $\iota:\mathcal{E}_+\to  \mathcal{E}^+_0$ from \eqref{eq:iota} to be 
$$\mathcal{I}\mapsto \delta^{\mathcal{I}},\qquad \nabla \mathcal{I} \mapsto \delta^{\nabla \mathcal{I}} \ .$$
\end{itemize}
Since the equation is rather simple, we do not introduce the vector bundle assignments and the indexing maps explicitly. 
Set 
$$|\mathcal{I}|\in (2-{1}/{10}, 2), \qquad  |\nabla \mathcal{I}| \in (1-{1}/{10}, 1), \qquad |\Xi|\in (-1-{1}/{10}, -1)\ $$
and define the naive rule
$$\mathring{R}(\mathcal{I})= \mathring{R}(\nabla \mathcal{I})= \{ [\delta^{f,n}] \sqcup[\mathcal{I}]^n\sqcup [\Xi] ,\  [\delta^{A,n}] \sqcup [\nabla \mathcal{I}]^2 \ : n\geq 0 \} $$
as well as its normalisation $R$. 
Let us introduce the the analogue of the usual graphical notation. We write $\Xi= \<X>$ (using the common convention of not writing it as an edge) and denote an edge of type $\mathcal{I}$ by $\<I>$ as well as an edge of type $\nabla \mathcal{I}$ by $\<bI>$.
Thus, we for example write 
$\mathcal{I}\Xi= \<1>$. Denote edges of type $\delta^{\mathcal{I}}$ by $\<1'>$.
Since it suffices to solve the abstract in $\mathcal{D}^\gamma$ for some $\gamma>1$,
we can work with 
$$\mathfrak{T}^{\mathcal{I}}:= \{ \<1>,\ \<1'> \}$$ to describe the abstract solution
and
\begin{align*}
\mathfrak{T}^r
&_= \{ \delta^{f,0}\Xi,\  \ \delta^{f,1} \<Xi2> ,\  \delta^{f,1}  \<1'Xi> , \  \delta^{A,0} \<b2> \}
\end{align*}
to describe the right-hand side.
We observe that the only negative subtrees appearing are given by 
$$\mathfrak{T}_-=\{  \<Xi2> ,\ \<b2> \} \ .$$

Thus, we next fix $\delta_0> 1$ and work with the regularity structure
constructed form the above trees. We quotient it, so that the product of jets agrees with the canonical product, c.f.\ Remark~\ref{rem:identifying_under_can product on jets} and
denote by $Z^\epsilon=(\Pi^\epsilon, \Gamma^\epsilon)$ the canonical model for the smooth noise $\xi_\epsilon$, such that the abstract Integration map $\mathcal{I}$ realises $K_t(p,q):=G_t(p,q) \kappa(t)\kappa(d(p,q)/\bar{r})$ and $\nabla\mathcal{I}$ realises $K^\nabla_t (p,q):= (\nabla_1 G_t(p,q)) \kappa(t)\kappa(d(p,q)/\bar{r})$.
Next, in order to identify the correct two dimensional solution family,
we restrict attention to the subgroup $\mathfrak{G}'_-$ of elements $\mathfrak{g}_{c,c' }\in \mathfrak{G}_-$such that 
$$\mathfrak{g}_{c,c' }( \<X> ) = 0,\qquad 
\mathfrak{g}_{c,c' }( \<Xi2>) = c,\qquad   \mathfrak{g}_{c,c' }( \<b2>)= c' g^{\sharp},$$
 for some $c,c'\in \mathbb{R}$ and where
where $g^{\sharp}\in \mathcal{C}(TM^{\otimes 2})$ denotes the adjoint tensor field of the metric $g$.
To lighten notation, we denote the renormalised model $(Z^\epsilon)^{\mathfrak{g}_{C_\epsilon, C'_\epsilon}}$ by 
$\hat{Z}^\epsilon= (\hat{\Pi}^\epsilon, \hat{\Gamma}^\epsilon)$. 

\begin{proof}[Proof of Theorem~\ref{thm:g-pam}]
In the setting described above, let $\eta, \gamma\in \mathbb{R}$ such that
$0<\eta<\alpha$ and $-|\Xi|<\gamma<\delta_0$.
It follows from 
Theorem~\ref{thm:abstract_fixed point} there exists a unique abstract solution $\hat{U}^\epsilon\in \mathcal{D}^{\gamma, \eta}(\mathcal{T}^{\mathcal{I}})$ 
of the abstract lift of 
$$
\partial_t u + \Delta u = A( \nabla u, \nabla u) +f(u)\xi, 
$$
for the renormalised model $\hat{Z}^\epsilon$.
Then, it is simple special case of Proposition~\ref{prop:renormalised equation} and the observation $A(g^{\sharp})= \tr A$, that 
$\hat{u}= \mathcal{R} U^{\epsilon}$ solves \eqref{gpam}.
Thus, the proof is complete once Proposition~\ref{prop:convergence of models} below is shown.
\end{proof}

%
%

\begin{prop}\label{prop:convergence of models}
The renormalised models ${\hat{Z}}_\epsilon$ converge to a limiting model $\hat{Z}$
as $\epsilon\to 0$.
\end{prop}
\begin{proof}
It follows from Proposition~\ref{prop:stochastic estimate for trees in nonlinearity} below combined with a Kolmogorov type criterion, obtained by going to charts,
that for each $p,q\in \mathbb{R}\times M$ , $\tau_p \in \mathcal{T}^r$ there exist random distributions
$$ \hat{\Pi}_p \tau_p \qquad \hat{\Pi}_q \hat{\Gamma}_{q,p} \tau_p,  $$
such that\footnote{Here we use this notation to denote the restriction of the corresponding expressions for $\|\ \cdot \  \|_{[-1,T+1]\times M} $ and $\llbracket\  \cdot \   \rrbracket_{[-1,T+1] \times M} $  in Definition~\ref{def:norm on models} to only $\tau\in\mathcal{T}^r $.} 
$$\|\  \hat{\Pi}^\epsilon-\hat{\Pi}\  | \ {\mathcal{T}^r }   \|_{[-1,T+1]\times M} + \llbracket \hat{Z}^\epsilon  \ ;\hat{Z}\ | \ {\mathcal{T}^r}  \rrbracket_{[-1,T+1] \times M} \to 0 \ .$$
Then, it follows from Theorem~\ref{thm:extension theorem} that 
$\|\hat{Z}^\epsilon;\hat {Z}\|_{[-1,T+1]\times M} \to 0$.
\end{proof}

\subsubsection{Stochastic estimates}\label{sec:stoch_estimates}
This section is dedicated to proving the following.
\begin{prop}\label{prop:stochastic estimate for trees in nonlinearity}
There exists a $\kappa>0$, such that for any $p,q\in \mathbb{R}\times M$, $\tau_p \in {\mathcal{T}}^r_{\alpha, \delta}$ and $N\in \mathbb{N}$, it holds that uniformly over $\lambda\in (0,1]$ and $\phi_q^\lambda\in  \mathcal{B}^{r, \lambda}_q$ that 
\begin{equation}\label{eq: stochastic bound on model}
\E\left[\left(\hat \Pi_q^\epsilon \tau_q) (\phi_q^\lambda)\right)^N \right]^{1/N}\ \lesssim \lambda^{\alpha+2\kappa} \ , \qquad
\E\left[\left((\hat{\Pi}^\epsilon_q\hat{\Gamma}^\epsilon_{q,p} \tau_p -\hat{\Pi}_p^\epsilon \tau_p) (\phi_q^\lambda)\right)^N \right]^{1/N}\ \lesssim \lambda^{\delta+2\kappa} 
\end{equation}
as well as the same estimate for $ \hat \Pi_p^\epsilon$ and $\hat{ \Gamma}^\epsilon_{p,q}$ replaced by $ \hat \Pi_p$ and $\hat{ \Gamma}_{p,q}$.
Furthermore, it holds that 
\begin{align*}
&\E\left[\left(\hat \Pi_p^\epsilon \tau_p- \hat \Pi_p \tau_p) (\phi_q^\lambda)\right)^N \right]^{1/N}\ \lesssim \epsilon^{\kappa} \lambda^{\alpha+\kappa} \\
&\E\left[\left((\hat{\Pi}^\epsilon_q\hat{\Gamma}^\epsilon_{q,p} \tau_p -\hat{\Pi}_p^\epsilon \tau_p) (\phi_q^\lambda)-(\hat{\Pi}_q\hat{\Gamma}_{q,p} \tau_p +\hat{\Pi}_p \tau_p) (\phi_q^\lambda)\right)^N \right]^{1/N}\ \lesssim \epsilon^{\kappa} \lambda^{\delta+\kappa} \ . 
\end{align*}
\end{prop}


We first introduce similar notation to \cite{HP15}, \cite{HQ18}.
Recall, that a random variable belonging to the $k$th homogeneous Wiener chaos can
be described by a kernel in $L^2( M)^{\otimes k}$, i.e.\ the
space of square-integrable functions in $k$ space variables via the correspondence 
$f \mapsto I_k(f)$ given in \cite[p.~8]{Nua06}.

We shall graphically represent $\bigl(\Pi_\star^\epsilon \tau\bigr)(\phi^\lambda_\star)$,
where $\phi^\lambda_\star\in \mathcal{B}^{r,\lambda}_\star$ denotes a test function that is localised around the
point $\star \in \RR\times M$. Nodes in our graph will represent variables in $M$ or $\mathbb{R}\times M$, with one
special node \tikz[baseline=-3] \node [root] {}; representing the point $\star\in \mathbb{R}\times M$.
The nodes \tikz[baseline=-3] \node [var] {}; represent the arguments of our kernel,
so that a random variable in the $k$th (homogeneous) Wiener chaos is represented by
a graph with exactly $k$ such nodes. The remaining nodes, which we draw as \tikz[baseline=-3] \node [dot] {};, represent dummy space time variables that are to be integrated out.

Each line then represents a space or space-time kernel, with 
\tikz[baseline=-0.1cm] \draw[kernel] (0,0) to (1,0) \enlarge;
representing the kernel $K$, 
\tikz[baseline=-0.1cm] \draw[rho] (0,0) to (1,0);
representing the kernel $\rho_\epsilon=G_{{\epsilon}^2}$ and
\tikz[baseline=-0.1cm] \draw[testfcn] (1,0) to (0.0,0) \enlarge;
representing a generic test function $\phi_\star^\lambda$. Note that since 
$\rho_\epsilon (p,q)= \rho_\epsilon (q,p)$, the notation \tikz[baseline=-0.1cm] \draw[rho] (0.0,0) to (1,0);
is unambiguous.


Whenever we draw a barred arrow 
\tikz[baseline=-0.1cm] \draw[kernel1] (0.0,0) to (1,0);
this represents a factor $K((t,p), (s,q)) - K(\star,  (s,q))$, where
$(s,q)\in \mathbb{R}\times M$ and $(t,p)\in \mathbb{R}\times M$ are the coordinates of the starting and end
point respectively. 

\subsubsection{A first stochastic term}
Using the notation introduced above, one can write
\begin{equation}
\bigl(\Pi_\star^{\epsilon} \<Xi2>\bigr)(\phi_\star^\lambda) = \;
\begin{tikzpicture}[scale=0.35,baseline=0.3cm]
	\node at (0,-1)  [root] (root) {};
	\node at (-2,1)  [dot] (left) {};
	\node at (-2,3)  [dot] (left1) {};
	\node at (0,1) [var] (variable1) {};
	\node at (0,3) [var] (variable2) {};
	
	\draw[testfcn] (left) to  (root);
	
	\draw[kernel1] (left1) to (left);
	\draw[rho] (variable2) to (left1); 
	\draw[rho] (variable1) to (left); 
\end{tikzpicture}\;
+ \;
\begin{tikzpicture}[scale=0.35,baseline=0.3cm]
	\node at (0,-1)  [root] (root) {};
	\node at (-2,1)  [dot] (left) {};
	\node at (-2,3)  [dot] (left1) {};
	\node at (0,2) [dot] (variable1) {};
	\node at (0,2) [dot] (variable2) {};
	
	\draw[testfcn] (left) to (root);
	
	\draw[kernel1] (left1) to (left);
	\draw[rho] (variable2) to (left1); 
	\draw[rho] (variable1) to (left); 
\end{tikzpicture}\;
=\;
\begin{tikzpicture}[scale=0.35,baseline=0.3cm]
	\node at (0,-1)  [root] (root) {};
	\node at (-2,1)  [dot] (left) {};
	\node at (-2,3)  [dot] (left1) {};
	\node at (0,1) [var] (variable1) {};
	\node at (0,3) [var] (variable2) {};
	
	\draw[testfcn] (left) to  (root);
	
	\draw[kernel1] (left1) to (left);
	\draw[rho] (variable2) to (left1); 
	\draw[rho] (variable1) to (left); 
\end{tikzpicture}\;
+\;
\begin{tikzpicture}[scale=0.35,baseline=0.3cm]
	\node at (0.0,-1)  [root] (root) {};
	\node at (-2,1)  [dot] (left) {};
	\node at (-2,3)  [dot] (left1) {};
	\node at (0,2) [dot] (variable1) {};
	\node at (0,2) [dot] (variable2) {};
	
	\draw[testfcn] (left) to (root);
	
	\draw[kernel] (left1) to (left);
	\draw[rho] (variable2) to (left1); 
	\draw[rho] (variable1) to (left); 
\end{tikzpicture}\;
- \;
\begin{tikzpicture}[scale=0.35,baseline=0.3cm]
	\node at (0.0,-1)  [root] (root) {};
	\node at (-1,1)  [dot] (left) {};
	\node at (0,3)  [dot] (top) {};
	\node at (1,1) [dot] (right) {};
	
	\draw[testfcn] (left) to  (root);
	
	\draw[kernel] (right) to (root);
	\draw[rho] (top) to (right); 
	\draw[rho] (top) to (left); 
\end{tikzpicture}\;
.
\end{equation}
%
%
%
%
%
%
We further introduce the graphical notation 
$$
\begin{tikzpicture}[scale=0.35,baseline=0.6cm]
	\node at (-2,1)  [root] (left) {};
	\node at (-2,3)  [reddot] (left1) {};
	\node at (0,2) [reddot] (variable) {};
	
	\draw[red, kernel] (left1) to (left);
	\draw[red, rho] (variable) to (left1); 
	\draw[red, rho] (variable) to (left); 
\end{tikzpicture}\; :=
C_\epsilon
$$
and decompose the function
\begin{equation}\label{eq:decom_counterterm1}
\begin{tikzpicture}[scale=0.35,baseline=0.6cm]
	\node at (-2,1)  [root] (left) {};
	\node at (-2,3)  [dot] (left1) {};
	\node at (0,2) [dot] (variable) {};
	
	\draw[kernel] (left1) to (left);
	\draw[rho] (variable) to (left1); 
	\draw[rho] (variable) to (left); 
\end{tikzpicture}\;
= \;
\begin{tikzpicture}[scale=0.35,baseline=0.6cm]
	\node at (-2,1)  [root] (left) {};
	\node at (-2,3)  [reddot] (left1) {};
	\node at (0,2) [reddot] (variable) {};
	
	\draw[red, kernel] (left1) to (left);
	\draw[red, rho] (variable) to (left1); 
	\draw[red, rho] (variable) to (left); 
\end{tikzpicture}\; +\;
\begin{tikzpicture}[scale=0.35,baseline=0.6cm]
	\node at (-2,1)  [root] (left) {};
	\node at (-2,3)  [bluedot] (left1) {};
	\node at (0,2) [bluedot] (variable) {};
	
	\draw[blue, kernel] (left1) to (left);
	\draw[blue, rho] (variable) to (left1); 
	\draw[blue, rho] (variable) to (left); 
\end{tikzpicture}\; ,
\end{equation}
where the blue term is determined by \eqref{eq:decom_counterterm1}.

\begin{lemma}\label{lem:first counterterm}
The function  $
\begin{tikzpicture}[scale=0.35,baseline=0.6cm]
	\node at (-2,1)  [root] (left) {};
	\node at (-2,3)  [bluedot] (left1) {};
	\node at (0,2) [bluedot] (variable) {};
	
	\draw[blue, kernel] (left1) to (left);
	\draw[blue, rho] (variable) to (left1); 
	\draw[blue, rho] (variable) to (left); 
\end{tikzpicture}\;
$ 
is uniformly bounded for $\epsilon\in (0,1]$ and converges uniformly as $\epsilon \to 0$.
\end{lemma}
\begin{proof}
Recall that
$$\pmb{\Pi}^{\epsilon} \<1>(t,p)=\int_{\mathbb{R}}\int_{M} K_{t-s}(p,q) \xi_{\epsilon}(q) \dVol_q ds
= \int_{\mathbb{R}} \langle \xi, G_{t-s+\epsilon^2} (p,\cdot )\rangle \kappa(t-s) ds + r^\epsilon(t,p)$$
and thus
\begin{align*}
\begin{tikzpicture}[scale=0.35,baseline=0.6cm]
	\node at (-2,1)  [root] (left) {};
	\node at (-2,3)  [dot] (left1) {};
	\node at (0.0,2) [dot] (variable) {};
	\draw[kernel] (left1) to (left);
	\draw[rho] (variable) to (left1); 
	\draw[rho] (variable) to (left); 
\end{tikzpicture}\;
&=\E\left[\pmb{\Pi}^{\epsilon} \<Xi2> (t,p)\right]\\
&= \E\left[ \int_{\mathbb{R}} \langle \xi, G_{t-s+\epsilon^2}(p,\cdot )\rangle \kappa(t-s)  \langle \xi, G_{\epsilon^2}(p,\cdot )\rangle  ds\right] +\E\left[ r^\epsilon (t,p)\xi_{\epsilon}(p)\right] \ .
\end{align*} 
Since the conclusion of the lemma holds for the map $p\mapsto \E\left[ R^\epsilon_1(t,p)\xi_{\epsilon}(p)\right]$ we proceed by noting that
\begin{align*}
& \E\left[ \int_{\mathbb{R}} \langle \xi, G_{t-s+\epsilon^2}(p,\cdot )\rangle \kappa(t-s)  \langle \xi, G_{\epsilon^2}(p,\cdot )\rangle  ds\right] \\
&= \int_{\mathbb{R}} \int_M G_{t-s+\epsilon^2}(p,z) G_{\epsilon^2}(p,z) \dVol_z \kappa(t-s) ds \\ 
&= \int_{\mathbb{R}} \int_{B_{r}(p)} Z_{t-s+\epsilon^2}(p,z) Z_{\epsilon^2}(p,z) \dVol_z \kappa(t-s) ds\\
&+ \int_{\mathbb{R}} \int_{M} \left(G_{t-s+\epsilon^2}(p,z) -Z_{t-s+\epsilon^2}(p,z) \right) G_{\epsilon^2}(p,z) \dVol_z \kappa(t-s) ds\\
&+\int_{\mathbb{R}} \int_{M} Z_{t-s+\epsilon^2}(p,z) \left( G_{\epsilon^2}(p,z)-Z_{\epsilon^2}(p,z)\right)  \dVol_z \kappa(t-s) ds \ , \\
\end{align*} 
where if follows from Theorem~\ref{thm:heat_kernel_expansion} that the latter two summands converge uniformly as $\epsilon \to 0$.
Finally, if follows from\eqref{pam counterterm1} and \eqref{eq:volume_asymptotics} that
$$ \int_{\mathbb{R}} \int_{B_{r}(p)} Z_{t-s+\epsilon^2}(p,z) Z_{\epsilon^2}(p,z) \dVol_z \kappa(t-s) ds- C_\epsilon$$ 
converges uniformly to a limiting function. 
\end{proof}
Thus, after renormalisation we find
\begin{equation}\label{eq:renormalised_chaos_dec_tree1}
\bigl(\hat \Pi_\star^{\epsilon} \<Xi2>\bigr)(\phi_\star^\lambda) = \;
\begin{tikzpicture}[scale=0.35,baseline=0.3cm]
	\node at (0,-1)  [root] (root) {};
	\node at (-2,1)  [dot] (left) {};
	\node at (-2,3)  [dot] (left1) {};
	\node at (0,1) [var] (variable1) {};
	\node at (0,3) [var] (variable2) {};
	
	\draw[testfcn] (left) to  (root);
	
	\draw[kernel1] (left1) to (left);
	\draw[rho] (variable2) to (left1); 
	\draw[rho] (variable1) to (left); 
\end{tikzpicture}\;
 - \;
\begin{tikzpicture}[scale=0.35,baseline=0.3cm]
	\node at (0.0,-1)  [root] (root) {};
	\node at (-1,1)  [dot] (left) {};
	\node at (0,3)  [dot] (top) {};
	\node at (1,1) [dot] (right) {};
	
	\draw[testfcn] (left) to  (root);
	
	\draw[kernel] (right) to (root);
	\draw[rho] (top) to (right); 
	\draw[rho] (top) to (left); 
\end{tikzpicture}\;
+\;
\begin{tikzpicture}[scale=0.35,baseline=0.3cm]
	\node at (0.0,-1)  [root] (root) {};
	\node at (-2,1)  [dot] (left) {};
	\node at (-2,3)  [bluedot] (left1) {};
	\node at (0,2) [bluedot] (variable1) {};
	
	\draw[testfcn] (left) to (root);
	
	\draw[blue,kernel] (left1) to (left);
	\draw[blue, rho] (variable1) to (left1); 
	\draw[blue, rho] (variable1) to (left); 
\end{tikzpicture}\;.
\end{equation}

\subsubsection{A second stochastic term}
Here we use a slight variation of the above notation and use
\tikz[baseline=-0.1cm] \draw[keps] (0,0) to (1,0) \enlarge;
to represent the kernel 
$$ \left( K^{\nabla}_t * \rho_\epsilon\right) (p, q)= \int_M   K^{\nabla}_t (p,z) \rho_\epsilon (z,q) \dVol_z \ ,  $$ and
\tikz[baseline=-0.1cm] \draw[testfcn] (1,0) to (0,0) \enlarge;
represents a localised test function $\phi_\star^\lambda\in \mathcal{C}^\infty (TM^{\otimes 2})$. So

\begin{equation*}
\bigl(\Pi_\star^{\epsilon}\<b2>\bigr)(\phi_\star^\lambda)
= 
\begin{tikzpicture}[scale=0.35,baseline=0.3cm]
	\node at (0,-1)  [root] (root) {};
	\node at (0,1)  [dot] (int) {};
	\node at (-1,3.5)  [var] (left) {};
	\node at (1,3.5)  [var] (right) {};
	
	\draw[testfcn] (int) to  (root);
	
	\draw[keps] (left) to (int);
	\draw[keps] (right) to (int);
\end{tikzpicture}\;
+ \;
\begin{tikzpicture}[scale=0.35,baseline=0.3cm]
	\node at (0,-1)  [root] (root) {};
	\node at (0,1)  [dot] (int) {};
	\node at (0,3.5) [dot] (top) {};
	
	\draw[testfcn] (int) to  (root);
	
	\draw[keps] (top) to[bend left=60] (int); 
	\draw[keps] (top) to[bend right=60] (int); 
\end{tikzpicture}\;.
\end{equation*}
Similarly to above, we write
$$
\begin{tikzpicture}[scale=0.35,baseline=0.3cm]
	\node at (0,0)  [root] (root) {};
	\node at (0,2.5) [reddot] (top) {};
	\draw[red, keps] (top) to[bend left=60] (root); 
	\draw[red, keps] (top) to[bend right=60] (root); 
\end{tikzpicture}\; =  C'_\epsilon g^{\sharp} 
\qquad \text{and} \qquad
\begin{tikzpicture}[scale=0.35,baseline=0.3cm]
	\node at (0,0)  [root] (root) {};
	\node at (0,2.5) [bluedot] (top) {};
	\draw[blue,keps] (top) to[bend left=60] (root); 
	\draw[blue,keps] (top) to[bend right=60] (root); 
\end{tikzpicture}\; = \;
\begin{tikzpicture}[scale=0.35,baseline=0.3cm]
	\node at (0,0)  [root] (root) {};
	\node at (0,2.5) [dot] (top) {};
	\draw[keps] (top) to[bend left=60] (root); 
	\draw[keps] (top) to[bend right=60] (root); 
\end{tikzpicture}\; - \;
\begin{tikzpicture}[scale=0.35,baseline=0.3cm]
	\node at (0,0)  [root] (root) {};
	\node at (0,2.5) [reddot] (top) {};
	\draw[red, keps] (top) to[bend left=60] (root); 
	\draw[red, keps] (top) to[bend right=60] (root); 
\end{tikzpicture}\;  \ .
$$
%
%
%
%
%
%
%
%
%
%

\begin{lemma}\label{lem:second counterterm}
The section $\begin{tikzpicture}[scale=0.35,baseline=0.3cm]
	\node at (0,0)  [root] (int) {};
	\node at (0,2.5) [bluedot] (top) {};
	\draw[blue, keps] (top) to[bend left=60] (int); 
	\draw[blue, keps] (top) to[bend right=60] (int); 
\end{tikzpicture}\;$ of $\mathcal{C}^\infty(T^*M^{\otimes_s 2})$ 
is uniformly bounded for $\epsilon\in (0,1]$ and converges uniformly.

\end{lemma}

\begin{proof}
Let $\{e_{p,i}\}_{i=1,2}$ be a basis of $T_p M$.
We proceed exactly as in Lemma~\ref{lem:first counterterm} and find that 
\begin{align*}
\begin{tikzpicture}[scale=0.35,baseline=0.3cm]
	\node at (0.0,0)  [root] (int) {};
	\node at (0,2.5) [dot] (top) {};
	\draw[keps] (top) to[bend left=60] (int); 
	\draw[keps] (top) to[bend right=60] (int); 
\end{tikzpicture}\;
  &= 
 \E\left[ \int_{\mathbb{R}} \langle \xi, \nabla G_{t-s+\epsilon^2}(p,\cdot )\rangle \kappa(t-s) ds \otimes \int_{\mathbb{R}} \langle \xi, \nabla G_{t-s'+\epsilon^2}(p,\cdot )\rangle \kappa(t-s') ds'\right] + R^\epsilon_1 (p)
 \\
&= \sum_{i,j} e_{p,i}\otimes e_{p,j} \int_{\mathbb{R}^2}	\int_{B_{r}(p)} \partial_i Z_{t-s+\epsilon^2}(p,z )\otimes \partial_j Z_{t-s'+\epsilon^2}(p,z )\dVol_z ds ds' +R^\epsilon_2\\
&=   \sum_{i,j}  {C}^{(i,j)}_\epsilon e_i\otimes e_j   + R^\epsilon_3\ ,
\end{align*} 	
where 
$$
{C}^{(i,j)}_\epsilon:=\int_{\mathbb{R}^2}	\int_{B_{r}(0)} \partial_i \bar{Z}_{t-s+\epsilon^2}(z )\otimes \partial_j \bar{Z}_{t-s'+\epsilon^2}(z ) dz ds ds' \ 
$$
and the sections $R^\epsilon_1, \, R^\epsilon_2, \, R^\epsilon_3$ of $TM^{\otimes 2}$ converge uniformly as $\epsilon \to 0$.
Noting that $C'_\epsilon={C}^{(1,1)}_\epsilon={C}^{(2,2)}_\epsilon$ and ${C}^{(1,2)}_\epsilon={C}^{(2,1)}_\epsilon=0$, this concludes the proof.
%
%
\end{proof}

Thus, after renormalisation,
\begin{equation}\label{eq:renormalised_chaos_dec_tree2}
\bigl(\hat \Pi_\star^{\epsilon}\<b2>\bigr)(\phi_\star^\lambda)
= 
\begin{tikzpicture}[scale=0.35,baseline=0.3cm]
	\node at (0,-1)  [root] (root) {};
	\node at (0,1)  [dot] (int) {};
	\node at (-1,3.5)  [var] (left) {};
	\node at (1,3.5)  [var] (right) {};
	
	\draw[testfcn] (int) to  (root);
	
	\draw[keps] (left) to (int);
	\draw[keps] (right) to (int);
\end{tikzpicture}\;
%
%
+\;
\begin{tikzpicture}[scale=0.35,baseline=0.3cm]
	\node at (0,-1)  [root] (root) {};
	\node at (0,1)  [dot] (int) {};
	\node at (0,3.5) [bluedot] (top) {};
	
	\draw[testfcn] (int) to  (root);
	
	\draw[blue, keps] (top) to[bend left=60] (int); 
	\draw[blue, keps] (top) to[bend right=60] (int); 
\end{tikzpicture}\; .
\end{equation}

\begin{proof}[Proof of Proposition~\ref{prop:stochastic estimate for trees in nonlinearity} ]
Note that it follows from \eqref{eq:renormalised_chaos_dec_tree1} and  \eqref{eq:renormalised_chaos_dec_tree2} that
\begin{equation}
\E \bigl|\bigl(\hat \Pi_\star^{\epsilon} \<Xi2>\bigr)(\phi_\star^\lambda)\bigr|^2 \le 2\;
\begin{tikzpicture}[scale=0.35,baseline=0.3cm]
	\node at (0,-1)  [root] (root) {};
	\node at (-1.5,1)  [dot] (left) {};
	\node at (-1.5,3)  [dot] (left1) {};
	\node at (1.5,1) [dot] (variable1) {};
	\node at (1.5,3) [dot] (variable2) {};
	
	\draw[testfcn] (left) to  (root);
	\draw[testfcn] (variable1) to  (root);
	
	\draw[kernel1] (left1) to (left);
	\draw[kernel1] (variable2) to (variable1);
	\draw[rho] (variable2) to node[dot] {} (left1); 
	\draw[rho] (variable1) to node[dot] {} (left); 
\end{tikzpicture}\; + \; 2 \;
\left(
\begin{tikzpicture}[scale=0.35,baseline=0.3cm]
	\node at (0,-1)  [root] (root) {};
	\node at (-1,1)  [dot] (left) {};
	\node at (0,3)  [dot] (top) {};
	\node at (1,1) [dot] (right) {};
	
	\draw[testfcn] (left) to  (root);
	
	\draw[kernel] (right) to (root);
	\draw[rho] (top) to (right); 
	\draw[rho] (top) to (left); 
\end{tikzpicture}
\right)^2\; +\; 2 \;
\left(
\begin{tikzpicture}[scale=0.35,baseline=0.3cm]
	\node at (0,-1)  [root] (root) {};
	\node at (-2,1)  [dot] (left) {};
	\node at (-2,3)  [bluedot] (left1) {};
	\node at (0,2) [bluedot] (variable1) {};
	\node at (0,2) [bluedot] (variable2) {};
	
	\draw[testfcn] (left) to (root);
	
	\draw[blue,kernel] (left1) to (left);
	\draw[rho] (variable2) to (left1); 
	\draw[rho] (variable1) to (left); 
\end{tikzpicture}
\right)^2\; 
\end{equation}
and
\begin{equation}
\E |\bigl(\hat \Pi_\star^{\epsilon}\<b2>\bigr)(\phi_\star^\lambda)|^2
\le 2\;
\begin{tikzpicture}[scale=0.35,baseline=0.25cm]
	\node at (0.0,-1)  [root] (root) {};
	\node at (-2,1)  [dot] (intl) {};
	\node at (2,1)  [dot] (intr) {};
	\node at (0,1)  [dot] (top1) {};
	\node at (0,3)  [dot] (top2) {};
	
	\draw[testfcn] (intl) to  (root);
	\draw[testfcn] (intr) to  (root);
	
	\draw[keps] (top1) to (intl);
	\draw[keps] (top1) to (intr);
	\draw[keps] (top2) to (intl);
	\draw[keps] (top2) to (intr);
\end{tikzpicture}\;.
%
%
+ \; \left(
\begin{tikzpicture}[scale=0.35,baseline=0.3cm]
	\node at (0.0,-1)  [root] (root) {};
	\node at (0,1)  [dot] (int) {};
	\node at (0,3.5) [bluedot] (top) {};
	
	\draw[testfcn] (int) to  (root);
	
	\draw[blue, keps] (top) to[bend left=60] (int); 
	\draw[blue, keps] (top) to[bend right=60] (int); 
\end{tikzpicture}
\; \right)^2\  .
\end{equation}
The last term of each has already been bounded in Lemma~\ref{lem:first counterterm}, resp.\ Lemma~\ref{lem:second counterterm}, while the rest can be estimated using the bounds of Section~\ref{sec:Kernels with prescribed singularities} exactly as in the proof of \cite[Theorem~10.19]{Hai14}.
Combining this  
with the fact that for
$\tau_p =\<Xi2>$ and $\<2>$ one has
$$ \hat{\Pi}^\epsilon_q \hat{\Gamma}^\epsilon_{q,p} \tau_p -\hat{\Pi}^\epsilon_p \tau_p= 0 \ ,$$
implies that for these elements \eqref{eq: stochastic bound on model} is satisfied. 
Since every element $\tau_p\in {\mathcal{T}}^r$ is the product of a jet with either $\<Xi2>$ or $ \<2>$ \eqref{eq: stochastic bound on model} follows generally.
In order to conclude the remainder of Proposition~\ref{prop:stochastic estimate for trees in nonlinearity} one proceeds as usual, edge by edge replacing  
$\rho_\epsilon$ by $\rho_\epsilon (p,q)- \delta_p(q)$. 
\end{proof}

\subsection{The $\Phi^4_3(E)$ equation}\label{section:phi43}
Let 
\begin{align}
C_\epsilon &= \int_\mathbb{R^2} \int_{B_r(0)}  \phi^\epsilon (s- r )\phi^\epsilon (s'- r )
\bar{Z}_{t-s+\epsilon^2}(z)
 \bar{Z}_{t-s'+\epsilon^2}( z)
 ds  ds' dr dz \label{phi4 counterterm1} \\
C'_\epsilon &= \int_{\mathbb{R}\times B_r(0)} \left(\bar{Z}^{(\epsilon)}(z)\right)^2 \bar{Z} (z) dz \label{phi4 counterterm2}
\end{align}
where $\bar{Z}^{(\epsilon)}(t,x)=\int_\mathbb{R^2} \int_{B_r}   \phi^\epsilon (s- r )\phi^\epsilon (s'- r )
\bar{Z}_{t-s-s'+2\epsilon^2}(x)
 ds  ds' dr dx
$.

We shall prove the following theorem.

\begin{theorem}\label{thm:phi4_3}
Let $E\to M$ be a vector bundle over a compact 3-manifold equipped with with a metric $h$ and compatible connection $\nabla^{E}$. Let $|\cdot|_h$ be the norm on $E$ induced by $h$ and $\Delta^E$ 
a generalised Laplacian on $E$ differing form the connection Laplacian only by a $0$-th order term.
Let $\xi$ be an $E$-valued space-time white noise, and denote by $\xi_\epsilon$ its regularisation as in \eqref{eq:mollified space-time white noise}.
Then, let $u_0\in \cC^{\alpha}(E)$ for $\alpha>-\frac{2}{3}$, let $C_\epsilon$, $C'_\epsilon$ as in \eqref{phi4 counterterm1} and \eqref{phi4 counterterm2} and let $u_\epsilon$ solve
\begin{equation}\label{phi4}
\partial_t u_\epsilon + \Delta^E u_\epsilon= -u_\epsilon |u_\epsilon |^2_h+ 3 C_\epsilon u_\epsilon - 9 C_\epsilon' u_\epsilon +\xi_\epsilon\ , \qquad
u_\epsilon(0) =u_0 \ .
\end{equation}
Then, there exists a (random) $T>0$ and $u\in  \mathcal{D}'((0,T)\times E)$, where we canonically identify $(0,T)\times E$ with a subset of $\pi^* E$, such that $u_\epsilon \to u$ in $\mathcal{D}'((0,T)\times E)$ as $\epsilon\to 0$ in probability.
\end{theorem}

\begin{remark}
The condition that $\Delta^E$ 
differs form the connection Laplacian by only a $0$-th order term is not a real loss of generality, c.f.\ Remark~\ref{rem:generalised_laplacian}.
\end{remark}

\begin{remark}
It is well known that one can take $T= \infty$ a.s.\ in Theorem~\ref{thm:phi4_3}, see \cite{MW17}, \cite{MW20}, \cite{GR23}, \cite{JP23}\cite{Duc21} \&  \cite{BDFT23a}.
Furthermore, it is well understood that $u$ can actually be realised as a continuous function $u : [0,\infty) \to \mathcal{D}'(E)$.
\end{remark}
%

First, we define a new symmetric nonlinearity
$$H: E^{\otimes 3}\to E, \qquad (u_1,u_2,u_3)\mapsto \frac{1}{3!}\sum_{\sigma\in S_3} u_{\sigma(1)} h(u_{\sigma(2)}, u_{\sigma(3)}) \ , $$
where $S_3$ denotes the permutation group of the set $\{1, 2, 3\}\subset \NN$
 and study
\begin{equation}\label{eq:phi}
\partial_t u + \Delta^E u =  -H(u,u,u) +\xi\ .
\end{equation}
We define the edge types:
\begin{itemize}
\item kernel types $\mathcal{E}_+= \{\mathcal{I} \}$,
\item noise types $\mathcal{E}_-= \{ \Xi \}$,
\item jet types $\mathcal{E}^0_0 =\{ \delta^{(3)} \}$ 
and $\mathcal{E}^+_0= \{\delta^{\mathcal{I}}\}$ with the injection $\iota:\mathcal{E}_+\to  \mathcal{E}^+_0$.
\end{itemize}

We fix homogeneities $|\mathcal{I}|\in (2-\frac{1}{10}, 2)$, $|\Xi|\in (-\frac{5}{2}-\frac{1}{10}, -\frac{5}{2})$
and
the naive rule characterised by $\mathring{R}(\mathcal{I})= \{ [\mathcal{I}]^3,\ [\Xi] \}$. Denote by $R$ its normalisation.
Using graphical notation we write $\Xi= \<X>$,  denote an edge of type $\mathcal{I}$ by $\<I>$ and edges of type $\delta^{\mathcal{I}}$ by $\<1'>$.
Since the nonlinearity is given by a parallel section of $H\in L(E^{\otimes 3}, E)$, we shall suppress the $\delta^{(3)} $ edges in the notation and use the convention that it is implicitly present, whenever a node has three incident edges from above. For example, we write
$$\delta^{(3)}(\mathcal{I}\Xi)^3= \<3>\ , \qquad \delta^{(3)} \delta^{\mathcal{I}} (\mathcal{I}\Xi)^2= \<3'>\ , \qquad \delta^{(3)} (\delta^{\mathcal{I}})^2 \mathcal{I}\Xi = \<3''>, \qquad  \delta^{(3)} (\delta^{\mathcal{I}})^3=  \<3'''>\ .$$ 

By exactly the same power counting as for the usual $\phi^4$ equation, it suffices to solve the equation in $\mathcal{D}^\gamma$ for some $\gamma>1$.
Thus, the relevant trees to describe the solution are
$$\mathfrak{T}^\mathcal{J}:=\left\{\<1'>, \; \<1>, \; \<30>, \; \<3'0> \right\} $$
and to describe the right-hand side we use
$${\mathfrak{T}}^r:=\left\{ \<X>, \; \<3>, \;\<3'>, \; \<3''>, \;  \<3'''>  ,\; \<32>,\; \<3'2>,\; \<32'>   \right\} \ .$$
For $\delta_0\in (1,2)$ we work with the regularity structure constructed in Section~\ref{section application to spdes}. Note that this time, as we are working with a multi-linear nonlinearity, we do not need to identify products of jets, c.f.\ Section~\ref{section Local operations}. Denote by $Z^\epsilon=(\Pi^\epsilon, \Gamma^\epsilon)$ the canonical model for the smooth noise $\xi_\epsilon$, such that the abstract Integration map $\mathcal{I}$ realises $K_t(p,q):=G_t(p,q) \kappa(t)\kappa(d(p,q)/\bar{r})$, where 
$G$ denotes the heat kernel of $\Delta^E$.

Finally, we shall work with the subgroup $\mathfrak{G}'_- \subset \mathfrak{G}_-$ consisting of elements  $\mathfrak{g}_{c,c'} \in \mathfrak{G}_-$ 
such that 
$$\mathfrak{g}(\<2>)= c h^{\sharp}, \qquad \mathfrak{g}( \<22>)= 2c' \id_E \   $$
for some $c, c'\in \mathbb{R}$ 
and such that for other $\tau \in \mathfrak{T}_-^{(-)}$ one has $\mathfrak{g}(\tau)= e(\tau)$, where $e\in \mathfrak{G}_-$ denotes the unit element.
We again write  $\hat{Z}^\epsilon:= (Z^\epsilon)^{\mathfrak{g}_{C_\epsilon, C'_\epsilon}}$ and
$\hat{Z}^\epsilon= (\hat{\Pi}^\epsilon, \hat{\Gamma}^\epsilon)$. 

\begin{remark}
Note that we have incorporated another simplification due to the fact that $H$ is a parallel non-linearity. A priory $\mathfrak{g}( \<22>)$ is a $0$-th order differential operator 
$$\mathfrak{g}( \<22>): \mathcal{C}^\infty\left( (E\otimes L(E^{\otimes_s 3}, E))\times L(E^{\otimes_s 3}, E) \right) \to \mathcal{C}^\infty(E)\ .$$ Only after contracting both sets of indices corresponding to $L(E^{\otimes_s 3}, E)$ with $H$ one obtains a $0$-th order differential operator $\mathcal{C}^\infty(E)\to \mathcal{C}^\infty(E)$.
\end{remark}

\begin{proof}[Proof of Theorem~\ref{thm:phi4_3}]
In the setting described above, let $\eta, \gamma\in \mathbb{R}$ such that
$\big| 2|\Xi| + 2|\mathcal{I}|\big|<\gamma <\delta_0$ and
$-{2}/3<\eta<\left( |\Xi| + |\mathcal{I}|\right)\wedge \alpha$ as well as $\eta\leq \alpha$.
Note that the homogeneity of $|\Xi|$ is a priori too low in order to apply the  Theorem~\ref{thm:abstract_fixed point} directly, but this can be resolved exactly as in \cite[Sec.~9.4]{Hai14} for the classical $\Phi^4_3$ equation.
One then finds that there exists a unique abstract (local) solution $\hat{U}^\epsilon\in \mathcal{D}^{\gamma, \eta}(\mathcal{T}^{\mathcal{I}})$ 
of the abstract lift of 
\eqref{eq:phi}
for the renormalised models $\hat{Z}^\epsilon$.
Then, it follows from Proposition~\ref{prop:renormalised equation} that
$\hat{u}= \mathcal{R} \hat{U}^{\epsilon}$ solves 
$$\partial_t \hat{u} + \Delta^E \hat{u}= -H(\hat{u}) + 3C_\epsilon H (\hat{u}\odot h^{\sharp})  - 9 C'_\epsilon H(\hat{u} \odot h^{\sharp}) +\xi_\epsilon$$
which is exactly \eqref{phi4}.
The proof is thus complete once Proposition~\ref{prop2:convergence of models} below is shown.
\end{proof}

%
%

\begin{prop}\label{prop2:convergence of models}
The restriction of the renormalised model ${\hat{Z}}_\epsilon$ to $\mathcal{T}^\mathcal{I}_{\leq 1, :}$ and $\mathcal{T}^r_{\leq 0, :}$ converges 
as $\epsilon\to 0$.
\end{prop}
\begin{proof}
This follows exactly as in the proof of Proposition~\ref{prop:convergence of models} from Proposition~\ref{prop2:stochastic estimate for trees in nonlinearity} below.
\end{proof}
\begin{remark}
A difference between Propositions~\ref{prop:convergence of models} and~\ref{prop2:convergence of models} is that the latter only claims convergence of the model restricted to objects relevant for the equation at hand. This is done purely for convenience, since this reduces the number of stochastic estimates we need to show below. An appropriate modification of the general results in \cite{CH16}, \cite{HS23} would be desirable and is expected to hold in our setting as well.
\end{remark}

\subsubsection{Stochastic estimates}\label{sec:stochastic_estimates2}
This section is dedicated to proving the following Proposition, which is the combination of Lemma~\ref{lem:stochstic estimats wick trees}, Lemma~\ref{lem:stochastic bounds 32 and 32'} and Lemma~\ref{lem:stochastic bounds 3'2} below.
\begin{prop}\label{prop2:stochastic estimate for trees in nonlinearity}
The conclusion of Proposition~\ref{prop:stochastic estimate for trees in nonlinearity} holds for $\tau_p \in {\mathcal{T}}^r_{\leq 1, :}$.
\end{prop}

%
We use similar notation as in Section~\ref{sec:stoch_estimates}, where
this time a random variable belonging to the $k$th homogeneous Wiener chaos can
be described by a kernel in $L^2(\pi^* E )^{\hat{\otimes} k}$, where $\pi: \mathbb{R}\times M \to M$ is the canonical projection. That is, the space of of square-integrable sections of $k$ space-time variables.

As previously, we introduce convenient graphical notation:
Nodes in our graph will represent variables in $\RR\times M$, with one
special node \tikz[baseline=-3] \node [root] {}; representing the point $\star\in \mathbb{R}\times M$.
As previously, the nodes \tikz[baseline=-3] \node [var] {}; represent the arguments of our kernel and the the nodes which we draw as \tikz[baseline=-3] \node [dot] {};, represent dummy variables that are to be integrated out.
Each line then represents a kernel, with 
\tikz[baseline=-0.1cm] \draw[kernel] (0,0) to (1,0) \enlarge;
representing the kernel  $K((t,p),(s,q)):=G_{t-s}(p,q) \kappa(t-s)\kappa(d(p,q)/\bar{r})$, 
with
\tikz[baseline=-0.1cm] \draw[keps] (0,0) to (1,0) \enlarge;
representing the kernel 
\begin{equation}\label{eq:space_time mollified kernel}
 \left(K * \rho_\epsilon\right) (z_1, z_2)= \int_{\mathbb{R}\times M}   K (z_1,z') \rho_\epsilon (z',z_2) \dVol_z 
\end{equation} 
for $\rho_\epsilon(z_1, z_2) = \phi^\epsilon(t_1-{t}_2) G_{{\epsilon}^2} (p_1 , {p}_2 )$
and
\tikz[baseline=-0.1cm] \draw[testfcn] (1,0) to (0,0) \enlarge;
representing a representing a localised test function $\phi_\star^\lambda\in \mathcal{C}^\infty (\pi^*E^{\otimes i})$, for some $i\in \{1,2,3\}$ depending on the context.
Whenever we draw a barred arrow 
\tikz[baseline=-0.1cm] \draw[kernel1] (0,0) to (1,0) \enlarge;
this represents a factor $K(z_1, z_2) - R \left(Q_{\leq 0} j_\star K(\cdot, z_2)\right) (z_1)$, where
$z_1$ and $z_2$ are the coordinates of the starting and end
point respectively. 
Similarly\newline \tikz[baseline=-0.1cm] \draw[kernel2] (0,0) to (1,0) \enlarge;
represents a factor $K(z_1, z_2) -R \left(Q_{\leq 1} j_\star K(\cdot, z_2)\right) (z_1)$. 
In the calculations that follow it will be convenient to recall that the heat kernel $G_t(p,q)\in  E|_p \otimes E^*|_q$ satisfies 
$G^{\star}_t(q,p)= G_t(p,q)$, c.f.\ Remark~\ref{rem:generalised_laplacian}. 
Similarly to above we can write
\begin{equation}
\bigl(\Pi_\star^{(\epsilon)}\<2>\bigr)(\phi_\star^\lambda)
= 
\begin{tikzpicture}[scale=0.35,baseline=0.3cm]
	\node at (0,-1)  [root] (root) {};
	\node at (0,1)  [dot] (int) {};
	\node at (-1,3.5)  [var] (left) {};
	\node at (1,3.5)  [var] (right) {};
	
	\draw[testfcn] (int) to  (root);
	
	\draw[keps] (left) to (int);
	\draw[keps] (right) to (int);
\end{tikzpicture}\;
+ \;
\begin{tikzpicture}[scale=0.35,baseline=0.3cm]
	\node at (0,-1)  [root] (root) {};
	\node at (0,1)  [dot] (int) {};
	\node at (0,3.5) [dot] (top) {};
	
	\draw[testfcn] (int) to  (root);
	
	\draw[keps] (top) to[bend left=60] (int); 
	\draw[keps] (top) to[bend right=60] (int); 
\end{tikzpicture}\;.
\end{equation}
as well as
$$
\begin{tikzpicture}[scale=0.35,baseline=0.3cm]
	\node at (0,0)  [root] (root) {};
	\node at (0,2.5) [reddot] (top) {};
	\draw[red, keps] (top) to[bend left=60] (root); 
	\draw[red, keps] (top) to[bend right=60] (root); 
\end{tikzpicture}\; =   C_\epsilon h^{\sharp} 
\qquad \text{and } \qquad
\begin{tikzpicture}[scale=0.35,baseline=0.3cm]
	\node at (0,0)  [root] (root) {};
	\node at (0,2.5) [bluedot] (top) {};
	\draw[blue,keps] (top) to[bend left=60] (root); 
	\draw[blue,keps] (top) to[bend right=60] (root); 
\end{tikzpicture}\; = \;
\begin{tikzpicture}[scale=0.35,baseline=0.3cm]
	\node at (0,0)  [root] (root) {};
	\node at (0,2.5) [dot] (top) {};
	\draw[keps] (top) to[bend left=60] (root); 
	\draw[keps] (top) to[bend right=60] (root); 
\end{tikzpicture}\; - \;
\begin{tikzpicture}[scale=0.35,baseline=0.3cm]
	\node at (0,0)  [root] (root) {};
	\node at (0,2.5) [reddot] (top) {};
	\draw[red, keps] (top) to[bend left=60] (root); 
	\draw[red, keps] (top) to[bend right=60] (root); 
\end{tikzpicture}\;  .
$$

\subsubsection{A first counterterm.}
\begin{lemma}\label{lem:phi43 first counterterm}
The section $\begin{tikzpicture}[scale=0.35,baseline=0.3cm]
	\node at (0,0)  [root] (int) {};
	\node at (0,2.5) [bluedot] (top) {};
	\draw[blue, keps] (top) to[bend left=60] (int); 
	\draw[blue, keps] (top) to[bend right=60] (int); 
\end{tikzpicture}\;$ of $\mathcal{C}^\infty(\pi^* E^{\otimes 2})$ 
is uniformly bounded for $\epsilon\in (0,1]$ and converges uniformly.
\end{lemma}
\begin{proof}
We proceed as in the proof of Lemma~\ref{lem:second counterterm}, the only slight difference being the fact the $\xi$ is an $E$-valued space time white noise, instead of a scalar valued spatial white noise.
First, fix a local frame $\{e_{p,i}\}_{i=1}^{\dim E} \in E|_p$ and let 
$\xi_{i,p}= \langle \tau_{z,p}e_{p,i}, \xi\rangle_E\in \mathcal{D}'(M)$.
%
Thus, we find writing $\left(\phi^\epsilon\right)^{\star 2}(s-s')= \int \phi^\epsilon (s- r )\phi^\epsilon (s'- r )  dr$ (recall that $\phi$ is even)
\begin{align*}
\begin{tikzpicture}[scale=0.35,baseline=0.3cm]
	\node at (0,0)  [root] (root) {};
	\node at (0,2.5) [dot] (top) {};
	\draw[keps] (top) to[bend left=60] (root); 
	\draw[keps] (top) to[bend right=60] (root); 
\end{tikzpicture} 
&= \sum_i \int_\mathbb{R^2}  \left(\phi^\epsilon\right)^{\star 2}(s-s') \\
&\qquad \times \left(\int_{B_r} 
\left( Z^E_{t-s+\epsilon^2}(p, z)(\tau_{z,p}e_{p,i}) \right) \otimes
 \left(Z^E_{t-s'+\epsilon^2}(p, z)(\tau_{z,p}e_{p,i})  \right)
  \dVol_z\right)  ds  ds' + R^\epsilon\\
 &= \sum_i e_{p,i}\otimes e_{p,i} \int_\mathbb{R^2} \int_{B_r(p)}  \left(\phi^\epsilon\right)^{\star 2}(s-s')
 {Z}_{t-s+\epsilon^2}(p, z)
 {Z}_{t-s'+\epsilon^2}(p, z)
 ds  ds'  \dVol_z+ R^\epsilon \ ,
\end{align*}
where we used that the $\xi_{i,p}(r,z)$ for $i=1,\ldots,d$ are independent. Finally, note the conclusion of the Lemma holds for $R^\epsilon$ and
\begin{align*}
\int_\mathbb{R^2} \int_{B_r(p)}  \phi^\epsilon (s- r )\phi^\epsilon (s'- r )&
 {Z}_{t-s+\epsilon^2}(p, z)
 {Z}_{t-s'+\epsilon^2}(p, z)
 ds  ds' dr \dVol_z  -C_\epsilon
 \end{align*}
 converges uniformly in $p$ as $\epsilon \to 0$.
\end{proof}

Thus, we find as previously,
\begin{equation}
\bigl(\hat \Pi_\star^{(\epsilon)}\<2>\bigr)(\phi_\star^\lambda)
= 
\begin{tikzpicture}[scale=0.35,baseline=0.3cm]
	\node at (0,-1)  [root] (root) {};
	\node at (0,1)  [dot] (int) {};
	\node at (-1,3.5)  [var] (left) {};
	\node at (1,3.5)  [var] (right) {};
	
	\draw[testfcn] (int) to  (root);
	
	\draw[keps] (left) to (int);
	\draw[keps] (right) to (int);
\end{tikzpicture}\;
%
%
+\;
\begin{tikzpicture}[scale=0.35,baseline=0.3cm]
	\node at (0,-1)  [root] (root) {};
	\node at (0,1)  [dot] (int) {};
	\node at (0,3.5) [bluedot] (top) {};
	
	\draw[testfcn] (int) to  (root);
	
	\draw[blue, keps] (top) to[bend left=60] (int); 
	\draw[blue, keps] (top) to[bend right=60] (int); 
\end{tikzpicture}\; .
\end{equation}
Since,
\begin{align}
\bigl(\Pi_\star^{(\epsilon)}\<3>\bigr)(\phi_\star^\lambda)&= 
\begin{tikzpicture}[scale=0.35,baseline=0.25cm]
	\node at (0,1)  [int] (int) {}; 
	\node at (-1.5,3.5)  [var] (left) {};
	\node at (0,3.8)  [var] (middle) {};
	\node at (1.5,3.5)  [var] (right) {};
	\node at (0,-1)  [root] (root) {};
	\draw[keps] (left) to  (int);	
	\draw[keps] (right) to (int);
	\draw[keps] (middle) to (int);
	\draw[testfcn] (int) to  (root);
\end{tikzpicture}\;
 + 3\;
\begin{tikzpicture}[scale=0.35,baseline=0.25cm]
   \node at (0,-1)  [root] (root) {};
	\node at (0,1)  [int] (int) {};
	\node at (-1,3.6)  [int] (top) {};
	\node at (1,3.6)  [var] (topr) {};
	\draw[keps,bend right = 60] (top) to  (int);	
	\draw[keps,bend left = 60] (top) to (int);
	\draw[keps,bend left = 60] (topr) to (int);
	\draw[testfcn] (int) to  (root);
\end{tikzpicture}\;,
\end{align}
we also find that
\begin{align}
\bigl(\hat \Pi_\star^{(\epsilon)}\<3>\bigr)(\phi_\star^\lambda)&= 
\begin{tikzpicture}[scale=0.35,baseline=0.25cm]
	\node at  (0,1)  [int] (int) {};
	\node at (-1.5,3.5)  [var] (left) {};
	\node at (0,3.8)  [var] (middle) {};
	\node at (1.5,3.5)  [var] (right) {};
	\node at (0,-1)  [root] (root) {};
	\draw[keps] (left) to  (int);	
	\draw[keps] (right) to (int);
	\draw[keps] (middle) to (int);
	\draw[testfcn] (int) to  (root);
\end{tikzpicture}\;
 + 3\;
\begin{tikzpicture}[scale=0.35,baseline=0.25cm]
   \node at (0,-1)  [root] (root) {};
	\node at (0,1)  [int] (int) {};
	\node at (-1,3.6)  [bluedot] (top) {};
	\node at (1,3.6)  [var] (topr) {};
	\draw[blue, keps,bend right = 60] (top) to  (int);	
	\draw[blue, keps,bend left = 60] (top) to (int);
	\draw[keps,bend left = 60] (topr) to (int);
	\draw[testfcn] (int) to  (root);
\end{tikzpicture}\;.
\end{align}

The next lemma summarises what we can conclude so far, arguing exactly as in the proof of Proposition~\ref{prop:stochastic estimate for trees in nonlinearity}.
\begin{lemma}\label{lem:stochstic estimats wick trees}
The conclusion of Proposition~\ref{prop:stochastic estimate for trees in nonlinearity} holds for $\tau_p\in \Func_W(T)$ for 
$$T\in\left\{ \<3>, \ \<3'>, \  \<3''>, \  \<3'''>   \right\}\ . $$
\end{lemma}

\subsubsection{A second counterterm}
				We introduce the kernel of two variables
\begin{equation}
Q_\epsilon = 
\begin{tikzpicture}[scale=0.35,baseline=-0.1cm]
	\node at (4,0)  [root] (root) {};
	\node at (0,0)  [root] (middle) {};
	\node at (2,-1.5)  [int] (left) {};
	\node at (2,1.5)  [int] (right) {};
	
	\draw[keps,bend right=30] (left) to  (root);	
	\draw[keps,bend left=30] (left) to  (middle);	
	\draw[keps,bend left=30] (right) to  (root);	
	\draw[keps,bend right=30] (right) to  (middle);	
	\draw[kernel] (middle) to (root);
\end{tikzpicture}\;.
\end{equation}
with the convention that the left green dot represents the variable $\star_1\in \mathbb{R}\times M$ and the right one $\star_2\in \mathbb{R}\times M$.
Further set
\begin{equation}\label{eq:renormalised sunset}
\begin{tikzpicture}[scale=0.35,baseline=-0.1cm]
	\node at (4,0)  [root] (root) {};
	\node at (0,0)  [root] (middle) {};
	\node at (2,-1.5)  [reddot] (left) {};
	\node at (2,1.5)  [reddot] (right) {};
	
	\draw[red,keps,bend right=30] (left) to  (root);	
	\draw[red,keps,bend left=30] (left) to  (middle);	
	\draw[red,keps,bend left=30] (right) to  (root);	
	\draw[red,keps,bend right=30] (right) to  (middle);	
	\draw[red,kernel] (middle) to (root);
\end{tikzpicture}\; := C_\epsilon \delta_{\star_2} (\star_1) 
\qquad  \textit{ and } \qquad
\begin{tikzpicture}[scale=0.35,baseline=-0.1cm]
	\node at (4,0)  [root] (root) {};
	\node at (0,0)  [root] (middle) {};
	\node at (2,-1.5)  [bluedot] (left) {};
	\node at (2,1.5)  [bluedot] (right) {};
	
	\draw[blue,keps,bend right=30] (left) to  (root);	
	\draw[blue,keps,bend left=30] (left) to  (middle);	
	\draw[blue,keps,bend left=30] (right) to  (root);	
	\draw[blue,keps,bend right=30] (right) to  (middle);	
	\draw[blue,kernel] (middle) to (root);
\end{tikzpicture}\; = Q_\epsilon (\star_1, \star_2) -C'_\epsilon \delta_{\star_2} (\star_1) \ .
\end{equation}
where $\delta_{\star_2} ( \ \cdot \ )$ denotes the Dirac distribution at $\star_2$ acting on $\phi\in \mathcal{D}(\pi^* E)$ by
$\langle \delta_{\star_2} , \phi\rangle= \phi(\star_2)$.

%
%
\begin{lemma}\label{lem:phi43 second counterterm}
One has $$\begin{tikzpicture}[scale=0.35,baseline=-0.1cm]
	\node at (4,0)  [root] (root) {};
	\node at (0,0)  [root] (middle) {};
	
	\node at (2,-1.5)  [bluedot] (left) {};
	\node at (2,1.5)  [bluedot] (right) {};
	
	\draw[blue,keps,bend right=30] (left) to  (root);	
	\draw[blue,keps,bend left=30] (left) to  (middle);	
	\draw[blue,keps,bend left=30] (right) to  (root);	
	\draw[blue,keps,bend right=30] (right) to  (middle);	
	\draw[blue,kernel] (middle) to (root);
\end{tikzpicture}\; = \mathcal{R}Q_\epsilon (\star_1, \star_2) -k_\epsilon(\star_2) \delta_{\star_2} (\star_1) \ $$ 
where the $k_\epsilon\in \mathcal{C}^\infty\left( L(\pi^*E, \pi^*E)\right)$ is uniformly bounded for $\epsilon \in (0,1]$ and converges uniformly as $\epsilon\to 0$.
%
\end{lemma}

\begin{proof}
Let us first calculate, denoting the first green node as $\star_1=(t,p)$ and the second as $\star_2=(t',q)$
\begin{align*}
&\begin{tikzpicture}[scale=0.35,baseline=-0.1cm]
	\node at (4,0)  [root] (root) {};
	\node at (0,0)  [root] (middle) {};
	\node at (2,1.5)  [int] (right) {};
	\draw[keps,bend left=30] (right) to  (root);	
	\draw[keps,bend right=30] (right) to  (middle);	
\end{tikzpicture}\;\\
 &= \sum_i e_{p,i}\otimes \tau_{q,p}e_{p,i} \int_\mathbb{R^2} \int_{B_r(p)}   \phi^\epsilon (s- r )\phi^\epsilon (s'- r )
Z_{t+t'-s-s'+2\epsilon^2}(p, q)
 ds  ds' dr \dVol_z\\
 &\qquad + R^\epsilon_1 
 \end{align*}
%
Writing 
$$\tilde{Z}^{(\epsilon)}((t,p),(t',q)):=\int_\mathbb{R^2} \int_{B_r}   \phi^\epsilon (s- r )\phi^\epsilon (s'- r )
Z_{t+t'-s-s'+2\epsilon^2}(p, q)
 ds  ds' dr \dVol_z$$
we define 
$S_\epsilon(\star_1, \star_2) = \tilde{Z}^{(\epsilon)} (\star_1, \star_2)^2 {Z}(\star_1, \star_2)$ and $S^E_\epsilon (\star_1, \star_2)=S_\epsilon (\star_1, \star_2) \tau_{q,p}$ and find
\begin{align*}
\begin{tikzpicture}[scale=0.35,baseline=-0.1cm]
	\node at (4,0)  [root] (root) {};
	\node at (0,0)  [root] (middle) {};
	\node at (2,-1.5)  [int] (left) {};
	\node at (2,1.5)  [int] (right) {};
	\draw[keps,bend right=30] (left) to  (root);	
	\draw[keps,bend left=30] (left) to  (middle);	
	\draw[keps,bend left=30] (right) to  (root);	
	\draw[keps,bend right=30] (right) to  (middle);	
	\draw[kernel] (middle) to (root);
\end{tikzpicture}\;&
=  H\left( \tau_{q,p} H(\cdot, e_{i,p}, e_{i,p}), \tau_{q,p} e_{i,p}, \tau_{q,p} e_{i,p}\right) S_\epsilon (\star_1, \star_2) +R_2^\epsilon(\star_1, \star_2) \\
&=  S^E_\epsilon (\star_1, \star_2) +R_2^\epsilon(\star_1, \star_2)  \ .
\end{align*}
where $R_2^\epsilon(p,q)$ satisfies $\vertiii{R_2^\epsilon}_{-|\fraks|+1/2, m}<\infty$ uniformly over $\epsilon$.
Since $\mathcal{R} S^E_\epsilon= S^E_\epsilon (\star_1, \star_2) - \int S_\epsilon(\star_1, z) dz \cdot \delta_{\star_1} (\star_2)$, 
we conclude by observing that
$\int S_\epsilon(\star_1, z) dz - C'_\epsilon$ converges uniformly as $\epsilon \to 0$.
\end{proof}

\subsubsection{Larger trees}
Note that
\begin{align*}
 \bigl( \Pi_\star^{(\epsilon)}\<32>\bigr)(\phi_\star^\lambda) &= 
\begin{tikzpicture}[scale=0.35,baseline=0.25cm]
\node at (0,-2)  [root] (root) {};
	\node at (0,0)  [int] (int) {};
	\node at (0,2.5)  [int] (middle) {};
	\node at (-1.6,2)  [var] (left) {};
	\node at (1.6,2)  [var] (right) {};
	\node at (-1.5,4.5)  [var] (tl) {};
	\node at (1.5,4.5)  [var] (tr) {};
	\node at (0,5)  [var] (tm) {};	
	\draw[keps] (left) to  (int);	
	\draw[keps] (right) to (int);
	\draw[keps] (tl) to  (middle);	
	\draw[keps] (tr) to (middle);
	\draw[keps] (tm) to  (middle);	
	\draw[kernel1] (middle) to (int);
	\draw[testfcn] (int) to  (root);
\end{tikzpicture}\; + 
6\;
\begin{tikzpicture}[scale=0.35,baseline=0.25cm]
\node at (0,-2)  [root] (root) {};
	\node at (0,-0)  [int] (int) {};
	\node at (0,2.5)  [int] (middle) {};
	\node at (-1.6,1.25)  [int] (left) {};
	\node at (1.6,2)  [var] (right) {};
	\node at (1.5,4.5)  [var] (tr) {};
	\node at (0,6)  [var] (tm) {};
	\draw[keps,bend right=40] (left) to  (int);	
	\draw[keps] (right) to (int);
	\draw[keps,bend left=40] (left) to  (middle);	
	\draw[keps] (tr) to (middle);
	\draw[keps] (tm) to  (middle);	
	\draw[kernel1] (middle) to (int);
	\draw[testfcn] (int) to  (root);
\end{tikzpicture}\; + 
3\;
\begin{tikzpicture}[scale=0.35,baseline=0.25cm]
\node at (0,-2)  [root] (root) {};
	\node at (0,0)  [int] (int) {};
	\node at (0,2.5)  [int] (middle) {};
	\node at (-1.6,2)  [var] (left) {};
	\node at (1.6,2)  [var] (right) {};
	\node at (-1,4.6)  [int] (tl) {};
	\node at (1,4.6)  [var] (tr) {};
	\draw[keps] (left) to  (int);	
	\draw[keps] (right) to (int);
	\draw[keps,bend right = 60] (tl) to  (middle);	
	\draw[keps,bend left = 60] (tl) to (middle);
	\draw[keps,bend left = 60] (tr) to (middle);
	\draw[kernel1] (middle) to (int);
	\draw[testfcn] (int) to  (root);
\end{tikzpicture}\; + \;
\begin{tikzpicture}[scale=0.35,baseline=0.25cm]
\node at (0,-2)  [root] (root) {};
	\node at (0,0)  [int] (int) {};
	\node at (1,2.6)  [int] (middle) {};
	\node at (-1,2.6)  [int] (left) {};
	\node at (-0.5,4.5)  [var] (tl) {};
	\node at (2.5,4.5)  [var] (tr) {};
	\node at (1,5)  [var] (tm) {};
	\draw[keps,bend right = 60] (left) to  (int);	
	\draw[keps,bend left = 60] (left) to (int);
	\draw[kernel1,bend left = 60] (middle) to (int);
	\draw[keps] (tl) to  (middle);	
	\draw[keps] (tr) to (middle);
	\draw[keps] (tm) to  (middle);	
	\draw[testfcn] (int) to  (root);
\end{tikzpicture}\\ 
&\qquad + 
3\;
\begin{tikzpicture}[scale=0.35,baseline=0.25cm]
\node at (0,-2)  [root] (root) {};
	\node at (0,0)  [int] (int) {};
	\node at (1,2.6)  [int] (middle) {};
	\node at (-1,2.6)  [int] (left) {};
	\node at (0,4.6)  [int] (tl) {};
	\node at (2,4.6)  [var] (tr) {};
	\draw[keps,bend right = 60] (left) to  (int);	
	\draw[keps,bend left = 60] (left) to (int);
	\draw[kernel1,bend left = 60] (middle) to (int);
	\draw[keps,bend right = 60] (tl) to  (middle);	
	\draw[keps,bend left = 60] (tl) to (middle);
	\draw[keps,bend left = 60] (tr) to (middle);
	\draw[testfcn] (int) to  (root);
\end{tikzpicture}\; + 6\;
\begin{tikzpicture}[scale=0.35,baseline=0.25cm]
\node at (0,-2)  [root] (root) {};
	\node at (0,0)  [int] (int) {};
	\node at (0,2.5)  [int] (middle) {};
	\node at (-1.6,1.25)  [int] (left) {};
	\node at (1.6,2)  [var] (right) {};
	\node at (0,4.5)  [int] (top) {};
	\draw[keps,bend right=40] (left) to  (int);	
	\draw[keps] (right) to (int);
	\draw[keps,bend left=40] (left) to  (middle);	
	\draw[keps,bend right = 60] (top) to (middle);
	\draw[keps,bend left = 60] (top) to  (middle);	
	\draw[kernel1] (middle) to (int);
	\draw[testfcn] (int) to  (root);
\end{tikzpicture}\; + 6\;
\begin{tikzpicture}[scale=0.35,baseline=0.25cm]
\node at (0,-2)  [root] (root) {};
	\node at (0,0)  [int] (int) {};
	\node at (0,2.5)  [int] (middle) {};
	\node at (-1.6,1.25)  [int] (left) {};
	\node at (1.6,1.25)  [int] (right) {};
	\node at (0,5)  [var] (tm) {};
	\draw[keps,bend right=40] (left) to  (int);	
	\draw[keps,bend left=40] (left) to  (middle);	
	\draw[keps,bend left=40] (right) to  (int);	
	\draw[keps,bend right=40] (right) to  (middle);	
	\draw[keps] (tm) to  (middle);	
	\draw[kernel1] (middle) to (int);
	\draw[testfcn] (int) to  (root);
\end{tikzpicture}\;,
\end{align*}
from which it follows that
\begin{align}\label{eq:renormalised 32}
 \bigl( \Pi_\star^{(\epsilon)}\<32>\bigr)(\phi_\star^\lambda) &= 
\begin{tikzpicture}[scale=0.35,baseline=0.25cm]
\node at (0,-2)  [root] (root) {};
	\node at (0,0)  [int] (int) {};
	\node at (0,2.5)  [int] (middle) {};
	\node at (-1.6,2)  [var] (left) {};
	\node at (1.6,2)  [var] (right) {};
	\node at (-1.5,4.5)  [var] (tl) {};
	\node at (1.5,4.5)  [var] (tr) {};
	\node at (0,5)  [var] (tm) {};	
	\draw[keps] (left) to  (int);	
	\draw[keps] (right) to (int);
	\draw[keps] (tl) to  (middle);	
	\draw[keps] (tr) to (middle);
	\draw[keps] (tm) to  (middle);	
	\draw[kernel1] (middle) to (int);
	\draw[testfcn] (int) to  (root);
\end{tikzpicture}\; + 
6\;
\begin{tikzpicture}[scale=0.35,baseline=0.25cm]
\node at (0,-2)  [root] (root) {};
	\node at (0,-0)  [int] (int) {};
	\node at (0,2.5)  [int] (middle) {};
	\node at (-1.6,1.25)  [int] (left) {};
	\node at (1.6,2)  [var] (right) {};
	\node at (1.5,4.5)  [var] (tr) {};
	\node at (0,6)  [var] (tm) {};
	\draw[keps,bend right=40] (left) to  (int);	
	\draw[keps] (right) to (int);
	\draw[keps,bend left=40] (left) to  (middle);	
	\draw[keps] (tr) to (middle);
	\draw[keps] (tm) to  (middle);	
	\draw[kernel1] (middle) to (int);
	\draw[testfcn] (int) to  (root);
\end{tikzpicture}\; + 
3\;
\begin{tikzpicture}[scale=0.35,baseline=0.25cm]
\node at (0,-2)  [root] (root) {};
	\node at (0,0)  [int] (int) {};
	\node at (0,2.5)  [int] (middle) {};
	\node at (-1.6,2)  [var] (left) {};
	\node at (1.6,2)  [var] (right) {};
	\node at (-1,4.6)  [bluedot] (tl) {};
	\node at (1,4.6)  [var] (tr) {};
	\draw[keps] (left) to  (int);	
	\draw[keps] (right) to (int);
	\draw[blue, keps,bend right = 60] (tl) to  (middle);	
	\draw[blue, keps,bend left = 60] (tl) to (middle);
	\draw[keps,bend left = 60] (tr) to (middle);
	\draw[kernel1] (middle) to (int);
	\draw[testfcn] (int) to  (root);
\end{tikzpicture}\; + \;
\begin{tikzpicture}[scale=0.35,baseline=0.25cm]
\node at (0,-2)  [root] (root) {};
	\node at (0,0)  [int] (int) {};
	\node at (1,2.6)  [int] (middle) {};
	\node at (-1,2.6)  [bluedot] (left) {};
	\node at (-0.5,4.5)  [var] (tl) {};
	\node at (2.5,4.5)  [var] (tr) {};
	\node at (1,5)  [var] (tm) {};
	\draw[blue, keps,bend right = 60] (left) to  (int);	
	\draw[blue, keps,bend left = 60] (left) to (int);
	\draw[kernel1,bend left = 60] (middle) to (int);
	\draw[keps] (tl) to  (middle);	
	\draw[keps] (tr) to (middle);
	\draw[keps] (tm) to  (middle);	
	\draw[testfcn] (int) to  (root);
\end{tikzpicture}\\ 
&\qquad + 
3\;
\begin{tikzpicture}[scale=0.35,baseline=0.25cm]
\node at (0,-2)  [root] (root) {};
	\node at (0,0)  [int] (int) {};
	\node at (1,2.6)  [int] (middle) {};
	\node at (-1,2.6)  [bluedot] (left) {};
	\node at (0,4.6)  [bluedot] (tl) {};
	\node at (2,4.6)  [var] (tr) {};
	\draw[blue, keps,bend right = 60] (left) to  (int);	
	\draw[blue, keps,bend left = 60] (left) to (int);
	\draw[kernel1,bend left = 60] (middle) to (int);
	\draw[blue, keps,bend right = 60] (tl) to  (middle);	
	\draw[blue, keps,bend left = 60] (tl) to (middle);
	\draw[keps,bend left = 60] (tr) to (middle);
	\draw[testfcn] (int) to  (root);
\end{tikzpicture}\; + 6\;
\begin{tikzpicture}[scale=0.35,baseline=0.25cm]
\node at (0,-2)  [root] (root) {};
	\node at (0,0)  [int] (int) {};
	\node at (0,2.5)  [int] (middle) {};
	\node at (-1.6,1.25)  [int] (left) {};
	\node at (1.6,2)  [var] (right) {};
	\node at (0,4.5)  [bluedot] (top) {};
	\draw[keps,bend right=40] (left) to  (int);	
	\draw[keps] (right) to (int);
	\draw[keps,bend left=40] (left) to  (middle);	
	\draw[blue, keps,bend right = 60] (top) to (middle);
	\draw[blue, keps,bend left = 60] (top) to  (middle);	
	\draw[kernel1] (middle) to (int);
	\draw[testfcn] (int) to  (root);
\end{tikzpicture}\; + 6\;
\begin{tikzpicture}[scale=0.35,baseline=0.25cm]
\node at (0,-2)  [root] (root) {};
	\node at (0,0)  [dot] (int) {};
	\node at (0,2.5)  [bluedot] (middle) {};
	\node at (-1.6,1.25)  [bluedot] (left) {};
	\node at (1.6,1.25)  [bluedot] (right) {};
	\node at (0,5)  [var] (tm) {};
	\draw[blue, keps,bend right=40] (left) to  (int);	
	\draw[blue, keps,bend left=40] (left) to  (middle);	
	\draw[blue, keps,bend left=40] (right) to  (int);	
	\draw[blue, keps,bend right=40] (right) to  (middle);	
	\draw[keps] (tm) to  (middle);	
	\draw[blue, kernel] (middle) to (int);
	\draw[testfcn] (int) to  (root);
\end{tikzpicture}\;
+ 6\;
\begin{tikzpicture}[scale=0.35,baseline=0.25cm]
\node at (0,-2)  [root] (root) {};
	\node at (0,0)  [int] (int) {};
	\node at (0,2.5)  [int] (middle) {};
	\node at (-1.6,1.25)  [int] (left) {};
	\node at (0,1.25)  [int] (right) {};
	\node at (0,5)  [var] (tm) {};
	\node[above] at (2,0) {${\scriptstyle 0}$};	 
	\draw[keps,bend right=40] (left) to  (int);	
	\draw[keps,bend left=40] (left) to  (middle);	
	\draw[keps] (right) to  (int);	
	\draw[keps] (right) to  (middle);	
	\draw[keps] (tm) to  (middle);	
	\draw[kernel, bend left= 90] (middle) to (root);
	\draw[gray, kernel, bend right= 70] (root) to (int);
	\draw[testfcn] (int) to  (root);
\end{tikzpicture}\;. \nonumber
\end{align}	 
where the gray arrow and the $0$ in the last digram are to be read as follows:  For $n\in \mathbb{N}$
$$
\begin{tikzpicture}[baseline=-0.1cm]
\draw[kernel] (0,0) to (0.95,0); \node at (1,0)  [root] (root) {}; \draw[gray, kernel] (1.1,0) to (2,0);
\node[above] at (0.5,0) {${\scriptstyle n}$}; 
\end{tikzpicture}
$$
represents $\left(R Q_{\leq n} j_\star K(\cdot,z_1)\right) (z_2)$ where $z_1$ is the variable where the left arrow starts, and $z_2$ is the variable where the right gray arrow ends.

\begin{lemma}\label{lem:stochastic bounds 32 and 32'}
The conclusion of Proposition~\ref{prop:stochastic estimate for trees in nonlinearity} holds for $\tau_p\in \Func_W(T)$ for 
$T\in\left\{ \<32>, \<32'>   \right\} $.
\end{lemma}
\begin{proof}
Equation~\ref{eq:renormalised 32} together with $ \hat{\Pi}_q^{(\epsilon)} \hat{\Gamma}_{q,p} \<32>= \hat{\Pi}_p^{(\epsilon)} \<32>$ imply the conclusion of the lemma for $T=\<32>$ by repeated application of the Lemmas in Section~\ref{sec:Kernels with prescribed singularities} exactly as in the proof of \cite[Thm.~10.22]{Hai14}.

In order to treat the case $T=\<32'>$, first note that $\<31>$ satisfies by a similar argument as above the conclusion of Proposition~\ref{prop:stochastic estimate for trees in nonlinearity} and then note that elements of $\Func_W(\<32'>)$ are obtained from it by multiplication with a jet.
\end{proof}


\subsubsection{A last tree}
We introduce the following notation. For $v\in \bigoplus_{p\in M} JE$, we write
\begin{equation}
\begin{tikzpicture}[scale=0.35,baseline=-0.1cm]
						\node at (4,0)  [root] (root) {};
						\node [left] at (0.2,0) {$v$};
	\draw[kernelBig] (0,0) to (root);
\end{tikzpicture}\;.
\end{equation}
for the function $R v (\star)$.
Then, writing  $\<3'2>_{(e_p)} \in \Func_W \<3'2>$ for the element canonically isomorphic to $e_p\in E|_p \subset J_p E$, one finds that
\begin{align*}
 \bigl( \Pi_\star^{(\epsilon)}\<3'2>_{(e_\star)}\bigr)(\phi_\star^\lambda) &= 
\begin{tikzpicture}[scale=0.35,baseline=0.25cm]
\node at (0,-2)  [root] (root) {};
	\node at (0,0)  [int] (int) {};
	\node at (0,2.5)  [int] (middle) {};
	\node at (-1.6,2)  [var] (left) {};
	\node at (1.6,2)  [var] (right) {};
	\node at (-1.5,4.5)  [var] (tl) {};
	\node at (1.5,4.5)  [var] (tr) {};
	\node [above] at (0,4.8) {$e_\star$};	  
	\draw[keps] (left) to  (int);	
	\draw[keps] (right) to (int);
	\draw[keps] (tl) to  (middle);	
	\draw[keps] (tr) to (middle);
	\draw[kernelBig] (0,5) to  (middle);	
	\draw[kernel1] (middle) to (int);
	\draw[testfcn] (int) to  (root);
\end{tikzpicture}\; + 
4\;
\begin{tikzpicture}[scale=0.35,baseline=0.25cm]
\node at (0,-2)  [root] (root) {};
	\node at (0,-0)  [int] (int) {};
	\node at (0,2.5)  [int] (middle) {};
	\node at (-1.6,1.25)  [int] (left) {};
	\node at (1.6,2)  [var] (right) {};
	\node at (1.5,4.5)  [var] (tr) {};
\node [above] at (0,4.8) {$e_\star$};	  
	\draw[keps,bend right=40] (left) to  (int);	
	\draw[keps] (right) to (int);
	\draw[keps,bend left=40] (left) to  (middle);	
	\draw[keps] (tr) to (middle);
	\draw[kernelBig] (0,5) to  (middle);	
	\draw[kernel1] (middle) to (int);
	\draw[testfcn] (int) to  (root);
\end{tikzpicture}\; + 
\;
\begin{tikzpicture}[scale=0.35,baseline=0.25cm]
\node at (0,-2)  [root] (root) {};
	\node at (0,0)  [int] (int) {};
	\node at (0,2.5)  [int] (middle) {};
	\node at (-1.6,2)  [var] (left) {};
	\node at (1.6,2)  [var] (right) {};
	\node at (-1,4.6)  [int] (tl) {};
\node [above] at (1.1,4.4) {$e_\star$};	  
	\draw[keps] (left) to  (int);	
	\draw[keps] (right) to (int);
	\draw[keps,bend right = 60] (tl) to  (middle);	
	\draw[keps,bend left = 60] (tl) to (middle);
	\draw[kernelBig,bend left = 60] (1,4.6)  to (middle);
	\draw[kernel1] (middle) to (int);
	\draw[testfcn] (int) to  (root);
\end{tikzpicture}\; + \;
\begin{tikzpicture}[scale=0.35,baseline=0.25cm]
\node at (0,-2)  [root] (root) {};
	\node at (0,0)  [int] (int) {};
	\node at (1,2.6)  [int] (middle) {};
	\node at (-1,2.6)  [int] (left) {};
	\node at (-0.5,4.5)  [var] (tl) {};
	\node at (1,5)  [var] (tm) {};
	\node [above] at (2.8,4.2) {$e_\star$};	  
	\draw[keps,bend right = 60] (left) to  (int);	
	\draw[keps,bend left = 60] (left) to (int);
	\draw[kernel1,bend left = 60] (middle) to (int);
	\draw[keps] (tl) to  (middle);	
	\draw[kernelBig] (2.5,4.5)  to (middle);
	\draw[keps] (tm) to  (middle);	
	\draw[testfcn] (int) to  (root);
\end{tikzpicture}\\ 
&\qquad + 
\;
\begin{tikzpicture}[scale=0.35,baseline=0.25cm]
\node at (0,-2)  [root] (root) {};
	\node at (0,0)  [int] (int) {};
	\node at (1,2.6)  [int] (middle) {};
	\node at (-1,2.6)  [int] (left) {};
	\node at (0,4.6)  [int] (tl) {};
\node [above] at (2.1,4.4) {$e_\star$};	  
	\draw[keps,bend right = 60] (left) to  (int);	
	\draw[keps,bend left = 60] (left) to (int);
	\draw[kernel1,bend left = 60] (middle) to (int);
	\draw[keps,bend right = 60] (tl) to  (middle);	
	\draw[keps,bend left = 60] (tl) to (middle);
	\draw[kernelBig,bend left = 60] (2,4.6) to (middle);
	\draw[testfcn] (int) to  (root);
\end{tikzpicture}\; +  \;
2\; \begin{tikzpicture}[scale=0.35,baseline=0.25cm]
\node at (0,-2)  [root] (root) {};
	\node at (0,0)  [int] (int) {};
	\node at (0,2.5)  [int] (middle) {};
	\node at (-1.6,1.25)  [int] (left) {};
	\node at (1.6,1.25)  [int] (right) {};
\node [above] at (0,4.8) {$e_\star$};	  
	\draw[keps,bend right=40] (left) to  (int);	
	\draw[keps,bend left=40] (left) to  (middle);	
	\draw[keps,bend left=40] (right) to  (int);	
	\draw[keps,bend right=40] (right) to  (middle);	
	\draw[kernelBig] (0,5) to  (middle);	
	\draw[kernel1] (middle) to (int);
	\draw[testfcn] (int) to  (root);
\end{tikzpicture}\;.  
\end{align*}
and thus
\begin{align}\label{eq:renormalised 3'2}
 \bigl( \hat \Pi_\star^{(\epsilon)}\<3'2>_{(e_\star)}\bigr)(\phi_\star^\lambda) &= 
\begin{tikzpicture}[scale=0.35,baseline=0.25cm]
\node at (0,-2)  [root] (root) {};
	\node at (0,0)  [int] (int) {};
	\node at (0,2.5)  [int] (middle) {};
	\node at (-1.6,2)  [var] (left) {};
	\node at (1.6,2)  [var] (right) {};
	\node at (-1.5,4.5)  [var] (tl) {};
	\node at (1.5,4.5)  [var] (tr) {};
\node [above] at (0,4.8) {$e_\star$};	  
	\draw[keps] (left) to  (int);	
	\draw[keps] (right) to (int);
	\draw[keps] (tl) to  (middle);	
	\draw[keps] (tr) to (middle);
	\draw[kernelBig] (0,5) to  (middle);	
	\draw[kernel1] (middle) to (int);
	\draw[testfcn] (int) to  (root);
\end{tikzpicture}\; + 
4\;
\begin{tikzpicture}[scale=0.35,baseline=0.25cm]
\node at (0,-2)  [root] (root) {};
	\node at (0,-0)  [int] (int) {};
	\node at (0,2.5)  [int] (middle) {};
	\node at (-1.6,1.25)  [int] (left) {};
	\node at (1.6,2)  [var] (right) {};
	\node at (1.5,4.5)  [var] (tr) {};
\node [above] at (0,4.8) {$e_\star$};	  
	\draw[keps,bend right=40] (left) to  (int);	
	\draw[keps] (right) to (int);
	\draw[keps,bend left=40] (left) to  (middle);	
	\draw[keps] (tr) to (middle);
	\draw[kernelBig] (0,5) to  (middle);	
	\draw[kernel1] (middle) to (int);
	\draw[testfcn] (int) to  (root);
\end{tikzpicture}\; + 
\;
\begin{tikzpicture}[scale=0.35,baseline=0.25cm]
\node at (0,-2)  [root] (root) {};
	\node at (0,0)  [int] (int) {};
	\node at (0,2.5)  [int] (middle) {};
	\node at (-1.6,2)  [var] (left) {};
	\node at (1.6,2)  [var] (right) {};
	\node at (-1,4.6)  [bluedot] (tl) {};
\node [above] at (1.1,4.4) {$e_\star$};	  
	\draw[keps] (left) to  (int);	
	\draw[keps] (right) to (int);
	\draw[blue, keps,bend right = 60] (tl) to  (middle);	
	\draw[blue, keps,bend left = 60] (tl) to (middle);
	\draw[kernelBig,bend left = 60] (1,4.6)  to (middle);
	\draw[kernel1] (middle) to (int);
	\draw[testfcn] (int) to  (root);
\end{tikzpicture}\; + \;
\begin{tikzpicture}[scale=0.35,baseline=0.25cm]
\node at (0,-2)  [root] (root) {};
	\node at (0,0)  [int] (int) {};
	\node at (1,2.6)  [int] (middle) {};
	\node at (-1,2.6)  [bluedot] (left) {};
	\node at (-0.5,4.5)  [var] (tl) {};
	\node at (1,5)  [var] (tm) {};
	\node [above] at (2.8,4.2) {$e_\star$};	  
	\draw[blue, keps,bend right = 60] (left) to  (int);	
	\draw[blue, keps,bend left = 60] (left) to (int);
	\draw[kernel1,bend left = 60] (middle) to (int);
	\draw[keps] (tl) to  (middle);	
	\draw[kernelBig] (2.5,4.5)  to (middle);
	\draw[keps] (tm) to  (middle);	
	\draw[testfcn] (int) to  (root);
\end{tikzpicture}\\ 
&\qquad + 
\;
\begin{tikzpicture}[scale=0.35,baseline=0.25cm]
\node at (0,-2)  [root] (root) {};
	\node at (0,0)  [int] (int) {};
	\node at (1,2.6)  [int] (middle) {};
	\node at (-1,2.6)  [bluedot] (left) {};
	\node at (0,4.6)  [bluedot] (tl) {};
\node [above] at (2.1,4.4) {$e_\star$};	  
	\draw[blue, keps,bend right = 60] (left) to  (int);	
	\draw[blue, keps,bend left = 60] (left) to (int);
	\draw[kernel1,bend left = 60] (middle) to (int);
	\draw[blue, keps,bend right = 60] (tl) to  (middle);	
	\draw[blue, keps,bend left = 60] (tl) to (middle);
	\draw[kernelBig,bend left = 60] (2,4.6) to (middle);
	\draw[testfcn] (int) to  (root);
\end{tikzpicture}\; +  \;
2\; 
\begin{tikzpicture}[scale=0.35,baseline=0.25cm]
\node at (0,-2)  [root] (root) {};
	\node at (0,0)  [int] (int) {};
	\node at (0,2.5)  [bluedot] (middle) {};
	\node at (-1.6,1.25)  [bluedot] (left) {};
	\node at (1.6,1.25)  [bluedot] (right) {};
\node [above] at (0,4.8) {$e_\star$};	  
	\draw[blue, keps,bend right=40] (left) to  (int);	
	\draw[blue, keps,bend left=40] (left) to  (middle);	
	\draw[blue, keps,bend left=40] (right) to  (int);	
	\draw[blue, keps,bend right=40] (right) to  (middle);	
	\draw[kernelBig] (0,5) to  (middle);	
	\draw[blue, kernel] (middle) to (int);
	\draw[testfcn] (int) to  (root);
\end{tikzpicture}\; + \;
2\; \begin{tikzpicture}[scale=0.35,baseline=0.25cm]
\node at (0,-2)  [root] (root) {};
	\node at (0,0)  [int] (int) {};
	\node at (0,2.5)  [int] (middle) {};
	\node at (-1.6,1.25)  [int] (left) {};
	\node at (0,1.25)  [int] (right) {};
\node [above] at (0,4.8) {$e_\star$};	
\node[above] at (2,0) {${\scriptstyle 0}$};	   
	\draw[keps,bend right=40] (left) to  (int);	
	\draw[keps,bend left=40] (left) to  (middle);	
	\draw[keps] (right) to  (int);	
	\draw[keps] (right) to  (middle);	
	\draw[kernelBig] (0,5) to  (middle);	
	\draw[kernel, bend left= 90] (middle) to (root);
	\draw[testfcn] (int) to  (root);
	\draw[gray, kernel, bend right= 70] (root) to (int);
\end{tikzpicture}\;. \nonumber
\end{align}

Next we find that for the canonical model
\begin{align*}
\Pi^\epsilon_q  \Gamma^\epsilon_{q,p} \<3'0>_{(e_p)}- \Pi^\epsilon_p \<3'0>_{(e_p)} =
& \pmb{\Pi}^\epsilon\left( \<3'0>_{(\Gamma_{q,p} e_p)} -\<3'0>_{(e_p)}\right) - R \left( Q_{\leq 1} j_q K(\pmb{\Pi}^\epsilon \<3'>_{(Q_0\Gamma_{q,p} e_p)}-\pmb{\Pi}^\epsilon\<3'>_{(e_p)})\right) \\
&- \left( \Pi^\epsilon_q  \Gamma^\epsilon_{q,p} - \Pi_q^\epsilon\right) Q_{\leq 0} j_{p} K( \pmb{\Pi}^\epsilon \<3'>_{(e_p)} )(p)\\
=& \pmb{\Pi}^\epsilon\left( \<3'0>_{(\Gamma_{q,p} e_p)} -\<3'0>_{(e_p)}\right) - R \left( Q_{\leq 1} j_q K(\pmb{\Pi}\epsilon \<3'>_{(Q_0\Gamma_{q,p} e_p)}-\pmb{\Pi}^\epsilon \<3'>_{(e_p)})\right) \\
&- R \left(  \Gamma_{q,p} - \id\right)  \left( Q_{\leq 0} j_{p}K( \pmb{\Pi}^\epsilon \<3'>_{(e_p)}) \right)(p)
\end{align*}
And therefore 
\begin{align*}
\Pi^\epsilon_q  \Gamma_{q,p}^\epsilon \<3'2>_{(e_p)}- \Pi^\epsilon_p \<3'2>_{(e_p)} =&\underbrace{
\pmb{\Pi}^\epsilon\left( \<3'2>_{(\Gamma_{q,p} e_p)} -\<3'2>_{(e_p)}\right)
- R \left( Q_{\leq 1} j_q K(\pmb{\Pi}^\epsilon \<3'>_{(Q_0\Gamma_{q,p} e_p -e_p)})\right) \cdot \left(\pmb{\Pi}\<2>\right))}_{=:\bigl( \Pi^\epsilon_q  \Gamma_{q,p}^\epsilon \<3'2>_{(e_p)}- \Pi^\epsilon_p \<3'2>_{(e_p)}\bigr)_A} \\
&-  \underbrace{\left(\pmb{\Pi}^\epsilon \<2>\right) \cdot R\left(  \Gamma_{q,p} - \id\right) Q_{\leq 0} j_{p}K( \pmb{\Pi}^\epsilon \<3'>_{(e_p)} )}_{=:\bigl( \Pi^\epsilon_q  \Gamma_{q,p}^\epsilon \<3'2>_{(e_p)}- \Pi^\epsilon_p \<3'2>_{(e_p)}\bigr)_B}
\end{align*}
We suggestively define
\begin{equation}\label{eq:renormalised error bound 3'2 B}
\bigl( \hat{\Pi}^\epsilon_q  \hat{\Gamma}_{q,p}^\epsilon \<3'2>_{(e_p)}- \hat{\Pi}^\epsilon_p  \<3'2>_{(e_p)} \bigr)_{\hat{B}}:=
\left(\hat{\pmb{\Pi}}^\epsilon\<2>\right) \cdot R\left(  {\Gamma}_{q,p} - \id\right)Q_{\leq 0} j_p  K( \hat{\pmb{\Pi}}^\epsilon (\<3'>_{(e_p)}) ) \ 
\end{equation}
and set 
\begin{align*}
\bigl( \hat{\Pi}^\epsilon_q  \hat{\Gamma}_{q,p}^\epsilon \<3'2>_{(e_p)}- \hat{\Pi}^\epsilon_p  \<3'2>_{(e_p)} \bigr)_{\hat{A}} &=
\hat{\Pi}^\epsilon_q  \hat{\Gamma}_{q,p}^\epsilon \<3'2>_{(e_p)}- \hat{\Pi}^\epsilon_p  \<3'2>_{(e_p)} \\
&-\bigl( \hat{\Pi}^\epsilon_q  \hat{\Gamma}_{q,p}^\epsilon \<3'2>_{(e_p)}- \hat{\Pi}^\epsilon_p  \<3'2>_{(e_p)} \bigr)_{\hat{B}} .
\end{align*}
Since
\begin{align*}
\bigl( \Pi^\epsilon_\star  \Gamma_{\star,p} \<3'2>_{(e_p)}- \Pi^\epsilon_p \<3'2>_{(e_p)}\bigr)_A  (\phi_\star^\lambda) &= 
\begin{tikzpicture}[scale=0.35,baseline=0.25cm]
\node at (0,-2)  [root] (root) {};
	\node at (0,0)  [int] (int) {};
	\node at (0,2.5)  [int] (middle) {};
	\node at (-1.6,2)  [var] (left) {};
	\node at (1.6,2)  [var] (right) {};
	\node at (-1.5,4.5)  [var] (tl) {};
	\node at (1.5,4.5)  [var] (tr) {};
\node [above] at (0,4.8) {${\scriptscriptstyle \Gamma_{\star,p} e_p -e_p}$};	  
	\draw[keps] (left) to  (int);	
	\draw[keps] (right) to (int);
	\draw[keps] (tl) to  (middle);	
	\draw[keps] (tr) to (middle);
	\draw[kernelBig] (0,5) to  (middle);	
	\draw[kernel2] (middle) to (int);
	\draw[testfcn] (int) to  (root);
\end{tikzpicture}\; + 
4\;
\begin{tikzpicture}[scale=0.35,baseline=0.25cm]
\node at (0,-2)  [root] (root) {};
	\node at (0,-0)  [int] (int) {};
	\node at (0,2.5)  [int] (middle) {};
	\node at (-1.6,1.25)  [int] (left) {};
	\node at (1.6,2)  [var] (right) {};
	\node at (1.5,4.5)  [var] (tr) {};
\node [above] at (0,4.8) {${\scriptscriptstyle \Gamma_{\star,p} e_p -e_p}$};	  
	\draw[keps,bend right=40] (left) to  (int);	
	\draw[keps] (right) to (int);
	\draw[keps,bend left=40] (left) to  (middle);	
	\draw[keps] (tr) to (middle);
	\draw[kernelBig] (0,5) to  (middle);	
	\draw[kernel2] (middle) to (int);
	\draw[testfcn] (int) to  (root);
\end{tikzpicture}\; + 
\;
\begin{tikzpicture}[scale=0.35,baseline=0.25cm]
\node at (0,-2)  [root] (root) {};
	\node at (0,0)  [int] (int) {};
	\node at (0,2.5)  [int] (middle) {};
	\node at (-1.6,2)  [var] (left) {};
	\node at (1.6,2)  [var] (right) {};
	\node at (-1,4.6)  [int] (tl) {};
\node [above] at (1.1,4.4) {${\scriptscriptstyle \Gamma_{\star,p} e_p -e_p}$};	  
	\draw[keps] (left) to  (int);	
	\draw[keps] (right) to (int);
	\draw[keps,bend right = 60] (tl) to  (middle);	
	\draw[keps,bend left = 60] (tl) to (middle);
	\draw[kernelBig,bend left = 60] (1,4.6)  to (middle);
	\draw[kernel2] (middle) to (int);
	\draw[testfcn] (int) to  (root);
\end{tikzpicture}
\\ 
&\qquad + 
\;
\begin{tikzpicture}[scale=0.35,baseline=0.25cm]
\node at (0,-2)  [root] (root) {};
	\node at (0,0)  [int] (int) {};
	\node at (1,2.6)  [int] (middle) {};
	\node at (-1,2.6)  [int] (left) {};
	\node at (-0.5,4.5)  [var] (tl) {};
	\node at (1,5)  [var] (tm) {};
	\node [above] at (3,4.2)  {${\scriptscriptstyle \Gamma_{\star,p} e_p -e_p}$};	  
	\draw[keps,bend right = 60] (left) to  (int);	
	\draw[keps,bend left = 60] (left) to (int);
	\draw[kernel2,bend left = 60] (middle) to (int);
	\draw[keps] (tl) to  (middle);	
	\draw[kernelBig] (2.5,4.5)  to (middle);
	\draw[keps] (tm) to  (middle);	
	\draw[testfcn] (int) to  (root);
\end{tikzpicture}\; +
\;
\begin{tikzpicture}[scale=0.35,baseline=0.25cm]
\node at (0,-2)  [root] (root) {};
	\node at (0,0)  [int] (int) {};
	\node at (1,2.6)  [int] (middle) {};
	\node at (-1,2.6)  [int] (left) {};
	\node at (0,4.6)  [int] (tl) {};
\node [above] at (2.1,4.4) {${\scriptscriptstyle \Gamma_{\star,p} e_p -e_p}$};	  
	\draw[keps,bend right = 60] (left) to  (int);	
	\draw[keps,bend left = 60] (left) to (int);
	\draw[kernel2,bend left = 60] (middle) to (int);
	\draw[keps,bend right = 60] (tl) to  (middle);	
	\draw[keps,bend left = 60] (tl) to (middle);
	\draw[kernelBig,bend left = 60] (2,4.6) to (middle);
	\draw[testfcn] (int) to  (root);
\end{tikzpicture}\; +  \;
2\; \begin{tikzpicture}[scale=0.35,baseline=0.25cm]
\node at (0,-2)  [root] (root) {};
	\node at (0,0)  [int] (int) {};
	\node at (0,2.5)  [int] (middle) {};
	\node at (-1.6,1.25)  [int] (left) {};
	\node at (1.6,1.25)  [int] (right) {};
\node [above] at (0,4.8) {${\scriptscriptstyle \Gamma_{\star,p} e_p -e_p}$};	  
	\draw[keps,bend right=40] (left) to  (int);	
	\draw[keps,bend left=40] (left) to  (middle);	
	\draw[keps,bend left=40] (right) to  (int);	
	\draw[keps,bend right=40] (right) to  (middle);	
	\draw[kernelBig] (0,5) to  (middle);	
	\draw[kernel2] (middle) to (int);
	\draw[testfcn] (int) to  (root);
\end{tikzpicture}\;,  
\end{align*}
we find that the latter term is given by
\begin{align}\label{eq:renormalised error bound 3'2 A}
\bigl( \hat{\Pi}^\epsilon_\star  \hat{\Gamma}_{\star,p} \<3'2>_{(e_p)}-& \hat{\Pi}^\epsilon_p \<3'2>_{(e_p)}\bigr)_{\hat{A}}  (\phi_\star^\lambda) = 
\begin{tikzpicture}[scale=0.35,baseline=0.25cm]
\node at (0,-2)  [root] (root) {};
	\node at (0,0)  [int] (int) {};
	\node at (0,2.5)  [int] (middle) {};
	\node at (-1.6,2)  [var] (left) {};
	\node at (1.6,2)  [var] (right) {};
	\node at (-1.5,4.5)  [var] (tl) {};
	\node at (1.5,4.5)  [var] (tr) {};
	\node [above] at (0,4.8) {${\scriptscriptstyle \Gamma_{\star,p} e_p -e_p}$};	  
	\draw[keps] (left) to  (int);	
	\draw[keps] (right) to (int);
	\draw[keps] (tl) to  (middle);	
	\draw[keps] (tr) to (middle);
	\draw[kernelBig] (0,5) to  (middle);	
	\draw[kernel2] (middle) to (int);
	\draw[testfcn] (int) to  (root);
\end{tikzpicture}\; + 
4\;
\begin{tikzpicture}[scale=0.35,baseline=0.25cm]
\node at (0,-2)  [root] (root) {};
	\node at (0,-0)  [int] (int) {};
	\node at (0,2.5)  [int] (middle) {};
	\node at (-1.6,1.25)  [int] (left) {};
	\node at (1.6,2)  [var] (right) {};
	\node at (1.5,4.5)  [var] (tr) {};
\node [above] at (0,4.8) {${\scriptscriptstyle \Gamma_{\star,p} e_p -e_p}$};	  
	\draw[keps,bend right=40] (left) to  (int);	
	\draw[keps] (right) to (int);
	\draw[keps,bend left=40] (left) to  (middle);	
	\draw[keps] (tr) to (middle);
	\draw[kernelBig] (0,5) to  (middle);	
	\draw[kernel2] (middle) to (int);
	\draw[testfcn] (int) to  (root);
\end{tikzpicture}\; + 
\;
\begin{tikzpicture}[scale=0.35,baseline=0.25cm]
\node at (0,-2)  [root] (root) {};
	\node at (0,0)  [int] (int) {};
	\node at (0,2.5)  [int] (middle) {};
	\node at (-1.6,2)  [var] (left) {};
	\node at (1.6,2)  [var] (right) {};
	\node at (-1,4.6)  [bluedot] (tl) {};
\node [above] at (1.1,4.4) {${\scriptscriptstyle \Gamma_{\star,p} e_p -e_p}$};	  
	\draw[keps] (left) to  (int);	
	\draw[keps] (right) to (int);
	\draw[blue, keps,bend right = 60] (tl) to  (middle);	
	\draw[blue, keps,bend left = 60] (tl) to (middle);
	\draw[kernelBig,bend left = 60] (1,4.6)  to (middle);
	\draw[kernel2] (middle) to (int);
	\draw[testfcn] (int) to  (root);
\end{tikzpicture}
\\ 
&\qquad + \;
\begin{tikzpicture}[scale=0.35,baseline=0.25cm]
\node at (0,-2)  [root] (root) {};
	\node at (0,0)  [int] (int) {};
	\node at (1,2.6)  [int] (middle) {};
	\node at (-1,2.6)  [bluedot] (left) {};
	\node at (-0.5,4.5)  [var] (tl) {};
	\node at (1,5)  [var] (tm) {};
		\node [above] at (3,4.2)  {${\scriptscriptstyle \Gamma_{\star,p} e_p -e_p}$};	  
	\draw[blue, keps,bend right = 60] (left) to  (int);	
	\draw[blue, keps,bend left = 60] (left) to (int);
	\draw[kernel2,bend left = 60] (middle) to (int);
	\draw[keps] (tl) to  (middle);	
	\draw[kernelBig] (2.5,4.5)  to (middle);
	\draw[keps] (tm) to  (middle);	
	\draw[testfcn] (int) to  (root);
\end{tikzpicture} + 
\;
\begin{tikzpicture}[scale=0.35,baseline=0.25cm]
\node at (0,-2)  [root] (root) {};
	\node at (0,0)  [int] (int) {};
	\node at (1,2.6)  [int] (middle) {};
	\node at (-1,2.6)  [bluedot] (left) {};
	\node at (0,4.6)  [bluedot] (tl) {};
\node [above] at (2.1,4.4) {${\scriptscriptstyle \Gamma_{\star,p} e_p -e_p}$};	  
	\draw[blue, keps,bend right = 60] (left) to  (int);	
	\draw[blue, keps,bend left = 60] (left) to (int);
	\draw[kernel2,bend left = 60] (middle) to (int);
	\draw[blue, keps,bend right = 60] (tl) to  (middle);	
	\draw[blue, keps,bend left = 60] (tl) to (middle);
	\draw[kernelBig,bend left = 60] (2,4.6) to (middle);
	\draw[testfcn] (int) to  (root);
\end{tikzpicture}\; + 
2\; 
\begin{tikzpicture}[scale=0.35,baseline=0.25cm]
\node at (0,-2)  [root] (root) {};
	\node at (0,0)  [int] (int) {};
	\node at (0,2.5)  [bluedot] (middle) {};
	\node at (-1.6,1.25)  [bluedot] (left) {};
	\node at (1.6,1.25)  [bluedot] (right) {};
\node [above] at (0,4.8) {${\scriptscriptstyle \Gamma_{\star,p} e_p -e_p}$};	  
	\draw[blue, keps,bend right=40] (left) to  (int);	
	\draw[blue, keps,bend left=40] (left) to  (middle);	
	\draw[blue, keps,bend left=40] (right) to  (int);	
	\draw[blue, keps,bend right=40] (right) to  (middle);	
	\draw[kernelBig] (0,5) to  (middle);	
	\draw[blue, kernel] (middle) to (int);
	\draw[testfcn] (int) to  (root);
\end{tikzpicture}\; + \;
2\; \begin{tikzpicture}[scale=0.35,baseline=0.25cm]
\node at (0,-2)  [root] (root) {};
	\node at (0,0)  [int] (int) {};
	\node at (0,2.5)  [int] (middle) {};
	\node at (-1.6,1.25)  [int] (left) {};
	\node at (0,1.25)  [int] (right) {};
\node [above] at (0,4.8) {${\scriptscriptstyle \Gamma_{\star,p} e_p -e_p}$};	  
\node[above] at (2,0) {${\scriptstyle 1}$};	 
	\draw[keps,bend right=40] (left) to  (int);	
	\draw[keps,bend left=40] (left) to  (middle);	
	\draw[keps] (right) to  (int);	
	\draw[keps] (right) to  (middle);	
	\draw[kernelBig] (0,5) to  (middle);	
	\draw[kernel, bend left= 90] (middle) to (root);
	\draw[testfcn] (int) to  (root);
	\draw[gray, kernel, bend right= 70] (root) to (int);
\end{tikzpicture}\;. \nonumber
\end{align}

\begin{lemma}\label{lem:stochastic bounds 3'2}
The conclusion of Proposition~\ref{prop:stochastic estimate for trees in nonlinearity} holds for $\tau_p\in \Func_W(\<3'2>)$.
\end{lemma}
\begin{proof}
The desired bounds on $\bigl( \hat \Pi_\star^{(\epsilon)}\<3'2>_{(e_\star)}\bigr)(\phi_\star^\lambda)$ follow from \eqref{eq:renormalised 3'2} as in the proof of Lemma~\ref{lem:stochastic bounds 32 and 32'}. 
In order to establish the desired bounds on 
$$\bigl( \hat{\Pi}^\epsilon_\star  \hat{\Gamma}_{\star,p} \<3'2>_{(e_p)}- \hat{\Pi}^\epsilon_p \<3'2>_{(e_p)}\bigr) (\phi_\star^\lambda) $$
we estimate \eqref{eq:renormalised error bound 3'2 B} and \eqref{eq:renormalised error bound 3'2 A} separately. 
Indeed, the former is treated easily using the bounds already established in Lemma~\ref{lem:stochstic estimats wick trees} since it only involves a product of a smooth function with a distribution. 
For the latter, note that for all terms in the Wiener Chaos decomposition except
\begin{equation}\label{eq:some more work}
\mathcal{P}^\epsilon (\phi^\lambda_\star):=\begin{tikzpicture}[scale=0.35,baseline=0.25cm]
\node at (0,-2)  [root] (root) {};
	\node at (0,0)  [int] (int) {};
	\node at (0,2.5)  [bluedot] (middle) {};
	\node at (-1.6,1.25)  [bluedot] (left) {};
	\node at (1.6,1.25)  [bluedot] (right) {};
\node [above] at (0,4.8) {${\scriptscriptstyle \Gamma_{\star,p} e_p -e_p}$};	  
	\draw[blue, keps,bend right=40] (left) to  (int);	
	\draw[blue, keps,bend left=40] (left) to  (middle);	
	\draw[blue, keps,bend left=40] (right) to  (int);	
	\draw[blue, keps,bend right=40] (right) to  (middle);	
	\draw[kernelBig] (0,5) to  (middle);	
	\draw[blue, kernel] (middle) to (int);
	\draw[testfcn] (int) to  (root);
\end{tikzpicture}\; + \;
 \begin{tikzpicture}[scale=0.35,baseline=0.25cm]
\node at (0,-2)  [root] (root) {};
	\node at (0,0)  [int] (int) {};
	\node at (0,2.5)  [int] (middle) {};
	\node at (-1.6,1.25)  [int] (left) {};
	\node at (0,1.25)  [int] (right) {};
\node [above] at (0,4.8) {${\scriptscriptstyle \Gamma_{\star,p} e_p -e_p}$};	  
\node[above] at (2,0) {${\scriptstyle 1}$};	 
	\draw[keps,bend right=40] (left) to  (int);	
	\draw[keps,bend left=40] (left) to  (middle);	
	\draw[keps] (right) to  (int);	
	\draw[keps] (right) to  (middle);	
	\draw[kernelBig] (0,5) to  (middle);	
	\draw[kernel, bend left= 90] (middle) to (root);
	\draw[testfcn] (int) to  (root);
	\draw[gray, kernel, bend right= 70] (root) to (int);
\end{tikzpicture}\;
\end{equation}
one can argue as for the terms in \eqref{eq:renormalised 3'2}.
In order to treat this last term, let us introduce the following notation:
\begin{itemize}
\item $f(z)= R\left(\Gamma_{\star,p} e_p -e_p\right)(z)$, which satisfies $|f(z)|\lesssim d_\fraks (z,\star)^{\delta_0}$
\item $K^\epsilon(\star_1,\star_2)= 
\tikzsetnextfilename{nondiv_kernel}
\begin{tikzpicture}[scale=0.35,baseline=-0.1cm]
	\node at (4,0)  [root] (root) {};
	\node at (0,0)  [root] (middle) {};
	
	\node at (2,-1.5)  [int] (left) {};
	\node at (2,1.5)  [int] (right) {};
	\draw[keps,bend right=30] (left) to  (root);	
	\draw[keps,bend left=30] (left) to  (middle);	
	\draw[keps,bend left=30] (right) to  (root);	
	\draw[keps,bend right=30] (right) to  (middle);	
\end{tikzpicture}\;,
$
which satisfies $K^\epsilon(\star_1,\star_2)\lesssim d_\fraks(\star_1,\star_2)^{-2}$,
\end{itemize}
and note that \eqref{eq:some more work} is given by
\begin{align*}
\mathcal{P}^\epsilon (\phi^\lambda_\star)=&\int_{\mathbb{R}\times M}\int_{\mathbb{R}\times M} K^\epsilon(x,z)K(x,z )f(z)\phi^\lambda(x)- K^\epsilon(x,z)K(x,z )\tau_{z,x} f(x)\phi^\lambda_\star(x)dx dz\\
&- \int_{\mathbb{R}\times M}\int_{\mathbb{R}\times M}K^\epsilon (x,z) R\left(Q_{\leq 1} j_\star K(\cdot, z)\right)(x)f(z)\phi_\star^\lambda(x) 
dx dz \\
&+ \int_{\mathbb{R}\times M}\int_{\mathbb{R}\times M}  k_\epsilon (x)f(x)\phi^\lambda_\star (x) dx dz.
\end{align*}
Clearly $\int_{\mathbb{R}\times M}\int_{\mathbb{R}\times M} | k_\epsilon (x)f(x)\phi^\lambda_\star (x) |dx dz\lesssim \lambda^{\delta_0}$, so we turn to the other term. On the one hand, first integrating over $x\in \mathbb{R}\times M$ one finds uniformly in $\epsilon>0$.
\begin{align*} 
&\Bigg| \int_{|x|<\frac{1}{2}d(x,z)} \Big( K^\epsilon(x,z)K(x,z )f(z)\phi^\lambda_\star (x) - K^\epsilon (x,z) R\left(Q_{\leq 1} j_\star K(\cdot, z)\right)(x)f(z)\phi^\lambda_\star(x) \Big)dx dz \\
&\qquad-\int_{|x|<\frac{1}{2}d(x,z)} K^\epsilon(x,z)K(x,z ) \tau_{z,x}f(x)\phi^\lambda_\star(x)
dx dz\Bigg| \\
&\lesssim  \int_{|x|<\frac{1}{2}d(x,z)} |K^\epsilon(x,z)|d(x,\star)^2  \sup_{|y|<\frac{1}{2}d(x,z)} |Q_{\leq 2} j_y K(\cdot, z)|  |f(z)| |\phi^\lambda_\star| dx dz \\
&\qquad +\int_{|x|<\frac{1}{2}d(x,z)} |K^\epsilon(x,z)K(x,z ) \tau_{z,x}f(x)\phi^\lambda_\star(x)| dx dz\\
&\lesssim \lambda^{\delta_0-\kappa}
\end{align*}
Similarly, integrating first over $z$
\begin{align*} 
&\Bigg| \int_{|x|<\frac{1}{2}d(x,z)} \Big( K^\epsilon(x,z)K(x,z )f(z)\phi^\lambda_\star (x) - K^\epsilon (x,z) R\left(Q_{\leq 1} j_\star K(\cdot, z)\right)(x)f(z)\phi^\lambda_\star(x) \Big) dx dz \\
&\qquad -\int_{|x|<\frac{1}{2}d(x,z)} K^\epsilon(x,z)K(x,z ) \tau_{z,x}f(x)\phi^\lambda_\star (x)
dx dz\Bigg| \\
&\leq \Bigg|\int_{|x|>\frac{1}{2}d(x,z)}  K^\epsilon(x,z)K(x,z )\left(f(z)-\tau_{z,x}f(x)\right)\phi^\lambda_\star(x)  dx dz\Bigg|\\
&\qquad + \int_{|x|>\frac{1}{2}d(x,z)}  | K^\epsilon (x,z) R\left(Q_{\leq 1} j_\star K(\cdot, z)\right)(x)f(z)\phi^\lambda_\star (x) 
dx dz\\
&\lesssim \lambda^{\delta_0}
\end{align*}
\end{proof}
\begin{remark}
Note that in order to estimate \eqref{eq:some more work}, we needed to leverage positive and negative renormalisation for the set of edges connecting the same two vertices. Thus, estimating this diagram falls outside the scope of (naively) applying Theorem~\ref{thm:HQ}.
\end{remark}

\begin{remark}
Note that the scalar $\phi^4_3$ equation can essentially be treated as in \cite{Hai14} by working with the choice of admissible realisation such that $R j_p 1= 1$ for all $p\in M$, since then $\hat{\Pi}^\epsilon_q  \hat{\Gamma}_{q,p}^\epsilon \<3'2>_{(j_p 1)}- \hat{\Pi}^\epsilon_p  \<3'2>_{(j_p 1)} =0 \ .$
\end{remark}

%
%
%
%

\subsection{The $\phi^3_4$ equation.}\label{sec:phi34}
Let 
$$
C^{ \<2>, 1}_\epsilon =\int_{\mathbb{R}^3}   \phi^\epsilon (t-s+r)\phi^\epsilon (t-s+r )
\bar{Z}_{s+s'+2\epsilon^2}(0)\ 
 ds  ds' dr \ ,
$$ and 
$$C^{ \<2>,2}_\epsilon =\frac{1}{3}\int_{\mathbb{R}^3}  \phi^\epsilon (t-s+r)\phi^\epsilon (t-s+r )
(s+s'+2\epsilon)  \bar{Z}_{s+s'+2\epsilon^2}(0)\ 
 ds  ds' dr  \ .
$$
In order to describe the remaining counterterms, we use again graphical notation, with 
$\tikzsetnextfilename{arrrr} \tikz[baseline=-0.1cm] \draw[red, kernel] (0,0) to (1,0) \enlarge;$
representing $\bar{Z}$,
$\tikzsetnextfilename{arr} \tikz[baseline=-0.1cm] \draw[red, keps] (0,0) to (1,0) \enlarge; $
representing the kernel 
$\bar{Z}* \bar{\rho}_\epsilon$
for $\bar{\rho}_\epsilon(t,z) = \phi^\epsilon(t) \bar{Z}_{{\epsilon}^2}(x)$.
The special node $\tikzsetnextfilename{rootdotbg} \tikz[baseline=-3] \node[broot] {};$ represents $0\in \mathbb{R}^5$, while 
$\tikzsetnextfilename{arbred} \tikz[baseline=-3] \node [reddot] {};$ represents a variable in $\mathbb{R}^5$ to be integrated out. 
Set
\begin{equation*}
C^{\!\!\<211>}_\epsilon = 2\;
\tikzsetnextfilename{phi34ct'brg}
\begin{tikzpicture}[scale=0.35,baseline=0.8cm]
	\node at (-2,1)  [broot] (left) {};
	\node at (0.0,3)  [reddot] (left1) {};
	\node at (-2,5)  [reddot] (left2) {};
	\node at (-2,3) [reddot] (variable3) {};
	\node at (0,5.7) [reddot] (variable4) {};
	\draw[red, kernel] (left1) to   (left);
	\draw[red, kernel] (left2) to  (left1);
	\draw[red, keps] (variable3) to  (left); 
	\draw[red, keps] (variable4) to   (left1); 
	\draw[red, keps] (variable3) to   (left2); 
	\draw[red, keps] (variable4) to   (left2);
\end{tikzpicture}\; ,
\qquad
C^{\<22j>}_\epsilon = 2\; \tikzsetnextfilename{phi34ct211''brg}
\begin{tikzpicture}[scale=0.35,baseline=.9cm]
	\node at (0.0,2)  [broot] (int) {};
	\node at (-2,3)  [reddot] (left) {};
	\node at (0,5.5)  [reddot] (ttop) {};
	\node at (2,3)  [reddot] (right) {};
	\node at (0,4) [reddot] (top) {};
	\draw[red, kernel] (left) to  (int);
	\draw[red, kernel] (right) to  (int);
	\draw[red, keps] (ttop) to  (right); 
	\draw[red, keps] (top) to  (right); 
	\draw[red, keps] (top) to  (left); 
	\draw[red, keps] (ttop) to (left);
	\end{tikzpicture} \ ,
	\qquad
C^{\<11> }_\epsilon =
\tikzsetnextfilename{phi34ct11'br'g}
\begin{tikzpicture}[scale=0.35,baseline=0.6cm]
	\node at (-2,1)  [broot] (left) {};
	\node at (-2,3)  [reddot] (left1) {};
	\node at (0.0,2) [reddot] (variable) {};
	\draw[red, kernel] (left1) to (left);
	\draw[red, keps] (variable) to (left1); 
	\draw[red, keps] (left) to (variable); 
\end{tikzpicture}\; .
\end{equation*}
Finally, set 
\begin{equation}\label{countertermm}
\bar{C}_\epsilon= (C^{ \<2>, 1}_\epsilon,
C^{ \<2>,2}_\epsilon ,
C^{\!\!\<211>}_\epsilon ,
C^{\<22j>}_\epsilon ,
C^{\<11> }_\epsilon )\in \mathbb{R}^5 \ .
\end{equation}
We shall prove the following theorem.
\begin{theorem}\label{thm:phi3_4}
Let $M$ be a compact four dimensional Riemannian manifold, $\xi$ be space time white noise on $M$ 
and $\xi_\epsilon$ its regularisation as described above. Let $u_0 \in \cC^{\alpha}(M) $ for $\alpha>-1$
and denote by $\Delta$ the Laplace--Beltrami operator and by $s:M\to \mathbb{R}$ the scalar curvature.
Then, for $\bar{C}_\epsilon$ as in \eqref{countertermm} and $u_\epsilon$ solving 
\begin{equation}\label{eq:phi3_4}
\partial_t u_\epsilon + \Delta u_\epsilon = u^2_\epsilon - C^{\<2>, 1}_\epsilon  - C_\epsilon^{\<2>, 2} s -4 C^{\<11>}_\epsilon u_\epsilon  - 4C^{\!\!\<211>}_\epsilon - C^{\<22j>}_\epsilon +  \xi_\epsilon, \qquad u_\epsilon(0) =u_0 \ ,
\end{equation}
there exists a (random) $T>0$ and $u\in  \mathcal{D}'((0,T)\times M)$ such that $u_\epsilon \to u$ as $\epsilon\to 0$ in probability.
\end{theorem}
\begin{remark}
Note that for this equation one expects that the maximal existence time of the solutions is a.s. finite. 
\end{remark}
We proceed as by now standard and define edge types
\begin{itemize}
\item $\mathcal{E}_+= \{\mathcal{I} \}$,
\item $\mathcal{E}_-= \{ \Xi \}$,
\item $\mathcal{E}^+_0= \{\delta^{\mathcal{I}}\}$ with the injection $\iota:\mathcal{E}_+\to  \mathcal{E}^+_0$.
\item Since the equations is scalar valued and the non-linearity parallel, we suppress the 
$\mathcal{E}^{0}_{0}$ type edges.
\end{itemize}
We fix homogeneities $|\mathcal{I}|\in (2-\frac{1}{10}, 2)$, $|\Xi|\in (-3-\frac{1}{10}, -3)$
and the rule characterised by
$${R}(\mathcal{I})= \big\{ [\delta^{\mathcal{I}}]^2,\ [\delta^{\mathcal{I}} \mathcal{I}],\ [\mathcal{I}]^2, \ [\Xi]  \big\}\ .$$
Again suffices to solve the equation in $\mathcal{D}^\gamma$ for some $\gamma>1$ and thus, using again graphical notation where $\Xi= \<X>$, $\mathcal{I}$ is written as $\<I>$  and $\delta^{\mathcal{I}}$ as $\<1'>$,  the relevant trees are
$$\mathfrak{T}^\mathcal{J}:=\left\{ \<1>, \; \<20>, \;\<210>
\;\<1'>, \; \<2'0>\ \right\}$$
and
$${\mathfrak{T}}^r:= 
\left\{\ \<X>, \; \<2>, \; \<21>,  \; \<211>, \; \<22j>,
\; \<2'>, \; \<2''>, \; \<2'1> ,\; \<21'>\ \right\}\ . $$
Counterterms are only needed for negative subtrees with an even number of noises, namely 
\<2>, \<211>, \<22j>, \<11>.
We denote for $$\bar{C}=(C^{( \<2>, 1)}, C^{ \<2>,2}, C^{\<211>}, C^{\<22j>}, C^{\<11>}) \in \mathbb{R}^5$$
by  $g_ {\bar C}\in \mathfrak{G}$ the element such that
$\mathfrak{g}_{\bar C}(\<2>)= C^{( \<2>, 1)} s + C^{ \<2>,2}, $
$\mathfrak{g}_{\bar C}(\tau)= C^\tau$ for $\tau\in \left\{\<211>, \ \<22j>,\ \<11> \right\} $ and $\mathfrak{g}_{\bar C}(\bar \tau)= e(\bar \tau)$ for any other negative subtree $\bar \tau\in \mathfrak{T}_-\setminus \left\{ \<2>,  \<211>, \<22j>,\<11> \right\}$ . We denote by $\mathfrak{G}_-'\subset \mathfrak{G}_-$ the subgroup generated by these $\mathfrak{g}_{\bar C} \in \mathfrak{G}_-$.
Denote by $Z^\epsilon=(\Pi^\epsilon, \Gamma^\epsilon)$ the canonical model for the smooth noise $\xi_\epsilon$, such that the abstract Integration map $\mathcal{I}$ realises $K_t(p,q):=G_t(p,q) \kappa(t)\kappa(d(p,q)/\bar{r})$, where 
$G$ denotes the heat kernel, and such that the admissible realisation has the property that $R(j_\cdot 1)=1$.
We again write  $\hat{Z}^\epsilon:= (Z^\epsilon)^{\mathfrak{g}_{\bar C}}$ and
$\hat{Z}^\epsilon= (\hat{\Pi}^\epsilon, \hat{\Gamma}^\epsilon)$. 
%
%
%

\begin{proof}[Proof of Theorem~\ref{thm:phi3_4}]
In the setting described above, let $\eta, \gamma\in \mathbb{R}$ such that
$-1<\eta<\alpha$ and $\big| |\Xi| + |\mathcal{I}|\big|<\gamma <\delta_0$.
Note that this time the homogeneity of $|\Xi|$ and $| \<2> |$ are 
too low to apply the fixed point theorem~\ref{thm:abstract_fixed point} directly. Fortunately, this can again be resolved in \cite[Sec.~9.4]{Hai14}.
One then finds that there exists a unique abstract (local) solution $\hat{U}^\epsilon\in \mathcal{D}^{\gamma, \eta}(\mathcal{T}^{\mathcal{I}})$ 
of the abstract lift of 
$$\partial_t u-\Delta u= u^2 +\xi$$
for the renormalised models $\hat{Z}^\epsilon$.
Then, by Proposition~\ref{prop:renormalised equation}
$\hat{u}= \mathcal{R} \hat{U}^{\epsilon}$ solves \eqref{eq:phi3_4}.
The proof is thus complete, once the analogue of Proposition~\ref{prop2:convergence of models} for the model $\hat{Z}^\epsilon$ of this section  is shown, which in turn follows Proposition~\ref{prop3:stochastic estimate for trees in nonlinearity} below.
\end{proof}

\subsubsection{Stochastic estimates}
In this section we prove the following proposition.
\begin{prop}\label{prop3:stochastic estimate for trees in nonlinearity}
The conclusion of Proposition~\ref{prop:stochastic estimate for trees in nonlinearity} holds for $\tau_p \in {\mathcal{T}}^r_{\alpha, :}$ whenever $\alpha\leq 0$.
\end{prop}
Throughout this section, we use (essentially) the same graphical as in Section~\ref{sec:stochastic_estimates2}.
Thus, we write
$$
\begin{tikzpicture}[scale=0.35,baseline=0.3cm]
	\node at (0,0.0)  [root] (root) {};
	\node at (0,2.5) [reddot] (top) {};
	\draw[red, keps] (top) to[bend left=60] (root); 
	\draw[red, keps] (top) to[bend right=60] (root); 
\end{tikzpicture}\; =   C^{( \<2>, 1)}_\epsilon + 
C^{ \<2>,2}_\epsilon s(\star)
\qquad \text{and } \qquad
\begin{tikzpicture}[scale=0.35,baseline=0.3cm]
	\node at (0.0,0)  [root] (root) {};
	\node at (0,2.5) [bluedot] (top) {};
	\draw[blue,keps] (top) to[bend left=60] (root); 
	\draw[blue,keps] (top) to[bend right=60] (root); 
\end{tikzpicture}\; = \;
\begin{tikzpicture}[scale=0.35,baseline=0.3cm]
	\node at (0.0,0)  [root] (root) {};
	\node at (0,2.5) [dot] (top) {};
	\draw[keps] (top) to[bend left=60] (root); 
	\draw[keps] (top) to[bend right=60] (root); 
\end{tikzpicture}\; - \;
\begin{tikzpicture}[scale=0.35,baseline=0.3cm]
	\node at (0.0,0)  [root] (root) {};
	\node at (0,2.5) [reddot] (top) {};
	\draw[red, keps] (top) to[bend left=60] (root); 
	\draw[red, keps] (top) to[bend right=60] (root); 
\end{tikzpicture}\;  .
$$

\begin{lemma}
The section $\begin{tikzpicture}[scale=0.35,baseline=0.3cm]
	\node at (0.0,0)  [root] (int) {};
	\node at (0,2.5) [bluedot] (top) {};
	\draw[blue, keps] (top) to[bend left=60] (int); 
	\draw[blue, keps] (top) to[bend right=60] (int); 
\end{tikzpicture}\;$ of $\mathcal{C}^\infty(\RR\times M)$ 
is uniformly bounded for $\epsilon\in (0,1]$ and converges uniformly.
\end{lemma}
\begin{proof}
We proceed as previously and find that
\begin{align*}
\pmb{\Pi}^{(\epsilon)}  \<1>  (t,p)&= \int_\mathbb{R}  \langle \xi , \phi^\epsilon (s- \cdot ) G_{t-s+\epsilon^2}(p, \cdot) \rangle \kappa (t-s) ds + r_\epsilon(t,p) .
\end{align*}
Therefore,
$$
\begin{tikzpicture}[scale=0.35,baseline=0.3cm]
	\node at (0.0,0)  [root] (root) {};
	\node at (0,2.5) [dot] (top) {};
	\draw[keps] (top) to[bend left=60] (root); 
	\draw[keps] (top) to[bend right=60] (root); 
\end{tikzpicture} 
 =  \int_{\mathbb{R}^3} \int_M  \phi^\epsilon (t-s+r)\phi^\epsilon (t-s+r )
G_{s+s'+2\epsilon^2}(p, p) \kappa(s)\kappa (s')
 ds  ds' dr + R_\epsilon \ ,
$$
where $R_\epsilon$ satisfies the conclusion of the lemma.
Thus, the lemma follows from Theorem~\ref{thm:heat_kernel_expansion}.
\end{proof}
The Lemma follows from a simplified ad verbatim adaptation of the proofs of Lemma~\ref{lem:phi43 first counterterm} and~\ref{lem:phi43 second counterterm}.
\begin{lemma}
The smooth functions 
$$
\tikzsetnextfilename{phi34ct'b}
\begin{tikzpicture}[scale=0.35,baseline=0.8cm]
	\node at (-2,1)  [root] (left) {};
	\node at (0.0,3)  [bluedot] (left1) {};
	\node at (-2,5)  [bluedot] (left2) {};
	\node at (-2,3) [bluedot] (variable3) {};
	\node at (0,5.7) [bluedot] (variable4) {};
	\draw[blue, kernel] (left1) to   (left);
	\draw[blue, kernel] (left2) to  (left1);
	\draw[blue, keps] (variable3) to  (left); 
	\draw[blue, keps] (variable4) to   (left1); 
	\draw[blue, keps] (variable3) to   (left2); 
	\draw[blue, keps] (variable4) to   (left2);
\end{tikzpicture}
\; 
:=
\tikzsetnextfilename{phi34ct'bl}
\begin{tikzpicture}[scale=0.35,baseline=0.8cm]
	\node at (-2,1)  [root] (left) {};
	\node at (0.0,3)  [dot] (left1) {};
	\node at (-2,5)  [dot] (left2) {};
	\node at (-2,3) [dot] (variable3) {};
	\node at (0,5.7) [dot] (variable4) {};
	\draw[kernel] (left1) to   (left);
	\draw[kernel] (left2) to  (left1);
	\draw[keps] (variable3) to  (left); 
	\draw[keps] (variable4) to   (left1); 
	\draw[keps] (variable3) to   (left2); 
	\draw[keps] (variable4) to   (left2);
\end{tikzpicture}\;
 -C^{\<211>}_\epsilon\ ,
\qquad
\tikzsetnextfilename{phi34ct211''b}
\begin{tikzpicture}[scale=0.35,baseline=.9cm]
	\node at (0,2)  [root] (int) {};
	\node at (-2,3)  [bluedot] (left) {};
	\node at (0.0,5.5)  [bluedot] (ttop) {};
	\node at (2,3)  [bluedot] (right) {};
	\node at (0,4) [bluedot] (top) {};
	\draw[blue, kernel] (left) to  (int);
	\draw[blue, kernel] (right) to  (int);
	\draw[blue, keps] (ttop) to  (right); 
	\draw[blue, keps] (top) to  (right); 
	\draw[blue, keps] (top) to  (left); 
	\draw[blue, keps] (ttop) to (left);
	\end{tikzpicture} \;
	:=
	\tikzsetnextfilename{phi34ct211''bl}
\begin{tikzpicture}[scale=0.35,baseline=.9cm]
	\node at (0.0,2)  [root] (int) {};
	\node at (-2,3)  [dot] (left) {};
	\node at (0,5.5)  [dot] (ttop) {};
	\node at (2,3)  [dot] (right) {};
	\node at (0,4) [dot] (top) {};
	\draw[kernel] (left) to  (int);
	\draw[kernel] (right) to  (int);
	\draw[keps] (ttop) to  (right); 
	\draw[keps] (top) to  (right); 
	\draw[keps] (top) to  (left); 
	\draw[keps] (ttop) to (left);
	\end{tikzpicture} \;
	-C^{\<22j>}_\epsilon 
$$
are uniformly bounded for $\epsilon\in (0,1]$ and converge uniformly as $\epsilon\to 0$. Furthermore, denoting the variable corresponding to the higher green dot 
by $\star_1$ and the one corresponding to the lower one by $\star_2$,
one has 
$$
\tikzsetnextfilename{phi34ct11''bl}
\begin{tikzpicture}[scale=0.35,baseline=0.6cm]
	\node at (-2,1)  [root] (left) {};
	\node at (-2,3)  [root] (left1) {};
	\node at (0.0,2) [bluedot] (variable) {};
	\draw[blue, kernel] (left1) to (left);
	\draw[blue, keps] (variable) to (left1); 
	\draw[blue, keps] (left) to (variable); 
\end{tikzpicture}\; 
:= 
\tikzsetnextfilename{phi34ct11''bldef}
\begin{tikzpicture}[scale=0.35,baseline=0.6cm]
	\node at (-2,1)  [root] (left) {};
	\node at (-2,3)  [root] (left1) {};
	\node at (0.0,2) [dot] (variable) {};
	\draw[kernel] (left1) to (left);
	\draw[keps] (variable) to (left1); 
	\draw[keps] (left) to (variable); 
\end{tikzpicture}\; - C^{\<11> }_\epsilon \delta_{\star_2} (\star_1) =
\mathcal{R}Q_\epsilon (\star_1, \star_2) -k_\epsilon(\star_2) \delta_{\star_2} (\star_1) 
$$
where the $k_\epsilon\in \mathcal{C}^\infty (\RR\times M)$ is uniformly bounded for $\epsilon \in (0,1]$ and converges uniformly as $\epsilon\to 0$.
\end{lemma}

\begin{proof}[Proof of Proposition~\ref{prop3:stochastic estimate for trees in nonlinearity}]
We proceed exactly as in the proof of Propositions~\ref{prop:stochastic estimate for trees in nonlinearity}~\&~\ref{prop2:stochastic estimate for trees in nonlinearity}. The term $\bigl(\hat \Pi_\star^{(\epsilon)}\<2>\bigr)(\phi_\star^\lambda)$ can be treated exactly as in the estimate of the analogue term in Proposition~\ref{prop:stochastic estimate for trees in nonlinearity}. 
We next study the Wiener Chaos expansion of
$$\bigl( \hat{\Pi}_\star^\epsilon \<21>\bigr)(\phi_\star^\lambda), \qquad \bigl( \hat{\Pi}_\star^\epsilon \<22j>\bigr)(\phi_\star^\lambda), \qquad
	\bigl( \hat{\Pi}^\epsilon_\star\<211>\bigr)(\phi_\star^\lambda) \ .$$ 
Since
 $$
\bigl( \Pi_{\star}^\epsilon \<21>\bigr) (\phi_\star^\lambda)
=
\begin{tikzpicture}[scale=0.35,baseline=1.1cm]
	\node at (0.8,1.5)  [root] (root) {};
	\node at (-1.3,3)  [dot] (left1) {};
	\node at (-1.3,5)  [dot] (left2) {};
	\node at (0.8,3) [var] (variable2) {};
	\node at (0.8,4.3) [var] (variable3) {};
	\node at (0.8,5.7) [var] (variable4) {};
	\draw[testfcn] (left1) to (root);
	\draw[kernel] (left2) to  (left1);
	\draw[keps] (variable2) to  (left1); 
	\draw[keps] (variable3) to  (left2); 
	\draw[keps] (variable4) to (left2);
\end{tikzpicture}
+
\begin{tikzpicture}[scale=0.35,baseline=1.1cm]
	\node at (0.8,1.5)  [root] (root) {};
	\node at (-1.3,3)  [dot] (left1) {};
	\node at (-1.3,5)  [dot] (left2) {};
	\node at (0.8,3) [var] (variable2) {};
	\node at (0.8,5) [dot] (variable4) {};
	\draw[testfcn] (left1) to (root);
	\draw[kernel] (left2) to  (left1);
	\draw[keps] (variable2) to  (left1); 
	\draw[keps,bend left=40] (variable4) to  (left2); 
	\draw[keps,bend right=40] (variable4) to (left2);
\end{tikzpicture}
+ 2\;
\begin{tikzpicture}[scale=0.35,baseline=1.1cm]
	\node at (0.8,1.5)  [root] (root) {};
	\node at (-1.3,3)  [dot] (left1) {};
	\node at (-1.3,5)  [dot] (left2) {};
	\node at (0.8,4) [dot] (variable2) {};
	\node at (0.8,5.7) [var] (variable3) {};	
	\draw[testfcn] (left1) to (root);
	\draw[kernel] (left2) to  (left1);
	\draw[keps] (variable2) to  (left1); 
	\draw[keps] (variable3) to  (left2); 
	\draw[keps] (variable2) to (left2); 
 \end{tikzpicture}\;,
 $$
we find that
  $$
\bigl( \hat{\Pi}_\star^\epsilon \<21>\bigr)(\phi_\star^\lambda)
=
\begin{tikzpicture}[scale=0.35,baseline=1.1cm]
	\node at (0.8,1.5)  [root] (root) {};
	\node at (-1.3,3)  [dot] (left1) {};
	\node at (-1.3,5)  [dot] (left2) {};
	\node at (0.8,3) [var] (variable2) {};
	\node at (0.8,4.3) [var] (variable3) {};
	\node at (0.8,5.7) [var] (variable4) {};	
	\draw[testfcn] (left1) to (root);
	\draw[kernel] (left2) to  (left1);
	\draw[keps] (variable2) to  (left1); 
	\draw[keps] (variable3) to  (left2); 
	\draw[keps] (variable4) to (left2);
\end{tikzpicture}
+
\begin{tikzpicture}[scale=0.35,baseline=1.1cm]
	\node at (0.8,1.5)  [root] (root) {};
	\node at (-1.3,3)  [dot] (left1) {};
	\node at (-1.3,5)  [bluedot] (left2) {};
	\node at (0.8,3) [var] (variable2) {};
	\node at (0.8,5) [bluedot] (variable4) {};
	\draw[testfcn] (left1) to (root);
	\draw[kernel] (left2) to  (left1);
	\draw[keps] (variable2) to  (left1); 
	\draw[blue, keps,bend left=40] (variable4) to  (left2); 
	\draw[blue, keps,bend right=40] (variable4) to (left2);
\end{tikzpicture}
+ 2\;
\begin{tikzpicture}[scale=0.35,baseline=1.1cm]
	\node at (0.8,1.5)  [root] (root) {};
	\node at (-1.3,3)  [dot] (left1) {};
	\node at (-1.3,5)  [dot] (left2) {};
	\node at (0.8,4) [bluedot] (variable2) {};
	\node at (0.8,5.7) [var] (variable3) {};	
	\draw[testfcn] (left1) to (root);
	\draw[blue, kernel] (left2) to  (left1);
	\draw[blue, keps] (variable2) to  (left1); 
	\draw[keps] (variable3) to  (left2); 
	\draw[blue, keps] (variable2) to (left2); 
 \end{tikzpicture}\;.
 $$
 We observe that all three trees can by bounded straightforwardly using the estimates of Section~\ref{sec:Kernels with prescribed singularities}.
Similarly one finds that
\begin{align*}
 \bigl( \Pi_\star^\epsilon \<22j>\bigr)(\phi_\star^\lambda)
=&
\begin{tikzpicture}[scale=0.35,baseline=.9cm]
	\node at (0,0)  [root] (root) {};
	\node at (0,2)  [dot] (int) {};
	\node at (-2,3)  [dot] (left) {};
	\node at (-1,5) [var] (variable3) {};
	\node at (-3,5) [var] (variable4) {};
	\node at (2,3)  [dot] (right) {};
	\node at (1,5) [var] (variable1) {};
	\node at (3,5) [var] (variable2) {};
	\draw[testfcn] (int) to (root);
	\draw[kernel] (left) to  (int);
	\draw[kernel] (right) to  (int);
	\draw[keps] (variable2) to  (right); 
	\draw[keps] (variable1) to  (right); 
	\draw[keps] (variable3) to  (left); 
	\draw[keps] (variable4) to (left);
\end{tikzpicture}
+ 2
\begin{tikzpicture}[scale=0.35,baseline=.9cm]
	\node at (0,0)  [root] (root) {};
	\node at (0,2)  [dot] (int) {};
	\node at (-2,3)  [dot] (left) {};
	\node at (-2,5) [dot] (variable3) {};
	\node at (2,3)  [dot] (right) {};
	\node at (1,5) [var] (variable1) {};
	\node at (3,5) [var] (variable2) {};
	\draw[testfcn] (int) to (root);
	\draw[kernel] (left) to  (int);
	\draw[kernel] (right) to  (int);
	\draw[keps] (variable2) to  (right); 
	\draw[keps] (variable1) to  (right); 
	\draw[keps, bend left=40] (variable3) to  (left); 
	\draw[keps, bend right=40] (variable3) to (left);
\end{tikzpicture}\\
&
+ 4\;
\begin{tikzpicture}[scale=0.35,baseline=.9cm]
	\node at (0,0)  [root] (root) {};
	\node at (0,2)  [dot] (int) {};
	\node at (-2,3)  [dot] (left) {};
	\node at (-3,5) [var] (variable4) {};
	\node at (2,3)  [dot] (right) {};
	\node at (0,4) [dot] (top) {};
	\node at (3,5) [var] (variable2) {};
	\draw[testfcn] (int) to (root);
	\draw[kernel] (left) to  (int);
	\draw[kernel] (right) to  (int);
	\draw[keps] (variable2) to  (right); 
	\draw[keps] (top) to  (right); 
	\draw[keps] (top) to  (left); 
	\draw[keps] (variable4) to (left);
	\end{tikzpicture}
	+ \begin{tikzpicture}[scale=0.35,baseline=.9cm]
	\node at (0,0)  [root] (root) {};
	\node at (0,2)  [dot] (int) {};
	\node at (-2,3)  [dot] (left) {};
	\node at (-2,5) [dot] (variable3) {};
	\node at (2,3)  [dot] (right) {};
	\node at (2,5) [dot] (variable1) {};
	\draw[testfcn] (int) to (root);
	\draw[kernel] (left) to  (int);
	\draw[kernel] (right) to  (int);
	\draw[keps, bend left=40] (variable1) to  (right); 
	\draw[keps, bend right=40] (variable1) to  (right); 
	\draw[keps, bend left=40] (variable3) to  (left); 
	\draw[keps, bend right=40] (variable3) to (left);
\end{tikzpicture} 
+2 
\begin{tikzpicture}[scale=0.35,baseline=.9cm]
	\node at (0,0)  [root] (root) {};
	\node at (0,2)  [dot] (int) {};
	\node at (-2,3)  [dot] (left) {};
	\node at (0,5.5)  [dot] (ttop) {};
	\node at (2,3)  [dot] (right) {};
	\node at (0,4) [dot] (top) {};
	\draw[testfcn] (int) to (root);
	\draw[kernel] (left) to  (int);
	\draw[kernel] (right) to  (int);
	\draw[keps] (ttop) to  (right); 
	\draw[keps] (top) to  (right); 
	\draw[keps] (top) to  (left); 
	\draw[keps] (ttop) to (left);
	\end{tikzpicture}
	\end{align*} 
	and thus
	
\begin{align*}
	\bigl( \hat{\Pi}_\star^\epsilon \<22j>\bigr)(\phi_\star^\lambda)
= &
\begin{tikzpicture}[scale=0.35,baseline=.9cm]
	\node at (0,0.0)  [root] (root) {};
	\node at (0,2)  [dot] (int) {};
	\node at (-2,3)  [dot] (left) {};
	\node at (-1,5) [var] (variable3) {};
	\node at (-3,5) [var] (variable4) {};
	\node at (2,3)  [dot] (right) {};
	\node at (1,5) [var] (variable1) {};
	\node at (3,5) [var] (variable2) {};
	\draw[testfcn] (int) to (root);
	\draw[kernel] (left) to  (int);
	\draw[kernel] (right) to  (int);
	\draw[keps] (variable2) to  (right); 
	\draw[keps] (variable1) to  (right); 
	\draw[keps] (variable3) to  (left); 
	\draw[keps] (variable4) to (left);
\end{tikzpicture}
+ 2
\begin{tikzpicture}[scale=0.35,baseline=.9cm]
	\node at (0,0.0)  [root] (root) {};
	\node at (0,2)  [dot] (int) {};
	\node at (-2,3)  [dot] (left) {};
	\node at (-2,5) [bluedot] (variable3) {};
	\node at (2,3)  [dot] (right) {};
	\node at (1,5) [var] (variable1) {};
	\node at (3,5) [var] (variable2) {};
	\draw[testfcn] (int) to (root);
	\draw[kernel] (left) to  (int);
	\draw[kernel] (right) to  (int);
	\draw[keps] (variable2) to  (right); 
	\draw[keps] (variable1) to  (right); 
	\draw[blue, keps, bend left=40] (variable3) to  (left); 
	\draw[blue, keps, bend right=40] (variable3) to (left);
\end{tikzpicture}\\
&
+ 4\;
\begin{tikzpicture}[scale=0.35,baseline=.9cm]
	\node at (0,0.0)  [root] (root) {};
	\node at (0,2)  [dot] (int) {};
	\node at (-2,3)  [dot] (left) {};
	\node at (-3,5) [var] (variable4) {};
	\node at (2,3)  [dot] (right) {};
	\node at (0,4) [dot] (top) {};
	\node at (3,5) [var] (variable2) {};
	\draw[testfcn] (int) to (root);
	\draw[kernel] (left) to  (int);
	\draw[kernel] (right) to  (int);
	\draw[keps] (variable2) to  (right); 
	\draw[keps] (top) to  (right); 
	\draw[keps] (top) to  (left); 
	\draw[keps] (variable4) to (left);
	\end{tikzpicture}
	+ \begin{tikzpicture}[scale=0.35,baseline=.9cm]
	\node at (0,0.0)  [root] (root) {};
	\node at (0,2)  [dot] (int) {};
	\node at (-2,3)  [dot] (left) {};
	\node at (-2,5) [bluedot] (variable3) {};
	\node at (2,3)  [dot] (right) {};
	\node at (2,5) [bluedot] (variable1) {};
	\draw[testfcn] (int) to (root);
	\draw[kernel] (left) to  (int);
	\draw[kernel] (right) to  (int);
	\draw[blue, keps, bend left=40] (variable1) to  (right); 
	\draw[blue, keps, bend right=40] (variable1) to  (right); 
	\draw[blue, keps, bend left=40] (variable3) to  (left); 
	\draw[blue, keps, bend right=40] (variable3) to (left);
\end{tikzpicture} 
+2 \;
\begin{tikzpicture}[scale=0.35,baseline=.9cm]
	\node at (0.0,0)  [root] (root) {};
	\node at (0,2)  [bluedot] (int) {};
	\node at (-2,3)  [bluedot] (left) {};
	\node at (0,5.5)  [bluedot] (ttop) {};
	\node at (2,3)  [bluedot] (right) {};
	\node at (0,4) [bluedot] (top) {};
	\draw[testfcn] (int) to (root);
	\draw[blue, kernel] (left) to  (int);
	\draw[blue, kernel] (right) to  (int);
	\draw[blue, keps] (ttop) to  (right); 
	\draw[blue, keps] (top) to  (right); 
	\draw[blue, keps] (top) to  (left); 
	\draw[blue, keps] (ttop) to (left);
	\end{tikzpicture} \ .
\end{align*}	
In order to estimate the term
\begin{equation}\label{eq:for HQ-1}
\tikzsetnextfilename{phi34a1}
\begin{tikzpicture}[scale=0.35,baseline=.9cm]
	\node at (0,0)  [root] (root) {};
	\node at (0,2)  [dot] (int) {};
	\node at (-2,3)  [dot] (left) {};
	\node at (-3,5) [var] (variable4) {};
	\node at (2,3)  [dot] (right) {};
	\node at (0,4) [dot] (top) {};
	\node at (3,5) [var] (variable2) {};
	\draw[testfcn] (int) to (root);
	\draw[kernel] (left) to  (int);
	\draw[kernel] (right) to  (int);
	\draw[keps] (variable2) to  (right); 
	\draw[keps] (top) to  (right); 
	\draw[keps] (top) to  (left); 
	\draw[keps] (variable4) to (left);
	\end{tikzpicture}
\end{equation}
we apply Theorem~\ref{thm:HQ}, while the other terms are easily bounded using the estimates of Section~\ref{sec:Kernels with prescribed singularities}. Next note that
$$
	\bigl( \Pi_\star^\epsilon\<211>\bigr)(\phi_\star^\lambda)
=
\begin{tikzpicture}[scale=0.35,baseline=0.8cm]
	\node at (0.0,-0.8)  [root] (root) {};
	\node at (-2,1)  [dot] (left) {};
	\node at (-2,3)  [dot] (left1) {};
	\node at (-2,5)  [dot] (left2) {};
	\node at (0,1) [var] (variable1) {};
	\node at (0,3) [var] (variable2) {};
	\node at (0,4.3) [var] (variable3) {};
	\node at (0,5.7) [var] (variable4) {};
	
	\draw[testfcn] (left) to (root);

	\draw[kernel1] (left1) to   (left);
	\draw[kernel] (left2) to  (left1);
	\draw[keps] (variable2) to  (left1); 
	\draw[keps] (variable1) to   (left); 
	\draw[keps] (variable3) to   (left2); 
	\draw[keps] (variable4) to   (left2);
\end{tikzpicture}
+\; 
\begin{tikzpicture}[scale=0.35,baseline=0.8cm]
	\node at (0.0,-0.8)  [root] (root) {};
	\node at (-2,1)  [dot] (left) {};
	\node at (-2,3)  [dot] (left1) {};
	\node at (-2,5)  [dot] (left2) {};
	\node at (0,1) [var] (variable1) {};
	\node at (0,3) [var] (variable2) {};
	\node at (0,5) [dot] (variable3) {};
	
	\draw[testfcn] (left) to (root);

	\draw[kernel1] (left1) to   (left);
	\draw[kernel] (left2) to  (left1);
	\draw[keps] (variable2) to  (left1); 
	\draw[keps] (variable1) to   (left); 
	\draw[keps, bend left=40] (variable3) to   (left2); 
	\draw[keps, bend right=40] (variable3) to   (left2);
\end{tikzpicture}
\;+\;
\begin{tikzpicture}[scale=0.35,baseline=0.8cm]
	\node at (0.0,-0.8)  [root] (root) {};
	\node at (-2,1)  [dot] (left) {};
	\node at (-2,3)  [dot] (left1) {};
	\node at (-2,5)  [dot] (left2) {};
	\node at (0,2) [dot] (variable1) {};
	\node at (0,4.3) [var] (variable3) {};
	\node at (0,5.7) [var] (variable4) {};
	
	\draw[testfcn] (left) to (root);

	\draw[kernel1] (left1) to   (left);
	\draw[kernel] (left2) to  (left1);
	\draw[keps] (variable1) to  (left1); 
	\draw[keps] (variable1) to   (left); 
	\draw[keps] (variable3) to   (left2); 
	\draw[keps] (variable4) to   (left2);
\end{tikzpicture}
\;+2\;
\begin{tikzpicture}[scale=0.35,baseline=0.8cm]
	\node at (0.0,-0.8)  [root] (root) {};
	\node at (-2,1)  [dot] (left) {};
	\node at (-2,3)  [dot] (left1) {};
	\node at (-2,5)  [dot] (left2) {};
	\node at (0,1) [var] (variable1) {};
	\node at (0,4) [dot] (variable3) {};
	\node at (0,5.7) [var] (variable4) {};
	
	\draw[testfcn] (left) to (root);

	\draw[kernel1] (left1) to   (left);
	\draw[kernel] (left2) to  (left1);
	\draw[keps] (variable3) to  (left1); 
	\draw[keps] (variable1) to   (left); 
	\draw[keps] (variable3) to   (left2); 
	\draw[keps] (variable4) to   (left2);
\end{tikzpicture}
\;+2\;
\begin{tikzpicture}[scale=0.35,baseline=0.8cm]
	\node at (0.0,-0.8)  [root] (root) {};
	\node at (-2,1)  [dot] (left) {};
	\node at (0,3)  [dot] (left1) {};
	\node at (-2,5)  [dot] (left2) {};
	\node at (0,5) [var] (variable1) {};
	\node at (-2,3) [dot] (variable3) {};
	\node at (-0.5,6) [var] (variable4) {};
	
	\draw[testfcn] (left) to (root);

	\draw[kernel1] (left1) to   (left);
	\draw[kernel] (left2) to  (left1);
	\draw[keps] (variable3) to  (left); 
	\draw[keps] (variable1) to   (left1); 
	\draw[keps] (variable3) to   (left2); 
	\draw[keps] (variable4) to   (left2);
\end{tikzpicture}
\;+2\;
\begin{tikzpicture}[scale=0.35,baseline=0.8cm]
	\node at (0.0,-0.8)  [root] (root) {};
	\node at (-2,1)  [dot] (left) {};
	\node at (0,3)  [dot] (left1) {};
	\node at (-2,5)  [dot] (left2) {};
	\node at (-2,3) [dot] (variable3) {};
	\node at (0,5.7) [dot] (variable4) {};
	
	\draw[testfcn] (left) to (root);

	\draw[kernel1] (left1) to   (left);
	\draw[kernel] (left2) to  (left1);
	\draw[keps] (variable3) to  (left); 
	\draw[keps] (variable4) to   (left1); 
	\draw[keps] (variable3) to   (left2); 
	\draw[keps] (variable4) to   (left2);
\end{tikzpicture}
+\;
\begin{tikzpicture}[scale=0.35,baseline=0.8cm]
	\node at (0.0,-0.8)  [root] (root) {};
	\node at (-2,1)  [dot] (left) {};
	\node at (-2,3)  [dot] (left1) {};
	\node at (-2,5)  [dot] (left2) {};
	\node at (0,2) [dot] (variable1) {};
	\node at (0,5) [dot] (variable3) {};
	
	\draw[testfcn] (left) to (root);

	\draw[kernel1] (left1) to   (left);
	\draw[kernel] (left2) to  (left1);
	\draw[keps] (variable1) to  (left1); 
	\draw[keps] (variable1) to   (left); 
	\draw[keps, bend left=40] (variable3) to   (left2); 
	\draw[keps, bend right=40] (variable3) to   (left2);
\end{tikzpicture}
$$
and thus
\begin{align*}
	\bigl( \hat{\Pi}^\epsilon_\star\<211>\bigr)(\phi_\star^\lambda)
=&
\begin{tikzpicture}[scale=0.35,baseline=0.8cm]
	\node at (0.0,-0.8)  [root] (root) {};
	\node at (-2,1)  [dot] (left) {};
	\node at (-2,3)  [dot] (left1) {};
	\node at (-2,5)  [dot] (left2) {};
	\node at (0,1) [var] (variable1) {};
	\node at (0,3) [var] (variable2) {};
	\node at (0,4.3) [var] (variable3) {};
	\node at (0,5.7) [var] (variable4) {};
	\draw[testfcn] (left) to (root);
	\draw[kernel1] (left1) to   (left);
	\draw[kernel] (left2) to  (left1);
	\draw[keps] (variable2) to  (left1); 
	\draw[keps] (variable1) to   (left); 
	\draw[keps] (variable3) to   (left2); 
	\draw[keps] (variable4) to   (left2);
\end{tikzpicture}
+\; 
\begin{tikzpicture}[scale=0.35,baseline=0.8cm]
	\node at (0.0,-0.8)  [root] (root) {};
	\node at (-2,1)  [dot] (left) {};
	\node at (-2,3)  [dot] (left1) {};
	\node at (-2,5)  [dot] (left2) {};
	\node at (0,1) [var] (variable1) {};
	\node at (0,3) [var] (variable2) {};
	\node at (0,5) [bluedot] (variable3) {};
	\draw[testfcn] (left) to (root);
	\draw[kernel1] (left1) to   (left);
	\draw[kernel] (left2) to  (left1);
	\draw[keps] (variable2) to  (left1); 
	\draw[keps] (variable1) to   (left); 
	\draw[blue, keps, bend left=40] (variable3) to   (left2); 
	\draw[blue, keps, bend right=40] (variable3) to   (left2);
\end{tikzpicture}
\;+\;
\begin{tikzpicture}[scale=0.35,baseline=0.8cm]
	\node at (0.0,-0.8)  [root] (root) {};
	\node at (-2,1)  [dot] (left) {};
	\node at (-2,3)  [dot] (left1) {};
	\node at (-2,5)  [dot] (left2) {};
	\node at (0,2) [bluedot] (variable1) {};
	\node at (0,4.3) [var] (variable3) {};
	\node at (0,5.7) [var] (variable4) {};
	\draw[testfcn] (left) to (root);
	\draw[blue, kernel] (left1) to   (left);
	\draw[kernel] (left2) to  (left1);
	\draw[blue, keps] (variable1) to  (left1); 
	\draw[blue, keps] (variable1) to   (left); 
	\draw[keps] (variable3) to   (left2); 
	\draw[keps] (variable4) to   (left2);
\end{tikzpicture}
\;+\;
\begin{tikzpicture}[scale=0.35,baseline=0.8cm]
	\node at (0.0,-0.8)  [root] (root) {};
	\node at (-2,1)  [dot] (left) {};
	\node at (-2,3)  [dot] (left1) {};
	\node at (-2,5)  [dot] (left2) {};
	\node at (0,2) [dot] (variable1) {};
	\node at (0,4.3) [var] (variable3) {};
	\node at (0,5.7) [var] (variable4) {};
	\draw[testfcn] (left) to (root);
	\draw[kernel, bend right=60] (left1) to   (root);
	\draw[kernel] (left2) to  (left1);
	\draw[keps] (variable1) to  (left1); 
	\draw[keps] (variable1) to   (left); 
	\draw[keps] (variable3) to   (left2); 
	\draw[keps] (variable4) to   (left2);
\end{tikzpicture}
\;+2\;
\begin{tikzpicture}[scale=0.35,baseline=0.8cm]
	\node at (0.0,-0.8)  [root] (root) {};
	\node at (-2,1)  [dot] (left) {};
	\node at (-2,3)  [dot] (left1) {};
	\node at (-2,5)  [bluedot] (left2) {};
	\node at (0,1) [var] (variable1) {};
	\node at (0,4) [bluedot] (variable3) {};
	\node at (0,5.7) [var] (variable4) {};
	\draw[testfcn] (left) to (root);
	\draw[kernel1] (left1) to   (left);
	\draw[kernel] (left2) to  (left1);
	\draw[blue, keps] (variable3) to  (left1); 
	\draw[keps] (variable1) to   (left); 
	\draw[blue, keps] (variable3) to   (left2); 
	\draw[keps] (variable4) to   (left2);
\end{tikzpicture}
\;+2\;
\begin{tikzpicture}[scale=0.35,baseline=0.8cm]
	\node at (0.0,-0.8)  [root] (root) {};
	\node at (-2,1)  [dot] (left) {};
	\node at (0,3)  [dot] (left1) {};
	\node at (-2,5)  [dot] (left2) {};
	\node at (0,5) [var] (variable1) {};
	\node at (-2,3) [dot] (variable3) {};
	\node at (-0.5,6) [var] (variable4) {};
	\draw[testfcn] (left) to (root);
	\draw[kernel1] (left1) to   (left);
	\draw[kernel] (left2) to  (left1);
	\draw[keps] (variable3) to  (left); 
	\draw[keps] (variable1) to   (left1); 
	\draw[keps] (variable3) to   (left2); 
	\draw[keps] (variable4) to   (left2);
\end{tikzpicture}
\;\\
&+2\;
\begin{tikzpicture}[scale=0.35,baseline=0.8cm]
	\node at (0.0,-0.8)  [root] (root) {};
	\node at (-2,1)  [dot] (left) {};
	\node at (0,3)  [bluedot] (left1) {};
	\node at (-2,5)  [bluedot] (left2) {};
	\node at (-2,3) [bluedot] (variable3) {};
	\node at (0,5.7) [bluedot] (variable4) {};
	\draw[testfcn] (left) to (root);
	\draw[blue, kernel] (left1) to   (left);
	\draw[blue, kernel] (left2) to  (left1);
	\draw[blue, keps] (variable3) to  (left); 
	\draw[blue, keps] (variable4) to   (left1); 
	\draw[blue, keps] (variable3) to   (left2); 
	\draw[blue, keps] (variable4) to   (left2);
\end{tikzpicture}
\;+2\;
\begin{tikzpicture}[scale=0.35,baseline=0.8cm]
	\node at (0.0,-0.8)  [root] (root) {};
	\node at (-2,1)  [dot] (left) {};
	\node at (0,3)  [dot] (left1) {};
	\node at (-2,5)  [dot] (left2) {};
	\node at (-2,3) [dot] (variable3) {};
	\node at (0,5.7) [dot] (variable4) {};
	\draw[testfcn] (left) to (root);
	\draw[kernel] (left1) to   (root);
	\draw[kernel] (left2) to  (left1);
	\draw[keps] (variable3) to  (left); 
	\draw[keps] (variable4) to   (left1); 
	\draw[keps] (variable3) to   (left2); 
	\draw[keps] (variable4) to   (left2);
\end{tikzpicture}
+\;
\begin{tikzpicture}[scale=0.35,baseline=0.8cm]
	\node at (0.0,-0.8)  [root] (root) {};
	\node at (-2,1)  [dot] (left) {};
	\node at (-2,3)  [bluedot] (left1) {};
	\node at (-2,5)  [dot] (left2) {};
	\node at (0,2) [bluedot] (variable1) {};
	\node at (0,5) [bluedot] (variable3) {};
	\draw[testfcn] (left) to (root);	
	\draw[blue, kernel] (left1) to   (left);
	\draw[kernel] (left2) to  (left1);
	\draw[blue, keps] (variable1) to  (left1); 
	\draw[blue, keps] (variable1) to   (left); 
	\draw[blue, keps, bend left=40] (variable3) to   (left2); 
	\draw[blue, keps, bend right=40] (variable3) to   (left2);
\end{tikzpicture} \;
+\;
\begin{tikzpicture}[scale=0.35,baseline=0.8cm]
	\node at (0.0,-0.8)  [root] (root) {};
	\node at (-2,1)  [dot] (left) {};
	\node at (-2,3)  [dot] (left1) {};
	\node at (-2,5)  [dot] (left2) {};
	\node at (0,2) [dot] (variable1) {};
	\node at (0,5) [bluedot] (variable3) {};
	\draw[testfcn] (left) to (root);
	\draw[kernel , bend right=60] (left1) to   (root);
	\draw[kernel] (left2) to  (left1);
	\draw[keps] (variable1) to  (left1); 
	\draw[keps] (variable1) to   (left); 
	\draw[blue, keps, bend left=40] (variable3) to   (left2); 
	\draw[blue, keps, bend right=40] (variable3) to   (left2);
\end{tikzpicture} \;.
\end{align*}
Here we apply Theorem~\ref{thm:HQ} to bound the terms
\begin{equation}\label{eq:for HQ-2}
\tikzsetnextfilename{phi34a2}
\begin{tikzpicture}[scale=0.35,baseline=0.8cm]
	\node at (0.0,-0.8)  [root] (root) {};
	\node at (-2,1)  [dot] (left) {};
	\node at (-2,3)  [dot] (left1) {};
	\node at (-2,5)  [dot] (left2) {};
	\node at (0,2) [dot] (variable1) {};
	\node at (0,4.3) [var] (variable3) {};
	\node at (0,5.7) [var] (variable4) {};
	\draw[testfcn] (left) to (root);
	\draw[kernel, bend right=60] (left1) to   (root);
	\draw[kernel] (left2) to  (left1);
	\draw[keps] (variable1) to  (left1); 
	\draw[keps] (variable1) to   (left); 
	\draw[keps] (variable3) to   (left2); 
	\draw[keps] (variable4) to   (left2);
\end{tikzpicture}
\; \qquad \text{and} \qquad
\tikzsetnextfilename{phi34a3}
\begin{tikzpicture}[scale=0.35,baseline=0.8cm]
	\node at (0.0,-0.8)  [root] (root) {};
	\node at (-2,1)  [dot] (left) {};
	\node at (0,3)  [dot] (left1) {};
	\node at (-2,5)  [dot] (left2) {};
	\node at (0,5) [var] (variable1) {};
	\node at (-2,3) [dot] (variable3) {};
	\node at (-0.5,6) [var] (variable4) {};
	\draw[testfcn] (left) to (root);
	\draw[kernel1] (left1) to   (left);
	\draw[kernel] (left2) to  (left1);
	\draw[keps] (variable3) to  (left); 
	\draw[keps] (variable1) to   (left1); 
	\draw[keps] (variable3) to   (left2); 
	\draw[keps] (variable4) to   (left2);
\end{tikzpicture} \; ,
\end{equation}
while the remaining terms can again be bounded using the estimates in Section~\ref{sec:Kernels with prescribed singularities}
The bounds on $$\bigl( \hat{\Pi}_\star^\epsilon \<2'>\bigr)(\phi_\star^\lambda), \qquad \bigl( \hat{\Pi}_\star^\epsilon \<21'>\bigr)(\phi_\star^\lambda) , \qquad \bigl( \hat{\Pi}_\star^\epsilon \<2'1>\bigr)(\phi_\star^\lambda),
$$ follow straightforwardly, where for the last term we use that admissible realisation satisfies $R(j_\cdot 1)=1$.
\end{proof}
\begin{remark}
Let us observe that the trees appearing for the $\phi^3_4$ equation are formally identical to those for the KPZ equation and the power-counting is also similar.
\end{remark}

\section{Kernels with Prescribed Singularities}\label{sec:Kernels estimates}
Throughout this section we fix vector bundles  $(E^{i},\nabla^E, \langle\cdot, \cdot\rangle_E)$ and $(F^{i},\nabla^F, \langle\cdot, \cdot\rangle_F)$ over $M$ for $i\in \mathbb{N}$. 
A smooth kernel $K\in \cC^\infty(F\hotimes E^* \setminus \triangle)$, where $\triangle\subset M\times M$ denotes the diagonal\footnote{
Throughout this section, we shall use the lighter notation 
$F\hotimes E^* \setminus \triangle$ to denote the restricted bundle with base space $M\times M\setminus \triangle$.
},
 is called of order $\zeta$ if for any for $m \in \mathbb{N}$ 
$$\vertiii{K}_{\zeta; m} := \sup_{k\leq m} \sup_{p,q\in M} d_\fraks (p,q)^{k-\zeta} |j_{p,q} K|_k <+\infty \ .$$

\subsection{Elementary estimates}\label{sec:Kernels with prescribed singularities}
In this section, we collect the analogue estimates to \cite[Section~10.3]{Hai14}, the proofs of which adapt mutatis mutandis.
\begin{lemma}
For $i=1,2$ let $K_{i}\in \cC^\infty(F^{(i)}\hotimes (E^{(i)})^* \setminus\triangle)$ be supported on $\{(p,q)\in (M\times M)\setminus \triangle\   | \ d_\fraks(p,q)<1\}$
and of order $\zeta_i$.
Then, up to the canonical isomorphism $\left(F^{(1)}\hotimes (E^{(1)})^*\right) \otimes\left( F^{(1)}\hotimes (E^{(1)})^*\right)\simeq (F^{(1)}\otimes F^{(2)})\hotimes (E^{(1)}\otimes E^{(2)} )^* \ , $
$$ K_1\otimes K_{2}\in \cC^\infty\left( (F^{(1)}\otimes F^{(2)})\hotimes (E^{(1)}\otimes E^{(2)} )^* \setminus \triangle\right)\ .$$  Furthermore,  there exists $C>0$ such that 
$$\vertiii{K_1\otimes K_2}_{\zeta_1+\zeta_2; m}\leq C\vertiii{K_1}_{\zeta_1; m}\vertiii{K_1}_{\zeta_2; m}\ .$$
\end{lemma}
%

\begin{lemma}
For $i=1,2$ let $K_{ i+1,i}\in \cC^\infty(E^{(i+1)}\hotimes (E^{(i)})^* \setminus\triangle)$ be supported on $\{(p,q)\in (M\times M)\setminus \triangle\   | \ d_\fraks(p,q)<1\}$ and of order $\zeta_i$.
If $\zeta_{1}\wedge\zeta_2> -|\fraks|$, we write
$$K_{3,2} * K_{2,1} (p,q):= \int_M   K_{3,2}(p,z) K_{2,1}(z,q) \dVol_z \ ,$$
where $K_{3,2}(p,z) K_{2,1}(z,q) \in \left( E^{(3)}\hotimes (E^{(1)})^* \right)|_{p,q}$ is obtained 
by contracting the factors in $(E^{(2)})^*$ and $E^{(2)}$.
Then,
$K_{3,2} * K_{2,1} \in \cC^\infty\left( E^{(3)}\hotimes (E^{(1)})^* \setminus\triangle \right)$.
If furthermore $\bar \zeta :=  \zeta_{1}+ \zeta_2 + |\fraks| < 0$, then
$$\vertiii{K_{3,2} * K_{2,1}}_{\bar{\zeta}; m}\leq C\vertiii{K_1}_{\zeta_1; m}\vertiii{K_1}_{\zeta_2; m}\ .$$
If $\bar{\zeta}\in \mathbb{R}_+\setminus \mathbb{N}$, then 
$K_{3,2} * K_{2,1} \in \cC^{[\bar \zeta ]}(E^{(3)}\hotimes (E^{(1)})^* \setminus\triangle)$ and
$$K_{3,1}(p,q):=  K_{3,2} * K_{2,1}(p,q) - R \left(   j^{[\bar \zeta ]}_q K_{3,1}(\cdot, q ) \right) (p)$$
satisfies 
$$\vertiii{K_{3,1}}_{\bar{\zeta}; m}\leq C\vertiii{K_1}_{\zeta_1; \bar{m}}\vertiii{K_1}_{\zeta_2; \bar{m}}\ ,$$
for $\bar{m}= m \vee ([\bar{\zeta}] + \max_{i} \fraks_i)$ for any admissible realisation $R$. 
\end{lemma}

\begin{definition}
Let $-|\fraks|-1<\zeta \leq |\fraks|$ and $K\in \cC^\infty(F\hotimes E^* \setminus \triangle)$ of order $\zeta$ 
such that $\supp(K(p,\cdot) )\subset B_{r_p}(p)$. 
We define the renormalised operator $\mathcal{R}K: \cC^\infty(E) \to \cC^\infty(F)$ by 
$$\mathcal{R}K (\varphi)= \int_M K(p,z) \left(\phi(z)- \tau_{z,p}\phi(p)\right) \dVol_z \ .$$
\end{definition}
\begin{lemma}
For $i=1,2$ let $K_{ i+1,i}\in \cC^\infty(E^{(i+1)}\hotimes (E^{(i)})^* \setminus\triangle)$ be supported on $\{(p,q)\in (M\times M)\setminus \triangle\   | \ d_\fraks(p,q)<1\}$ and of order $\zeta_i$.
If $-|\fraks|-1<\zeta_{2}\leq -|\fraks|$ and  $-2|\fraks|-\zeta_2 <\zeta_1<0$, then
$$\mathcal{R}K_{3,2} * K_{2,1} \in \cC^\infty\left( E^{(3)}\hotimes (E^{(1)})^* \setminus \triangle \right)\ ,$$
satisfies for $\bar{\zeta}= 0\wedge (\zeta_1+\zeta_2+|\fraks|)$,
$$\vertiii{\mathcal{R}K_{3,2} * K_{2,1}}_{\bar{\zeta}; m}\leq C\vertiii{K_1}_{\zeta_1; m}\vertiii{K_1}_{\zeta_2; \bar{m}}\ ,$$
for every $m \geq 0$ and 
$\bar{m} = m + \max_i \fraks_i$.
\end{lemma}

\begin{lemma}
Let $K\in \cC^\infty(F\hotimes E^* \setminus \triangle)$ be supported on $\{(p,q)\in (M\times M)\setminus \triangle\   | \ d_\fraks(p,q)<1\}$
and of order $\zeta$. 
If $M=\bar{M}$ has constant scaling, set $\rho_\epsilon= G_{\epsilon^2}$, while if $M=\RR\times \bar{M}$ has scaling $\fraks=(2, 1)$ let $\rho_\epsilon$ be given as in \eqref{eq:space_time mollified kernel}. 
Then,
 $K_\epsilon:= K * \rho_\epsilon$ has bounded derivatives of all orders and
$$|j_{p,q} K_\epsilon|_k\leq C (d_\fraks(p,q)+\epsilon)^{\zeta-k} \vertiii{K}_{\zeta, k} \ .$$
Finally, for $\bar{\zeta}\in [\zeta-1, {\zeta})$ and $m\geq 0$, one has the bound 
$$\vertiii{K-K_\epsilon}_{\bar{\zeta}, m}\lesssim \epsilon^{\zeta-\bar{\zeta}} \vertiii{K}_{\zeta,\bar{m}} \ ,$$
where $\bar{m}= n + \max_i \fraks_i$. 
\end{lemma}

\begin{lemma}
Let $K\in \cC^\infty(F\hotimes E^* \setminus \triangle)$ be supported on $\{(p,q)\in (M\times M)\setminus \triangle\   | \ d_\fraks(p,q)<1\}$
and of order $\zeta\leq 0$. Then, for every $\alpha\in [0,1]$, one has the bound
$$|K( z_1,z)-\tau_{z_1, z_2}K(z_2,z)| \lesssim d(z_1,z_2)^\alpha ( d(z,z_1)^{\zeta-\alpha}+ d(z,z_2)^{\zeta-\alpha} ) \vertiii{K}_{\zeta,m}, $$
where $m= \sup_i \fraks_i$.
\end{lemma}

\subsection{A Hairer--Quastel type criterion}
In this section we state a non-translation invariant weaker version of \cite[Theorem~A.3]{HQ18} which is sufficient to estimate the terms in \eqref{eq:for HQ-1} and \eqref{eq:for HQ-2}. 
Let us first recall the setting: We are given a finite directed multigraph $\CCG= (\CCV, \CCE)$ with edges $e \in \CCE$ labelled by $(a_e, r_e)\in \mathbb{R}_+\times \{0,1\}$, and kernels $K_e: M\times M\setminus \triangle \to \mathbb{R}$ compactly supported on $\{(p,q)\in (M\times M)\setminus \triangle\   | \ d_\fraks(p,q)<1\}$ and satisfying $\vertiii{K_e}_{a_e; 2}<\infty .$
We assume $\CCG$ contains $M\geq 0$ distinguished edges\footnote{The green arrows in Section~\ref{sec:concrete applications}.} $e_{\star,1},\ldots,e_{\star,M}$ connecting $\star \in \CCV$ distinguished vertices $v_{\star,1},\ldots,v_{\star,M}$. We write $\CCV_\star= \{ \star,  v_{\star,1},\ldots,v_{\star,M}\}$ 
$\CCV_0= \CCV\setminus \{\star \}$.
For $e\in E$, we write $(e_-,e_+)$ for the pair of vertices such that that $e$ is directed from the vertex $e_-$ to $e_+$. 
For simplicity\footnote{This is not needed for the argument but generally satisfied for diagrams appearing in regularity structures.}
assume that for any pair of vertices 
$v_1, v_2$ there exists at most one edge connecting them such that $r_e=1$.

Then, we define a new fully connected directed graph $\hat \CCG =(\CCV , \hat \CCE)$ where for each edge $e\in \hat \CCE$ the label
$(\hat a_e, r_e)$ is the sum of the labels over all edges directed from $e_-$ to $e_+$ (with the empty sum being understood to be $0$ and $K_e=1$ for all such newly added edges).

For a subset $\bar\CCV \subset \CCV$ we call 
\begin{itemize}
\item $\CCE^\uparrow (\bar \CCV)  =  
 \{ e\in \CCE : e\cap \bar\CCV = e_-\}$ outgoing edges,
 \item $\CCE^\downarrow (\bar \CCV)  =   \{ e\in \CCE : e\cap \bar\CCV = e_+\}$ incoming edges,
 \item $\CCE_0 (\bar \CCV)  =   \{ e\in \CCE : e\cap \bar\CCV = e\}$ internal edges,
 \item $\CCE (\bar \CCV)  =   \{ e\in \CCE : e\cap \bar\CCV \neq \emptyset\}$ incident edges
\end{itemize}
and write
$\CCE_+ (\bar \CCV)  =   \{ e\in \CCE(\bar \CCV) :  r_e>0\}$, $\CCE_+^\uparrow = \CCE_+ \cap \CCE^\uparrow$ and $\CCE_+^\downarrow = \CCE_+ \cap \CCE^\downarrow$. 

\begin{assumption}\label{ass:mainGraph}
The resulting directed graph $(\CCV,\hat\CCE)$ with labels $(\hat a_e,r_e)$ satisfies
$r_e = 0$ for all edges connected to $\star$ or connecting two elements of $\CCV_\star$. Furthermore assume the following.
\begin{enumerate}
\item For every subset $\bar \CCV \subset \CCV_0$ of cardinality at least $2$,  
\begin{equation}\label{e:assEdges}
\sum_{e  \in \hat\CCE_0 (\bar \CCV)} \hat a_e < (|\bar \CCV| - 1)|\fraks|\;;
\end{equation}
\item For every subset $\bar \CCV \subset \CCV$ containing $0$ of cardinality at least $2$,  
\begin{equation}\label{e:assEdges2}
\sum_{e  \in \hat \CCE_0 (\bar \CCV) } \hat a_e 
+\sum_{e  \in \hat\CCE^{\uparrow}_+ (\bar \CCV)  } (\hat a_e+ r_e -1)  - \sum_{e  \in \hat\CCE^{\downarrow}_+ (\bar \CCV)  }  r_e< (|\bar \CCV| - 1)|\fraks|\;;
\end{equation}
\item  For every non-empty subset $\bar \CCV \subset \CCV\setminus \CCV_\star$,
\begin{equation}\label{e:assEdges3}
 \sum_{e\in \hat\CCE(\bar\CCV)\setminus \hat\CCE^{\downarrow}_+(\bar \CCV) }  \hat a_e 
 +\sum_{e\in \hat \CCE^{\uparrow}_+(\bar\CCV)}  r_e
- \sum_{e \in \hat \CCE^\downarrow_+(\bar \CCV)} (r_e-1)
> |\bar \CCV||\fraks|\;.
\end{equation}
\end{enumerate}
\end{assumption}
Next, for edges $e\in \CCE$ such that $r_e=0$ set $\hat{K}_e(p,q)={K}_e(p,q)$ while for edges $e\in \CCE$ such that $r_e=1$ set
$\hat{K}_e(p,q)={K}_e(p,q)-K_e(\star. q)$ and let
$$\mathcal{I}^{\CCG}(\phi^\lambda_\star, K):= \int_{M^{\CCV_0}}  \prod_{e\in \CCE} \hat{K}_e(x_{e_+},x_{e_-})\prod_{i=1}^M \phi^\lambda ({x_{v_{\star,i}}}) \dVol_x \ .$$

\begin{theorem}[Theorem~A.3 \cite{HQ18}]\label{thm:HQ}
 Let $\CCG=(\CCV,\CCE)$ be a finite directed multigraph with labels $\{a_e,r_e\}_{e\in \CCE}$ and kernels
$\{K_e\}_{e\in \CCE}$ with  the resulting graph satisfying Assumption~\ref{ass:mainGraph} and its preamble.
Then, there exist  $C<\infty$ depending only on the 
structure of the graph $(\CCV,\CCE)$ and the labels $r_e$ such that 
\begin{equation}\label{genconvBound}
{I}^\CCG(\phi_\lambda,K) 
\le C\lambda^{\tilde \alpha}\prod_{e\in \CCE} \|K_e\|_{a_e;2}\;,
\end{equation}
for $0<\lambda\le 1$, where \begin{equation}
\tilde \alpha = |\fraks||\CCV\setminus \CCV_\star| - \sum_{e\in \CCE} a_e. 
\end{equation}
\end{theorem}
This theorem is a simpler version of the one in \cite{HQ18}, so we shall only outline the (minor) changes necessary to adapt it to manifolds.
One first introduces $\phi: \mathbb{R}\to [0,1]$ supported on $[3/8,1]$ such that for $\Phi^{(n)}(s)= \varphi(2^n s)$ one has for every $s>0$
$$\sum_{n\in \mathbb{Z}} \varphi^{(n)} (s)=1\ .$$
Then, for $\mathbf{n}\in \mathbb{N}^3$ and an edges $e\in \CCE$ such that $r_e=0$ set
$$\hat{K}^{k,l,m}_e (p,q) = 
\begin{cases}  
\varphi^{(k)}(d(p,q))  K_e (p,q)  & \text{if }  l=m=0 \ , \\
0 & \text{else. } 
\end{cases}
$$
One observes that the analogue of \cite[Lemma~A.4]{HQ18} holds.
If $r_e=1$, set 
$$\hat{K}^{k,l,m}_e (p,q) = \varphi^{(k)}(d(p,q)) \varphi^{(l)}(d(\star, p)) \varphi^{(m)} (d(\star, q)) \left( K_e (p,q)- K(p, \star)\right) \ . $$ 
After possibly enhancing the edge set $\CCE$ to include for any $(v,w)\in \CCV^{2}$ an edge $e$ from $v$ to $w$ and setting $K_e=1$ for such edges, define for $\mathbf{n}: \CCE \to \mathbb{N}^3$
$$ \hat{K}^{(n)}(x)= \prod_{e\in \CCE} \hat{K}^{(n_e)}_e (x_{e_-}, x_{e_+}) \ .$$ One observes
that if we set $\mathcal{N}_\lambda=\big\{\mathbf{n}: \CCE \to \mathbb{N}^3 \ : \ 2^{-{\mathbf{n}_{e_{\star, i}} }}  \leq \lambda\ \big\} $
it suffices to show that the absolute value of
$$
I^\mathbb{G}_{\lambda}(K) = \sum_{\mathbf{n}\in \mathcal{N}_\lambda} \int_{M^{\CCV_0}} \hat{K}^{(\mathbf{n})}(x) \dVol_x$$
 is bounded by the right-hand side of \eqref{genconvBound}, see \cite[Lemma~A.7]{HQ18}.
 
At this point, we can apply the multiscale clustering argument in \cite{HQ18} essentially \textit{ad verbatim}, even with the further simplification that in our case the naive bound of \cite[Section~A.4]{HQ18} is sufficient. Let us therefore only recall the main cornerstones of that argument.
To every configuration $x\in M^{\CCV}$ one associates a labelled binary rooted tree $(T,l)$ where
the leave set $L_T$ of $T$ is given by $\CCV$ and 
the inner nodes $T^\circ= N_T\setminus L_T$ have a labels 
$$l:T^\circ \to \mathbb{N}$$
satisfying the following properties.
\begin{enumerate}
\item The labels are order preserving, i.e.\ for inner nodes $\nu, \omega\in T^\circ$,
$$\nu\geq \omega \Rightarrow l(\nu)\geq l(\omega) \ .$$
\item For every $v,w\in \CCV$, one has 
$$2^{-l(v\wedge w)}\leq d_\fraks(x_v,x_w) \leq |\CCV| 2^{-l(v \wedge w)} \ ,$$
where $v\wedge w\in T^\circ$ denotes the largest common ancestor of $v$ and $w$.
\end{enumerate}
%
After defining $\mathcal{N}(T,l)$ and $\CCT_\lambda(\CCV)$ as in \cite[Definition A.8]{HQ18} and right thereafter,
it follows by the same argument as for \cite[Lemma~A.13]{HQ18} that 
\begin{equation}
|{I}_\lambda^\CCG (K)| \lesssim \sum_{(T,\ell) \in \CCT_\lambda(\CCV)} \sum_{n \in \mathcal{N}(T,\ell)} \Bigl|  \int_{M^{\CCV}} \hat K^{(\mathbf{n} )}(x)\,\dVol_x \Bigr|\;.
\end{equation}
Next, it follows as in \cite[Lemma~A.15]{HQ18}
that for  $\eta\colon T^\circ \to \RR$ given by
$\eta(v) =|\fraks|+ \sum_{e \in \hat\CCE} \eta_e(v)$, where
\begin{align*}
\eta_e(v)   &= - \hat{a}_e \mathbf{1}_{e_\uparrow} (v)
+ r_e  (\mathbf{1}_{e_+\wedge 0} (v) 
- \mathbf{1}_{e_\uparrow} (v)\bigr) \mathbf{1}_{r_e >0, e_+\wedge 0 > e_\uparrow} \\
&\quad + (1-r_e -  \hat{a}_e) (\mathbf{1}_{e_-\wedge 0} (v)
- \mathbf{1}_{e_\uparrow} (v)\bigr) \mathbf{1}_{r_e >0, e_-\wedge 0 > e_\uparrow}\;,
\end{align*}
with $\mathbf{1}_v(v) = 1$ and $\mathbf{1}_v(w) = 0$ for $w \neq v$, one has the bound
\begin{equation}\label{naive bound}
\Bigl(\prod_{v \in T^\circ} 2^{-\ell_v |\fraks|}\Bigr)\sup_x |\hat K^{(\mathbf{n})}(x)| \lesssim \prod_{v \in T^\circ} 2^{-\ell_v \eta(v)}\; .
\end{equation}
Finally, combining \eqref{naive bound} with \cite[Lemma~A.10 \& Lemma~A.16]{HQ18} concludes the argument. Note that we do not consider any negative renormalisation here and therefore we can omit the considerations of \cite[Section~A.5]{HQ18}.
	\bibliographystyle{Martin}
		\bibliography{./Manifolds.bib}

\end{document}